\documentclass[preprint,10pt]{elsarticle}

\usepackage{graphics}
\usepackage{tikz,pgf}
\usetikzlibrary{decorations.markings, decorations.pathmorphing}
\usetikzlibrary{arrows,decorations,backgrounds,scopes,plotmarks}
\usepackage{subfigure}

\usepackage[english]{babel}
\usepackage[all]{xy}
\usepackage{marvosym}
\usepackage{stmaryrd,mathrsfs,amsmath,amssymb,bm,amsfonts}
\usepackage{amsmath}
\usepackage{amsmath,amsthm,amssymb,amscd}
\newtheorem{thm}{Theorem}[subsection]
\newtheorem{cor}[thm]{Corollary}
\newtheorem{lem}[thm]{Lemma}
\newtheorem{prop}[thm]{Proposition}
\newtheorem{defn}[thm]{Definition}
\newtheorem{rem}[thm]{Remark}
\newtheorem{ex}[thm]{Example}
\numberwithin{equation}{subsection}

\begin{document}

\begin{frontmatter}



\author[a,b]{Sen Hu}
\ead[a]{shu@ustc.edu.cn}

\author[a]{Xuexing Lu}
\ead[a]{xxlu@mail.ustc.edu.cn, xuexinglu@gmail.com}

\author[a,b]{Yu Ye}
\ead[a,b]{yeyu@ustc.edu.cn}

\title{Combinatorics and algebra of tensor calculus}

\address[a]{Department of Mathematics,University of Science and Technology of China,PR China}
\address[b]{Wu Wen-Tsun key laboratory of Mathematics,PR China}

\author{}

\address{}

\begin{abstract}
In this paper, motivated by the theory of operads and PROPs we reveal the combinatorial nature of tensor calculus for strict tensor categories and show that there exists a monad which is described by the coarse-graining of graphs and characterizes the algebraic nature of tensor calculus.
More concretely, what we have done are listed in the following:

1. We give a combinatorial formulation of a progressive plane graph introduced by Joyal and Street and some of their properties are investigated.

2. We introduce the category $\mathbf{T.Sch}$ of tensor schemes to make the construct of a free strict tensor category a functor $F:\mathbf{T.Sch}\rightarrow\mathbf{Str.T}$ from the category $\mathbf{T.Sch}$  of tensor schemes to the category $\mathbf{Str.T}$ of strict tensor categories. We also construct a right adjunction $U:\mathbf{Str.T}\rightarrow \mathbf{T.sch}$ of $F.$

3. We analysis the associated monad of the adjunction which is named as the monad of tensor calculus and show that it is described by the coarse-graining of  graphs. A algebra of this monad is named as a tensor manifold.

4. Identity morphisms in a tensor manifold and several operations on a tensor manifold such as tensor product, composition and fusion are introduced. We show that they satisfy some natural compatible conditions. We also show that under these compatible conditions the appointment of identity morphisms and these operations can totally characterize the algebraic structure of a tensor manifold.

5. We construct a functor $\Psi:\mathbf{T.Sch}^T\rightarrow\mathbf{Str.T}$ from the category $\mathbf{T.Sch}^T$ of tensor manifolds to the category of strict tensor categories, and we show that $\Psi$ is a left inverse of the natural comparison functor $\Phi:\mathbf{Str.T}\rightarrow\mathbf{T.Sch}^T$ which means that $\Phi$ is an embedding.  We also show that the adjunction $F,U$ is not monadic, hence we can interpret a strict tensor category as a special kind of tensor manifold.

6. We prove that $\Psi$ is also a left adjunction of the comparison functor $\Phi$.

\end{abstract}


\end{frontmatter}
\tableofcontents

\section{Introduction}

As the title shows, the theme of this paper has two parts. One is to introduce a combinatorial language formulating the mathematical theory of tensor calculus established by Joyal and Street in the famous paper \cite{[JS91]}. The other one is to excavate the monad of tensor calculus which is just hidden behind the work of Joyal and Street. The first part reveals the combinatorial nature of tensor calculus, and in the combinatorial language the monad of tensor calculus is described by the coarse-graining of planar graphs. But a very interesting fact is that the construction of coarse-graining appears naturally in a wide range of theories in physics such as the theories of topological order \cite{[LW04],[LW05],[CGW10]} and quantum order \cite{[Vi06]}, the theory of loop quantum gravity \cite{[AKRZ14]}. This fact has strongly motivated us to conjecture a big picture called a theory of tensor geometry.

\begin{center}
    \textbf{Combinatorial formulation of progressive plane graphs}
\end{center}

In \cite{[JS91]}, Joyal and Street proved a firm theoretical foundation of graphical calculus for strict tensor categories. In the first chapter of \cite{[JS91]}, they introduced  several important  notions  such as (leveled/boxed) progressive plane graph, tensor scheme, diagram in a strict tensor category, diagram in a tensor scheme, evaluation of diagrams in a strict tensor category, free strict tensor category, etc, which are essential for formulating graphical calculus as a mathematical theory.

In this paper, we will recall these notions in a combinatorial language and  prove a combinatorial theory of tensor calculus for strict tensor categories. The combinatorial definition of a graph is well-known and has been used for many years in the theory of operads and PROPs \cite{[Mar06], [LV12],[MSS02]}. The key difficult in the theory is to find a simple and coherent way to characterize the progressive planar structure on a progressive plane graph.

The notion of a (combinatorial) progressive planar graph  is introduced in the beginning of section $3.1$.  We think that it is a right combinatorial formulation of a progressive plane graph in the sense that:

$\bullet$ By definition a progressive planar structure is a linear extension of a partial order with some properties, so we can see that for any two isomorphic progressive planar graphs there is an unique isomorphism between them, that is, the groupoid of progressive planar graphs is a clique.

$\bullet$ We have proved that there is a natural (preserving tensor products and compositions) bijection between the set of progressive planar graphs and the set of progressive plane graphs(section $3.6$).  We call the natural bijection this planar geometric realization of progressive planar graphs.
$$\xymatrix{progressive\ planar\ graphs\ar[rrrrr]^{planar\ geometric\ realization}_{\cong}&&&&&progressive\ plane\ graphs}$$

A direct consequence of this fact should be that  the combinatorial theory  affirms the conjecture of P.Selinger in \cite{[Se09]} that the "recumbent" planar isotopy of progressive plane graphs can be generalized to arbitrary planar isotopy.  As a combinatorial date, a progressive planar graph can be stored in a computer, so the combinatorial theory with these advantages may provides us an effective tool to do tensor calculus by a computer.

\begin{center}
    \textbf{What is tensor geometry}
\end{center}

In fact, there is a deeper theory behind the work of Joyal and Street, namely, a theory of tensor geometry. In our opinion, this new geometry is a kind of categorical  noncommutative geometry or a nonpertubative generalization of (non-symmetric) operads and PROPs. In this new geometry, we should replace classical objects such as groups and  associative algebras by "quantum object" such as PRO (non-symmetric version of PROP \cite{[Mar06], [LV12],[MSS02]}) and \textbf{tensor manifolds} (section $6.1$). The notion of a tensor manifold is a natural generalization of both a PRO and a strict tensor category which is seen as the most general setting making the notion of a tensor and their calculus meaningful \cite{[JS91]}. We hope in the future to introduce the notions of  structured tensor manifolds which may generalize symmetric, braided, rigid categories, etc.

More precisely, the construct of a free strict tensor category is in fact a functor $$F:\mathbf{T.Sch}\rightarrow \mathbf{Str.T} $$ from the category $\mathbf{T.Sch}$ of tensor schemes to the category $\mathbf{Str.T}$ of strict tensor categories, and in sections $5.2$, $5.3$ we will show that there exists a right adjoint functor of $F$ $$U:\mathbf{Str.T}\rightarrow \mathbf{T.Sch},$$ which is essentially given by the construction of space of prime diagrams in a strict tensor category and is a resolution of strict tensor categories in some sense.

By the general theory of adjunctions \cite{[Mac71]}, this adjunction will give rise to a monad $(T,\mu,\eta)$ in the category $\mathbf{T.Sch}$ of tensor schemes and we call it the monad of tensor calculus. A tensor manifold is by definition an algebra of this monad. All tensor manifolds and their morphisms form a category denoted by $\mathbf{T.Sch}^{T}$ and the theory of tensor geometry is the study of this category or its structured version.

According to the general theory of monads \cite{[Mac71]}, there should be an adjunction naturally associated to $(T,\mu,\eta)$
$$(F^T,U^T, \varepsilon^T, \eta^T):\mathbf{T.Sch}\rightharpoonup \mathbf{T.Sch}^T,$$
and there is an unique comparison functor
$$\Phi:\mathbf{Str.T}\rightarrow \mathbf{T.Sch}^{T},$$
such that $F^T=\Phi\circ F$ and $U=U^T\circ \Phi$, that is, in the following diagram both the $F$-square and the $U$-square commute
$$\xymatrix{\mathbf{Str.T}\ar@{-->}[rr]^{\Phi}\ar@<0.8mm>[d]^{U}&&\ar@<0.8mm>[d]^{U^T}\mathbf{T.Sch}^{T}\\\mathbf{T.Sch}\ar@<0.8mm>[u]^{F}\ar@{=}[rr]&&\ar@<0.8mm>[u]^{F^T}\mathbf{T.Sch}.}$$
The functor $\Phi$ sends every strict tensor category  to its associated tensor manifold. A left inverse of $\Phi$ $$\Psi:\mathbf{T.Sch}^{T}\rightarrow\mathbf{Str.T}$$  can be constructed. We will prove that the comparison functor is not an equivalence of categories, but an embedding (section $6.3$). Thus strict tensor categories are just special kinds of tensor manifolds (which are called \textbf{critical} tensor manifolds in section $6.4$) and tensor manifolds are nontrivial generalization of strict tensor categories. Moreover, we show that $\Psi$ is left adjoint to $\Phi$ (Theorem $6.4.6$).

\begin{center}
   \textbf{ Motivations}
\end{center}
A strong belief in our work is that strict tensor categories may not be the most general setting for defining algebraic structures on  tensor systems (= tensor schemes, for definition see section 4.1). This belief mainly comes from our motivations to the field of tensor calculus.

The theory of tensor manifolds was strongly motivated by the theory of  operads and PROPs \cite{[Mar06], [LV12]}, and has some differences in the following aspects:

$\bullet$ operads, PROPs and their colored version are free on the level of objects, while tensor manifolds can be non-free on the level of objects.

$\bullet$ The monads controlling operads/PROPs and tensor manifolds are different slightly. Monads characterizing algebraic structures on a tensor systems of operadic type are described by the construction of substitution of graphs \cite{[Mar06]}. While the monad characterizing a tensor manifold is described by fine-graining of graphs (section $5.6$) and a similar construction also appears in recent work in loop quantum gravity \cite{[AKRZ14]}.

Other motivations to the theory of tensor geometry come from our several interests. We categorize them into three related classes. The first is our interest in formal structures of non-perturbative quantum field theory, aiming to answer the question that: are there some formal mathematical structures to help us to unify string theory, loop quantum gravity \cite{[Ro97], [Sm97]} and string-net condensation \cite{[LW04],[LW05]}? The idea is that the formal structure of a tensor manifold is very similar to that of sigma models in perturbative string theory in the sense that in perturbative string theory we consider the space of fields  $$world\ sheet\rightarrow background\  space$$ which are maps from world-sheets to a background space, while in theory of tensor geometry we consider the space of spin-networks  $$planar\ graph\rightarrow tensor\ manifold$$ which can be seen as maps from planar graphs ("Feynman graphs") to a tensor manifold. The action in string theory is determined by the geometry structure of and other fields on the background space while in the theory of tensor manifold we just have an algebraic structure of a monad. The two theories are different but this "anolog" motivates us very much.

The second is trying to find a cohomology theory of strict tensor categories generalizing that of operads and props \cite{[Mar94],[Mar96], [Mar06],[MV07]}. Our work shows that we can  interpret a strict tensor category as an algebra of the monad of tensor calculus, so there would be a Barr-Beck homology \cite{[BB69]} of a strict tensor category. We hope in the future we can come back to this issue.

The third is mainly concern with possible connections with various kinds of  renormalization theories, such as entanglement renormalization [Vi06] for quantum phase transition of  quantum lattice systems,  wave function renormalization \cite{[CGW10]} in the theory of topological phases and algebraic/integrable renormalization \cite{[CK00],[CK01]} in perturbative quantum field theory.

Although we can not predict the futures of all this three directions,  we emphasize that there is a common mathematical concept which should be treated as a basic one, that is, spin networks \cite{[Ma99],[Ba94]} and the different directions just interpret them in different ways, such as morphisms in tensor categories, quantum entanglements, non-local excitation of gauge fields, quanta of geometry, histories or circuits of quantum computation \cite{[BS09]}. So it is reasonable to hope that a theory of tensor manifolds can provide a helpful insight to the "big unification".

Many constructions in algebraic geometry, super geometry and  noncommutative geometry may be generalized in the setting of tensor geometry, such as functor of points, odd functor of points, etc. We think that these generalizations are deserved to investigate and some of them will be connected to Feynman transformations in theory of operads and props \cite{[KWZ12]}, where "odd" point in super algebraic geometry is replaced by a family of "odd" Feynman graphs in super tensor geometry. In our opinion, the works of D.N.Yetter \cite{[Yet03]}, A.A.Davydov \cite{[Dav97]} and  T.Maszczyk \cite{[Mas06],[Mas11]} may suggest possible connections between classical commutative/noncommutative geometry with tensor geometry. There are also many people  whose works have impacted on our ideas greatly, so we list some of their articles, such as \cite{ [Co04],[Fi02], [FI05],[FM01],[GK94], [Io00], [Io07],[K93], [K94],[Ko11], [Ko13], [Lo12], [QZ11]}.\\
\\

The paper is organized as follows.

In section $2$, we introduce the notions of a pre-graph, a graph and their geometric realizations. Pre-graphs and their geometric realizations are discussed in section $2.1$. Graphs and their geometric realizations are discussed in section $2.2$. Some operations such as tensor product, grafting, substitution and coarse-graining of graphs are introduced in section $2.3$.  In section $2.4$, we introduce some structures on graphs such as orientation, polarization and anchored structures.

Section $3$ is devoted to a combinatorial theory of progressive planar graphs (planar graphs for short). In section $3.1$, we give the definition of a progressive planar graph and its several equivalent descriptions. In section $3.2$, we show some properties of a progressive planar graph, especially we prove that there is at most one compatible planar structure on a progressive, polarized and anchored graph. In section $3.3$, we show that every planar structure on a progressive graph can induces a planar structure on its set of vertices. The tensor product of planar graphs are discussed in section $3.4$. In section $3.5$, we show that on the set of planar graphs there is a well-defined composition for two composable planar graphs which is associative. Moreover, we show every planar graph can be decomposed as a composition of several essential prime planar graphs.  In section $3.6$, we introduce the planar geometric realization of a progressive planar graph and show that it provides  a canononical (preserving tensor product and composition) isomorphism between the set $\mathsf{\Gamma}$ of isomorphic classes of our progressive planar graphs and the set $\mathsf{G}$ of isomorphic classes of Joyal and Street's progressive plane graphs. The notions of a fissus planar graph and its coarse-graining are introduced in section $3.7$.

In section $4.1$, we introduce the category of tensor schemes. The notions of a diagram in a tensor scheme and a diagram in a strict tensor category are introduced in section $4.2$ and $4.3$, respectively. In section $4.4$, we prove that there is a well-defined value of a diagram in a strict tensor category.  The notions of a compound fissus planar graph and its coarse-graining are introduced in section $4.5$.

In section $5$, we introduce two functors $F:\mathbf{T.Sch}\rightarrow \mathbf{Str.T}$, $U: \mathbf{Str.T}\rightarrow \mathbf{T.Sch}$ and prove that they form an adjunction. The functor $F$ is just the construction of a free strict tensor category generated by a tensor scheme which was introduced by Joyal and Street in \cite{[JS91]}. The functor $U$ is given essentially by the construction of prime diagrams in a strict tensor category. We also give a detailed description of the unit, counit, multiplication and comultiplication of this adjunction. The functor $F$ is introduced in section $5.1$ and the functor $U$ is introduced in section $5.2$. Section $5.3$ is devoted to prove the fact that $F$ and $U$ form an adjunction. In order to simplify the descriptions of functors $G=FU$ and $T=UF$, we introduce two functors $\mathbf{\Gamma}^{\otimes}$ and $\mathbf{\Gamma}_F$ in section $5.4$ and show that $G$ and $T$ are isomorphic to $\mathbf{\Gamma}$ and $\mathbf{\Gamma}_F$, respectively. The unit $\eta:I\rightarrow UF$ and counit $\varepsilon:FU\rightarrow I$ of tensor calculus are discussed in section $5.5$ and the multiplication $\mu:T\circ T\rightarrow T$ and comultiplication $\delta:G\rightarrow G\circ G$ are discussed in section $5.6$.

Section $6$ is devote to a theory of tensor manifolds. In section $6.1$, we give the definition of a tensor manifold. The category of tensor manifolds is introduced and we denote is as $\mathbf{T.Sch}^T.$ Some basic properties are also investigated.  Some operations on a tensor manifold is introduced in section $6.2$, such as tensor product, composition, fusions associated with a pair of linear partitions (defined in section $3.7$). Moreover, we show that in addition with a family of identity morphisms and some compatible conditions these operations determine the structure of a tensor manifold uniquely.  In section $6.3$, we construct a left inverse $\Psi$ of the comparison functor $\Phi:\mathbf{Str.T}\rightarrow \mathbf{T.Sch}^T$ and prove that $\Phi$ is not an equivalence of categories, hence the the adjunction of tensor calculus is not monadic.  In section $6.4$, we show that $\Psi$ is left adjoint to $\Phi$.

In section $7.1$, we review definitions about strict tensor categoies and  strict tensor functors. In section $7.2$, we review the definition of a progressive plane graph introduced in \cite{[JS91]} and related notions.  Section $7.3$ is devoted to a basic review of adjunctions and their associated (co)monads.

\section{Combinatorics of graphs}

\subsection{Pre-graphs}
In this section, we will introduce the notion of a pre-graph and its geometric realization. First recall that a partition of a non-empty set $X$ is a set of nonempty subsets of $X$ such that every element $x$ in $X$ is in exactly one of these subsets (page 28 in \cite{[Ha60]}) (i.e., $X$ is a disjoint union of the subsets). In this paper, we define the partition of an empty set to be itself.  The notion of a pre-graph is essentially the notion of a wheeled graph in \cite{[MMS06]}, but some terminologies are different from theirs in this paper.
\begin{defn}
A $($combinatorial$)$ \textbf{pre-graph} $\Gamma=(H,P,\sigma)$ consists of

$\bullet$ a finite  set $H$ whose elements are called \textbf{half-edges} or \textbf{flags} of $\Gamma$,

$\bullet$ a partition $P$ of $H$ whose blocks are called \textbf{vertices} of $\Gamma$,

$\bullet$ an involution $\sigma:H\rightarrow H$, whose orbits are called \textbf{edges} of $\Gamma$.
\end{defn}
The sets of half-edges, vertices and edges of $\Gamma$ are denoted by $H(\Gamma)$, $V(\Gamma)$ and $E(\Gamma)$, respectively. If $H(\Gamma)$ is an empty set, we call $\Gamma$ an \textbf{empty graph}. The partition $P$ defines a function  $\pi:H(\Gamma)\rightarrow V(\Gamma)$ which sends a half-edge to its block and also defines an equivalence relation $\overset{P}{\thicksim}$ on $H(\Gamma)$ such that $h_1\overset{P}{\thicksim} h_2$ if and only if $h_1, h_2$ belongs to a same block of $P$. We say that  a vertex $v\in V(\Gamma)$ has \textbf{valency} or \textbf{degree} $n$ if it consists of $n$ half-edges. A half-edge $h\in H(\Gamma)$ is \textbf{incident to} $v\in V(\Gamma)$ if $h\in v$, an edge $e$ is incident to $v$ if at least one of its half-edges is incident to $v$. The involution $\sigma$ defines another function from $H(\Gamma)$ to $E(\Gamma)$ which sends a half-edge to the edge which it belongs to. For any half-edge $h$,  we denote by $\overline{h}$ the edge which $h$ belongs to.

Following rules in page 8 of \cite{[GK94]}, we may associate to each pre-graph $\Gamma$ a finite CW complex
$$|\Gamma|=\frac{H(\Gamma)\times [0,\frac{1}{2}]}{\sim},$$
which is a quotient space of $H(\Gamma)\times [0,\frac{1}{2}]$ with respect to an equivalence relation $\sim$ defined as: $(h_1,l_1)\sim (h_2,l_2)$, for $h_1, h_2\in H(\Gamma)$ and $0\leq l_1,l_2\leq\frac{1}{2}$ if one of three conditions is satisfied: $(a)$ $h_1=h_2$ and $l_1=l_2$, $(b)$ $h_1\overset{P}{\thicksim}h_2$, $l_1=l_2=\frac{1}{2}$, $(c)$   $h_1=\sigma(h_2)$, $l_1=l_2=0$. That is, in topological space $H(\Gamma)\times [0,\frac{1}{2}]$, we identify $(h_1,\frac{1}{2})$, $(h_2,\frac{1}{2})$ if $h_1, h_2$ are in the same block, and identify $(h_1,0)$, $(h_2,0)$ if $h_1=\sigma(h_2)$.  The space  $|\Gamma|$ is called the \textbf{geometric realization} of $\Gamma$.  Now we give an example of a pre-graph and its geometric realization.

\begin{ex}\label{example}
Let $\Gamma$ be a pre-graph $(H,P,\sigma)$, where
\begin{align*}
&H=\{a,b,c,d,e,f,g,h,i\},\\
&P=\{\{a,b,c,d,e\},\{f,g,h,i,j\},\{k,l\}\},\\
&\sigma=(a)(b)(c)(df)(eg)(hi)(j)(kl).
\end{align*}
The set of edges is $E(\Gamma)=\{\{a\},\{b\},\{c\},\{d,f\},\{e,g\},\{h,i\},\{j\},\{k,l\}\}$ and
there are  three vertices $v_1=\{a,b,c,d,e\}$, $v_2=\{f,g,h,i,j\}$, $v_3=\{k,l\}$. Picturally, its geometric realization is as follows:

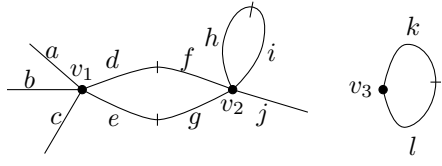
\begin{center}
\begin{tikzpicture}

\node [above] at (0,0) {$v_1$};
\draw[fill] (0,0) circle [radius=0.055];
\draw  (0,0)--(-0.7,0.61);
\node  at (-0.4,0.5) {$a$};
\draw (0,0)--(-1,0);
\node  at (-0.7,0.15) {$b$};
\draw (0,0)--(-0.5,-0.85);
\node  at (-0.35,-0.38) {$c$};

\draw plot[smooth]  coordinates {(0,0)  (1,0.3)   (2,0)};
\draw (1,0.22)--(1,0.38);
\node  at (0.4,0.41) {$d$}; \node  at (1.4,0.4) {$f$};
\draw plot[smooth]  coordinates {(0,0)  (1,-0.4)   (2,0)};
\draw (1,-0.32)--(1,-0.48);
\node  at (0.4,-0.41) {e}; \node  at (1.5,-0.41) {$g$};

\node [below ] at (2,0) {$v_2$};
\draw[fill] (2,0) circle [radius=0.055];

\draw (2,0)--(3,-0.3);
\node  at (2.4,-0.33) {$j$};

\draw plot[smooth, tension=1]  coordinates {(2,0)  (2.4,0.6) (2.25,1.1) (1.88,0.65)   (2,0)};
\draw (2.27,1.17)--(2.23,0.96);
\node [right] at (2.33,0.5) {$i$};
\node [left] at (1.95,0.7) {$h$};
\node [above] at (4.4,0.6) {$k$};
\draw[fill] (4,0) circle [radius=0.055];
\draw plot[smooth, tension=1]  coordinates {(4,0)  (4.3,-0.5) (4.7,0.1) (4.3,0.6)  (4,0)};

\node [below] at (4.4,-0.5) {$l$};
\draw (4.77,0.1)--(4.63,0.1);
\node [left] at (4,0) {$v_3$};
\end{tikzpicture}.
\end{center}
To make a vertex distinct, we put one $\bullet$ on the place of the vertex. We will usually make no distinction between a pre-graph and its geometric realization.
\end{ex}

\begin{defn}
Let $\Gamma_1=(H_1,P_1,\sigma_1)$ and $\Gamma_2=(H_2,P_2,\sigma_2)$ be two pre-graphs, an \textbf{equivalence} or \textbf{isomorphism} of them is a bijection $\phi:H_1\rightarrow H_2$ that preserves partitions and commutes  with involutions, that is,

$\bullet$ for every $h\in H_1$, then $\sigma_2(\phi(h))=\phi(\sigma_{1}(h)),$

$\bullet$ $h_1 \overset{P_1}{\thicksim} h_2$ if and only if $\phi(h)\overset{P_2}{\thicksim}\phi(h_2)$, for every $h_1,h_2\in H_1$,
where $\overset{P_1}{\thicksim}, \overset{P_2}{\thicksim}$ are equivalence relations defined by $P_1, P_2$ respectively.
\end{defn}
Obviously, any equivalence $\phi:H(\Gamma_1)\rightarrow H(\Gamma_2)$  of two pre-graphs naturally induces a homeomorphism $|\phi|:|\Gamma_1|\rightarrow |\Gamma_2|$  of their geometric realizations. It can be easily checked that, if $\phi_1:H(\Gamma_1)\rightarrow H(\Gamma_2)$ and $\phi_2:H(\Gamma_2)\rightarrow H(\Gamma_3)$ are two equivalence of pre-graphs, then their composition $\phi_{2}\circ \phi_{1}: H(\Gamma_1)\rightarrow H(\Gamma_3)$ naturally is an equivalence of pre-graph $\Gamma_1$ and pre-graph $\Gamma_2$, we denote it by $\phi_2\circ\phi_1:\Gamma_1\rightarrow \Gamma_2$. Thus all pre-graphs and their equivalences form a locally  small groupoid. The following are some common notions associated with pre-graphs.

\begin{center}
   \textbf{ Real and virtual edges}
\end{center}
An edge $e$ is called an \textbf{virtual edge} if it contains two half-edges and is contained in a vertex, that is, $e=\{h_1,h_2\}$ with $\pi(h_1)=\pi(h_2)$. A virtual edge is also called a loop or a wheel. An edge is called an \textbf{real edge} if it is not an virtual edge. The sets of virtual and real edges of graph $\Gamma$ is denoted by $E_{vir}(\Gamma)$ and $E_{re}(\Gamma)$, respectively and we have $E(\Gamma)=E_{vir}(\Gamma)\sqcup E_{re}(\Gamma)$, where $\sqcup$ denotes disjoint union of sets.  In example \ref{example}, $e_1=\{h,i\}$ and $e_2=\{k,l\}$ are virtual edges, $e_3=\{d,f\}$ and $e_4=\{j\}$ are real edges. $h,i,k,l$ are virtual half-edges and $d,f,j$ are real half-edges.

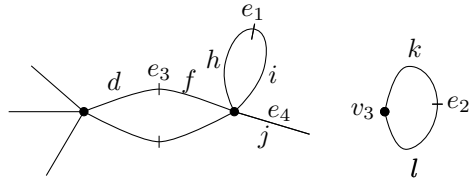
\begin{center}
\begin{tikzpicture}
\draw[fill] (0,0) circle [radius=0.055];
\draw (0,0)--(-0.7,0.61);
\draw (0,0)--(-1,0);
\draw (0,0)--(-0.5,-0.85);

\draw plot[smooth]  coordinates {(0,0)  (1,0.3)   (2,0)};
\draw (1,0.22)--(1,0.38);
\node  at (0.4,0.41) {$d$}; \node  at (1.4,0.4) {$f$};
\draw plot[smooth]  coordinates {(0,0)  (1,-0.4)   (2,0)};
\draw (1,-0.32)--(1,-0.48);

\draw[fill] (2,0) circle [radius=0.055];

\draw (2,0)--(3,-0.3);
\node  at (2.4,-0.33) {$j$};

\draw (2.27,1.17)--(2.23,0.96);

\draw[fill] (2,0) circle [radius=0.025];
\draw (2,0)--(3,-0.3);
\node [below] at (2.6,0.2)  {$e_4$};
\node [above] at (1,0.3) {$e_3$};
\node [above] at (2.25,1.1) {$e_1$};
\draw[fill] (4,0) circle [radius=0.055];

\node [below] at (4.4,-0.5) {$l$};
\draw (4.77,0.1)--(4.63,0.1);
\node [right] at (4.7,0.1) {$e_2$};

\draw plot[smooth, tension=1]  coordinates {(2,0)  (2.4,0.6) (2.25,1.1) (1.88,0.65)   (2,0)};
\draw (2.27,1.17)--(2.23,0.96);
\node [right] at (2.33,0.5) {$i$};
\node [left] at (1.95,0.7) {$h$};
\draw[fill] (4,0) circle [radius=0.055];
\draw plot[smooth, tension=1]  coordinates {(4,0)  (4.3,-0.5) (4.7,0.1) (4.3,0.6)  (4,0)};
\node [above] at (4.4,0.6) {$k$};
\node [below] at (4.4,-0.5) {$l$};
\draw (4.77,0.1)--(4.63,0.1);
\node [left] at (4,0) {$v_3$};
\end{tikzpicture}
\end{center}

\begin{center}
   \textbf{ Real and virtual vertices}
\end{center}
A vertex $v$ is called a \textbf{virtual vertex} if it contains a virtual edge, more precisely, there is a two-element $\sigma$-orbit $\{h_1,h_2\}\subset v $ with $h_1\neq h_2, h_1=\sigma(h_2)$, otherwise we call it a \textbf{real vertex}.  The sets of virtual and real vertices of graph $\Gamma$ is denoted by $V_{vir}(\Gamma)$ and $V_{re}(\Gamma)$, respectively.  Obviously, $V(\Gamma)=V_{vir}(\Gamma)\sqcup V_{re}(\Gamma)$. In example \ref{example}, $v_1$ is a real vertex and $v_2, v_3$ are  virtual vertices.
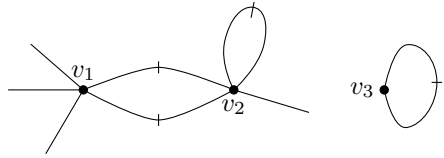
\begin{center}
\begin{tikzpicture}
\node [above] at (0,0) {$v_1$};
\draw[fill] (0,0) circle [radius=0.055];
\draw (0,0)--(-0.7,0.61);
\draw (0,0)--(-1,0);
\draw (0,0)--(-0.5,-0.85);
\draw plot[smooth]  coordinates {(0,0)  (1,0.3)   (2,0)};
\draw (1,0.22)--(1,0.38);
\draw plot[smooth]  coordinates {(0,0)  (1,-0.4)   (2,0)};
\draw (1,-0.32)--(1,-0.48);
\node [below ] at (2,0) {$v_2$};
\draw[fill] (2,0) circle [radius=0.055];
\draw (2,0)--(3,-0.3);
\draw plot[smooth, tension=1]  coordinates {(2,0)  (2.4,0.6) (2.25,1.1) (1.88,0.65)   (2,0)};
\draw (2.27,1.17)--(2.23,0.96);
\draw[fill] (4,0) circle [radius=0.055];
\draw plot[smooth, tension=1]  coordinates {(4,0)  (4.3,-0.5) (4.7,0.1) (4.3,0.6)  (4,0)};
\draw (4.77,0.1)--(4.63,0.1);
\node [left] at (4,0) {$v_3$};
\end{tikzpicture}
\end{center}
\begin{center}
   \textbf{ Real and virtual legs}
\end{center}

A  half-edge $h$ is called a \textbf{leg} if it forms an one-element orbit of $\sigma$ $($i.e. an edge with one element$)$  or belongs to a virtual edge. In the former case, we call it a \textbf{real leg} and in the later case we call it a \textbf{virtual leg}. For a half-edge $h\in H(\Gamma)$, if $h\neq \sigma(h)$ and $\pi(h)=\pi(\sigma(h))$, then $h$ a virtual leg of $\Gamma$ and $\overline{h}=(h,\sigma(h))$ is called a \textbf{virtual boundary edge} of $\Gamma$. By definition, a leg is virtual if and only if it belongs to a virtual edge. If $h= \sigma(h)$, we call $h$ a real leg and $\overline{h}=\{h\}$ a \textbf{real boundary edge}.
We denote the sets of legs, real legs and virtual legs of $\Gamma$ by  $Leg(\Gamma),Leg_{re}(\Gamma)$ and $Leg_{vir}(\Gamma)$. Obviously, we have $Leg(\Gamma)=Leg_{re}(\Gamma)\sqcup Leg_{vir}(\Gamma)$. In example \ref{example}, $Leg_{re}(\Gamma)=\{a,b,c,j\}$ and $Leg_{vir}(\Gamma)=\{h,i,k,l\}$.
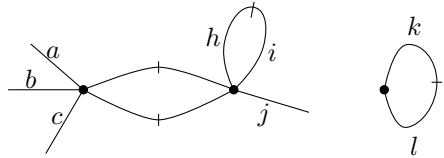
\begin{center}
\begin{tikzpicture}
\draw[fill] (0,0) circle [radius=0.055];
\draw (0,0)--(-0.7,0.61);
\node  at (-0.4,0.5) {$a$};
\draw (0,0)--(-1,0);
\node  at (-0.7,0.15) {$b$};
\draw (0,0)--(-0.5,-0.85);
\node  at (-0.35,-0.38) {$c$};

\draw plot[smooth]  coordinates {(0,0)  (1,0.3)   (2,0)};
\draw (1,0.22)--(1,0.38);

\draw plot[smooth]  coordinates {(0,0)  (1,-0.4)   (2,0)};
\draw (1,-0.32)--(1,-0.48);

\draw[fill] (2,0) circle [radius=0.055];

\draw (2,0)--(3,-0.3);
\node  at (2.4,-0.33) {$j$};

\draw plot[smooth, tension=1]  coordinates {(2,0)  (2.4,0.6) (2.25,1.1) (1.88,0.65)   (2,0)};
\draw (2.27,1.17)--(2.23,0.96);

\draw[fill] (4,0) circle [radius=0.055];
\draw plot[smooth, tension=1]  coordinates {(4,0)  (4.3,-0.5) (4.7,0.1) (4.3,0.6)  (4,0)};
\node [right] at (2.33,0.5) {$i$};
\node [left] at (1.95,0.7) {$h$};
\node [above] at (4.4,0.6) {$k$};
\node [below] at (4.4,-0.5) {$l$};
\draw (4.77,0.1)--(4.63,0.1);
\end{tikzpicture}
\end{center}

\begin{center}
   \textbf{ Inner and external edges}
\end{center}

An edge is called an \textbf{inner edge} if it is  a real and two-element edge. Otherwise, we call it a \textbf{boundary edge} or an \textbf{external edge}. Virtual edges are external edges by definition. We denote the set of inner edges and the set of external  edges by $Inn(\Gamma)$ and $Ext(\Gamma)$, respectively, and we have $E(\Gamma)=Inn(\Gamma)\sqcup Ext(\Gamma)$. In example \ref{example}, $Inn(\Gamma)=\{\{d,f\},\{e,g\}\}$ and $Ext(\Gamma)=\{\{a\},\{b\},\{c\},\{h,i\},\{j\},\{k,l\}\}$. Please do not confuse a leg with its corresponding external edge, for example, $a, b,c,j$ are real legs and their corresponding external edges are $\overline{a}=\{a\},\overline{b}=\{b\},\overline{c}=\{c\},\overline{j}=\{j\}$, and $h,i,k,l$ are virtual legs and their corresponding external edges are $\overline{h}=\overline{i}=\{h,i\},\overline{k}=\overline{l}=\{k,l\}.$
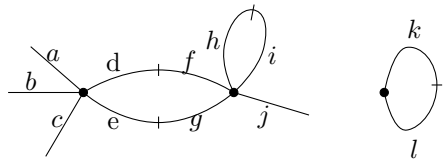
\begin{center}
\begin{tikzpicture}
\draw[fill] (0,0) circle [radius=0.055];
\draw (0,0)--(-0.7,0.61);
\node  at (-0.4,0.5) {$a$};
\draw (0,0)--(-1,0);
\node  at (-0.7,0.15) {$b$};
\draw (0,0)--(-0.5,-0.85);
\node  at (-0.35,-0.38) {$c$};

\draw plot[smooth, tension=1]  coordinates {(0,0)  (1,0.3)   (2,0)};
\draw (1,0.22)--(1,0.38);
\node  at (0.4,0.41) {d}; \node  at (1.4,0.4) {$f$};
\draw plot[smooth, tension=1]  coordinates {(0,0)  (1,-0.4)   (2,0)};
\draw (1,-0.32)--(1,-0.48);
\node  at (0.4,-0.41) {e}; \node  at (1.5,-0.41) {$g$};

\draw[fill] (2,0) circle [radius=0.055];

\draw (2,0)--(3,-0.3);
\node  at (2.4,-0.33) {$j$};

\draw plot[smooth,tension=1]  coordinates {(2,0)  (2.4,0.6) (2.25,1.1) (1.88,0.65)   (2,0)};
\draw (2.27,1.17)--(2.23,0.96);

\draw[fill] (4,0) circle [radius=0.055];
\draw plot[smooth, tension=1]  coordinates {(4,0)  (4.3,-0.5) (4.7,0.1) (4.3,0.6)  (4,0)};
\node [right] at (2.33,0.5) {$i$};
\node [left] at (1.95,0.7) {$h$};
\node [above] at (4.4,0.6) {$k$};
\node [below] at (4.4,-0.5) {$l$};
\draw (4.77,0.1)--(4.63,0.1);

\end{tikzpicture}
\end{center}

\begin{center}
   \textbf{ Inner and external vertices}
\end{center}

A vertex is called an \textbf{external vertex} if there is  a leg incident to it. Otherwise we call it an \textbf{inner vertex}. By definition, we see that a vertex is external if and only if  there is an external edge incident to it.
For example, in the following  pre-graph, $v_1, v_2, v_3$ are external vertices and $v_4$ is an inner vertex.
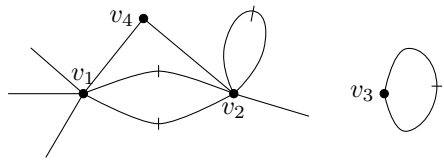
\begin{center}
\begin{tikzpicture}
\node [above] at (0,0) {$v_1$};
\draw[fill] (0,0) circle [radius=0.055];
\draw (0,0)--(-0.7,0.61);
\draw (0,0)--(-1,0);
\draw (0,0)--(-0.5,-0.85);
\draw plot[smooth]  coordinates {(0,0)  (1,0.3)   (2,0)};
\draw (1,0.22)--(1,0.38);
\draw plot[smooth]  coordinates {(0,0)  (1,-0.4)   (2,0)};
\draw (1,-0.32)--(1,-0.48);
\node [below ] at (2,0) {$v_2$};
\draw[fill] (2,0) circle [radius=0.055];
\draw (2,0)--(3,-0.3);
\draw plot[smooth, tension=1]  coordinates {(2,0)  (2.4,0.6) (2.25,1.1) (1.88,0.65)   (2,0)};
\draw (2.27,1.17)--(2.23,0.96);
\draw[fill] (4,0) circle [radius=0.055];
\draw plot[smooth,tension=1]  coordinates {(4,0)  (4.3,-0.5) (4.7,0.1) (4.3,0.6)  (4,0)};
\draw (4.77,0.1)--(4.63,0.1);
\node [left] at (4,0) {$v_3$};
\draw[fill] (0.8,1) circle [radius=0.055];
\draw (0,0)--(0.8,1);
\draw (2,0)--(0.8,1);
\node [left] at (0.8,1) {$v_4$};
\end{tikzpicture}
\end{center}

\subsection{Graphs and their geometric realization}
In this paper, we will mainly focus on a sub-class of pre-graphs which are pre-graphs with their loops dissociated, and we call them graphs. More precisely,
\begin{defn}
A pre-graph is called a \textbf{graph}  when it has no virtual vertex with degree $\geq 3$, and
a graph is called \textbf{reduced}  when it is non-empty and has no virtual edge.
\end{defn}
Empty pre-graphs are graphs, and we can associate any non-empty pre-graph an unique reduced graph by cutting off all its virtual edges.
$$
\begin{tabular}{|l|c|}
\hline
pre-graph & $(H,P, \sigma)$\\ \hline
 graph & no virtual vertice of deg$\geq$3\\\hline
reduced graph &non-empty $\&$ no virtual edges\\\hline
\end{tabular}
$$

Notice that all virtual vertices of a graph are of degree $2$, so all virtual edges are "dissociative", that is, each virtual edge of a graph viewed as a pre-graph is a loop with its geometric realization being
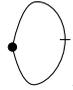
\begin{center}
\begin{tikzpicture}
\draw[fill] (4,0) circle [radius=0.055];
\draw plot[smooth, tension=1]  coordinates {(4,0)  (4.3,-0.5) (4.7,0.1) (4.3,0.6)  (4,0)};

\draw (4.77,0.1)--(4.63,0.1);
\end{tikzpicture}.
\end{center}
Thus every virtual edge in a graph $\Gamma$ forms a connected component of $|\Gamma|$.
In order to compatible with Joyal and Street's geometric theory of graphs, so we define the geometric realization of a virtual edge of a graph as a "decorated segment", that is, for every virtual edge, its geometric realization will be a topological space like this:
\begin{tikzpicture}
\draw[] (0,0) circle [radius=0.055];
\draw(-0.7,0)--(-0.055,0);
\draw(0.7,0)--(0.055,0);
\end{tikzpicture}.
This means that for any graph its geometric realization will be a geometric graph with vertices decorated: real vertices will be decorated by $\bullet$ and virtual vertices will be decorated by $\circ$.

\begin{rem}
For a graph $\Gamma$ we define  its geometric realization as the quotient space of $H(\Gamma)\times [0,\frac{1}{2}]$ with respect to a new equivalence $\sim'$ obtained by modifying condition $(c)$ of $\sim$ in the definition of geometric realization of a pre-graph.

More precisely, the geometric realization of $\Gamma$ is defined as  $|\Gamma|=\frac{H(\Gamma)\times [0,\frac{1}{2}]}{\sim'}$, where $\sim'$ is defined as:
$(h_1,l_1)\sim (h_2,l_2)$, for $h_1, h_2\in H(\Gamma)$ and $0\leq l_1,l_2\leq\frac{1}{2}$ if one of three conditions is satisfied: $(a)$ $h_1=h_2$ and $l_1=l_2$, $(b)$ $h_1\overset{P}{\thicksim}h_2$, $l_1=l_2=\frac{1}{2}$, $(c')$  $h_1\not\overset{P}{\thicksim}h_2$, $h_1=\sigma(h_2)$, $l_1=l_2=0$. In order to distinguish real and virtual vertices in its geometric realization, we put a $\bullet$ in the place of a real vertex but put a $\circ$ in the place of a virtual vertex.
\end{rem}

Now we give an example of a graph and its geometric realization.
\begin{ex}\label{example2}
In example \ref{example}, if we remove the virtual edge $\{h,i\}$, we get a graph and its geometric realization would be
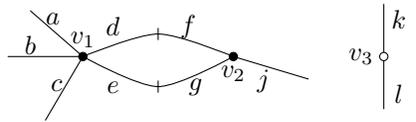
\begin{center}
\begin{tikzpicture}

\node [above] at (0,0) {$v_1$};
\draw[fill] (0,0) circle [radius=0.055];
\draw  (0,0)--(-0.7,0.61);
\node  at (-0.4,0.5) {$a$};
\draw (0,0)--(-1,0);
\node  at (-0.7,0.15) {$b$};
\draw (0,0)--(-0.5,-0.85);
\node  at (-0.35,-0.38) {$c$};

\draw plot[smooth]  coordinates {(0,0)  (1,0.3)   (2,0)};
\draw (1,0.22)--(1,0.38);
\node  at (0.4,0.41) {$d$}; \node  at (1.4,0.4) {$f$};
\draw plot[smooth]  coordinates {(0,0)  (1,-0.4)   (2,0)};
\draw (1,-0.32)--(1,-0.48);
\node  at (0.4,-0.41) {$e$}; \node  at (1.5,-0.41) {$g$};

\node [below ] at (2,0) {$v_2$};
\draw[fill] (2,0) circle [radius=0.055];

\draw (2,0)--(3,-0.3);
\node  at (2.4,-0.33) {$j$};

\node  at (4.2,0.5) {$k$};
\draw (4,0) circle [radius=0.055];
\draw (4,-0.055)--(4,-0.7);
\draw (4,0.055)--(4,0.7);
\node at (4.2,-0.5) {$l$};

\node  at (3.7,0) {$v_3$};
\end{tikzpicture}.
\end{center}
If we remove $\{k,l\}$ further, we get its associated reduced graph with its geometric realization being

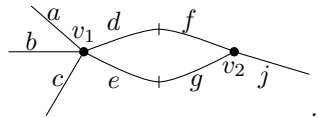
\begin{center}
\begin{tikzpicture}
\node [above] at (0,0) {$v_1$};
\draw[fill] (0,0) circle [radius=0.055];
\draw  (0,0)--(-0.7,0.61);
\node  at (-0.4,0.5) {$a$};
\draw (0,0)--(-1,0);
\node  at (-0.7,0.15) {$b$};
\draw (0,0)--(-0.5,-0.85);
\node  at (-0.35,-0.38) {$c$};

\draw plot[smooth]  coordinates {(0,0)  (1,0.3)   (2,0)};
\draw (1,0.22)--(1,0.38);
\node  at (0.4,0.41) {$d$}; \node  at (1.4,0.4) {$f$};
\draw plot[smooth]  coordinates {(0,0)  (1,-0.4)   (2,0)};
\draw (1,-0.32)--(1,-0.48);
\node  at (0.4,-0.41) {$e$}; \node  at (1.5,-0.41) {$g$};

\node [below ] at (2,0) {$v_2$};
\draw[fill] (2,0) circle [radius=0.055];

\draw (2,0)--(3,-0.3);
\node  at (2.4,-0.33) {$j$};
\end{tikzpicture}.
\end{center}
\end{ex}

We will not distinct a graph with its geometric realization usually.
The following are some useful notions.
\begin{center}
   \textbf{ Unitary graph}
\end{center}

A graph is called an \textbf{unitary graph} if it has only two half-edges and one virtual vertex.

For example, the dates $H=\{a,b\}$, $P=\{\{a,b\}\}$ and $\sigma(a)=b, \sigma(b)=a$ defines an unitary graph, its geometrical realization is
\begin{tikzpicture}
\draw[] (0,0) circle [radius=0.055];
\draw(-0.7,0)--(-0.055,0);
\draw(0.7,0)--(0.055,0);
\node [above] at (-0.4,0) {$a$};
\node [above] at (0.4,0) {$b$};
\node [above] at (0,0) {$v$};
\end{tikzpicture}, where $v$ denotes the unique virtual vertex $\{a,b\}$.

\begin{center}
   \textbf{ Prime graph}
\end{center}
A non-unitary graph with only one vertex is called a \textbf{prime graph}. A prime graph is also called a  corolla. By definition, tt is obvious that a prime graph is a pre-graph with exactly one real vertex, and  we call it a $n$-corolla if the degree of its vertex is $n$. The following graph is an example of  prime graph with degree $5$.
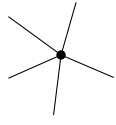
\begin{center}
\begin{tikzpicture}
\draw[fill] (0,0) circle [radius=0.055];
\draw (0,0)--(-0.7,0.51);
\draw (0,0)--(-0.7,-0.31);
\draw (0,0)--(0.2,0.7);
\draw (0,0)--(-0.1,-0.8);
\draw (0,0)-- (0.7,-0.3);
\end{tikzpicture}
\end{center}

\begin{center}
   \textbf{Closed graph}
\end{center}
A graph without legs is called a \textbf{closed graph}.
An \textbf{ideal} of graph $\Gamma=(H,P,\sigma)$ is a graph $\Gamma^o=(H^o,P^o,\sigma^o)$ such that $H^o=H-Leg(\Gamma)$, $P^o=P|_{H^o}$ and $\sigma^o=\sigma|_{H^o}$. Ideals of unitary and prime graphs are empty graphs.

\subsection{Operations on/of graphs}
In this section, we will introduce some related notions and constructions related to graphs and operations which can implement on graphs. It will be easy to see  that  the combinatorial definition of a graph has some advantages to make these constructions precise and coherent.

\begin{center}
   \textbf{ Sub-graph}
\end{center}
Following [Gr06], a \textbf{subgraph} of graph $\Gamma=(H,P,\sigma)$ is a graph $\Gamma'=(H',P',\sigma')$ such that $H'\subseteq H$, $P'\subseteq P$
and for any $h\in H'$,
 \begin{equation*}\sigma'(h)=
\begin{cases}
\sigma(h),&  \text{if}\ \sigma(h)\in H',\\
h,&  \text{if}\ \sigma(h)\notin H'.\\
\end{cases}
\end{equation*}
By definition, we see that subgraphs of $\Gamma$ are one-to-one correspondence to subsets of the partition $P(\Gamma)$.
If  $H'\subsetneq H$, we call $\Gamma'$ a real subgraph of $\Gamma$.
If $\Gamma'$ is a subgraph of $\Gamma$, we usually denote this fact as $\Gamma'\subseteq \Gamma$. If $\Gamma'$ is a  real subgraph of $\Gamma$, we  denote this as $\Gamma'\subsetneq \Gamma$.

\begin{center}
   \textbf{ Path }
\end{center}

Two vertices $v$ and $v'$ are connected if there is a sequence of inner edges $(\{h_0,h'_0\},...,\{h_n,h'_n\})$ such that $\pi(h_0)=v$, $\pi(h'_i)=\pi(h_{i+1})$ for $0\leq i<n$ and $\pi(h'_n)=v'$. We call such a sequence a \textbf{path} between $v$ and $v'$. If $v=v'$, we call it a closed path or circuit. A graph is \textbf{connected} if each pair of vertices is connected.

Obviously,  we can define  connected components of a graph to be its  maximal connected subgraphs. Corollas and unitary graphs are connected.

\begin{center}
   \textbf{ Quotient graph and division}
\end{center}
First recall that for two set $X$ and $X'$ with $X'\subset X$, the complement set $X-X'$ of $X'$ is defined to be $\{x|x\in X, x\not\in X'\}$, and if $P$ is a partition of $X$, then there will be a partition of $X'$ induced from $P$ naturally, denoted by $P|_{X'}$. Now let us give the definition of a \textbf{quotient graph}.
If $\Gamma'\subseteq \Gamma$ be a subgraph, we define a new graph $\Gamma''=(H'',P'',\sigma'')$ as follows:

$\bullet$ $H''=(H-H')\sqcup Leg(\Gamma')$;

$\bullet$ $P''=P|_{H-H'}\sqcup \{Leg(\Gamma')\}$;

$\bullet$ $\sigma''=\sigma|_{H''}$, that is, $\sigma''$ is the restriction of $\sigma$ on $H''$.

We call $\Gamma''$ the quotient graph of $\Gamma$ with respect to the subgraph $\Gamma'$, and write this fact as $\Gamma''=\Gamma/\Gamma'$ or $\Gamma'\rightarrowtail\Gamma\twoheadrightarrow\Gamma''.$

On their geometric realizations, $|\Gamma''|$ can be obtained by a contraction of $|\Gamma|$ along $|(\Gamma')^o|$. For any graph $\Gamma$, $\Gamma/\Gamma$ would be a corolla if $Leg(\Gamma)\neq \varnothing$ or an empty graph if $Leg(\Gamma)= \varnothing$.

A \textbf{division} of graph $\Gamma$ is a finite set of subgraphs $D=\{\Gamma_1,...,\Gamma_n\}$ such that $H(\Gamma)=\sqcup_{i=1}^nH(\Gamma_i)$. Each $\Gamma_i$ $(1\leq i\leq n)$ is called a \textbf{cell} or a\textbf{component} of this division. We also call a division a vertex-partition. If a graph $\Gamma$ has a division $D=\{\Gamma_1,...,\Gamma_n\}$, then it is obvious that for any different $i,j\in\{1,...,n\}$ $$(\Gamma/\Gamma_i)/\Gamma_j=(\Gamma/\Gamma_j)/\Gamma_i,$$ thus the totally quotient graph $\Gamma/D=\Gamma/\{\Gamma_1,...,\Gamma_n\}$ is well-defined, and we call it the \textbf{contraction} of $\Gamma$ according to the division $D$.

\begin{center}
   \textbf{ Splitting and fission}
\end{center}
Following \cite{[Gr06]}, given two vertices $v_1$ and $v_2$ of $\Gamma$, we call the set $E(v_1,v_2)$ of all edges
incident to  $v_1$ and $v_2$ an \textbf{inner thick edge}, that is, $$E(v_1,v_2)=\{e\in E(\Gamma)|e\rightarrow v_1, e\rightarrow v_2\}.$$
If $v$ is an external vertex of $\Gamma$, we call the set $E(v)$ of all boundary edges
incident to  $v$ an \textbf{external thick edge}, that is, $$E(v)=\{\{h\}\in E(\Gamma)|h\in Leg(\Gamma),h\rightarrow v\}.$$
A \textbf{splitting} of a graph  is an assignment of a partition to each external and internal thick edges and a partition of virtual edges. A splitting is also called as a thick-edge-partition.

Let $\Gamma=(H,P,\sigma)$ be a graph with a splitting, then it is easy to see that the splitting structure can be equivalently  defined as an equivalent relation $\sim$ on $H$ such that for every $h_1, h_2\in H(\Gamma)$

$\bullet$ $h_1=\sigma(h_2)$ and $\pi(h_1)\neq\pi(h_2)$ imply that $h_1\sim h_2$;

$\bullet$ $h_1\sim h_2$ and $\pi(h_1)=\pi(h_2)$ imply that $\sigma(h_1)\sim\sigma(h_2)$, where $\pi$ is the map $\pi:H(\Gamma)\rightarrow V(\Gamma)$ defined by $P$;

$\bullet$ $h_1=\sigma(h_2)$ and $\pi(h_1)=\pi(h_2)$ implies that $h_1\nsim h_2$.

By definition, we see that on the quotient space $H/\sim$ there will be a well-defined  partition $P/\sim$ and a well-defined involution $\sigma/\sim$, thus we can get a new graph $\Gamma/\sim=(H/\sim, P/\sim, \sigma/\sim)$ and call it the \textbf{fusion} of $\Gamma$ according to the splitting.

\begin{center}
   \textbf{ Coarse-graining }
\end{center}

When a graph $\Gamma$ have both a splitting and a division, we call them \textbf{compatible} if there are both a naturally induced splitting on condensation of $\Gamma$ and a naturally induced division on fusion of $\Gamma$, and in this case we have $\Gamma/(D,\sim)=(\Gamma/D)/\sim=(\Gamma/\sim)/D$.

So when a graph $\Gamma$ is equipped with both  a dividing structure $D$ and a compatible splitting structure $\sim$, we call the pair $(D,\sim)$ a partition of $\Gamma$ and call $\Gamma/(D,\sim)$ the \textbf{coarse-graining} or \textbf{residue} of $\Gamma$ according to partition $(D,\sim)$. A procedure to represent a graph as coarse-graining of another graph with a splitting and a compatible division  is called \textbf{fine-graining}.

\begin{center}
   \textbf{ Substitution }
\end{center}

Let $\Gamma_1=(H_1,P_1,\sigma_1)$  and $\Gamma_2=(H_2,P_2,\sigma_2)$ be two graphs, $v\in V_{re}(\Gamma_2)$ be a \textbf{real} vertex of $\Gamma_2$ and $\theta:v\rightarrow Leg(\Gamma_1)\subseteq H_1$ is a bijection of sets. If $\Gamma_1=(H_1,P_1,\sigma_1)$ is reduced, we define a new graph $\Gamma=(H,P,\sigma)$ as follows:

$\bullet$ $H=H_1\sqcup (H_2-v)$;

$\bullet$ $P=P_1\sqcup (P_2-\{v\})$;

$\bullet$ for any $h\in H$,
 \begin{equation*}\sigma(h)=
\begin{cases}
\sigma_2(h),&  \text{if}\ h\in H_2-v,\sigma_2(h)\notin v,\\
\theta(h),&  \text{if}\ h\in H_2-v,  \sigma_2(h)\in v,\\
\sigma_1(h),&  \text{if}\ h\in H_1-Leg(\Gamma_1),\\
\sigma_2(\theta^{-1}(h)),&  \text{if}\ h\in Leg(\Gamma_1),\theta^{-1}(h)\not\in Leg(\Gamma_2),\\
h,&  \text{if}\ h\in Leg(\Gamma_1),\theta^{-1}(h)\in Leg(\Gamma_2).
\end{cases}
\end{equation*}

We call $\Gamma$ the \textbf{substitution} of $\Gamma_2$ by $\Gamma_1$ with respect to $(v,\theta)$, and  write this fact as $\Gamma=\Gamma_2\triangleleft_{(v, \theta)} \Gamma_1$. It is easy to see that $\Gamma_1$ is canonically isomorphic to a subgraph of  $\Gamma$, and $\Gamma_2$ is canonically isomorphic to a quotient graph of $\Gamma$. We write this fact as $\Gamma_1\rightarrowtail \Gamma_2\triangleleft_{(v, \theta)} \Gamma_1\twoheadrightarrow \Gamma_2$. Please notice that for any virtual vertex $v\in V_{vir}(\Gamma_1)$, after subustition it will become a real vertex of $\Gamma$, that is, $v\in V_{re}(\Gamma)$.

The construction of substitution is essential for the description of  monads characterizing operads or PROPs \cite{[Mar06],[LV12],[Lo12]} and is closely related with divisions of graphs.

\begin{center}
   \textbf{ Tensor product of graphs}
\end{center}
Recall that for a set $X_1$ with a partition $P_1$ and a set $X_2$ with a partition $P_2$, there is a partition on their disjoint union $X_1\sqcup X_2$ naturally, we denoted it by $P_1\sqcup P_2$.  If $\sigma_1$ and $\sigma_2$ are involutions on $X_1$ and $X_2$ respectively, then there will be an involution of $X_1\sqcup X_2$ naturally, we denote it by $\sigma_1\sqcup \sigma_2$.

The notion of tensor product can be introduced for pre-graphs. For any two pre-graphs $\Gamma_1=(H_1,P_1,\sigma_1)$ and $\Gamma_2=(H_2,P_2,\sigma_2)$, we define their \textbf{tensor product} to be a triple $(H,P,\sigma)$  with $H=H_1\sqcup H_2$, $P=P_1\sqcup P_2$ and  $\sigma=\sigma_1\sqcup \sigma_2$. It is easily to check that the triple forms a pre-graph, we denote it by $\Gamma_1\otimes\Gamma_2$. Evidently, we have $V(\Gamma_1\otimes\Gamma_2)=V(\Gamma_1)\sqcup V(\Gamma_2)$, and similar results also hold for the set of (real/virtual) edges and (real/virtual) legs, etc.
The following proposition can be directly checked.
\begin{prop}
The groupoid of pre-graphs equipped with the tensor product defined above forms a symmetric tensor category with empty pre-graph as the unit object.
\end{prop}

\begin{center}
   \textbf{ Elementary, essential prime and invertible graphs }
\end{center}

A graph is called \textbf{elementary} if it is tensor product of  finite number of prime graphs and unitary graphs. A graph is called \textbf{essential prime} if it is an elementary graph and contains exactly one real vertex. A graph is called \textbf{invertible}  if it is tensor product of finite number of unitary graphs. It is obvious that  an elementary graph is a corolla iff it is connected,  and  each connected component of an invertible graph is an unitary graph.

\begin{center}
   \textbf{ Merge }
\end{center}
Following \cite{[BM08]} (see also \cite{[KW13]}), we recall the notion of\textbf{ merge}.
Let $\Gamma=(H,P,\sigma)$ be a graph, $v_1,v_2\in V(\Gamma)$ be two real vertices. We define a new graph $\Gamma'=(H',P',\sigma')$ as follows:

$\bullet$ $H'=H$;

$\bullet$ $P'=(P-\{v_1,v_2\})\sqcup \{v_1\sqcup v_2\}$;

$\bullet$ $\sigma'=\sigma$.

We call $\Gamma'$ the merge of $\Gamma$ along $v_1,v_2$, and write $\Gamma'=\vee_{v_1,v_2}\Gamma$.

If $\Gamma_1$ and $\Gamma_2$ are two graphs with $v_1\in V(\Gamma_1)$ and $v_2\in V(\Gamma_2)$, we define the merge of $\Gamma_1$ and $\Gamma_2$ along $v_1,v_2$ as $\vee_{v_1,v_2}(\Gamma_1\otimes\Gamma_2)$. Intuitively,  merge is just fusion of vertices and can be viewed as a special kind of contraction.

\begin{center}
   \textbf{ Grafting of two graphs }
\end{center}

Now we define the notion of grafting of two graphs. Let $\Gamma_1=(H_1,P_1,\sigma_1)$ and $\Gamma_2=(H_2,P_2,\sigma_2)$ be two graphs. Let $h_1\in Leg(\Gamma_1)$ and $h_2\in Leg(\Gamma_2)$ be legs. We define a new graph $\Gamma=(H,P,\sigma)$ called \textbf{grafting} of $\Gamma_1$ and $\Gamma_2$ along $h_1,h_2$ as follows:

$\bullet$ if both $h_1,h_2$ are real legs, $H=H_1\sqcup H_2$, $P=P_1\sqcup P_2$ and
\begin{equation*}\sigma(h)=
\begin{cases}
\sigma_1(h),&  \text{if}\ h\in H_1- \{h_1\},\\
\sigma_2(h),&  \text{if}\ h\in H_2- \{h_2\},\\
h_2,  & \text{if}\ h=h_1,\\
h_1,  & \text{if}\ h=h_2;
\end{cases}
\end{equation*}

$\bullet$ if both $h_1,h_2$ are virtual legs, $H=(H_1-\{h_1\})\sqcup (H_2-\{h_2\})$, $P=(P_1\sqcup P_2)|_{(H_1- \{h_1,\sigma_1(h_1)\})\sqcup (H_2- \{h_2,\sigma_2(h_2)\})}\sqcup \{\{\sigma_1(h_1),\sigma_2(h_2)\}\}$ and
\begin{equation*}\sigma(h)=
\begin{cases}
\sigma_1(h),&  \text{if}\ h\in H_1- \{h_1,\sigma(h_1)\},\\
\sigma_2(h), & \text{if}\ h\in H_2- \{h_2,\sigma(h_2)\},\\
\sigma_2(h_2), &  \text{if}\ h=\sigma_1(h_1),\\
\sigma_1(h_1), &  \text{if}\ h=\sigma_2(h_2);
\end{cases}
\end{equation*}

$\bullet$ if $h_1$ is a real leg and $h_2$ is a virtual leg, $H=H_1\sqcup (H_2-\{h_2,\sigma_2(h_2)\})$, $P=(P_1\sqcup P_2)|_{H} $ and
\begin{equation*}\sigma(h)=
\begin{cases}
\sigma_1(h),&  \text{if}\ h\in H_1,\\
\sigma_2(h), & \text{if}\ h\in H_2- \{h_2,\sigma_2(h_2)\};
\end{cases}
\end{equation*}

$\bullet$ if $h_1$ is a virtual leg and $h_2$ is a real leg, $H=(H_1-\{h_1,\sigma_1(h_1)\})\sqcup H_1$, $P=(P_1\sqcup P_2)|_{H} $ and
\begin{equation*}\sigma(h)=
\begin{cases}
\sigma_1(h),&  \text{if}\ h\in H_1- \{h_1,\sigma_1(h_1)\},\\
\sigma_2(h), & \text{if}\ h\in H_2.
\end{cases}
\end{equation*}
In this situation we write $\Gamma=\Gamma_1\circ^{h_1}_{h_2}\Gamma_2$.  The class of graphs are closed under grafting and evidently we have $V_{re}(\Gamma)=V_{re}(\Gamma_1)\sqcup V_{re}(\Gamma_2)$, $E_{re}(\Gamma)=E_{re}(\Gamma_1)\sqcup E_{re}(\Gamma_2)$.

\begin{ex}
\

$\bullet$ If
$\Gamma_1=
\begin{tikzpicture}
\draw[fill] (0,0) circle [radius=0.055];
\draw(0.7,0)--(0.055,0);
\draw(-0.7,0)--(-0.055,0);
\node [above] at (-0.4,0) {$h_1$};
\node [above] at (0.4,0) {$h_2$};
\end{tikzpicture}
$,
$\Gamma_2=
\begin{tikzpicture}
\draw[fill] (0,0) circle [radius=0.055];
\draw(0.7,0)--(0.055,0);
\draw(-0.7,0)--(-0.055,0);
\node [above] at (-0.4,0) {$h_3$};
\node [above] at (0.4,0) {$h_4$};
\end{tikzpicture}
$, then $\Gamma_1\circ^{h_2}_{h_3}\Gamma_2=
\begin{tikzpicture}
\draw[fill] (0,0) circle [radius=0.055];
\draw[fill] (1.6,0) circle [radius=0.055];
\draw(0,0)--(1.6,0);
\draw(-0.7,0)--(-0.055,0);
\draw(0.8,0.1)--(0.8,-0.1);
\draw(2.3,0)--(1.6,0);
\node [above] at (-0.4,0) {$h_1$};
\node [above] at (0.4,0) {$h_2$};
\node [above] at (2.0,0) {$h_4$};
\node [above] at (1.2,0) {$h_3$};
\end{tikzpicture}$.

$\bullet$ If
$\Gamma_1=
\begin{tikzpicture}
\draw[fill] (0,0) circle [radius=0.055];
\draw(0.7,0)--(0.055,0);
\draw(-0.7,0)--(-0.055,0);
\node [above] at (-0.4,0) {$h_1$};
\node [above] at (0.4,0) {$h_2$};
\end{tikzpicture}
$,
$\Gamma_2=
\begin{tikzpicture}
\draw (0,0) circle [radius=0.055];
\draw(0.7,0)--(0.055,0);
\draw(-0.7,0)--(-0.055,0);
\node [above] at (-0.4,0) {$h_3$};
\node [above] at (0.4,0) {$h_4$};
\end{tikzpicture}$, then $\Gamma_1\circ^{h_2}_{h_3}\Gamma_2=\begin{tikzpicture}
\draw[fill] (0,0) circle [radius=0.055];
\draw(0.7,0)--(0.055,0);
\draw(-0.7,0)--(-0.055,0);
\node [above] at (-0.4,0) {$h_1$};
\node [above] at (0.4,0) {$h_2$};
\end{tikzpicture}
$.

$\bullet$ If
$\Gamma_1=
\begin{tikzpicture}
\draw(0,0) circle [radius=0.055];
\draw(0.7,0)--(0.055,0);
\draw(-0.7,0)--(-0.055,0);
\node [above] at (-0.4,0) {$h_1$};
\node [above] at (0.4,0) {$h_2$};
\end{tikzpicture}
$,
$\Gamma_2=
\begin{tikzpicture}
\draw (0,0) circle [radius=0.055];
\draw(0.7,0)--(0.055,0);
\draw(-0.7,0)--(-0.055,0);
\node [above] at (-0.4,0) {$h_3$};
\node [above] at (0.4,0) {$h_4$};
\end{tikzpicture}$, then $\Gamma_1\circ^{h_2}_{h_3}\Gamma_2=\begin{tikzpicture}
\draw (0,0) circle [radius=0.055];
\draw(0.7,0)--(0.055,0);
\draw(-0.7,0)--(-0.055,0);
\node [above] at (-0.4,0) {$h_1$};
\node [above] at (0.4,0) {$h_4$};
\end{tikzpicture}
$.

\end{ex}\ \\

If $h_1,h_2$ are two legs of graph $\Gamma=(H,P,\sigma)$, we define a new pre-graph $\Gamma'=(H',P',\sigma')$ as follows:

$\bullet$ if both $h_1,h_2$ are real legs, $H'=H$, $P'=P$ and
\begin{equation*}\sigma'(h)=
\begin{cases}
\sigma(h),&  \text{if}\ h\in H- \{h_1,h_2\},\\
h_2,  & \text{if}\ h=h_1,\\
h_1,  & \text{if}\ h=h_2;
\end{cases}
\end{equation*}

$\bullet$ if both $h_1,h_2$ are virtual legs, $H'=H-\{h_1,h_2\}$, $P'=P|_{H'}\sqcup \{\{\sigma(h_1),\sigma(h_2)\}\}$ and
\begin{equation*}\sigma'(h)=
\begin{cases}
\sigma(h),&  \text{if}\ h\in H- \{h_1,\sigma(h_1),h_2,\sigma(h_2)\},\\
\sigma(h_2), &  \text{if}\ h=\sigma(h_1),\\
\sigma(h_1), &  \text{if}\ h=\sigma(h_2);
\end{cases}
\end{equation*}

$\bullet$ if $h_1$ is a real leg and $h_2$ is a virtual leg, $H'=H-\{h_2,\sigma(h_2)\}$, $P'=P|_{H'} $ and
\begin{equation*}\sigma'(h)=
\begin{cases}
\sigma(h),&  \text{if}\ h\in H- \{h_1,h_2,\sigma(h_2)\},\\
h, & \text{if}\ h=h_1 ;
\end{cases}
\end{equation*}

$\bullet$ if $h_1$ is a virtual leg and $h_2$ is a real leg, $H'=H-\{h_1,\sigma(h_1)\}$, $P'=P|_{H'} $ and
\begin{equation*}\sigma'(h)=
\begin{cases}
\sigma(h),&  \text{if}\ h\in H- \{h_1,\sigma(h_1),h_2\},\\
h, & \text{if}\ h=h_2.
\end{cases}
\end{equation*}
We also write $\Gamma'$ as $\circ^{h_1}_{h_2}\Gamma$ and call it the \textbf{self-grafting} of $\Gamma$ along legs $h_1,h_2$. Obviously, the class of graphs is not closed under grafting.

Using the notions of self-grafting and tensor product of graphs, we can write $\Gamma_1\circ^{h_1}_{h_2}\Gamma_2$ as $\circ^{h_1}_{h_2}(\Gamma_1\otimes \Gamma_2)$. If $h_1,...,h_m$ are legs of $\Gamma_1$, $l_1,...,l_m$ are legs of $\Gamma_2$, we denote by $\Gamma_1\circ^{h_1,...,h_m}_{l_1,...,l_m}\Gamma_2$ the graph $\circ^{h_1}_{l_1}(\cdot\cdot\cdot\circ^{h_{m-1}}_{l_{m-1}}(\circ^{h_m}_{l_m}(\Gamma_1\otimes\Gamma_2))\cdot\cdot\cdot)$ obtained by successively grafting $\Gamma_1$ and $\Gamma_2$ along $\{h_i,l_i\}$s if each step of grafting is well-defined. Please notice that there are some symmetries  in above notations, that is, $\Gamma_1\circ^{h_1}_{h_2}\Gamma_2=\Gamma_2\circ^{h_2}_{h_1}\Gamma_1=\Gamma_1\circ^{h_2}_{h_1}\Gamma_2=\Gamma_2\circ^{h_1}_{h_2}\Gamma_1$ and $\circ^{h_1}_{h_2}\Gamma=\circ^{h_2}_{h_1}\Gamma$.

\subsection{Some structures on graphs}
 In this section, we will recall  definitions of an oriented graph, directed graph and progressive graph and other related notions following \cite{[JS91]}.
\begin{center}
   \textbf{ Oriented graph}
\end{center}

An \textbf{oriented graph} is a graph $\Gamma$ together with  a function $sgn:H(\Gamma)\rightarrow \{+,-\}$ (called orientation) such that for each edge $e=\{h_1,h_2\}$ with $h_1\neq h_2$, $sgn:e\rightarrow \{+,-\} $ is an isomorphism of sets.

It is easy to see that any subgraph of an oriented graph is an oriented graph with the naturally induced orientation.
We usually write an oriented graph as  $\overrightarrow{\Gamma}$ and one of its inner edges  as $\overrightarrow{e}$ to emphasize their orientation structures. The input set $In(v)$ of a vertex $v$ by defined is $\{h\in v| sgn(h)=+\}$ and output set $Out(v)$ by definition is $\{h\in v|sgn(h)=-\}$. For an oriented  edge $\overrightarrow{e}=\{h_{-},h_{+}\}$ with $sgn(h_{-})=-, sgn(h_{+})=+$, we define its source $s(\overrightarrow{e})$ and target $t(\overrightarrow{e})$ to be vertices $\pi(h_{-})$ and $\pi(h_{+})$, respectively. A leg $h$ is called an input of $\overrightarrow{\Gamma}$ if $sgn(h)=+$, and a leg $h$ is called an output of $\overrightarrow{\Gamma}$ if $sgn(h)=-$.
We denote by $In(\overrightarrow{\Gamma})$ and $Out(\overrightarrow{\Gamma})$ the sets of inputs and outputs of $\overrightarrow{\Gamma}$.

\begin{ex} Take $\Gamma$ be the graph in example \ref{example2}, we define $sgn$ as follows:

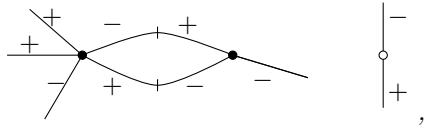
\begin{center}
\begin{tikzpicture}

\draw[fill] (0,0) circle [radius=0.055];
\draw  (0,0)--(-0.7,0.61);
\node  at (-0.4,0.5) {$+$};
\draw (0,0)--(-1,0);
\node  at (-0.7,0.15) {$+$};
\draw (0,0)--(-0.5,-0.85);
\node  at (-0.35,-0.38) {$-$};

\draw plot[smooth]  coordinates {(0,0)  (1,0.3)   (2,0)};
\draw (1,0.22)--(1,0.38);
\node  at (0.4,0.41) {$-$}; \node  at (1.4,0.4) {$+$};
\draw plot[smooth]  coordinates {(0,0)  (1,-0.4)   (2,0)};
\draw (1,-0.32)--(1,-0.48);
\node  at (0.4,-0.41) {$+$}; \node  at (1.5,-0.41) {$-$};

\draw[fill] (2,0) circle [radius=0.055];

\draw (2,0)--(3,-0.3);
\node  at (2.4,-0.33) {$-$};

\draw (2,0)--(3,-0.3);
\node  at (4.2,0.5) {$-$};
\draw (4,0) circle [radius=0.055];
\draw (4,-0.055)--(4,-0.7);
\draw (4,0.055)--(4,0.7);
\node at (4.2,-0.5) {$+$};
\end{tikzpicture},
\end{center} that is, $sgn^{-1}(+)=\{a,b,e,f,l\}$ and $sgn^{-1}(-)=\{c,d,j,k\}$.
More intuitively, we specify an orientation structure of a graph by drawing an arrow on each edges, for example  we picture the oriented graph above as

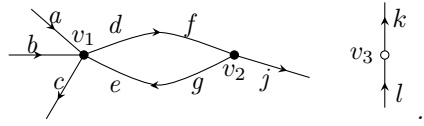
\begin{center}
\begin{tikzpicture}

\draw[fill] (0,0) circle [radius=0.055];
\draw  [postaction={decorate,decoration={markings,mark=at position 0.7 with {\arrowreversed[black]{stealth}}}}] (0,0)--(-0.7,0.61);
\draw [postaction={decorate,decoration={markings,mark=at position 0.7 with {\arrowreversed[black]{stealth}}}}] (0,0)--(-1,0);
\draw  [postaction={decorate,decoration={markings,mark=at position 0.7 with {\arrow[black]{stealth}}}}] (0,0)--(-0.5,-0.85);

\draw plot[smooth] coordinates {(0,0)  (1,-0.4)   (2,0)}[postaction={decorate, decoration={markings,mark=at position 0.511 with {\arrowreversed[black]{stealth}}}}];

\draw plot[smooth] coordinates {(0,0)  (1,0.3)   (2,0)}[postaction={decorate, decoration={markings,mark=at position .51 with {\arrow[black]{stealth}}}}];

\draw[fill] (2,0) circle [radius=0.055];

\draw  [postaction={decorate,decoration={markings,mark=at position 0.7 with {\arrow[black]{stealth}}}}] (2,0)--(3,-0.3);

\draw (4,0) circle [radius=0.055];
\draw  [postaction={decorate,decoration={markings,mark=at position 0.7 with {\arrowreversed[black]{stealth}}}}](4,-0.055)--(4,-0.7);
\draw  [postaction={decorate,decoration={markings,mark=at position 0.7 with {\arrow[black]{stealth}}}}](4,0.055)--(4,0.7);

\node [above] at (0,0) {$v_1$};
\node  at (-0.4,0.5) {a};
\node  at (-0.7,0.15) {b};
\node  at (-0.35,-0.38) {c};
\node  at (0.4,0.41) {d}; \node  at (1.4,0.4) {f};
\node  at (0.4,-0.41) {e}; \node  at (1.5,-0.41) {g};
\node [below ] at (2,0) {$v_2$};
\node  at (2.4,-0.33) {j};

\node  at (4.2,0.5) {$k$};

\node at (4.2,-0.5) {$l$};

\node  at (3.7,0) {$v_3$};

\end{tikzpicture}.
\end{center}
In this example, $In(v_1)=\{a,b,e\}$, $Out(v_1)=\{c,d\}$, $In(v_2)=\{f\}$, $Out(v_2)=\{g,j\}$, $In(v_3)=\{l\}$ and $Out(v_3)=\{k\}$. For source and target of edges, we have $v_1=s(\{d,f\})=t(\{e,g\})=t(\{a\})=t(\{b\})=s(\{c\}),v_2=s(\{e,g\})=t(\{d,f\})=$ and  $s(\{k,l\})= t(\{k,l\})=v_3$. The input set $In(\overrightarrow{\Gamma})$  is $\{a,b,l\}$ and the output set $Out(\overrightarrow{\Gamma})$ is $\{c,j,k\}$.
\end{ex}

\begin{center}
   \textbf{Directed graph and directed path }
\end{center}

A \textbf{directed graph} is an oriented graph $\overrightarrow{\Gamma}$ such that for any vertex $v\in V(\overrightarrow{\Gamma})$, both $In(v)$ and $Out(v)$ are non-empty sets. Any subgraph of a directed graph is a directed graph naturally. Any quotient graph of an oriented graph is an directed graph naturally. A vertex $v\in V(\Gamma)$ is called an input vertex if $In(\Gamma)\cap In(v)\neq \varnothing$, the set of all input vertices is denoted by $V_{in}(\Gamma)$.  Similarly we call a vertex $v\in V(\Gamma)$  an output vertex if $Out(\Gamma)\cap Out(v)\neq \varnothing$,denote the set of all output vertices by $V_{out}(\Gamma)$.

Let $\overrightarrow{\Gamma}_1=(\Gamma_1, sgn_1)$ and $\overrightarrow{\Gamma}_2=(\Gamma_2,sgn_2)$ be two directed graphs. We define their tensor product to be a directed graph $\overrightarrow{\Gamma}=(\Gamma,sgn)$ with $\Gamma=\Gamma_1\otimes \Gamma_2$ and $sgn=(sgn_1\sqcup sgn_2)$. Usually, we write $\overrightarrow{\Gamma}$ as $\overrightarrow{\Gamma}_1\otimes \overrightarrow{\Gamma}_2,$ and obviously this product is associative.

In an oriented graph $\overrightarrow{\Gamma}$, a \textbf{directed path} of length $n (\geq 1)$ is  a sequence of directed edges $e_1e_2\cdot\cdot\cdot e_n$ such that $t(e_i)=s(e_{i+1})$, $1\leq i\leq  n-1$. If both $v=s(e_1)$ and $v'=t(e_n)$ exist, we call $v$ the starting point and $v'$ the ending point of the directed path. A directed path with identical starting point and ending point is called a directed closed path or a circuit. A directed graph is called \textbf{acyclic} if it contains no directed circuit.

For any two directed edges $e_1$ and $e_2$, we use  the notation $e_1\rightarrow e_2$ to denote that there is a directed path starting  from $e_1$  and ending with $e_2$ and use the notation $e_1\nrightarrow e_2$ to denote that there is no directed path starting  from $e_1$  and ending with $e_2$. Similarly, for two directed edges $v_1$ and $v_2$, we denote the fact that there is a directed path starting with $v_1$ and ending with $v_2$ as $v_1\rightarrow v_2$. If there is no directed path from $v$ to $v'$, we denote this as $v\nrightarrow v_2$. We also introduce the notation $e\rightarrow v$ and $e\nrightarrow v$ to denote that  there is a directed path starting from $e$, ending with $v$ and there is no directed path starting from $e$, ending with $v$, respectively. For a subgraph $\overrightarrow{\Gamma'}$ of $\overrightarrow{\Gamma}$ and an edges $e\notin E(\Gamma')$, the notation $e\rightarrow \Gamma'$ denotes that there is at least one vertex $v\in V(\Gamma')$ such that $e\rightarrow v$ in $\overrightarrow{\Gamma}$, otherwise we denote as $e\nrightarrow \Gamma'$.

A subgraph $\overrightarrow{\Gamma'}$ of directed graph $\overrightarrow{\Gamma}$ is called admissible if it satisfies that
for any directed path $e_1e_2\cdot\cdot\cdot e_n$ in $\overrightarrow{\Gamma}$,
$$e_1, e_n \in E(\overrightarrow{\Gamma})\Longleftrightarrow e_i\in E(\overrightarrow{\Gamma'}), \  1\leq i\leq n.$$

The following proposition is obvious.
\begin{prop}
$\Gamma'\subset\Gamma$ is admissible if and only if $\Gamma/\Gamma'$ is a directed graph.
\end{prop}
\begin{center}
   \textbf{ Polarized, anchored and progressive graph}
\end{center}

A \textbf{polarized} graph is an oriented graph together with a choice of linear
order on $In(v)$ and $Out(v)$ for each vertex $v$. A \textbf{anchored} graph $\overrightarrow{\Gamma}$ is a directed graph together with a choice of linear
orders on $In(\Gamma)$ and $Out(\Gamma)$. A \textbf{progressive} graph is a directed and acyclic graph. In a progressive graph $\overrightarrow{\Gamma}$, we can define a partial order $<$ on the set of edges $E(\Gamma)$ as follows: for any two edges $e_1, e_2\in E(\Gamma)$,
$$e_1< e_2\Longleftrightarrow e_1\rightarrow e_2.$$
So roughly speaking, a progressive graph is an oriented graphs with a  "global flow" induced by the orientation.

\begin{center}
   \textbf{Composition of anchored graphs}
\end{center}

Let $\overrightarrow{\Gamma}_1=(\Gamma_1, sgn_1)$ and $\overrightarrow{\Gamma}_2=(\Gamma_2,sgn_2)$ be two anchored graphs with $Out(\overrightarrow{\Gamma}_1)=\{o_1,...,o_n\}$ and $In(\overrightarrow{\Gamma}_1)=\{i_1,...,i_n\}$. The composition of $\overrightarrow{\Gamma}_1$ and $\overrightarrow{\Gamma}_2$ is a anchored  graph $\overrightarrow{\Gamma}=(\Gamma,sgn)$ with $\Gamma=\Gamma_1\circ^{o_1,...,o_n}_{i_1,...,i_n}\Gamma_2$ and $sgn=(sgn_1\sqcup sgn_2)|_{H(\Gamma)}$.

It can be easily checked that the composition is well-defined, and $In(\overrightarrow{\Gamma})=In(\overrightarrow{\Gamma}_1)$, $Out(\overrightarrow{\Gamma})=Out(\overrightarrow{\Gamma}_2).$ In this situation, we also write $\overrightarrow{\Gamma}=\overrightarrow{\Gamma}_2\circ\overrightarrow{\Gamma}_1.$

\begin{rem}
Unlike non-oriented cases, the symmetry of our notations disappears in oriented cases, that is, $\overrightarrow{\Gamma}_2\circ\overrightarrow{\Gamma}_1$ and $\overrightarrow{\Gamma}_1\circ\overrightarrow{\Gamma}_2$ are different notions, they may not isomorphic as oriented graphs even if both of them exist.
\end{rem}
In the graph $\overrightarrow{\Gamma}=\overrightarrow{\Gamma}_2\circ\overrightarrow{\Gamma}_1$, edges are classified into  three classes: "old" edges that coming from $E(\overrightarrow{\Gamma}_1)-\{\overline{o_1},..., \overline{o_n}\}$ and that coming from $E(\overrightarrow{\Gamma}_2)-\{\overline{i_1},..., \overline{i_n}\}$ and "new" edges  formed through grafting of legs. We denote by $E_o(\overrightarrow{\Gamma_1})$, $E_o(\overrightarrow{\Gamma_2})$ and $E_n(\overrightarrow{\Gamma})$ the sets of the three classes of edges, that is,  $E_o(\overrightarrow{\Gamma_1})=E(\overrightarrow{\Gamma}_1)-\{\overline{o_1},..., \overline{o_n}\}$, $E_o(\overrightarrow{\Gamma_2})=E(\overrightarrow{\Gamma}_2)-\{\overline{i_1},..., \overline{i_n}\}$ and $E_n(\overrightarrow{\Gamma})=\{\overline{e}_1,...,\overline{e}_n\}$ where  any  $\overline{e}_k\in E_n(\overrightarrow{\Gamma})$ is of the form  \begin{equation*}\overline{e}_k=
\begin{cases}
\{o_k,i_k\},& \text{if both }\ o_k\ \text{and}\ i_k\  \text{are real};\\
\{o_k\},& \text{if}\ o_k\  \text{is real and } i_k\  \text{is virtual};\\
\{\sigma_1(o_k),\sigma_2(i_k)\},&\text{if both }\ o_k\  \text{and}\  i_k\  \text{are virtual};\\
\{i_k\},&\text{if}\ o_k\  \text{is virtual and } i_k\  \text{is real},
\end{cases}
\end{equation*}
for some $1\leq k \leq n.$
Obviously, we have $E(\overrightarrow{\Gamma})=E_o(\overrightarrow{\Gamma}_1)\sqcup E_n(\overrightarrow{\Gamma})\sqcup E_o(\overrightarrow{\Gamma}_2).$

\begin{ex}
\

$\bullet$ If
$\overrightarrow{\Gamma}_1=
\begin{tikzpicture}
\draw[fill] (0,0) circle [radius=0.055];
\draw(0.7,0)--(0.055,0)[postaction={decorate, decoration={markings,mark=at position .5 with {\arrowreversed[black]{stealth}}}}];
\draw(-0.7,0)--(-0.055,0)[postaction={decorate, decoration={markings,mark=at position .5 with {\arrow[black]{stealth}}}}];
\node [above] at (-0.4,0) {$h_1$};
\node [above] at (0.4,0) {$h_2$};
\end{tikzpicture}
$,
$\overrightarrow{\Gamma}_2=
\begin{tikzpicture}
\draw[fill] (0,0) circle [radius=0.055];
\draw(0.7,0)--(0.055,0)[postaction={decorate, decoration={markings,mark=at position .5 with {\arrowreversed[black]{stealth}}}}];
\draw(-0.7,0)--(-0.055,0)[postaction={decorate, decoration={markings,mark=at position .5 with {\arrow[black]{stealth}}}}];
\node [above] at (-0.4,0) {$h_3$};
\node [above] at (0.4,0) {$h_4$};
\end{tikzpicture}
$, then $\overrightarrow{\Gamma}_2\circ\overrightarrow{\Gamma}_1=
\begin{tikzpicture}
\draw[fill] (0,0) circle [radius=0.055];
\draw[fill] (1.6,0) circle [radius=0.055];
\draw(0,0)--(1.6,0)[postaction={decorate, decoration={markings,mark=at position .5 with {\arrow[black]{stealth}}}}];
\draw(-0.7,0)--(-0.055,0)[postaction={decorate, decoration={markings,mark=at position .5 with {\arrow[black]{stealth}}}}];

\draw(2.3,0)--(1.6,0)[postaction={decorate, decoration={markings,mark=at position .5 with {\arrowreversed[black]{stealth}}}}];
\node [above] at (-0.4,0) {$h_1$};
\node [above] at (0.4,0) {$h_2$};
\node [above] at (2.0,0) {$h_4$};
\node [above] at (1.2,0) {$h_3$};
\end{tikzpicture}$.

$\bullet$ If
$\overrightarrow{\Gamma}_1=
\begin{tikzpicture}
\draw[fill] (0,0) circle [radius=0.055];
\draw(0.7,0)--(0.055,0)[postaction={decorate, decoration={markings,mark=at position .5 with {\arrowreversed[black]{stealth}}}}];
\draw(-0.7,0)--(-0.055,0)[postaction={decorate, decoration={markings,mark=at position .5 with {\arrow[black]{stealth}}}}];
\node [above] at (-0.4,0) {$h_1$};
\node [above] at (0.4,0) {$h_2$};
\end{tikzpicture}
$,
$\overrightarrow{\Gamma}_2=
\begin{tikzpicture}
\draw (0,0) circle [radius=0.055];
\draw(0.7,0)--(0.055,0)[postaction={decorate, decoration={markings,mark=at position .5 with {\arrowreversed[black]{stealth}}}}];
\draw(-0.7,0)--(-0.055,0)[postaction={decorate, decoration={markings,mark=at position .5 with {\arrow[black]{stealth}}}}];
\node [above] at (-0.4,0) {$h_3$};
\node [above] at (0.4,0) {$h_4$};
\end{tikzpicture}
$,
then
$\overrightarrow{\Gamma}_2\circ\overrightarrow{\Gamma}_1=
\begin{tikzpicture}
\draw [fill](0,0) circle [radius=0.055];
\draw(0.7,0)--(0.055,0)[postaction={decorate, decoration={markings,mark=at position .5 with {\arrowreversed[black]{stealth}}}}];
\draw(-0.7,0)--(-0.055,0)[postaction={decorate, decoration={markings,mark=at position .5 with {\arrow[black]{stealth}}}}];
\node [above] at (-0.4,0) {$h_1$};
\node [above] at (0.4,0) {$h_2$};
\end{tikzpicture}
$.

$\bullet$ If
$\overrightarrow{\Gamma}_1=
\begin{tikzpicture}
\draw (0,0) circle [radius=0.055];
\draw(0.7,0)--(0.055,0)[postaction={decorate, decoration={markings,mark=at position .5 with {\arrowreversed[black]{stealth}}}}];
\draw(-0.7,0)--(-0.055,0)[postaction={decorate, decoration={markings,mark=at position .5 with {\arrow[black]{stealth}}}}];
\node [above] at (-0.4,0) {$h_1$};
\node [above] at (0.4,0) {$h_2$};
\end{tikzpicture}
$,
$\overrightarrow{\Gamma}_2=
\begin{tikzpicture}
\draw [fill] (0,0) circle [radius=0.055];
\draw(0.7,0)--(0.055,0)[postaction={decorate, decoration={markings,mark=at position .5 with {\arrowreversed[black]{stealth}}}}];
\draw(-0.7,0)--(-0.055,0)[postaction={decorate, decoration={markings,mark=at position .5 with {\arrow[black]{stealth}}}}];
\node [above] at (-0.4,0) {$h_3$};
\node [above] at (0.4,0) {$h_4$};
\end{tikzpicture}
$,
then
$\overrightarrow{\Gamma}_2\circ\overrightarrow{\Gamma}_1=
\begin{tikzpicture}
\draw [fill] (0,0) circle [radius=0.055];
\draw(0.7,0)--(0.055,0)[postaction={decorate, decoration={markings,mark=at position .5 with {\arrowreversed[black]{stealth}}}}];
\draw(-0.7,0)--(-0.055,0)[postaction={decorate, decoration={markings,mark=at position .5 with {\arrow[black]{stealth}}}}];
\node [above] at (-0.4,0) {$h_3$};
\node [above] at (0.4,0) {$h_4$};
\end{tikzpicture}
$.

$\bullet$ If
$\overrightarrow{\Gamma}_1=
\begin{tikzpicture}
\draw (0,0) circle [radius=0.055];
\draw(0.7,0)--(0.055,0)[postaction={decorate, decoration={markings,mark=at position .5 with {\arrowreversed[black]{stealth}}}}];
\draw(-0.7,0)--(-0.055,0)[postaction={decorate, decoration={markings,mark=at position .5 with {\arrow[black]{stealth}}}}];
\node [above] at (-0.4,0) {$h_1$};
\node [above] at (0.4,0) {$h_2$};
\end{tikzpicture}
$,
$\overrightarrow{\Gamma}_2=
\begin{tikzpicture}
\draw (0,0) circle [radius=0.055];
\draw(0.7,0)--(0.055,0)[postaction={decorate, decoration={markings,mark=at position .5 with {\arrowreversed[black]{stealth}}}}];
\draw(-0.7,0)--(-0.055,0)[postaction={decorate, decoration={markings,mark=at position .5 with {\arrow[black]{stealth}}}}];
\node [above] at (-0.4,0) {$h_3$};
\node [above] at (0.4,0) {$h_4$};
\end{tikzpicture}
$, then $\overrightarrow{\Gamma}_2\circ\overrightarrow{\Gamma}_1=\begin{tikzpicture}
\draw (0,0) circle [radius=0.055];
\draw(0.7,0)--(0.055,0)[postaction={decorate, decoration={markings,mark=at position .5 with {\arrowreversed[black]{stealth}}}}];
\draw(-0.7,0)--(-0.055,0)[postaction={decorate, decoration={markings,mark=at position .5 with {\arrow[black]{stealth}}}}];
\node [above] at (-0.4,0) {$h_1$};
\node [above] at (0.4,0) {$h_4$};
\end{tikzpicture}
$.
\end{ex}
The following proposition explains the advantage of the way we define the composition.
\begin{prop}
Composition of anchored graphs is associative.
\end{prop}
\begin{proof}
Directly check case by case.
\end{proof}

\section{Progressive planar graphs}
In this section, we will establish a combinatorial theory of  Joyal-Street's (boxed/leveled) progressive plane graphs. We will introduce the notion of a (combinatorial) progressive planar graph and show some of its basic properties.   We will also give the definitions of tensor product and composition of planar graphs. At last we will introduce the notion of the planar geometric realization of a planar graph and prove that planar geometric realization provides an equivalence between our combinatorial theory and the geometric theory of Joyal-Street.

\subsection{Definition of a progressive planar graph}
Here we give a combinatorial definition of a progressive planar graph and show some of its basic properties.
\begin{defn}
A (combinatorial) \textbf{progressive planar structure} on a directed graph $\overrightarrow{\Gamma}$ is a linear order $\preceq$ on the set of edges $E(\Gamma)$ satisfying

$(P_1)$ $e_1\rightarrow e_2$ implies $e_1\prec e_2$, for every two distinct edges $e_1,e_2$;

$(P_2)$    for any distinct edges $e_1,e_2,e_3$, if $e_1\rightarrow e_2$ and $e_1\prec e_3\prec e_2$ then $e_3\rightarrow e_2$ or $e_1\rightarrow e_3$.
\end{defn}
A directed graph $\overrightarrow{\Gamma}$ with a (combinatorial) progressive  planar structure $\prec$ is called a progressive planar graph and is denoted by $(\overrightarrow{\Gamma},\prec)$ or $(\Gamma, \prec)$, $\Gamma$ for convenient. We will say a planar structure and a planar graph for short in stead of a (combinatorial) progressive planar structure and a (combinatorial) progressive planar graph in this paper.  A planar graph is called elementary (prime, essential prime, unitary, invertible ) if the underlying graph is elementary (prime, essential prime, unitary, invertible). If the underly graph of a planar graph is reduced, we call this planar graph reduced.

\begin{ex}
Here we show an example of elementary planar graph.

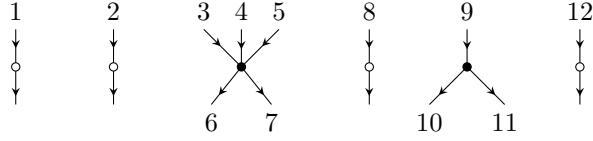
\begin{center}
\begin{tikzpicture}
\draw (0,0) circle [radius=0.055];
\draw (0,-0.055)-- (0,-0.5)[postaction={decorate, decoration={markings,mark=at position .75 with {\arrow[black]{stealth}}}}];
\draw (0,0.055)-- (0,0.5)[postaction={decorate, decoration={markings,mark=at position .75 with {\arrowreversed[black]{stealth}}}}];
\node [above] at (0,0.5) {$1$};

\draw (1.3,0) circle [radius=0.055];
\draw (1.3,-0.055)-- (1.3,-0.5)[postaction={decorate, decoration={markings,mark=at position .75 with {\arrow[black]{stealth}}}}];
\draw (1.3,0.055)-- (1.3,0.5)[postaction={decorate, decoration={markings,mark=at position .75 with {\arrowreversed[black]{stealth}}}}];

\node [above] at (1.3,0.5) {$2$};
\draw[fill] (3,0) circle [radius=0.055];
\draw (3,0)-- (3.5,0.5)[postaction={decorate, decoration={markings,mark=at position .75 with {\arrowreversed[black]{stealth}}}}];
\draw (3,0)-- (2.5,0.5)[postaction={decorate, decoration={markings,mark=at position .75 with {\arrowreversed[black]{stealth}}}}];
\draw (3,0)-- (3,0.5)[postaction={decorate, decoration={markings,mark=at position .75 with {\arrowreversed[black]{stealth}}}}];
\node [above] at (3.5,0.5) {$5$};
\node [above] at (3,0.5) {$4$};
\node [above] at (2.5,0.5) {$3$};

\draw (3,0)-- (2.6,-0.5)[postaction={decorate, decoration={markings,mark=at position .75 with {\arrow[black]{stealth}}}}];
\draw (3,0)-- (3.4,-0.5)[postaction={decorate, decoration={markings,mark=at position .75 with {\arrow[black]{stealth}}}}];
\node [below] at (2.6,-0.5) {$6$};
\node [below] at (3.4,-0.5) {$7$};

\draw (4.7,0) circle [radius=0.055];
\draw (4.7,-0.055)-- (4.7,-0.5)[postaction={decorate, decoration={markings,mark=at position .75 with {\arrow[black]{stealth}}}}];
\draw (4.7,0.055)-- (4.7,0.5)[postaction={decorate, decoration={markings,mark=at position .75 with {\arrowreversed[black]{stealth}}}}];

\node [above] at (4.7,0.5) {$8$};
\draw[fill] (6,0) circle [radius=0.055];
\draw (6,0)-- (6,0.5)[postaction={decorate, decoration={markings,mark=at position .75 with {\arrowreversed[black]{stealth}}}}];
\draw (6,0)-- (6.5,-0.5)[postaction={decorate, decoration={markings,mark=at position .75 with {\arrow[black]{stealth}}}}];
\draw (6,0)-- (5.5,-0.5)[postaction={decorate, decoration={markings,mark=at position .75 with {\arrow[black]{stealth}}}}];

\node [above] at (6,0.5) {$9$};
\node [below] at (6.5,-0.5) {$11$};
\node [below] at (5.5,-0.5) {$10$};

\draw (7.5,0) circle [radius=0.055];
\draw (7.5,-0.055)-- (7.5,-0.5)[postaction={decorate, decoration={markings,mark=at position .75 with {\arrow[black]{stealth}}}}];
\draw (7.5,0.055)-- (7.5,0.5)[postaction={decorate, decoration={markings,mark=at position .75 with {\arrowreversed[black]{stealth}}}}];

\node [above] at (7.5,0.5) {$12$};

\end{tikzpicture}
\end{center}
where the numbers labeling edges represent the linear order $\prec$ on the set of edges.
\end{ex}

\begin{ex}
We will show an example of general reduced planar graph.

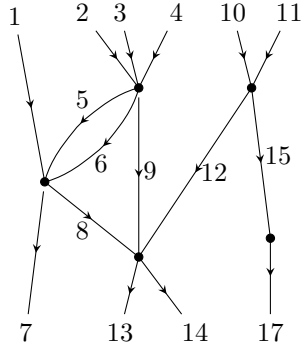
\begin{center}
\begin{tikzpicture}[scale=0.5]

\node (v2) at (-4,3) {};
\node (v1) at (-1.5,5.5) {};
\node (v7) at (-1.5,1) {};
\node (v9) at (1.5,5.5) {};
\node (v14) at (2,1.5) {};
\node (v3) at (-3,7.5) {$2$};
\node (v4) at (-2,7.5) {$3$};
\node (v5) at (-0.5,7.5) {$4$};
\node (v6) at (-4.8,7.4) {$1$};
\node (v11) at (-4.5,-1) {$7$};
\node (v12) at (-2,-1) {$13$};
\node (v13) at (0,-1) {$14$};
\node (v15) at (2,-1) {$17$};
\node (v8) at (1,7.5) {$10$};
\node (v10) at (2.5,7.5) {$11$};
\node  at (-2.5,3.5) {$6$};
\node  at (-3,5.2) {$5$};
\node  at (-1.2,3.3) {$9$};
\node  at (0.5,3.25) {$12$};
\node  at (2.2,3.7) {$15$};
\node  at (-3,1.7) {$8$};
\draw[fill] (-4,3) circle [radius=0.11];
\draw[fill] (v1) circle [radius=0.11];
\draw[fill] (v7) circle [radius=0.11];
\draw[fill] (v9) circle [radius=0.11];
\draw[fill] (v14) circle [radius=0.11];
\draw  plot[smooth, tension=1] coordinates {(v1) (-2.5,5)  (-3.5,4) (v2)}[postaction={decorate, decoration={markings,mark=at position .5 with {\arrow[black]{stealth}}}}];
\draw  plot[smooth, tension=1] coordinates {(v1) (-2,4.5)  (-3,3.5) (v2)}[postaction={decorate, decoration={markings,mark=at position .5 with {\arrow[black]{stealth}}}}];

\draw  (v3) -- (-1.5,5.5)[postaction={decorate, decoration={markings,mark=at position .5 with {\arrow[black]{stealth}}}}];
\draw  (v4) -- (-1.5,5.5)[postaction={decorate, decoration={markings,mark=at position .5 with {\arrow[black]{stealth}}}}];
\draw  (v5) -- (-1.5,5.5)[postaction={decorate, decoration={markings,mark=at position .5 with {\arrow[black]{stealth}}}}];

\draw  (v6) -- (-4,3)[postaction={decorate, decoration={markings,mark=at position .5 with {\arrow[black]{stealth}}}}];
\draw  (v1) -- (-1.5,1)[postaction={decorate, decoration={markings,mark=at position .5 with {\arrow[black]{stealth}}}}];
\draw  (-4,3) -- (-1.5,1)[postaction={decorate, decoration={markings,mark=at position .5 with {\arrow[black]{stealth}}}}];

\draw  (v8)--(1.5,5.5)[postaction={decorate, decoration={markings,mark=at position .5 with {\arrow[black]{stealth}}}}];
\draw  (v10) -- (1.5,5.5)[postaction={decorate, decoration={markings,mark=at position .5 with {\arrow[black]{stealth}}}}];
\draw  (1.5,5.5) -- (-1.5,1)[postaction={decorate, decoration={markings,mark=at position .5 with {\arrow[black]{stealth}}}}];
\draw  (v2) -- (v11)[postaction={decorate, decoration={markings,mark=at position .5 with {\arrow[black]{stealth}}}}];
\draw  (-1.5,1) -- (v12)[postaction={decorate, decoration={markings,mark=at position .65 with {\arrow[black]{stealth}}}}];
\draw  (v13) -- (-1.5,1)[postaction={decorate, decoration={markings,mark=at position .5 with {\arrowreversed[black]{stealth}}}}];
\draw  (1.5,5.5) -- (2,1.5)[postaction={decorate, decoration={markings,mark=at position .5 with {\arrow[black]{stealth}}}}];
\draw  (2,1.5) -- (v15)[postaction={decorate, decoration={markings,mark=at position .5 with {\arrow[black]{stealth}}}}];

\end{tikzpicture}
\end{center}
\end{ex}

Evidently, for any planar graph $(\overrightarrow{\Gamma}, \prec)$, the orientation $sgn$ and planar structure $\prec$ together can define an unique linear order $\prec_H$ on the sets $H(\Gamma)$ of half-edges such that:

$\bullet$ if $h_2\neq\sigma(h_1)$, then $h_1\prec_H h_2\Longleftrightarrow \overline{h_1}\prec \overline{h_2}$;

$\bullet$  if $h_2=\sigma(h_1)$, then $h_1\prec_H h_2\Longleftrightarrow sgn(h_1)=-.$

An \textbf{equivalence} or \textbf{isomorphism} of $(\overrightarrow{\Gamma}_1, \prec_1)$ and $(\overrightarrow{\Gamma}_2, \prec_2)$ is an equivalence $\phi:\Gamma_1\rightarrow \Gamma_2$ of graphs that preserves the linear orders on sets of half-edges $H(\Gamma_1)$ and $H(\Gamma_2)$, i.e, for any $h_1,h_2\in H(\Gamma_1)$ $$h_1\prec_{H(\Gamma_1)} h_2\Longleftrightarrow \phi(h_1)\prec_{H(\Gamma_2)} \phi(h_2).$$
We say $(\overrightarrow{\Gamma}_1, \prec_1)$ is equivalent to  $(\overrightarrow{\Gamma}_2, \prec_2)$ if there is an equivalence between them, and we denote this fact as $(\overrightarrow{\Gamma}_1, \prec_1)\cong(\overrightarrow{\Gamma}_2, \prec_2)$. All planar graphs and their equivalence form a groupoid. We denote the set of isomorphic classes of planar graphs by $\mathsf{\Gamma}$. We usually do not make a distinction  between a planar graph and its isomorphism class. A planar graph $(\overrightarrow{\Gamma},\prec)$ with $|In(\overrightarrow{\Gamma})|=m$ and $|Out(\overrightarrow{\Gamma})|=n$ is called an $(m,n)$-planar graph.

\begin{ex} The linear order on half-edges of the planar graph in example 3.1.2 is shown as
\begin{center}
\begin{tikzpicture}

\draw (0,0) circle [radius=0.055];
\draw (0,-0.055)-- (0,-0.5)[postaction={decorate, decoration={markings,mark=at position .75 with {\arrow[black]{stealth}}}}];
\draw (0,0.055)-- (0,0.5)[postaction={decorate, decoration={markings,mark=at position .75 with {\arrowreversed[black]{stealth}}}}];
\draw (1.3,0) circle [radius=0.055];
\draw (1.3,-0.055)-- (1.3,-0.5)[postaction={decorate, decoration={markings,mark=at position .75 with {\arrow[black]{stealth}}}}];
\draw (1.3,0.055)-- (1.3,0.5)[postaction={decorate, decoration={markings,mark=at position .75 with {\arrowreversed[black]{stealth}}}}];
\draw[fill] (3,0) circle [radius=0.055];
\draw (3,0)-- (3.5,0.5)[postaction={decorate, decoration={markings,mark=at position .75 with {\arrowreversed[black]{stealth}}}}];
\draw (3,0)-- (2.5,0.5)[postaction={decorate, decoration={markings,mark=at position .75 with {\arrowreversed[black]{stealth}}}}];
\draw (3,0)-- (3,0.5)[postaction={decorate, decoration={markings,mark=at position .75 with {\arrowreversed[black]{stealth}}}}];
\draw (3,0)-- (2.6,-0.5)[postaction={decorate, decoration={markings,mark=at position .75 with {\arrow[black]{stealth}}}}];
\draw (3,0)-- (3.4,-0.5)[postaction={decorate, decoration={markings,mark=at position .75 with {\arrow[black]{stealth}}}}];
\draw (4.7,0) circle [radius=0.055];
\draw (4.7,-0.055)-- (4.7,-0.5)[postaction={decorate, decoration={markings,mark=at position .75 with {\arrow[black]{stealth}}}}];
\draw (4.7,0.055)-- (4.7,0.5)[postaction={decorate, decoration={markings,mark=at position .75 with {\arrowreversed[black]{stealth}}}}];
\draw[fill] (6,0) circle [radius=0.055];
\draw (6,0)-- (6,0.5)[postaction={decorate, decoration={markings,mark=at position .75 with {\arrowreversed[black]{stealth}}}}];
\draw (6,0)-- (6.5,-0.5)[postaction={decorate, decoration={markings,mark=at position .75 with {\arrow[black]{stealth}}}}];
\draw (6,0)-- (5.5,-0.5)[postaction={decorate, decoration={markings,mark=at position .75 with {\arrow[black]{stealth}}}}];
\draw (7.5,0) circle [radius=0.055];
\draw (7.5,-0.055)-- (7.5,-0.5)[postaction={decorate, decoration={markings,mark=at position .75 with {\arrow[black]{stealth}}}}];
\draw (7.5,0.055)-- (7.5,0.5)[postaction={decorate, decoration={markings,mark=at position .75 with {\arrowreversed[black]{stealth}}}}];

\node [below] at (0,-0.5) {$2$};
\node [above] at (0,0.5) {$1$};
\node [below] at (1.3,-0.5) {$4$};
\node [above] at (1.3,0.5) {$3$};
\node [above] at (3.5,0.5) {$7$};
\node [above] at (3,0.5) {$6$};
\node [above] at (2.5,0.5) {$5$};
\node [below] at (2.6,-0.5) {$8$};
\node [below] at (3.4,-0.5) {$9$};
\node [below] at (4.7,-0.5) {$11$};
\node [above] at (4.7,0.5) {$10$};
\node [above] at (6,0.5) {$12$};
\node [below] at (6.5,-0.5) {$14$};
\node [below] at (5.5,-0.5) {$13$};
\node [below] at (7.5,-0.5) {$16$};
\node [above] at (7.5,0.5) {$15$};

\end{tikzpicture}.
\end{center}

\end{ex}

\begin{ex}
 The linear order on half-edges of the planar graph in example 3.1.3 is shown as
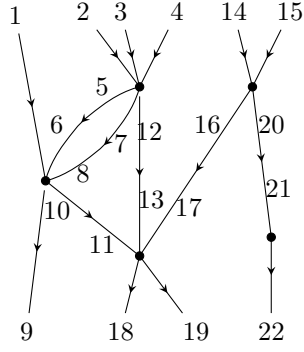
\begin{center}
\begin{tikzpicture}[scale=0.5]

\node (v2) at (-4,3) {};
\node (v1) at (-1.5,5.5) {};
\node (v7) at (-1.5,1) {};
\node (v9) at (1.5,5.5) {};
\node (v14) at (2,1.5) {};
\node (v3) at (-3,7.5) {$2$};
\node (v4) at (-2,7.5) {$3$};
\node (v5) at (-0.5,7.5) {$4$};
\node (v6) at (-4.8,7.4) {$1$};
\node (v11) at (-4.5,-1) {$9$};
\node (v12) at (-2,-1) {$18$};
\node (v13) at (0,-1) {$19$};
\node (v15) at (2,-1) {$22$};
\node (v8) at (1,7.5) {$14$};
\node (v10) at (2.5,7.5) {$15$};
\node  at (-2,4) {$7$};
\node  at (-3,3.2) {$8$};

\node  at (-3.7,2.2) {$10$};
\node  at (-2.5,1.3) {$11$};

\node  at (-2.5,5.5) {$5$};
\node  at (-3.7,4.5) {$6$};

\node  at (-1.25,4.3) {$12$};
\node  at (-1.2,2.5) {$13$};

\node  at (0.3,4.5) {$16$};
\node  at (-0.2,2.3) {$17$};

\node  at (2,4.5) {$20$};
\node  at (2.2,2.8) {$21$};

\draw[fill] (-4,3) circle [radius=0.11];
\draw[fill] (v1) circle [radius=0.11];
\draw[fill] (v7) circle [radius=0.11];
\draw[fill] (v9) circle [radius=0.11];
\draw[fill] (v14) circle [radius=0.11];
\draw  plot[smooth, tension=1] coordinates {(v1) (-2.5,5)  (-3.5,4) (v2)}[postaction={decorate, decoration={markings,mark=at position .5 with {\arrow[black]{stealth}}}}];
\draw  plot[smooth, tension=1] coordinates {(v1) (-2,4.5)  (-3,3.5) (v2)}[postaction={decorate, decoration={markings,mark=at position .5 with {\arrow[black]{stealth}}}}];

\draw  (v3) -- (-1.5,5.5)[postaction={decorate, decoration={markings,mark=at position .5 with {\arrow[black]{stealth}}}}];
\draw  (v4) -- (-1.5,5.5)[postaction={decorate, decoration={markings,mark=at position .5 with {\arrow[black]{stealth}}}}];
\draw  (v5) -- (-1.5,5.5)[postaction={decorate, decoration={markings,mark=at position .5 with {\arrow[black]{stealth}}}}];

\draw  (v6) -- (-4,3)[postaction={decorate, decoration={markings,mark=at position .5 with {\arrow[black]{stealth}}}}];
\draw  (v1) -- (-1.5,1)[postaction={decorate, decoration={markings,mark=at position .5 with {\arrow[black]{stealth}}}}];
\draw  (-4,3) -- (-1.5,1)[postaction={decorate, decoration={markings,mark=at position .5 with {\arrow[black]{stealth}}}}];

\draw  (v8)--(1.5,5.5)[postaction={decorate, decoration={markings,mark=at position .5 with {\arrow[black]{stealth}}}}];
\draw  (v10) -- (1.5,5.5)[postaction={decorate, decoration={markings,mark=at position .5 with {\arrow[black]{stealth}}}}];
\draw  (1.5,5.5) -- (-1.5,1)[postaction={decorate, decoration={markings,mark=at position .5 with {\arrow[black]{stealth}}}}];
\draw  (v2) -- (v11)[postaction={decorate, decoration={markings,mark=at position .5 with {\arrow[black]{stealth}}}}];
\draw  (-1.5,1) -- (v12)[postaction={decorate, decoration={markings,mark=at position .65 with {\arrow[black]{stealth}}}}];
\draw  (v13) -- (-1.5,1)[postaction={decorate, decoration={markings,mark=at position .5 with {\arrowreversed[black]{stealth}}}}];
\draw  (1.5,5.5) -- (2,1.5)[postaction={decorate, decoration={markings,mark=at position .5 with {\arrow[black]{stealth}}}}];
\draw  (2,1.5) -- (v15)[postaction={decorate, decoration={markings,mark=at position .5 with {\arrow[black]{stealth}}}}];

\end{tikzpicture}.
\end{center}

\end{ex}

Here we summarize some notions we have introduced:

\begin{center}
\begin{tabular}{|l|l|}
\hline
oriented graph & graph+orientation\\ \hline
directed  graph & orientation+$In(v)$\&$Out(v)$ nonempty\\\hline
progressive graph & directed + acyclic\\\hline
polarized graph& $In(v)\& Out(v)$ totally ordered\\\hline
anchored graph& directed + $In(\Gamma)\& Out(\Gamma)$ totally ordered\\\hline
planar graph & progressive + $\prec$ \\\hline
\end{tabular}
\end{center}

For simplicity, in this paper we use the notation $e_1\prec e_2$ to denote $e_1\preceq e_2$ and  $e_1\neq e_2$. And in almost of all  cases we discuss  about a linear order, we will restrict on the cases of different elements if we do not point out clearly.

The following proposition shows that in the definition of a planar graph, the condition $(P_2)$ can be replaced by other equivalent conditions.

\begin{prop}\label{three conditions}
If $\overrightarrow{\Gamma}$ is a directed graph with a linear order $\prec$ on the set of edges $E(\Gamma)$.
Then the following three conditions are equivalent:

$(P_2)$    for any three edges $e_1,e_2,e_3$, if $e_1\rightarrow e_2$ and $e_1\prec e_3\prec e_2$ then $e_3\rightarrow e_2$ or $e_1\rightarrow e_3$;

$(P_2^{r})$ for any three edges $e_1,e_2,e_3$, if $e_1\rightarrow e_2$, $e_1\prec e_3$ and $e_1 \nrightarrow e_3$ then $e_2\prec e_3$ or $e_3\rightarrow e_2$;

$(P_2^{l})$  for any three edges $e_1,e_2,e_3$, if $e_1\rightarrow e_2$, $e_3\prec e_2$ and $e_3 \nrightarrow e_2$ then $e_3\prec e_1$ or $e_1\rightarrow e_3$.

$(\widetilde{P_2})$    for any three edges $e_1,e_2,e_3$, if $e_1e_2$ is a directed path, that is, $t(e_1)=s(e_2)$, and $e_1\prec e_3\prec e_2$ then $e_3\rightarrow e_2$ or $e_1\rightarrow e_3$;

$(\widetilde{P_2^{r}})$ for any three edges $e_1,e_2,e_3$, if $e_1e_2$ is a directed path, $e_1\prec e_3$ and $e_1 \nrightarrow e_3$ then $e_2\prec e_3$ or $e_3\rightarrow e_2$;

$(\widetilde{P_2^{l}})$  for any three edges $e_1,e_2,e_3$,  if $e_1e_2$ is a directed path, $e_3\prec e_2$ and $e_3 \nrightarrow e_2$ then $e_3\prec e_1$ or $e_1\rightarrow e_3$.
\end{prop}

\begin{proof}
$\bullet$  $(P_2)\Longrightarrow  (P_2^{r})$. Assume  $e_1\rightarrow e_2$, $e_1\prec e_3$ and $e_1 \nrightarrow e_3$, we want to prove $e_2\prec e_3$ or $e_3\rightarrow e_2$ using condition $(P_2)$. By the linearity of $\prec$ we have  $e_2\prec e_3$ or $e_3\prec e_2$. If $e_2\prec e_3$, we complete our proof. If  $e_3\prec e_2$, by the assumption  we have  $e_1\prec e_3\prec e_2$. By condition $(P_2)$, we have $e_3\rightarrow e_2$ or $e_1\rightarrow e_3$. But we have assumed $e_1 \nrightarrow e_3$, so we must have $e_3\rightarrow e_2$.

$\bullet$ $(P^r_2)\Longrightarrow  (P_2)$. Assume $e_1\rightarrow e_2$ and $e_1\prec e_3\prec e_2$, we want to prove $e_3\rightarrow e_2$ or $e_1\rightarrow e_3$ using condition $(P_2^{r})$. If $e_1\rightarrow e_3$, we complete the proof. If $e_1 \nrightarrow e_3$, we have $e_2\prec e_3$ or $e_3\rightarrow e_2$ according to condition $(P^r_2)$. If $e_3\rightarrow e_2$, we complete the proof. If $e_2\prec e_3$, we have a contradiction with assumption $e_1\prec e_3\prec e_2$, so we must have $e_3\rightarrow e_2$.

$\bullet$ $(P_2)\Longrightarrow  (P_2^{l})$. Assume  $e_1\rightarrow e_2$, $e_3\prec e_2$ and $e_3 \nrightarrow e_2$, we want to prove that  $e_3\prec e_1$ or $e_1\rightarrow e_3$ using condition $(P_2)$. By the linearity of $\prec$ we have  $e_1\prec e_3$ or $e_3\prec e_1$. If $e_3\prec e_1$, we complete the proof. If $e_1\prec e_3$, then by  assumption we have $e_1\prec e_3\prec e_2$. By condition $(P_2)$, we have $e_3\rightarrow e_2$ or $e_1\rightarrow e_3$. But we have assumed $e_3 \nrightarrow e_2$, we must have $e_1\rightarrow e_3$, thus condition $(P_2^{l})$ is satisfied.

$\bullet$ $(P^l_2)\Longrightarrow  (P_2)$. Assume $e_1\rightarrow e_2$ and $e_1\prec e_3\prec e_2$, we want to prove $e_3\rightarrow e_2$ or $e_1\rightarrow e_3$ using condition $(P_2^{l})$. If $e_3\rightarrow e_2$, we complete the proof. If $e_3 \nrightarrow e_2$, we have $e_3\prec e_1$ or $e_1\rightarrow e_3$ according to condition $(P^l_2)$. If $e_1\rightarrow e_3$, we complete the proof. If $e_3\prec e_1$, we have a contradiction with assumption $e_1\prec e_3\prec e_2$, so we must have $e_1\rightarrow e_3$.

$\bullet$  $(P_2)\Longrightarrow  (\widetilde{P_2})$.  Obviously.

$\bullet$  $(\widetilde{P_2})\Longrightarrow  (P_2)$. If $e_1\rightarrow e_2$ and $e_1\prec e_3\prec e_2$, the linearity of $\prec$ implies that there must exists an unique direct path  $e_ie_j$ such that $e_1\rightarrow e_i$, $e_j\rightarrow e_2$ and $e_i\prec e_3\prec e_j$. Then condition $(\widetilde{P_2})$ says that $e_i\rightarrow e_3$ or $e_3\rightarrow e_j$, which imply that $e_1\rightarrow e_3$ or $e_3\rightarrow e_2$ respectively. So condition $(P_2)$ is satisfied.

The proof of the  two equivalences  $\widetilde{P_2^{l}}\Longleftrightarrow P_2^{l}$ and $\widetilde{P_2^{r}}\Longleftrightarrow P_2^{r}$ are similar.
\end{proof}

It is not too difficult to prove the following proposition:

\begin{prop}
If $\overrightarrow{\Gamma}'\subseteq\overrightarrow{\Gamma}$ is a directed subgraph of $\overrightarrow{\Gamma}$, then any planar structure $\prec$ on $\overrightarrow{\Gamma}$ naturally induces a planar structure $\prec'$ on $\overrightarrow{\Gamma}'$.
\end{prop}

Evidently, for Joyal and Street's progressive plane graphs, their opposite and mirror are progressive plane graphs naturally.  These equivalence of conditions  $(P_2), (P_2^r)$ and $(P_2^l)$ in proposition $3.1.6$ can be seen as an reflection of these phenomena in some sense. Also this fact has some connections with the fact that to any strict tensor categories, we can associate other three strict tensor category as discussed in Remark $7.1.1$.

\begin{rem}
We define the opposite graph of an oriented graph $\overrightarrow{\Gamma}$ to be an oriented graph obtained by reversing the orientation of every edges of $\overrightarrow{\Gamma}$, and denote it by $\overrightarrow{\Gamma}^{op}$. If there is an planar structure $\prec$ on $\overrightarrow{\Gamma}$, then there is a natural planar structure $\prec^{op}$ on $\overrightarrow{\Gamma}^{op}$ defined as $e_1\prec^{op}e_2\Longleftrightarrow e_2\prec e_1$, and we call the planar graph $(\overrightarrow{\Gamma}^{op},\prec^{op})$ opposite planar graph of $(\overrightarrow{\Gamma},\prec)$. The operation $(-)^{op}$ is a involution on the groupoid of planar graphs.
The condition $(P^l_2)$ and $(P^r_2)$ are dual to each other in the sense that $(\overrightarrow{\Gamma},\prec)$ satisfies condition $(P^l_2)$ if and only if $(\overrightarrow{\Gamma}^{op},\prec^{op})$ satisfies condition $(P^r_2)$.
\end{rem}

In fact, proposition \ref{three conditions}  has an abstract form, that is, we have the following proposition:
\begin{prop}\label{equivalence}
Let $(S,\rightarrow)$ be a set with a partial order $\rightarrow$. Let $\prec$ be a linear order which  extends the partial order $\rightarrow$. Then the following properties are equivalent:

$(P_2)$    for any three elements $s_1,s_2,s_3$, if $s_1\rightarrow s_2$ and $s_1\prec s_3\prec s_2$ then $s_3\rightarrow s_2$ or $s_1\rightarrow s_3$;

$(P_2^{r})$ for any three elements $s_1,s_2,s_3$, if $s_1\rightarrow s_2$, $s_1\prec s_3$ and $s_1 \nrightarrow s_3$ then $s_2\prec s_3$ or $s_3\rightarrow s_2$;

$(P_2^{l})$  for any three elements $s_1,s_2,s_3$, if $s_1\rightarrow s_2$, $s_3\prec s_2$ and $s_3 \nrightarrow s_2$ then $s_3\prec s_1$ or $s_1\rightarrow s_3$.

\end{prop}
\begin{proof}
$\bullet$ $(P_2)\Longrightarrow (P_2^{r}).$ If  $s_1\rightarrow s_2$, $s_1\prec s_3$ and $s_1 \nrightarrow s_3$, we want to prove that $s_2\prec s_3$ or $s_3\rightarrow s_2$ using condition $(P_2^{r})$. If $s_2\prec s_3$, we complete the proof. Otherwise, we have $s_3\prec s_2$, hence $s_1\prec s_3\prec s_2$, using the fact that $s_1\rightarrow s_2$ and condition $(P_2)$, we have $s_3\rightarrow s_2$ or $s_1\rightarrow s_3$. But we have assumed $s_1 \nrightarrow s_3$, thus we must have $s_3\rightarrow s_2$.

$\bullet$ $(P_2^{r})\Longrightarrow (P_2).$  If $s_1\rightarrow s_2$ and $s_1\prec s_3\prec s_2$, we want to prove $s_3\rightarrow s_2$ or $s_1\rightarrow s_3$ using condition $(P_2)$. If $s_1\rightarrow s_3$, we complete the proof. Otherwise, we have $s_1\nrightarrow s_3$, we have $s_2\prec s_3$ or $s_3\rightarrow s_2$ according to condition $(P_2)$. Noticing that we have assumed $s_3\prec s_2$,  we must have $s_3\rightarrow s_2$.

$\bullet$ $(P_2^{r})\Longrightarrow (P_2^l).$ If $s_1\rightarrow s_2$, $s_3\prec s_2$ and $s_3 \nrightarrow s_2$, we want to prove  $s_3\prec s_1$ or $s_1\rightarrow s_3$ using condition $(P_2^r)$. If $s_1\rightarrow s_3$, we complete the proof. Otherwise, we have $s_1\nrightarrow s_3$ and we discuss in the following two cases: if $s_3\prec s_1$, we complete the proof. Otherwise, we have $s_1\prec s_3$, notice that $s_1\nrightarrow s_3$ and $s_1\rightarrow s_2$, then by condition $(P_2^{r})$, we have $s_2\prec s_3$ or $s_3\rightarrow s_2$. But we have assumed $s_3\prec s_2$ and $s_3 \nrightarrow s_2$, we get a contradiction, thus we must have $s_3\prec s_1$.

$\bullet$ $(P_2^{l})\Longrightarrow (P_2^r).$ If $s_1\rightarrow s_2$, $s_1\prec s_3$ and $s_1 \nrightarrow s_3$, we want to prove $s_2\prec s_3$ or $s_3\rightarrow s_2$ using condition $(P_2^{l})$. If $s_3\rightarrow s_2$, we complete the proof. Otherwise, we have $s_3\nrightarrow s_2$ and we discuss int the following two cases: if $s_2\prec s_3$, we complete the proof. Otherwise, we have $s_3\prec s_2$, notice that $s_3 \nrightarrow s_2$ and $s_1\rightarrow s_2$, then by condition $(P_2^{l})$, we have  $s_3\prec s_1$ or $s_1\rightarrow s_3$. But we have assumed $s_1\prec s_3$ and $s_1 \nrightarrow s_3$, we get a contradiction, thus we must have $s_2\prec s_3$.
\end{proof}

Now it is reasonable to introduce the following notion:
\begin{defn}
A planar structure on a finite  partial set $(S,\rightarrow)$ is a linear order $\prec$ such that

$\bullet$ $\prec$ is a linear extension of the partial order  $\rightarrow$, i.e, $\prec$ is a linear order such that $s_1\rightarrow s_2$ implies $s_1\prec s_2$, for every two distinct elements $s_1. s_2$ of $S$;

$\bullet$  for any three elements $s_1,s_2,s_3$ of $S$, if $s_1\rightarrow s_2$ and $s_1\prec s_3\prec s_2$ then $s_3\rightarrow s_2$ or $s_1\rightarrow s_3$.

\end{defn}
A finite partial set with a planar structure  is called a planar set. Recall that Hasse diagram is a geometric way to represent a finite partial set, so it is reasonable to think that the Hasse diagram of a planar set should be drawn in the plane. In section 3.6, we will show that it is indeed the case.

Similarly to proposition 3.1.7, we have
\begin{prop}
Let $(S,\rightarrow,\prec)$ be a planar set, $T\subseteq S$ be a non-empty subset, then there is a natural planar structure on $T$ induced from $(S,\rightarrow,\prec)$.
\end{prop}

\subsection{Some properties of a planar graph}
This section is devoted to introduce some common properties of a planar graph, and some of them will be useful to prove that composition of planar graphs are planar graphs (Theorem 3.5.5).

\begin{lem}[A]
If $e_1\rightarrow e\leftarrow e_2$ and $e_1\prec e'\prec e_2$, then $e_1\nrightarrow e'$ implies $e'\rightarrow e$.
\end{lem}
\begin{proof}
By condition $(P_1)$, the fact $e\leftarrow e_2$ implies $e_2\prec e$. By the linearity of $\prec$ and the fact $e'\prec e_2$ , we have $e'\prec e$. According to condition $(P^r_2)$, $e_1\rightarrow e$ and $e_1\prec e'$ imply $e\prec e'$ or $e'\rightarrow e$. If $e'\rightarrow e$, we complete the proof. If $e\prec e'$, then due to $e'\prec e_2$ we have $e\prec e_2$ which is contract to $e_2\prec e$. So $e\prec e'$ is impossible and thus in this case we must have $e'\rightarrow e$.
\end{proof}

\begin{lem}[B]
If $e_1\leftarrow e\rightarrow e_2$ and $e_1\prec e'\prec e_2$, then $e'\nrightarrow e_2$ implies $e\rightarrow e'$.
\end{lem}

\begin{proof}
By condition $(P_1)$, the fact $e\rightarrow e_1$ implies $e\prec e_1$. By the linearity of $\prec$ and the fact $e_1\prec e'$ , we have $e\prec e'$. According to condition $(P^l_2)$, $e\rightarrow e_2$ and $e'\prec e_2$ imply $e'\prec e$ or $e\rightarrow e'$. If $e\rightarrow e'$, we complete the proof. If $e'\prec e$, then due to $e_1\prec e'$ we have $e_1\prec e$ which is contract to $e\prec e_1$. So $e'\prec e$ is impossible and thus in this case we must have $e'\rightarrow e$.
\end{proof}

For an edge $e$ of planar graph $(\overrightarrow{\Gamma}, \prec)$, we introduce some notations:

$$i_{min}(e)=min\{i_k\in In(\overrightarrow{\Gamma})| \overline{i_k}\rightarrow e\},$$
$$i_{max}(e)=max\{i_k\in In(\overrightarrow{\Gamma})| \overline{i_k}\rightarrow e\},$$
$$o_{min}(e)=min\{o_k\in Out(\overrightarrow{\Gamma})| e\rightarrow \overline{o_k}\},$$
$$o_{max}(e)=max\{o_k\in Out(\overrightarrow{\Gamma})| e\rightarrow \overline{o_k}\}.$$

\begin{lem}[A'] For any input $i\in In(\overrightarrow{\Gamma})$ and any edge $e\in E(\overrightarrow{\Gamma})$, we have
 $$\overline{i_{min}(e)}\preceq \overline{i}\preceq \overline{i_{max}(e)}\Longleftrightarrow \overline{i}\rightarrow e.$$
\end{lem}
\begin{proof} The $\Longleftarrow$ direction is obvious. To show that the $\Longleftarrow$ direction is a direct consequence of lemma A, we only need to  notice that $\overline{i_{min}(e)}\nrightarrow  \overline{i}$ which is implied by the fact that $i$ is an input.
\end{proof}

\begin{lem}[B'] For any output $o\in In(\overrightarrow{\Gamma})$ and any edge $e\in E(\overrightarrow{\Gamma})$, we have
 $$\overline{o_{min}(e)}\preceq \overline{o}\preceq \overline{o_{max}(e)}\Longleftrightarrow e\rightarrow \overline{o}.$$
\end{lem}
\begin{proof} The $\Longleftarrow$ direction is obvious. To show that the $\Longleftarrow$ direction is a direct consequence of lemma B, we only need to  notice that $\overline{o}\nrightarrow \overline{o_{max}(e)}$ which is implied by the fact that $o$ is an output.
\end{proof}

\begin{lem}[A''] For any edge $e\in E(\overrightarrow{\Gamma})$, if  $\overline{i_k}\prec e$ , $1\leq k\leq n$, then we have
$$\overline{i_k}\preceq \overline{i_{max}(e)}.$$ Moreover, if $k=n$, we have $\overline{ i_{max}(e)}=\overline{i_n}.$
\end{lem}
\begin{proof}
If  $\overline{i_k}\prec e$ and assume $ \overline{i_{max}(e)}\prec \overline{i_k}$, then we have $\overline{i_{max}(e)}\prec \overline{i_k}\prec e$. Notice that $\overline{i_{max}(e)}\rightarrow e$,  so according to condition $(P_2)$, we have $\overline{i_{max}(e)}\rightarrow \overline{i_k}$ or $\overline{i_k}\rightarrow e$. But both of them are impossible. In fact, if $\overline{i_k}\rightarrow e$, then by the definition of $\overline{i_{max}(e)}$ we have $\overline{i_k}\preceq\overline{i_{max}(e)}$ which is contract to $ \overline{i_{max}(e)}\prec \overline{i_k}$, so $\overline{i_k}\rightarrow e$ is impossible. Notice that $\overline{i_k}$ is an input edge, so $\overline{i_{max}(e)}\rightarrow \overline{i_k}$ is impossible. Thus we must have $\overline{i_k}\preceq \overline{i_{max}(e)}.$ It is obviously, if $k=n$, we have $\overline{ i_{max}(e)}=\overline{i_n}.$
\end{proof}

\begin{lem}[B''] For any edge $e\in E(\overrightarrow{\Gamma})$, if $e\prec \overline{o_k}$, $1\leq k\leq n$, then we have $$\overline{o_{min}(e)}\preceq \overline{o_k}.$$ Moreover, if $k=1$, we have $\overline{ o_{min}(e)}=\overline{o_1}.$
\end{lem}
\begin{proof}
If $e\prec \overline{o_k}$ and assume $\overline{o_k}\prec\overline{ o_{min}(e)}$, then we have $e\prec \overline{o_k}\prec\overline{ o_{min}(e)}$. Notice that $e\rightarrow \overline{ o_{min}(e)}$, so according to condition $(P_2)$, we have $e\rightarrow \overline{o_k}$ or $\overline{o_k} \rightarrow \overline{o_{min}(e)}$.  But both of them are impossible. In fact, if $e\rightarrow \overline{o_k}$, then by the definition of $ o_{min}(e)$  we have $\overline{ o_{min}(e)} \preceq\overline{o_k}$, which is contract to $\overline{o_k}\prec\overline{ o_{min}(e)}$, so $e\rightarrow \overline{o_k}$ is impossible. Notice that $\overline{o_k}$ is an output edge, so $\overline{o_k}\rightarrow\overline{ o_{min}(e)}$ is impossible. Thus we must have $\overline{ o_{min}(e)}\preceq\overline{o_k}$. Obviously, if $k=1$, we must have $\overline{ o_{min}(e)}=\overline{o_1}.$
\end{proof}

As in Joyal and Street's geometric theory of progressive plane graphs,  a combinatorial progressive planar structure on a combinatorial graph can  also induce some "more local" structures on the combinatorial graph.
\begin{prop}
Planar graphs are progressive, polarized and anchored graphs.
\end{prop}
\begin{proof}
$\bullet$ The fact that planar graphs are progressive is a direct consequence of the fact that planar structures are  linear orders. In fact, if $e_1\cdot\cdot\cdot e_n$ is a directed circuit of planar graph $(\overrightarrow{\Gamma}, \prec)$, then by the definition of $\prec$, we have $e_1\prec e_n$ and $e_n\prec e_1$ which is contrary to the linearity of $\prec$.

$\bullet$ The fact that planar graphs are polarized can be easily proved. In fact, let $v$ be a real vertex of planar graph $(\overrightarrow{\Gamma}, \prec)$,  we define the orders on $In(v)$ and $Out(v)$ as follows: $(1)$ $i_1< i_{2}$ iff $\overline{i_{1}}\prec \overline{i_{2}}$  for $i_1,i_2\in In(v)$ and $(2)$ $o_{1}< o_{2}$ iff $\overline{o_{1}}\prec \overline{o_{2}}$ for $o_1,o_2\in Out(v)$. The linearity of $\prec$ implies the linearity of $<$.

$\bullet$ The anchor structure is given as:  $(1)$ $i_1< i_{2}$ iff $\overline{i_{1}}\prec \overline{i_{2}}$  for $i_1,i_2\in In(\Gamma)$ and $(2)$ $o_{1}< o_{2}$ iff $\overline{o_{1}}\prec \overline{o_{2}}$ for $o_1,o_2\in Out(\Gamma)$. The linearity of $\prec$ implies the linearity of $<$.

\end{proof}

Conversely, we have the following proposition:

\begin{prop}\label{uniqueness}
If $\overrightarrow{\Gamma}$ is a progressive, polarized and anchored graph, then there exists at most one planar structure $\prec$ compatible with these structures, or more precisely, such that:

$\bullet$ $e_1\rightarrow e_2$ implies $e_1\prec e_2$ for any distinct edges $e_1, e_2$;

$\bullet$ for any vertex $v$,  $(1)$ $i_1< i_{2}$ iff $\overline{i_{1}}\prec \overline{i_{2}}$  for $i_1,i_2\in In(v)$ and $(2)$ $o_{1}< o_{2}$ iff $\overline{o_{1}}\prec \overline{o_{2}}$ for $o_1,o_2\in Out(v)$;

$\bullet$ $(1)$ $i_1< i_{2}$ iff $\overline{i_{1}}\prec \overline{i_{2}}$  for $i_1,i_2\in In(\Gamma)$ and $(2)$ $o_{1}< o_{2}$ iff $\overline{o_{1}}\prec \overline{o_{2}}$ for $o_1,o_2\in Out(\Gamma)$.
\end{prop}

\begin{proof}
To proof this composition, we only need to show that if there is a planar structure $\prec$, it will be determined by the progressive, polarization and anchor structures. In fact, for any distinct edges $e_1, e_2$, we set $$V(e_1, e_2)=\{v\in V(\Gamma)|v\rightarrow e_1, v\rightarrow e_2\},$$ then we have the following cases:

$\bullet$ Case $1$: $e_1\rightarrow e_2$ or $e_2\rightarrow e_1$. By the compatible condition between progressive structure  and $\prec$ (just $(P_1)$ condition), we have  $e_1\rightarrow e_2$ if and only if  $e_1\prec e_2$. That is, in this case, the order between $e_1, e_2$ is determined by the progressive structure.

$\bullet$ Case $2$: $V(e_1, e_2)=\varnothing$, $e_1\nrightarrow e_2$ and $e_2\nrightarrow e_1$. In this case, we must have $\overline{i_{max}(e_1)}\prec \overline{i_{min}(e_2)}$ or $\overline{i_{max}(e_2)}\prec \overline{i_{min}(e_1)}$. Otherwise, by lemma $(A')$ there will be an edge $i$ with $\overline{i}\rightarrow e_1$ and $\overline{i}\rightarrow e_2$, hence the target vertex of $i$ would be an element of $V(e_1, e_2)$, a contradiction. Thus, by property $(P^r_2)$ of $\prec$ and lemma $(A')$, we can easily prove that $e_1\prec e_2$ if and only if $\overline{i_{max}(e_1)}< \overline{i_{max}(e_2)}$ which is equivalent to $\overline{i_{max}(e_1)}\prec \overline{i_{max}(e_2)}$ by the compatible condition between $\prec$ and the anchor structure.

Now we want to prove that  $\overline{i_{max}(e_1)}< \overline{i_{max}(e_2)}$ implies $e_1\prec e_2$. Notice the facts that $\overline{i_{max}(e_1)}\rightarrow e_1$, $\overline{i_{max}(e_1)}\prec e_2$ (by $\overline{i_{max}(e_1)}\prec \overline{i_{min}(e_2)}$ and $\overline{i_{min}(e_2)}\rightarrow e_2$) and $\overline{i_{max}(e_1)}\nrightarrow e_2$ (by $\overline{i_{max}(e_1)}\prec \overline{i_{min}(e_2)}$ and lemma $(A')$), then by property $(P^r_2)$ of $\prec$, we have $e_1\prec e_2$ or $e_2\rightarrow e_1$. But  we have assumed $e_2\nrightarrow e_1$, thus we must have $e_1\prec e_2$.

Now if $e_1\prec e_2$, we want to prove $\overline{i_{max}(e_1)}< \overline{i_{max}(e_2)}$. Notice that $\overline{i_{min}(e_2)}\rightarrow e_2$, $e_1\prec e_2$ and $e_1\nrightarrow e_2$, by property $(P_2^l)$ of $\prec$, we have $e_1\prec \overline{i_{min}(e_2)}$ or $\overline{i_{min}(e_2)}\rightarrow e_1$.  If $\overline{i_{min}(e_2)}\rightarrow e_1$ and notice that $\overline{i_{min}(e_2)}\rightarrow e_2$, then the target vertex of  $\overline{i_{min}(e_2)}$ will  be an element of  $V(e_1, e_2)$, a contradiction. So we must have $e_1\prec \overline{i_{min}(e_2)}$. Notice that $\overline{i_{max}(e_1)}\prec e_1$ (by the fact $\overline{i_{max}(e_1)}\rightarrow e_1$ and fact we have proved in Case $1$ ), hence  $\overline{i_{max}(e_1)}\prec \overline{i_{max}(e_2)}$ by the linearity of $\prec$.

In summary, in this case, we prove that the order between $e_1, e_2$ is determined by the anchor structure.

$\bullet$ Case $3$: $V(e_1, e_2)\not=\varnothing$, $e_1\nrightarrow e_2$ and $e_2\nrightarrow e_1$. In this case, we first define a partial order $\widetilde{<}$ on $V(e_1, e_2)$ as $$v_1\widetilde{<}v_2\Longleftrightarrow v_1\rightarrow v_2 ,$$ for any distinct $v_1,v_2\in V(e_1, e_2), $ and the fact that $\widetilde{<}$ is indeed a partial order is consequence of the fact that $\overrightarrow{\Gamma}$ is a progressive graph. Let $v_{max}\in V(e_1, e_2)$ be a maximal element under $\widetilde{<}$, then there must be two distinct output  $o_1, o_2\in Out(v_{max})$ such that $\overline{o_1}\rightarrow e_1$ and $\overline{o_2}\rightarrow e_2$, otherwise the target of $\overline{o_1}(=\overline{o_2})$ will be an element larger than $v_{max}$. Another consequence of the maximality of $v_{max}$ is that $\overline{o_1}\nrightarrow e_2$ and $\overline{o_2}\nrightarrow e_1$.  Now we want to show that $e_1\prec e_2$ if and only if $o_1<o_2$ (or equivalently, $\overline{o_1}\prec \overline{o_2}$ by the compatible condition of polarization structure and  $\prec$). The proof is very similar to that of Case $2$.

Now we prove that $\overline{o_1}\prec \overline{o_2}$ implies $e_1\prec e_2$.  Notice that $\overline{o_2}\rightarrow e_2$, then by the fact we have proved in Case $1$ we have  $\overline{o_2}\prec e_2$, hence $\overline{o_1}\prec e_2$ (by $\overline{o_1}\prec \overline{o_2}$ and linearity of $\prec$).
So we have $\overline{o_1}\rightarrow e_1$, $\overline{o_1}\prec e_2$ and $\overline{o_1}\nrightarrow e_2$, by property $(P^r_2)$ of $\prec$ we have $e_1\prec e_2$ or $e_2\rightarrow e_1$. But $e_2\nrightarrow e_1$,  we must have $e_1\prec e_2$.

Now we prove that $e_1\prec e_2$ implies $\overline{o_1}\prec \overline{o_2}$.  Notice that $\overline{o_1}\rightarrow e_1$, then by the fact we have proved in Case $1$ we have  $\overline{o_1}\prec e_1$. By the linearity of $\prec$, we see that  to prove $\overline{o_1}\prec \overline{o_2}$ we only need to prove $e_1\prec \overline{o_2}.$ Now we have the facts that $\overline{o_2}\rightarrow e_2$, $e_1\prec e_2$ and $e_1\nrightarrow e_2$, using the property $(P_2^l)$ of $\prec$ we get $e_1\prec \overline{o_2}$ or $\overline{o_2}\rightarrow e_1$. But $\overline{o_2}\nrightarrow e_1$, we must have $e_1\prec \overline{o_2}$, hence complete the proof.

In summary, in this case, we prove that the order between $e_1, e_2$ is determined by the polarization structure.

\end{proof}

\subsection{Planar structure on vertices}
In this section, we want to show that the set of vertices of a planar graph will naturally be a planar set (definition 3.1.10).
In fact for a planar graph $(\overrightarrow{\Gamma}, \prec)$, we can define a partial order $<_V$ on the set $V_{re}(\overrightarrow{\Gamma})$ of real vertices as follows:
$$v_1<_V v_2\Longleftrightarrow v_1\rightarrow v_2.$$

It is easily to see that $<_V$ is indeed a partial order and the fact that $\overrightarrow{\Gamma}$ is a progressive graph implies that maximal and minimal elements always exist.
Now we want to prove that for planar graph $(\overrightarrow{\Gamma},\prec)$ the partial order $<_V$ defined above can be naturally extended to  a linear order $\prec_V$ on $V(\overrightarrow{\Gamma})$.  Similar to the Kontsevich's complex of graphs we can define complex of planar graphs, and this linear order can make the sign rule of differential explicit, we hope to come back to this issue in the further.

For a vertex $v$ of planar graph $(\overrightarrow{\Gamma}, \prec)$, we introduce some notations:

$$i_{min}(v)=min\  In(v),$$
$$i_{max}(v)=max\  In(v),$$
$$o_{min}(v)=min\  Out(v),$$
$$o_{max}(v)=max\  Out(v).$$

For any two  different vertices $v_1, v_2$,  we define the new order $\prec_V$ as follows:

$$v_1\prec_V v_2\Longleftrightarrow \ v_1\rightarrow v_2\ or\  \overline{o_{max}(v_1)}\prec \overline{i_{min}(v_2)}.$$

\begin{lem}
If $v_1\prec_V v_2$, then $\overline{i_{min}(v_1)}\prec\overline{o_{max}(v_2)}$.
\end{lem}
\begin{proof}
If $v_1\rightarrow v_2$, it is obvious that $\overline{i_{min}(v_1)}\rightarrow \overline{o_{max}(v_2)}$, hence that $\overline{i_{min}(v_1)}\prec\overline{o_{max}(v_2)}$. If $\overline{o_{max}(v_1)}\prec \overline{i_{min}(v_2)}$, noticing  that $\overline{i_{min}(v_1)}\prec \overline{o_{max}(v_1)}$ and  $\overline{i_{min}(v_2)}\prec \overline{o_{max}(v_2)}$, thus the fact that $\prec$ is a linear order implies that $\overline{i_{min}(v_1)}\prec\overline{o_{max}(v_2)}$.
\end{proof}

\begin{prop}
$\prec_V$ is a linear order on $V(\Gamma)$.
\end{prop}
\begin{proof}
$\bullet$ We want to prove that  $v_1\prec_V v_2$ and  $v_2 \prec_V v_1$ contradict with each other.

Case $1$: $v_1\rightarrow v_2$ and $v_2\rightarrow v_1$. In this case, these conditions contradict with the fact that $(\overrightarrow{\Gamma},\prec)$ is progressive.

Case $2$:  $v_1\rightarrow v_2$ and $\overline{o_{max}(v_2)}\prec \overline{i_{min}(v_1)}$. $v_1\rightarrow v_2$ implies that $\overline{i_{min}(v_1)}\prec\overline{o_{max}(v_2)}$ which  contradicts with $\overline{o_{max}(v_2)}\prec \overline{i_{min}(v_1)}$.

Case $3$: $v_2\rightarrow v_1$ and $\overline{o_{max}(v_1)}\prec \overline{i_{min}(v_2)}$. Similar to Case $2$.

Case $4$: $\overline{o_{max}(v_1)}\prec \overline{i_{min}(v_2)}$ and $\overline{o_{max}(v_2)}\prec \overline{i_{min}(v_1)}$. In this case, these conditions  contradict with the fact that $\prec$ is a linear order.

$\bullet$ If $v_1\prec_V v_2$ and $v_2\prec_V v_3$, we want to prove that $v_1\prec_V v_3$. By definition of  $\prec_V$, we have that $v_1\rightarrow v_2\ or\  \overline{o_{max}(v_1)}\prec \overline{i_{min}(v_2)}$ and $v_2\rightarrow v_3\ or\  \overline{o_{max}(v_2)}\prec \overline{i_{min}(v_3)}$.

Case $1$: $v_1\rightarrow v_2$ and $v_2\rightarrow v_3$. In this case, we have $v_1\rightarrow v_3$, hence $v_1\prec_V v_3$.

Case $2$: $v_1\rightarrow v_2$ and $\overline{o_{max}(v_2)}\prec \overline{i_{min}(v_3)}$. The fact that $v_1\rightarrow v_2$ implies that $\overline{i_{min}(v_1)}\rightarrow\overline{o_{max}(v_2)} $, thus $\overline{i_{min}(v_1)}\prec\overline{o_{max}(v_2)}$ and hence $\overline{i_{min}(v_1)}\prec \overline{i_{min}(v_3)}.$ We will prove that $v_1\rightarrow v_3$ or $\overline{o_{max}(v_1)}\prec \overline{i_{min}(v_3)}$ in the following two cases.

case $2.1$:  If  $\overline{o_{max}(v_1)}\prec \overline{i_{min}(v_3)}$, we get that $v_1\prec_V v_3$, thus we complete the proof.

case $2.2$:  If  $\overline{i_{min}(v_3)}\prec \overline{o_{max}(v_1)}$, we get that $\overline{i_{min}(v_1)}\prec\overline{i_{min}(v_3)}\prec \overline{o_{max}(v_1)}.$ Noticing that $\overline{i_{min}(v_1)}\rightarrow \overline{o_{max}(v_1)}$, we have $\overline{i_{min}(v_1)}\rightarrow\overline{i_{min}(v_3)}$ or $\overline{i_{min}(v_3)}\rightarrow \overline{o_{max}(v_1)}.$ In the first case, we have $v_1\rightarrow v_3$, hence $v_1\prec_V v_3.$ In the latter case, we have $v_3\rightarrow v_1$, noticing that $v_1\rightarrow v_2$ we get $v_3\rightarrow v_2$
hence $v_3\prec_V v_2$. But we have proved that it contracts with $v_2\prec_V v_3$, hence the latter case is impossible.

Case $3$:  $\overline{o_{max}(v_1)}\prec \overline{i_{min}(v_2)}$ and $\overline{o_{max}(v_2)}\prec \overline{i_{min}(v_3)}$. In this case, noticing that $\overline{i_{min}(v_2)}\rightarrow \overline{o_{max}(v_2)}$, we have $\overline{i_{min}(v_2)}\prec\overline{o_{max}(v_2)}$, hence $\overline{o_{max}(v_1)}\prec \overline{i_{min}(v_3)}$ which implies $v_1\prec_V v_3.$

Case $4$:  $\overline{o_{max}(v_1)}\prec \overline{i_{min}(v_2)}$ and $v_2\rightarrow v_3$. This case is similar to Case $2$ but we still give a complete proof. The fact that $v_2\rightarrow v_3$ implies $\overline{i_{min}(v_2)}\rightarrow \overline{o_{max}(v_3)}$, hence $\overline{i_{min}(v_2)}\prec \overline{o_{max}(v_3)}$,thus we have $\overline{o_{max}(v_1)}\prec \overline{o_{max}(v_3)}$. We will prove that $v_1\rightarrow v_3$ or $\overline{o_{max}(v_1)}\prec \overline{i_{min}(v_3)}$ in the following two cases.

case $4.1$: If $\overline{o_{max}(v_1)}\prec\overline{i_{min}(v_3)}$, we get $v_1\prec_V v_3$, thus complete the proof.

case $4.2$: If $\overline{i_{min}(v_3)}\prec\overline{o_{max}(v_1)}$, we get that $\overline{i_{min}(v_3)}\prec\overline{o_{max}(v_1)}\prec \overline{o_{max}(v_3)}$, noticing that $\overline{i_{min}(v_3)}\rightarrow \overline{o_{max}(v_3)}$, we must have $\overline{i_{min}(v_3)}\rightarrow \overline{o_{max}(v_1)}$ or $\overline{o_{max}(v_1)}\rightarrow \overline{o_{max}(v_3)}$. In the first case, we get $v_3\rightarrow v_1$, and notice that $v_2\rightarrow v_3$ we get $v_2\rightarrow v_1$, hence $v_2\prec_V v_1$ which contradict with the fact that $v_1\prec_V v_2$. In the latter case, we have $v_1\rightarrow v_3$, hence $v_1\prec_V v_3$.

$\bullet$ We want to prove that $v_1\prec_V v_2$ or $v_2 \prec_V v_1$ for different vertices $v_1,v_2$, which is equivalent to  at least one of the following conditions:
$(1)\ v_1\rightarrow v_2,\ (2)\ \overline{o_{max}(v_1)}\prec \overline{i_{min}(v_2)},\ (3)\ v_2\rightarrow v_1,\ (4)\ \overline{o_{max}(v_2)}\prec \overline{i_{min}(v_1)}$. We discuss in the following cases:

Case $1$: if $\overline{i_{min}(v_1)}\prec \overline{i_{min}(v_2)}\prec \overline{o_{max}(v_1)}$, then  by the fact that $\overline{i_{min}(v_1)}\rightarrow \overline{o_{max}(v_1)}$ and the condition  $(\widetilde{P_2})$ of $\prec$, we have $\overline{i_{min}(v_2)}\rightarrow \overline{o_{max}(v_1)} $ or $\overline{i_{min}(v_1)}\rightarrow \overline{i_{min}(v_2)} $ which implies $v_2\rightarrow v_1$ or $v_1\rightarrow v_2$, thus we have condition $(1)$ or condition $(3)$.

Case $2$: $ \overline{i_{min}(v_2)}\prec\overline{i_{min}(v_1)}$ or $\overline{o_{max}(v_1)}\prec\overline{i_{min}(v_2)}$. If $\overline{o_{max}(v_1)}\prec\overline{i_{min}(v_2)}$, we get condition $(2)$. If $ \overline{i_{min}(v_2)}\prec\overline{i_{min}(v_1)}$, we discuss in the following cases: $\overline{o_{max}(v_2)}\prec \overline{i_{min}(v_1)}$ or $ \overline{i_{min}(v_1)}\prec\overline{o_{max}(v_2)}.$

case $2.1$: If $ \overline{i_{min}(v_2)}\prec\overline{i_{min}(v_1)}$ and $\overline{o_{max}(v_2)}\prec \overline{i_{min}(v_1)}$, we get condition $(4)$.

case $2.2$: If $ \overline{i_{min}(v_2)}\prec\overline{i_{min}(v_1)}$ and $\overline{i_{min}(v_1)}\prec \overline{o_{max}(v_2)}$, that is, $\overline{i_{min}(v_2)}\prec\overline{i_{min}(v_1)}\prec \overline{o_{max}(v_2)}$, then by  the fact that $\overline{i_{min}(v_2)}\rightarrow \overline{o_{max}(v_2)}$ and the condition  $(\widetilde{P_2})$ of $\prec$, we must have $\overline{i_{min}(v_2)}\rightarrow \overline{i_{min}(v_1)}$ or $\overline{i_{min}(v_1)}\rightarrow \overline{o_{max}(v_2)}$, which implies $v_2\rightarrow v_1$ or $v_1\rightarrow v_2$ thus we have  conditions $(1)$ or condition $(3)$.

\end{proof}

\begin{ex}
For  the planar graph in example 3.1.3, the linear order on vertices are shown as
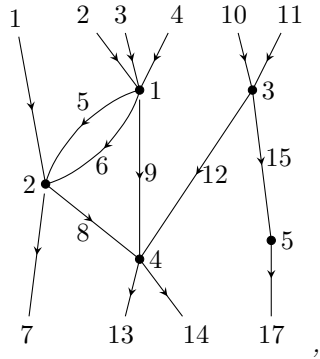
\begin{center}
\begin{tikzpicture}[scale=0.5]

\node (v2) at (-4,3) {};
\node[left] at (v2) {$2$};
\node (v1) at (-1.5,5.5) {};
\node[right] at (v1) {$1$};
\node (v7) at (-1.5,1) {};
\node[right] at (v7) {$4$};
\node (v9) at (1.5,5.5) {};
\node[right] at (v9) {$3$};
\node (v14) at (2,1.5) {};
\node[right] at (v14) {$5$};
\node (v3) at (-3,7.5) {$2$};
\node (v4) at (-2,7.5) {$3$};
\node (v5) at (-0.5,7.5) {$4$};
\node (v6) at (-4.8,7.4) {$1$};
\node (v11) at (-4.5,-1) {$7$};
\node (v12) at (-2,-1) {$13$};
\node (v13) at (0,-1) {$14$};
\node (v15) at (2,-1) {$17$};
\node (v8) at (1,7.5) {$10$};
\node (v10) at (2.5,7.5) {$11$};
\node  at (-2.5,3.5) {$6$};
\node  at (-3,5.2) {$5$};
\node  at (-1.2,3.3) {$9$};
\node  at (0.5,3.25) {$12$};
\node  at (2.2,3.7) {$15$};
\node  at (-3,1.7) {$8$};
\draw[fill] (-4,3) circle [radius=0.11];
\draw[fill] (v1) circle [radius=0.11];
\draw[fill] (v7) circle [radius=0.11];
\draw[fill] (v9) circle [radius=0.11];
\draw[fill] (v14) circle [radius=0.11];
\draw  plot[smooth, tension=1] coordinates {(v1) (-2.5,5)  (-3.5,4) (v2)}[postaction={decorate, decoration={markings,mark=at position .5 with {\arrow[black]{stealth}}}}];
\draw  plot[smooth, tension=1] coordinates {(v1) (-2,4.5)  (-3,3.5) (v2)}[postaction={decorate, decoration={markings,mark=at position .5 with {\arrow[black]{stealth}}}}];

\draw  (v3) -- (-1.5,5.5)[postaction={decorate, decoration={markings,mark=at position .5 with {\arrow[black]{stealth}}}}];
\draw  (v4) -- (-1.5,5.5)[postaction={decorate, decoration={markings,mark=at position .5 with {\arrow[black]{stealth}}}}];
\draw  (v5) -- (-1.5,5.5)[postaction={decorate, decoration={markings,mark=at position .5 with {\arrow[black]{stealth}}}}];

\draw  (v6) -- (-4,3)[postaction={decorate, decoration={markings,mark=at position .5 with {\arrow[black]{stealth}}}}];
\draw  (v1) -- (-1.5,1)[postaction={decorate, decoration={markings,mark=at position .5 with {\arrow[black]{stealth}}}}];
\draw  (-4,3) -- (-1.5,1)[postaction={decorate, decoration={markings,mark=at position .5 with {\arrow[black]{stealth}}}}];

\draw  (v8)--(1.5,5.5)[postaction={decorate, decoration={markings,mark=at position .5 with {\arrow[black]{stealth}}}}];
\draw  (v10) -- (1.5,5.5)[postaction={decorate, decoration={markings,mark=at position .5 with {\arrow[black]{stealth}}}}];
\draw  (1.5,5.5) -- (-1.5,1)[postaction={decorate, decoration={markings,mark=at position .5 with {\arrow[black]{stealth}}}}];
\draw  (v2) -- (v11)[postaction={decorate, decoration={markings,mark=at position .5 with {\arrow[black]{stealth}}}}];
\draw  (-1.5,1) -- (v12)[postaction={decorate, decoration={markings,mark=at position .65 with {\arrow[black]{stealth}}}}];
\draw  (v13) -- (-1.5,1)[postaction={decorate, decoration={markings,mark=at position .5 with {\arrowreversed[black]{stealth}}}}];
\draw  (1.5,5.5) -- (2,1.5)[postaction={decorate, decoration={markings,mark=at position .5 with {\arrow[black]{stealth}}}}];
\draw  (2,1.5) -- (v15)[postaction={decorate, decoration={markings,mark=at position .5 with {\arrow[black]{stealth}}}}];

\end{tikzpicture},
\end{center}
where the linear order on vertices is given by their labels $1,2,3,4,5.$
\end{ex}

In the following proposition, we will show that $\prec_V$ satisfies  $(P_2)$ condition.
\begin{prop}
For any three vertices $v_1$,$v_2$,$v_3$, if $v_1\prec_V v_3\prec_V v_2$ and $v_1\rightarrow v_2$, then $v_1\rightarrow v_3$ or $v_3\rightarrow v_2.$
\end{prop}
\begin{proof} We prove this proposition by contradiction.
Suppose that $v_1\nrightarrow v_3$ and $v_3\nrightarrow v_2$,then the fact that $v_1\prec_V v_3\prec_V v_2$  implies that $\overline{o_{max}(v_1)}\prec \overline{i_{min}(v_3)}$ and $\overline{o_{max}(v_3)}\prec \overline{i_{min}(v_2)}$. It is obvious that $$\overline{i_{min}(v_1)}\rightarrow \overline{o_{max}(v_1)}\prec \overline{i_{min}(v_3)}\rightarrow \overline{o_{max}(v_3)}\prec\overline{i_{min}(v_2)}\rightarrow \overline{o_{max}(v_2)}.$$
The fact $v_1\rightarrow v_2$ implies $\overline{i_{min}(v_1)}\rightarrow \overline{o_{max}(v_2)}$. Using the property $(P_2)$ of $\prec$, we have $\overline{i_{min}(v_1)}\rightarrow \overline{i_{min}(v_3)}$ or $\overline{i_{min}(v_3)}\rightarrow \overline{o_{max}(v_2)}$ which imply $v_1\rightarrow v_3$ or $v_3\rightarrow v_2$, hence we arrive at a contradiction.

\end{proof}
Thus we have shown that the set of vertices of any planar graph is a planar set. As a corollary of proposition \ref{equivalence}, we have
\begin{cor}
The following three conditions are equivalent:

$(P_2)$ for any three vertices $v_1$, $v_2$, $v_3$, if $v_1\rightarrow v_2$ and $v_1\prec_V v_3\prec_V v_2$, then $v_1\rightarrow v_3$ or $v_3\rightarrow v_2$;

$(P_2^r)$ for any three vertices $v_1$, $v_2$, $v_3$, if $v_1\rightarrow v_2$ and $v_1\prec_V v_3$ and $v_1\nrightarrow  v_3$, then $v_2\prec_V v_3$ or $v_3\rightarrow v_2$;

$(P_2^l)$  for any three vertices $v_1$, $v_2$, $v_3$, if $v_1\rightarrow v_2$ and $v_3\prec_V v_2$ and $v_3\nrightarrow  v_2$, then $v_3\prec_V v_1$ or $v_1\rightarrow v_3$.
\end{cor}

\subsection{Tensor product of planar graphs}

Now we give the notion of tensor product of two planar graphs.
\begin{defn}
Let $(\overrightarrow{\Gamma}_1,\prec_1)$ and $(\overrightarrow{\Gamma}_2,\prec_2)$ be two planar graph, we define their tensor product $(\overrightarrow{\Gamma},\prec)$ to be a directed graph $\overrightarrow{\Gamma}$ with a partial  order $\prec$ such that

$\bullet$ $\overrightarrow{\Gamma}=\overrightarrow{\Gamma}_1\otimes\overrightarrow{\Gamma}_2$ as a directed graph,

$\bullet$ for any different edges $e_1,e_2\in E(\Gamma)$, \begin{equation*} e_1\prec e_2 \Longleftrightarrow
\begin{cases}
e_1,e_2\in E(\Gamma_1)\  \text{and}\  e_1\prec_1 e_2,\  \text{or},\\
e_1,e_2\in E(\Gamma_2)\  \text{and}\  e_1\prec_2 e_2,\ \text{or},\\
e_1\in E(\Gamma_1)\  \text{and}\  e_2\in E(\Gamma_2) .
\end{cases}
\end{equation*}
\end{defn}
\begin{prop}
Tensor product of two planar  graphs is a planar graph.
\end{prop}
\begin{proof}
We only need to show that $\prec$ is a planar structure.

$\bullet$ The linearity is obvious from the definition of $\prec$.

$\bullet$ Now we want to show condition $(P_1)$ is satisfied. If $e_1\rightarrow e_2$, then we must have $e_1,e_2\in E(\Gamma_1)$ or $e_1,e_2\in E(\Gamma_2)$ due to the fact that $\Gamma_1$ and $\Gamma_2$ are not connected in $\Gamma_1\otimes \Gamma_2$. In both case, we have $e_1\prec e_2$.

$\bullet$ Now we want to prove that $\prec$  satisfies condition $(P^r_2)$.  If $e_1\rightarrow e_2$,$e_3\nrightarrow e_2$ and $e_1\prec e_3$, we want to prove that $e_2\prec e_3$ or $e_3\rightarrow e_2$. Just as above, $e_1\rightarrow e_2$  implies $e_1,e_2\in E(\Gamma_1)$ or $e_1,e_2\in E(\Gamma_2)$. Thus we only need discuss the follow cases:

Case $1$:  $e_1,e_2,e_3\in E(\Gamma_1)$. The fact that $\prec_1$ is a planar structure implies that  $e_2\prec_1 e_3$ or $e_3\rightarrow e_2$. Hence, we have $e_2\prec e_3$ or $e_3\rightarrow e_2$ due to the definition of $\prec$.

Case $2$:  $e_1,e_2,e_3\in E(\Gamma_2)$. In this case, the proof is similar to Case $1$.

Case $3$:  $e_1,e_2\in E(\Gamma_1)$ and $e_3\in E(\Gamma_2)$. In this case, we have $e_2\prec e_3$ which is directly from the definition of $\prec$.

Case $4$: $e_1,e_2\in E(\Gamma_2)$ and $e_3\in E(\Gamma_1)$. In this case, from the definition of $\prec$, we have $e_3\prec e_1$ which is contrary to our assumption $e_1\prec e_3$. So this case can be dropped.
\end{proof}

\begin{prop}
If $\prec$ and $\widetilde{\prec}$ are two planar structures on $\overrightarrow{\Gamma}_1\otimes\overrightarrow{\Gamma}_2$ and satisfy $(1)$  $\prec|_{E(\Gamma_1)}= \widetilde{\prec}|_{E(\Gamma_1)}$, $\prec|_{E(\Gamma_2)}= \widetilde{\prec}|_{E(\Gamma_2)}$, and, $(2)$  for any $e_1\in E(\Gamma_1),e_2\in E(\Gamma_2) $, $e_1\prec e_2$ and $ e_1\widetilde{\prec} e_2$,\ then $\prec=\widetilde{\prec}.$
\end{prop}
\begin{proof}
From condition $(1)$, we can define a linear order $\prec_1=\prec|_{E(\Gamma_1)}= \widetilde{\prec}|_{E(\Gamma_1)}$ on $\overrightarrow{\Gamma}_1$ and a linear order $\prec_2=\prec|_{E(\Gamma_2)}= \widetilde{\prec}|_{E(\Gamma_2)}$ on $\overrightarrow{\Gamma}_2$. To prove this proposition, we need to show that for any two different edges $e_1,e_2\in  E(\overrightarrow{\Gamma}_1\otimes\overrightarrow{\Gamma}_2)$, $e_1\prec e_2 \Longleftrightarrow e_1\widetilde{\prec} e_2 .$  We will  prove this in three cases:

Case $1$: $e_1, e_2\in E(\overrightarrow{\Gamma}_1)$. In this case, we have $$e_1\prec e_2 \Longleftrightarrow e_1\prec_1 e_2 \Longleftrightarrow e_1\widetilde{\prec} e_2.$$

Case $2$: $e_1, e_2\in E(\overrightarrow{\Gamma}_2)$. In this case, we have $$e_1\prec e_2 \Longleftrightarrow e_1\prec_2 e_2 \Longleftrightarrow e_1\widetilde{\prec} e_2.$$

Case $3$: $e_1\in E(\overrightarrow{\Gamma}_1),e_2\in E(\overrightarrow{\Gamma}_2)$. In this case,  $e_1\prec e_2 \Longleftrightarrow e_1\widetilde{\prec} e_2$ is a direct consequence of condition $(2).$
\end{proof}

From the above two propositions, we see that the tensor product is a natural and  well-defined operation/bi-functor on the groupoid of planar graphs. We denote by $(\overrightarrow{\Gamma}_1\otimes \overrightarrow{\Gamma}_2,\prec_1\otimes\prec_2)=(\overrightarrow{\Gamma}_1,\prec_1)\otimes (\overrightarrow{\Gamma}_2,\prec_2)$ the tensor product of $(\overrightarrow{\Gamma}_1,\prec_1)$ and $(\overrightarrow{\Gamma}_2,\prec_2)$. The following proposition is obvious:
\begin{prop}
The tensor product $\otimes$ is associative, that is, for any three planar graphs  $(\overrightarrow{\Gamma}_1,\prec_1), (\overrightarrow{\Gamma}_2,\prec_2), (\overrightarrow{\Gamma}_3,\prec_3),$ we have $$(\overrightarrow{\Gamma}_1\otimes \overrightarrow{\Gamma}_2,\prec_1\otimes\prec_2)\otimes(\overrightarrow{\Gamma}_3,\prec_3)=(\overrightarrow{\Gamma}_1,\prec_1)\otimes(\overrightarrow{\Gamma}_2\otimes \overrightarrow{\Gamma}_3,\prec_2\otimes\prec_3).$$
\end{prop}

In summary, we have the following proposition:
\begin{prop}
The tensor product and empty graph makes the groupoid of planar graphs a tensor category.
\end{prop}

\subsection{Composition of planar graphs}
In this section we will prove that on equivalence classes of planar graphs there is a well-defined composition. Let us start with  introducing the notion of composition of two planar graphs.
\begin{defn}
Let $(\overrightarrow{\Gamma}_1,\prec_1)$ and $(\overrightarrow{\Gamma}_2,\prec_2)$ be two planar graph, and let $Out(\overrightarrow{\Gamma}_1)=\{o_1,..., o_n\}$ with $ \overline{o_1}\prec_1\cdot\cdot\cdot\prec_1 \overline{o_n}$ and $In(\overrightarrow{\Gamma}_1)=\{i_1,..., i_n\}$ with $\overline{i_1}\prec_2\cdot\cdot\cdot\prec_2 \overline{i_n}$. Their composition $(\overrightarrow{\Gamma},\prec)$ is defined to be a directed graph $\overrightarrow{\Gamma}$ with a binary relation $\prec$ such that:

$\bullet$ $\overrightarrow{\Gamma}=\overrightarrow{\Gamma}_2\circ\overrightarrow{\Gamma}_1$ as an anchored graph;

$\bullet$ for any two different edges $e_1, e_2\in E(\overrightarrow{\Gamma})$, the following conditions are satisfied:

$(1)$ If $ e_1, e_2 \in E_o(\overrightarrow{\Gamma}_1)$, then $e_1\prec e_2 \Longleftrightarrow e_1\prec_1 e_2$;

$(2)$ If $ e_1, e_2 \in E_o(\overrightarrow{\Gamma}_2)$, then $e_1\prec e_2 \Longleftrightarrow e_1\prec_2 e_2$;

$(3)$ If $e_1\in E_o(\overrightarrow{\Gamma}_1)$ and $e_2 \in E_o(\overrightarrow{\Gamma}_2)$, then
\begin{equation*}
\begin{cases}
e_1\prec e_2&  \text{if}\  e_1\prec_1 \overline{o_1};\\
e_1\prec e_2&  \text{if}\ \overline{i_n}\prec_2 e_2;\\
e_1\prec e_2&  \text{if}\ \overline{o_{k}} \prec_1 e_1\prec_1 \overline{o_{k+1}}\  \text{and}\   e_2\prec_2 \overline{i_{k+1}},\ 1\leq k< n;\\
e_2\prec e_1&  \text{if}\ \overline{o_{k}} \prec_1 e_1\prec_1 \overline{o_{k+1}}\  \text{and}\  \overline{i_{k+1}}\prec_2 e_2,\ 1\leq k< n;
\end{cases}
\end{equation*}

$(4)$ If $e_1\in E_n(\overrightarrow{\Gamma})$,  then
\begin{equation*}
\begin{cases}
e_2\prec e_1&  \text{if}\  e_2\in E_o(\overrightarrow{\Gamma}_1)\  \text{and}\  e_2\prec_1 \overline{o_{k}};\\
e_1\prec e_2&  \text{if}\  e_2\in E_o(\overrightarrow{\Gamma}_1)\  \text{and}\ \overline{o_{k}}\prec_1 e_2;\\
e_2\prec e_1&  \text{if}\  e_2\in E_o(\overrightarrow{\Gamma}_2)\  \text{and}\  e_2\prec_2 \overline{i_{k}};\\
e_1\prec e_2&  \text{if}\  e_2\in E_o(\overrightarrow{\Gamma}_2)\  \text{and}\ \overline{i_{k}}\prec_2 e_2;
\end{cases}
\end{equation*}

$(5)$ If $e_1,e_2\in E_n(\overrightarrow{\Gamma})$, that is, $e_1=\overline{e}_k, e_2=\overline{e}_l$ for some $k,l\in\{1,...,n\}$, then $e_1\prec e_2\Longleftrightarrow k< l.$
\end{defn}

To illustrate the rules of the definition of $\prec$ and show that it is indeed a linear order, we give a constructive description of $\prec$. We label edges of $\overrightarrow{\Gamma}_1$, $\overrightarrow{\Gamma}_2$  by numbers, that is,  $E(\overrightarrow{\Gamma}_1)=\{e_1,....,e_s\}$, $E(\overrightarrow{\Gamma}_2)=\{e'_1,....,e'_t\}$  such that

$\bullet$ $e_i\prec_1 e_j\Longleftrightarrow i<j$, for $1\leq i,j\leq s,$

$\bullet$  $e'_p\prec_2 e'_q\Longleftrightarrow p<q$, for $1\leq p,q\leq t,$ \\
and let

$\bullet$  $\overline{o_{k}}=e_{q_k}$, for $1\leq k \leq n$, $q_k\in\{1,...,s\}$,

$\bullet$  $\overline{i_{k}}=e'_{p_k}$, for $1\leq k \leq n$, $p_k\in\{1,...,t\}$.

Thus, $E_o(\overrightarrow{\Gamma}_1)$ and  $E_o(\overrightarrow{\Gamma}_2)$ together with their linear orders can be equivalently written as sequences $$e_1\cdots e_{q_k-1}\widehat{e_{q_k}}e_{q_k+1}\cdots e_{s-1}\widehat{e_{q_n}}$$ and  $$\widehat{e'_{p_1}}e'_2\cdots e'_{p_k-1}\widehat{e'_{p_k}}e'_{p_k
+1}\cdots e'_t,$$  where $\widehat{e_{q_k}}$ and $\widehat{e'_{p_k}}$  denote to remove the edges $e_{q_k}$ and $e'_{p_k}$ in the two sequence, $1\leq k\leq n$.
Now we split them into segments
$$E_o(\overrightarrow{\Gamma}_1)=\underbrace{e_{Q_1}\sqcup\widehat{\{e_{q_1}\}}}\sqcup...\sqcup \underbrace{e_{Q_k}\sqcup\widehat{\{e_{q_k}\}} } \sqcup...\sqcup \underbrace{e_{Q_n}\sqcup\widehat{\{e_{q_n}\}}}$$ and
$$E_o(\overrightarrow{\Gamma}_2)=\underbrace{\widehat{\{e'_{p_1}\}}\sqcup e'_{P_1}}\sqcup...\sqcup \underbrace{\widehat{\{e'_{p_k}\}}\sqcup e'_{P_k}} \sqcup...\sqcup \underbrace{\widehat{\{e'_{p_n}\}}\sqcup e'_{P_n}},$$
where $e_{Q_k}=\{e_i, q_{k-1}+1\leq i \leq q_k-1\}$, $e_{P_k}=\{e'_p, p_k+1\leq p \leq p_{k+1}-1\}$, for $2\leq k \leq n-1$ and $e_{Q_1}=\{e_i, 1\leq i \leq q_1-1\}$, $e_{P_n}=\{e'_p, p_n+1\leq p \leq t\}$. We call these $e_{Q_k}$s and $e'_{P_k}$s \textbf{basic segments} of $(\overrightarrow{\Gamma}_1,\prec_1)$ and $(\overrightarrow{\Gamma}_2,\prec_2)$.

We  assume $E_n(\overrightarrow{\Gamma})=\{\overline{e}_1,...,\overline{e}_n\}$. Notice that $E(\overrightarrow{\Gamma})=E_o(\overrightarrow{\Gamma}_1) \sqcup E_n(\overrightarrow{\Gamma})\sqcup E_o(\overrightarrow{\Gamma}_2)$, so we define the linear order  $\prec=\prec_2\circ \prec_1$ on $E(\overrightarrow{\Gamma})$ to be the following "shuffled" order: $$E(\overrightarrow{\Gamma})=\underbrace{e_{Q_1}\sqcup\{\overline{e}_1\}\sqcup e'_{P_1}}\sqcup...\sqcup \underbrace{e_{Q_k}\sqcup\{\overline{e}_k\}\sqcup e'_{P_k}} \sqcup...\sqcup \underbrace{e_{Q_n}\sqcup\{\overline{e}_n\}\sqcup e'_{P_n}}.$$

It is easy to check that this linear order satisfying conditions $(1)-(5)$ in the definition above. So we get the following  proposition.
\begin{prop}
On the composition $\overrightarrow{\Gamma}=\overrightarrow{\Gamma}_2\circ\overrightarrow{\Gamma}_1$, there exists an unique linear order $\prec$ satisfying the conditions $(1)-(5)$ in the definition above.
\end{prop}

The following propositions give characterizations of basic segments.

\begin{prop}
$$e_{Q_1}=\{e\in E(\Gamma_1)|e\rightarrow \overline{o_1}\}$$ and for $2\leq k \leq n$, $$e_{Q_k}=\{e\in E(\Gamma_1)|e\rightarrow \overline{o_k}, e\nrightarrow \overline{o_{k-1}}\}.$$
\end{prop}
\begin{proof}
$\bullet$  Notice that $e\in e_{Q_1}\Longleftrightarrow e\prec_1 \overline{o_1}$. On one hand, if $e\prec_1 \overline{o_1}$, by lemma $(B'')$ we have $\overline{o_{min}(e)}=\overline{o_1}$, thus $e\rightarrow \overline{o_1}$. On the other hand, if $e\rightarrow \overline{o_1}$, by property $(P_1)$ we have $e\prec_1 \overline{o_1}$. Thus we prove that $e_{Q_1}=\{e\in E(\Gamma_1)|e\rightarrow \overline{o_1}\}$.

$\bullet$  Notice that $e\in e_{Q_k}\Longleftrightarrow \overline{o_{k-1}} \prec_1 e\prec_1 \overline{o_k}$.  On one hand, if $\overline{o_{k-1}} \prec_1 e\prec_1 \overline{o_k}$, by property $(P_2)$ we have $\overline{o_{k-1}} \rightarrow e$ or $e\rightarrow \overline{o_k}$. Because $\overline{o_{k-1}}$ is an output edge, thus $\overline{o_{k-1}} \nrightarrow e$, so we must have $e\rightarrow \overline{o_k}$. If $e\rightarrow \overline{o_{k-1}}$, then $e\prec_1 \overline{o_{k-1}}$ contradicts with $\overline{o_{k-1}} \prec_1 e$, thus $e\nrightarrow \overline{o_{k-1}}$. So $e_{Q_k}\subseteq\{e\in E(\Gamma_1)|e\rightarrow \overline{o_k}, e\nrightarrow \overline{o_{k-1}}\}$.

On the other hand, if $e\rightarrow \overline{o_k}$, then by property $(P_1)$ and lemma $(B')$ we have $e\prec \overline{o_k}$ and $\overline{o_{min}(e)}\preceq_1\overline{o_k}\preceq_1 \overline{o_{max}(e)}$. If $e\nrightarrow \overline{o_{k-1}}$, then by lemma $(B')$ we have $\overline{o_{max}(e)}\prec\overline{o_{k-1}}$ or $\overline{o_{k-1}}\prec_1\overline{o_{min}(e)}$. But $\overline{o_{max}(e)}\prec_1\overline{o_{k-1}}$ is contradicts with $\overline{o_{min}(e)}\preceq_1\overline{o_k}\preceq_1 \overline{o_{max}(e)}$, thus we must have  $\overline{o_{k-1}}\prec_1\overline{o_{min}(e)}$. Now we have $e\rightarrow \overline{o_{min}(e)}$, $\overline{o_{k-1}}\prec_1\overline{o_{min}(e)}$ and $\overline{o_{k-1}}\nrightarrow\overline{o_{min}(e)}$, by property $(P^l_2)$, we get $\overline{o_{k-1}}\prec_1 e$ or $e\rightarrow \overline{o_{k-1}}$.
Notice that $e\rightarrow \overline{o_{k-1}}$  contradicts with  $\overline{o_{k-1}}\prec_1\overline{o_{min}(e)}$ due to lemma $(B')$, so we must have  $\overline{o_{k-1}}\prec_1 e$. Thus we have proved that $\overline{o_{k-1}} \prec_1 e\prec_1 \overline{o_k}$, that is, $\{e\in E(\Gamma_1)|e\rightarrow \overline{o_k}, e\nrightarrow \overline{o_{k-1}}\}\subseteq e_{Q_k}$.

\end{proof}

\begin{prop} For  $1\leq k \leq n-1$,
$$e'_{P_k}=\{e'\in E(\Gamma_2)|\overline{i_k}\rightarrow e',\overline{i_{k+1}}\nrightarrow e'\}$$ and  $$e'_{P_n}=\{e'\in E(\Gamma_2)|\overline{i_n}\rightarrow e'\}.$$
\end{prop}
\begin{proof}
$\bullet$ Notice that $e'\in e'_{P_k}\Longleftrightarrow \overline{i_k}\prec_2 e'\prec_2 \overline{i_{k+1}}$. One on hand, if $\overline{i_k}\prec_2 e'\prec_2 \overline{i_{k+1}}$, by property $(P_2)$ we have $\overline{i_k}\rightarrow e'$ or $e'\rightarrow \overline{i_{k+1}}$. But $e'\rightarrow \overline{i_{k+1}}$ is impossible, because $\overline{i_{k+1}}$ is an input edge, thus we must have $\overline{i_k}\rightarrow e'$. If $\overline{i_{k+1}}\rightarrow e'$,by property $(P_1)$ we have $\overline{i_{k+1}}\prec_2 e'$ which contradicts with $e'\prec_2 \overline{i_{k+1}}$, thus we must have $\overline{i_{k+1}}\nrightarrow e'$. So we get $e'_{P_k}\subseteq\{e'\in E(\Gamma_2)|\overline{i_k}\rightarrow e',\overline{i_{k+1}}\nrightarrow e'\}$.

On the other hand, if $\overline{i_k}\rightarrow e'$, by property $(P_1)$ and lemma $(A')$ we have $\overline{i_k}\prec_2 e'$ and $\overline{i_{min}(e')}\preceq_2 \overline{i_k}\preceq_2\overline{i_{max}(e')}$ . If $\overline{i_{k+1}}\nrightarrow e'$, by lemma $(A')$, we have $\overline{i_{k+1}}\prec_2 \overline{i_{min}(e')}$ or $ \overline{i_{max}(e')}\prec_2\overline{i_{k+1}}$. But $\overline{i_{k+1}}\prec_2 \overline{i_{min}(e')}$ contradicts with $\overline{i_{min}(e')}\preceq_2 \overline{i_k}\preceq_2\overline{i_{max}(e')}$, thus we must have $\overline{i_{max}(e')}\prec_2\overline{i_{k+1}}$. Now we have get $\overline{i_{max}(e')}\rightarrow e'$, $\overline{i_{max}(e')}\prec_2\overline{i_{k+1}}$ and $\overline{i_{max}(e')}\nrightarrow \overline{i_{k+1}}$, by property $(P^r_2)$ we get $e' \prec_2 \overline{i_{k+1}} $ or $\overline{i_{k+1}}\rightarrow e'$. Notice that $\overline{i_{k+1}}\rightarrow e'$ contradicts with $ \overline{i_{max}(e')}\prec_2\overline{i_{k+1}}$ due to lemma $(A')$, so we must have $e' \prec_2 \overline{i_{k+1}} $. Thus we have proved that $\overline{i_k}\prec_2 e' \prec_2 \overline{i_{k+1}}$ which means that $\{e'\in E(\Gamma_2)|\overline{i_k}\rightarrow e',\overline{i_{k+1}}\nrightarrow e'\}\subseteq e'_{P_k}$.

$\bullet$ Notice that $e'\in e'_{P_n}\Longleftrightarrow \overline{i_n}\prec_2 e'$. On one hand, if $\overline{i_n}\prec_2 e'$, by lemma $(A'')$ we have $\overline{i_n}=\overline{i_{max}(e')}$. Thus $\overline{i_n}\rightarrow e'$. On the other hand, if $\overline{i_n}\rightarrow e'$, by property $(P_1)$ we have $\overline{i_n}\prec_2 e'$. Thus we  prove that $e'_{P_n}=\{e'\in E(\Gamma_2)|\overline{i_n}\rightarrow e'\}.$
\end{proof}

It is easy to see that the above two propositions can be stated and proved for any planar graph without any restriction.

\begin{thm}
According to the definition of planar graph, the composition $(\overrightarrow{\Gamma},\prec)$ of $(\overrightarrow{\Gamma}_1,\prec_1)$ and $(\overrightarrow{\Gamma}_2,\prec_2)$ is a planar graph.
\end{thm}
\begin{proof}[Sketch of proof]
According to above proposition, we only need to show that $\prec$ satisfies condition $(P_1)$ and $(\widetilde{P_2})$.

$\bullet$ Condition $(P_1)$ is satisfied.  Let $e_1,e_2$ be two edges of $(\overrightarrow{\Gamma},\prec)$. If $e_1\rightarrow e_2$, we want to prove $e_1\prec e_2$. We will prove this in several cases:

Case $1$: $e_1,e_2\in E_o(\overrightarrow{\Gamma}_1) $ or $e_1,e_2\in E_o(\overrightarrow{\Gamma}_2) $. Obvious.

Case $2$: $e_1\in E_o(\overrightarrow{\Gamma}_1) $, $e_2\in E_o(\overrightarrow{\Gamma}_2) $. In this case, $e_1\rightarrow e_2$ implies there must exist an edge $\overline{e}_k=\{o_k,i_k\}\in E_n(\overrightarrow{\Gamma})$ such that $e_1\rightarrow e_2=e_1\rightarrow\overline{e}_k\rightarrow e_2$, thus $e_1\prec_1 \overline{o_k}$ and $\overline{i_k}\prec_2 e_2$. By the definition  and linearity of $\prec$, we have $e_1\prec e_2$.

Case $3$: $e_1\in E_o(\overrightarrow{\Gamma}_2) $, $e_2\in E_o(\overrightarrow{\Gamma}_1) $. Similar to case $2$.

Case $4$:  $e_1\in E_n(\overrightarrow{\Gamma}) $, $e_2\in E_o(\overrightarrow{\Gamma}_1) $. Similar to case $2$.

Case $5$: $e_1\in E_o(\overrightarrow{\Gamma}_1) $, $e_2\in E_n(\overrightarrow{\Gamma}) $. Similar to case $2$.

Case $6$: $e_1\in E_o(\overrightarrow{\Gamma}_2) $, $e_2\in E_n(\overrightarrow{\Gamma}) $. Similar to case $2$.

Case $7$:  $e_1\in E_n(\overrightarrow{\Gamma}) $, $e_2\in E_o(\overrightarrow{\Gamma}_2) $. Similar to case $2$.

Case $8$:  $e_1, e_2 \in E_n(\overrightarrow{\Gamma}) $.  In this case, $e_1\rightarrow e_2$ implies there must be an edge $e\in E_o(\overrightarrow{\Gamma}_1) $ or $E_o(\overrightarrow{\Gamma}_2)$, such that $e_1\rightarrow e_2=e_1\rightarrow e\rightarrow e_2$. By case $4$ or $5$, we have $e_1\prec e\prec e_2$, the linearity of $\prec$ implies $e_1\prec e_2$.

$\bullet$  Condition $(\widetilde{P_2})$ is satisfied.

Let $e_1,e_2,e_3$ be three edges of $(\overrightarrow{\Gamma},\prec)$. If $e_1e_2$ is a directed path in  $\overrightarrow{\Gamma}$, and $e_1\prec e_3 \prec e_2$, we want to prove that $e_1\rightarrow e_3$ or $e_3\rightarrow e_2$. We also prove this in several cases:

Case $1$:  $e_1, e_2, e_3\in E_o(\overrightarrow{\Gamma}_1)$ or $e_1, e_2, e_3\in E_o(\overrightarrow{\Gamma}_2)$. Obvious.

Case $2$: $e_1\in E_o(\overrightarrow{\Gamma}_1)$ and $e_2\in E_n(\overrightarrow{\Gamma})$. We assume $e_2=\overline{e}_k$ for some $k\in\{1,...,n\}$. In this case, we discuss in the following three cases:

\begin{flushleft}
$\circ$ case $2.1$:  $e_3\in E_o(\overrightarrow{\Gamma}_1)$.
\end{flushleft}

We translate the conditions about  $e_1,e_2 , e_3$ in $\overrightarrow{\Gamma}$ into conditions in  $\overrightarrow{\Gamma}_1$ about edges $e_1,\overline{ o_k}, e_3$. In fact, we have  the fact that $e_1e_2$ is a directed path in  $\overrightarrow{\Gamma}$ if and only if $e_1\overline{ o_k}$ is a direct path in $\overrightarrow{\Gamma}_1$ which is obvious from the definition of $\overrightarrow{\Gamma}$, and the facts  $$e_1\prec e_3\Longleftrightarrow e_1\prec_1 e_3$$ and $$e_3\prec e_2\Longleftrightarrow e_3\prec_1 \overline{ o_k}$$ which can be directly checked from the definition of $\prec$. Thus we have $e_1\prec_1 e_3\prec_1\overline{o_k}$ and then due to the fact that $(\overrightarrow{\Gamma}_1,\prec_1)$ is a planar graph we get the fact that
$e_1\rightarrow e_3$ or $e_3\rightarrow \overline{o_k}$ in $\overrightarrow{\Gamma}_1$. Translating back into conditions in $\overrightarrow{\Gamma}$ we get exactly that $e_1\rightarrow e_3$ or $e_3\rightarrow e_2$.

\begin{flushleft}
$\circ$ case  $2.2$:  $e_3\in E_o(\overrightarrow{\Gamma}_2)$.
\end{flushleft}

In this case, it is easy to see that $e_2\nrightarrow e_3$, thus  to prove $(P'_2)$  we only need to show that $e_1\rightarrow e_3$ in $\overrightarrow{\Gamma}$. Taking $e_1$ as an edge in $\overrightarrow{\Gamma}_1$ and $e_3$ as an edge in  $\overrightarrow{\Gamma}_2$ , we assume $$o_{min}(e_1)=o_p,\ o_{max}(e_1)=o_q,\ i_{min}(e_3)=i_u,\ i_{max}(e_3)=i_v$$ and $$e_1\in e_{Q_\alpha},\ \  e_2\in e'_{P_\beta},$$  with $ 1\leq \alpha,\beta, p, q, u,v\leq n$.

Due to lemma $(A')$ and $(B')$ and definition of $\overrightarrow{\Gamma}$, to show $e_1\rightarrow e_3$ in $\overrightarrow{\Gamma}$ we only need  to show that $\overline{o_p}\preceq\overline{o_v}\preceq\overline{o_q}$, or equivlently, $p\leq v\leq q$.

From the definition of $\prec$ and the fact that $e_1\prec e_3$ in $\overrightarrow{\Gamma}$, we have $$\alpha\leq \beta.$$

Noticing that $e_1\rightarrow e_2=\overline{e}_k$ in $\overrightarrow{\Gamma}$, which is equivalent to $e_1\rightarrow \overline{o_k}$ in $\overrightarrow{\Gamma}_1$, due to lemma $(A')$ we have $\overline{o_p}\preceq\overline{o_k}\preceq\overline{o_q}$, or equivalently, $$p\leq k\leq q.$$

As $e_1\in e_{I_\alpha}$, we  have $e_1\prec_1 \overline{o_{\alpha}}$. So according to lemma $(B'')$, we have $\overline{o_p}\preceq_1  \overline{o_{\alpha}}$, or equivalently, $$p\leq \alpha.$$

As $e_3\in e'_{P_\beta}$, we  have $\overline{i_{\beta}}\prec_2 e_3$. So according to lemma $(B'')$, we have $\overline{i_{\beta}}\prec_2 \overline{i_v}$, or equivalently, $$\beta\leq v.$$

According to the  linearity and  $(P_1)$ condition of $\prec$ we have proved above, $e_3\prec e_2$ implies $e_2\nrightarrow e_3$ in $\overrightarrow{\Gamma}$, which implies that $\overline{i_k}\nrightarrow e_3$ in $\overrightarrow{\Gamma}_2$. So by lemma $(A')$ we have $\overline{i_k}\prec_2\overline{i_u}$ or $\overline{i_v}\prec_2\overline{i_k}$, or equivalently, $k<u$ or $v< k$. But $u>k$ is impossible.  In fact, we translate the condition $e_3\prec e_2$ in $\overrightarrow{\Gamma}$ to condition in $\overrightarrow{\Gamma}_2$ about $e_3$ and $i_k$, that is, $$e_3\prec e_2\Longleftrightarrow e_3 \prec_2 i_k$$ which can be directly checked from the definition of $\prec$. If $\overline{i_k}\prec_2\overline{i_u}$, then one the one hand, due to the fact that  $e_3 \prec_2 i_k$ we have $e_3 \prec_2 \overline{i_u}$; on the other hand, from the definition of $\overline{i_u}$, we have $\overline{i_u}\rightarrow e_3$ in $\overrightarrow{\Gamma}_2$ which implies that $\overline{i_u}\prec_2 e_3$, a contradiction! Thus we must have $$v< k.$$

To summary, we have proved that $p\leq\alpha\leq\beta\leq v<k\leq q$, thus we have proved $e_1\rightarrow e_3$ in $\overrightarrow{\Gamma}$.

\begin{flushleft}
$\circ$ case  $2.3$:  $e_3\in E_n(\overrightarrow{\Gamma})$.
\end{flushleft}
 We assume $e_2=\overline{e}_k, e_3=\overline{e}_l$ for some $k,l\in\{1,...,n\}$. Now we translate the condition about $e_1, e_2 , e_3$ in $\overrightarrow{\Gamma}$ into conditions in  $\overrightarrow{\Gamma}_1$ about edges $e_1,\overline{ o_k}, \overline{ o_l}$. In fact, just as case 2.1, we have the fact that $e_1e_2$ is a directed path in  $\overrightarrow{\Gamma}$ if and only if $e_1\overline{ o_k}$ is a direct path in $\overrightarrow{\Gamma}_1$ which is obvious from the definition of $\overrightarrow{\Gamma}$, and the facts $$e_1\prec e_3\Longleftrightarrow e_1\prec_1 \overline{ o_l}$$ and $$e_3\prec e_2\Longleftrightarrow \overline{ o_l}\prec_1 \overline{ o_k}\Longleftrightarrow l<k$$ which can be directly checked from the definition of $\prec$. Now  we have two facts  in $\overrightarrow{\Gamma}_1$ that $e_1\rightarrow\overline{ o_k}$ and $e_1\prec_1 \overline{ o_l}\prec_1 \overline{ o_k}$ and apply condition $(P_2)$, we get that $e_1\rightarrow \overline{ o_l} $ or $\overline{ o_l}\rightarrow \overline{ o_k}$. Translating back into conditions in $\overrightarrow{\Gamma}$ we get that $e_1\rightarrow e_3$ or $e_3\rightarrow e_2$.

Case $3$:  $e_1\in E_n(\overrightarrow{\Gamma})$ and $e_2\in E_o(\overrightarrow{\Gamma}_2)$. In this case, we discuss in the following three cases:

\begin{flushleft}
$\circ$ case $3.1$:  $e_3\in E_o(\overrightarrow{\Gamma}_1)$.
\end{flushleft}
The proof is similar as case 2.1.
\begin{flushleft}
$\circ$ case  $3.2$: $e_3\in E_o(\overrightarrow{\Gamma}_2)$.
\end{flushleft}
The proof is similar as case 2.2.
\begin{flushleft}
$\circ$ case  $3.3$: $e_3\in E_n(\overrightarrow{\Gamma})$.
\end{flushleft}
The proof is similar as case 2.3.

\end{proof}

This theorem shows that  composition is a natural and well-defined operation on the class of planar graphs.  We denote by $(\overrightarrow{\Gamma}_2\circ\overrightarrow{\Gamma}_1,\prec_2\circ\prec_1)=(\overrightarrow{\Gamma}_2,\prec_2)\circ(\overrightarrow{\Gamma}_1,\prec_1)$  the composition of $(\overrightarrow{\Gamma}_1,\prec_1)$ and $(\overrightarrow{\Gamma}_2,\prec_2)$ by grafting the input legs of $\overrightarrow{\Gamma}_2$ and output legs of $\overrightarrow{\Gamma}_1$. The following proposition can be directly checked from proposition 2.4.5 and the definition of composition $\prec_2\circ\prec_1$, so we leave the proof to readers.
\begin{prop}
The composition $\circ$ of planar graphs is associative, that is, for any three composable planar graphs $(\overrightarrow{\Gamma}_1,\prec_1)$, $(\overrightarrow{\Gamma}_2,\prec_2)$, $(\overrightarrow{\Gamma}_3,\prec_3)$, we have
$$(\overrightarrow{\Gamma}_3,\prec_3)\circ(\overrightarrow{\Gamma}_2\circ\overrightarrow{\Gamma}_1,\prec_2\circ\prec_1)=(\overrightarrow{\Gamma}_3\circ\overrightarrow{\Gamma}_2,\prec_3\circ\prec_2)\circ(\overrightarrow{\Gamma}_1,\prec_1).$$
\end{prop}

\begin{prop}
For two pair of composable planar graphs $(\overrightarrow{\Gamma}_1,\prec_1),(\overrightarrow{\Gamma}_2,\prec_2)$ and $(\overrightarrow{\Gamma}_1',\prec_1'), (\overrightarrow{\Gamma}_2',\prec_2')$, the middle-four-interchange law satisfies:
$$((\overrightarrow{\Gamma}_2,\prec_2)\circ(\overrightarrow{\Gamma}_1,\prec_1))\otimes((\overrightarrow{\Gamma}_2',\prec_2')\circ(\overrightarrow{\Gamma}_1',\prec_1'))=((\overrightarrow{\Gamma}_2,\prec_2)\otimes(\overrightarrow{\Gamma}_2',\prec_2'))\circ((\overrightarrow{\Gamma}_1,\prec_1)\otimes(\overrightarrow{\Gamma}_1',\prec_1')).$$
\end{prop}

\begin{proof}
Directly check by definitions.
\end{proof}

The following theorem shows that any planar graph has a finest decomposition under the composition $\circ$ with essential prime planar graphs as building blocks.
\begin{thm}\label{decomposition}
Any non-invertible planar graph $(\overrightarrow{\Gamma}, \prec)$ with $n$ real vertices can be written as a composition $(\overrightarrow{\Gamma}_n, \prec_n)\circ\cdot\cdot\cdot\circ(\overrightarrow{\Gamma}_1, \prec_1)$ with each $(\overrightarrow{\Gamma}_k, \prec_k)$ $(1\leq k\leq n)$ being an essential  prime  graph.
\end{thm}
\begin{proof}[Sketch of proof]
 If $v\in V(\overrightarrow{\Gamma})$ be a maximal or minimal real vertex under $<_V$, we define two sub-planar graphs $(\overrightarrow{\Gamma_v},\prec_{\Gamma_v})$ and $(\overrightarrow{\Gamma-v},\prec_{\Gamma-v})$ as follows:

$\bullet $ $H(\Gamma_v)=v$,

$\bullet $ $P(\Gamma_v)=\{v\}$,

$\bullet $ for any $h\in H(\Gamma_v)$,
$\sigma_{\Gamma_v}(h)=h$,

$\bullet $ $sgn_{\Gamma_v}=sgn_{\Gamma}|_{H(\Gamma_v)}$,

$\bullet $ $\prec_{\Gamma_v}$ is the induced planar structure from $\prec$, \

\begin{flushleft}
and
\end{flushleft}

$\bullet $ $H(\Gamma-v)=H(\Gamma)-v$,

$\bullet $ $P(\Gamma-v)=P(\Gamma)-\{v\}$,

$\bullet $ for any $h\in H(\Gamma-v)$,
 \begin{equation*}\sigma_{\Gamma-v}(h)=
\begin{cases}
\sigma_{\Gamma}(h),&  \text{if}\ \sigma_{\Gamma}(h)\notin v,\\
h,&  \text{if}\ \sigma_{\Gamma}(h)\in v,\\
\end{cases}
\end{equation*}

$\bullet $ $sgn_{\Gamma-v}=sgn_{\Gamma}|_{H(\Gamma-v)}$,

$\bullet $ $\prec_{\Gamma-v}$ is the induced planar structure from $\prec$.

Evidently, $(\overrightarrow{\Gamma_v},\prec_{\Gamma_v})$ is a planar prime graph which we call a vertex subgraph, and it can be easily checked that $(\overrightarrow{\Gamma_v},\prec_{\Gamma_v})$ and $(\overrightarrow{\Gamma-v},\prec_{\Gamma-v})$ are indeed  sub-planar graphs of $(\overrightarrow{\Gamma},\prec)$.

Now we define a planar unitary graph $(\overrightarrow{\Gamma_U},\prec_U)$ as follows:

$\bullet $ $H(\Gamma_U)=\{+,-\}$,

$\bullet $ $P(\Gamma_U)=\{\{+,-\}\}$,

$\bullet $ $\sigma_{\Gamma_U}(+)=-,\  \sigma_{\Gamma_U}(-)=+$,

$\bullet $  $sgn(+)=+,\  sgn(-)=-$,

$\bullet $  $\{-\}\prec_U \{+\}$.

We now will prove this proposition by showing that any planar graph with at least two real vertices has a decomposition under composition.
That is, if  $(\overrightarrow{\Gamma}, \prec)$ is a planar graph with $n\geq2$ real vertices, and assume $v\in V_{re}(\overrightarrow{\Gamma})$ be a maximal vertex under $<_V$, we will define two new planar graphs  $(\overrightarrow{\Gamma}_1, \prec_1)$ and  $(\overrightarrow{\Gamma}_2, \prec_2)$ and show that
$$ (\overrightarrow{\Gamma}, \prec)= (\overrightarrow{\Gamma}_2, \prec_2)\circ (\overrightarrow{\Gamma}_1, \prec_1).$$

We assume $Out(\Gamma)=\{o_1,...,o_t\}$ such that $$\overline{o_{p}}\prec \overline{o_{q}} \Longleftrightarrow  p< q.$$
Obviously, we can assume $$Out(v)=\{o_{\mu}\in Out(\Gamma)| k\leq\mu\leq l,\  1\leq k\leq l\leq t\}$$ be a segment of $Out(\Gamma)$.

Then we define $$(\overrightarrow{\Gamma}_1, \prec_1)=(\overrightarrow{\Gamma-v}, \prec_{\Gamma-v})$$ and $$(\overrightarrow{\Gamma}_2, \prec_2)=(\overrightarrow{\Gamma_U}, \prec_U)^{\otimes (k-1)}\otimes (\overrightarrow{\Gamma_{v_2}}, \prec_{\Gamma_{v_2}}) \otimes (\overrightarrow{\Gamma_U}, \prec_U)^{\otimes (t-l)},$$
and by directly checking, we have $ (\overrightarrow{\Gamma}, \prec)= (\overrightarrow{\Gamma}_2, \prec_2)\circ (\overrightarrow{\Gamma}_1, \prec_1)$.

\end{proof}

On the geometric realization, Theorem \ref{decomposition} just tell that we can cut a planar graph into two planar graphs with one being an essential prime. The following is an example, we cut the planar graph in example 3.1.3 into two parts along the dashed line.
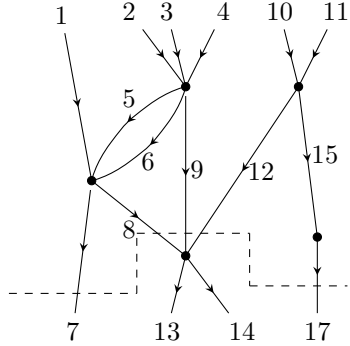
\begin{center}
\begin{tikzpicture}[scale=0.5]

\node (v2) at (-4,3) {};
\node (v1) at (-1.5,5.5) {};
\node (v7) at (-1.5,1) {};
\node (v9) at (1.5,5.5) {};
\node (v14) at (2,1.5) {};
\node (v3) at (-3,7.5) {$2$};
\node (v4) at (-2,7.5) {$3$};
\node (v5) at (-0.5,7.5) {$4$};
\node (v6) at (-4.8,7.4) {$1$};
\node (v11) at (-4.5,-1) {$7$};
\node (v12) at (-2,-1) {$13$};
\node (v13) at (0,-1) {$14$};
\node (v15) at (2,-1) {$17$};
\node (v8) at (1,7.5) {$10$};
\node (v10) at (2.5,7.5) {$11$};
\node  at (-2.5,3.5) {$6$};
\node  at (-3,5.2) {$5$};
\node  at (-1.2,3.3) {$9$};
\node  at (0.5,3.25) {$12$};
\node  at (2.2,3.7) {$15$};
\node  at (-3,1.7) {$8$};
\draw[fill] (-4,3) circle [radius=0.11];
\draw[fill] (v1) circle [radius=0.11];
\draw[fill] (v7) circle [radius=0.11];
\draw[fill] (v9) circle [radius=0.11];
\draw[fill] (v14) circle [radius=0.11];
\draw  plot[smooth, tension=1] coordinates {(v1) (-2.5,5)  (-3.5,4) (v2)}[postaction={decorate, decoration={markings,mark=at position .5 with {\arrow[black]{stealth}}}}];
\draw  plot[smooth, tension=1] coordinates {(v1) (-2,4.5)  (-3,3.5) (v2)}[postaction={decorate, decoration={markings,mark=at position .5 with {\arrow[black]{stealth}}}}];

\draw  (v3) -- (-1.5,5.5)[postaction={decorate, decoration={markings,mark=at position .5 with {\arrow[black]{stealth}}}}];
\draw  (v4) -- (-1.5,5.5)[postaction={decorate, decoration={markings,mark=at position .5 with {\arrow[black]{stealth}}}}];
\draw  (v5) -- (-1.5,5.5)[postaction={decorate, decoration={markings,mark=at position .5 with {\arrow[black]{stealth}}}}];

\draw  (v6) -- (-4,3)[postaction={decorate, decoration={markings,mark=at position .5 with {\arrow[black]{stealth}}}}];
\draw  (v1) -- (-1.5,1)[postaction={decorate, decoration={markings,mark=at position .5 with {\arrow[black]{stealth}}}}];
\draw  (-4,3) -- (-1.5,1)[postaction={decorate, decoration={markings,mark=at position .5 with {\arrow[black]{stealth}}}}];

\draw  (v8)--(1.5,5.5)[postaction={decorate, decoration={markings,mark=at position .5 with {\arrow[black]{stealth}}}}];
\draw  (v10) -- (1.5,5.5)[postaction={decorate, decoration={markings,mark=at position .5 with {\arrow[black]{stealth}}}}];
\draw  (1.5,5.5) -- (-1.5,1)[postaction={decorate, decoration={markings,mark=at position .5 with {\arrow[black]{stealth}}}}];
\draw  (v2) -- (v11)[postaction={decorate, decoration={markings,mark=at position .5 with {\arrow[black]{stealth}}}}];
\draw  (-1.5,1) -- (v12)[postaction={decorate, decoration={markings,mark=at position .65 with {\arrow[black]{stealth}}}}];
\draw  (v13) -- (-1.5,1)[postaction={decorate, decoration={markings,mark=at position .5 with {\arrowreversed[black]{stealth}}}}];
\draw  (1.5,5.5) -- (2,1.5)[postaction={decorate, decoration={markings,mark=at position .5 with {\arrow[black]{stealth}}}}];
\draw  (2,1.5) -- (v15)[postaction={decorate, decoration={markings,mark=at position .5 with {\arrow[black]{stealth}}}}];

\node (v16) at (-6.2,0) {};
\node (v21) at (2.8,0.2) {};
\node (v17) at (-2.8,0) {};
\node (v20) at (0.2,0.2) {};
\node (v18) at (-2.8,1.6) {};
\node (v19) at (0.2,1.6) {};
\draw [dashed]   (-6.2,0)-- (-2.8,0);
\draw [dashed]  (-2.8,0) edge (-2.8,1.6);
\draw [dashed]  (-2.8,1.6) edge (0.2,1.6);
\draw [dashed]  (0.2,1.6) edge (0.2,0.2);
\draw [dashed]  (0.2,0.2) edge (2.8,0.2);

\end{tikzpicture}
\end{center}

\subsection{Canonical isomorphism between $\mathsf{\Gamma}$ and  $\mathsf{G}$  }
In this section, we want to construction a canonical (preserving tensor product and composition)  isomorphism between the set $\mathsf{\Gamma}$ of planar graphs (isomorphic classes of (combinatorial) progressive planar graphs) and  the set $\mathsf{G}$ of plane graphs (isotopy classes of Joyal and Street's boxed/leveled progressive plane graphs). That is, we want to prove the following theorem:

\begin{thm}
There is an canonical isomorphism between $\mathsf{\Gamma}$ and  $\mathsf{G}$.
\end{thm}
\begin{proof}[Sketch of proof]
In order not to make the paper too long, we will not give a proof in all details but show the strategies to make the theorem obvious.

$\bullet$ First we give the definition of \textbf{planar geometric realization} of a combinatorial progressive planar graph and show that it is a leveled progressive plane graph unique up to a planar isotopy, in other word, we will define a function $\parallel\cdot\parallel:\mathsf{\Gamma}\rightarrow \mathsf{G}$.

For a prime planar graph, its planar geometric realization is defined to be a planar isotopic class of the boxed/leveled progressive plane graph  as
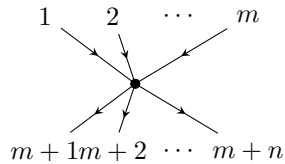
\begin{center}
\begin{tikzpicture}[scale=.6]
\node (v1) at (0,0) {};
\draw[fill] (0,0) circle [radius=0.1];
\node (v2) at (-2,1.5) {$1$};
\node (v3) at (-0.5,1.5) {$2$};
\node (v4) at (1,1.5) {$\cdots$};
\node (v5) at (2.5,1.5) {$m$};
\node (v6) at (-2,-1.5) {$m+1$};
\node (v7) at (-0.5,-1.5) {$m+2$};
\node at (1,-1.5) {$\cdots$};
\node (v8) at (2.5,-1.5) {$m+n$};
\draw  (v2) -- (0,0)[postaction={decorate, decoration={markings,mark=at position .50 with {\arrow[black]{stealth}}}}];
\draw  (v3) -- (0,0)[postaction={decorate, decoration={markings,mark=at position .50 with {\arrow[black]{stealth}}}}];
\draw  (v5) -- (0,0)[postaction={decorate, decoration={markings,mark=at position .50 with {\arrow[black]{stealth}}}}];
\draw  (v6) -- (0,0)[postaction={decorate, decoration={markings,mark=at position .50 with {\arrowreversed[black]{stealth}}}}];
\draw  (v7) -- (0,0)[postaction={decorate, decoration={markings,mark=at position .50 with {\arrowreversed[black]{stealth}}}}];
\draw  (v8) -- (0,0)[postaction={decorate, decoration={markings,mark=at position .50 with {\arrowreversed[black]{stealth}}}}];
\end{tikzpicture}
\end{center}
with the labels $\{1,2,...,m+n\}$ denoting the planar order of edges and arrows denoting orientation of edges.
For an unitary graph, its planar geometric realization is defined to be a planar isotopic class of the boxed/leveled progressive plane graph  as

\begin{center}
\begin{tikzpicture}[scale=.5]
\node (v2) at (0,0) {};
\draw (0,0) circle [radius=0.1];
\node (v1) at (0,2) {$1$};
\node (v3) at (0,-2) {$2$};
\draw  (v1) -- (0,0.1)[postaction={decorate, decoration={markings,mark=at position .50 with {\arrow[black]{stealth}}}}];
\draw  (v3) -- (0,-0.1)[postaction={decorate, decoration={markings,mark=at position .50 with {\arrowreversed[black]{stealth}}}}];
\end{tikzpicture}
\end{center}
with the labels $\{1,2\}$ denoting the planar order of edges and arrows denoting the orientations of edges.

For an elementary graph $(\overrightarrow{\Gamma}_1,\prec_1)\otimes(\overrightarrow{\Gamma}_2,\prec_2)\otimes\cdots \otimes(\overrightarrow{\Gamma}_n,\prec_n)$ with $(\overrightarrow{\Gamma}_i,\prec_i)$ $(1\leq i\leq n)$ being a prime or unitary planar graph, we define its planar geometric realization to be $\parallel(\overrightarrow{\Gamma}_1,\prec_1)\parallel\otimes_{JS}\parallel(\overrightarrow{\Gamma}_2,\prec_2)\parallel\otimes_{JS}\cdots \otimes_{JS}\parallel(\overrightarrow{\Gamma}_n,\prec_n)\parallel$ where $\otimes_{JS}$ denotes the tensor product of planar isotopic classes of boxed/leveled progressive plane graphs in \cite{[JS91]} (see appendix 7.2), and we draw its planar geometric realization schematically as
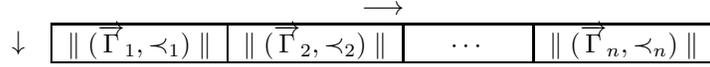
\[\begin{matrix}&\longrightarrow \\
\downarrow&\begin{tabular}{|c|c|ccc|c|}
\hline
$\parallel(\overrightarrow{\Gamma}_1,\prec_1)\parallel$&$\parallel(\overrightarrow{\Gamma}_2,\prec_2)\parallel$&&$\cdots$&&$\parallel(\overrightarrow{\Gamma}_n,\prec_n)\parallel$\\
\hline
\end{tabular}
\end{matrix}\]

From the definition of tensor product of $\otimes$, edge labels of every $\parallel(\overrightarrow{\Gamma}_1,\prec_i)\parallel$ $(1\leq i\leq n)$ form a block, and the blocks will increase from left to right.
Thus it is easily to see that planar geometric realizations of an elementary graphs is a progressive plane graphs  of Joyal and Street and is unique up to a planar isotropy, as example 3.1.2 shows.

Recall that in Theorem \ref{decomposition}, we have proved that every planar graph $(\overrightarrow{\Gamma},\prec)$ can be decomposed as composition of several essential prime graphs $(\overrightarrow{\Gamma}_n,\prec_n)\circ(\overrightarrow{\Gamma}_{n-1},\prec_{n-1})\circ\cdots\circ(\overrightarrow{\Gamma}_1,\prec_1)$, so we  define its  planar geometric realization  to be   $\parallel(\overrightarrow{\Gamma}_n,\prec_n)\parallel\circ_{JS}\parallel(\overrightarrow{\Gamma}_{n-1},\prec_{n-1})\parallel\circ_{JS}\cdots\circ_{JS}\parallel(\overrightarrow{\Gamma}_1,\prec_1)\parallel$,  where $\circ_{JS}$ denotes the composition of geometrical planar graphs in \cite{[JS91]}(see appendix 7.2), and we schematize the planar geometric realization of $(\overrightarrow{\Gamma},\prec)$ associated with this decomposition as
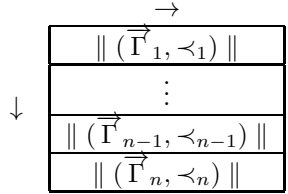
\[\begin{matrix}&\rightarrow \\
\downarrow& \begin{tabular}{|c|}
\hline
$\parallel(\overrightarrow{\Gamma}_1,\prec_1)\parallel$\\\hline
\vdots\\\hline
$\parallel(\overrightarrow{\Gamma}_{n-1},\prec_{n-1})\parallel$\\\hline
$\parallel(\overrightarrow{\Gamma}_n,\prec_n)\parallel$\\\hline
\end{tabular}
\end{matrix}\]
with the global  flow from top to down, thus its planar geometrical realization should be a boxed/leveled progressive plane graph as example 3.1.3 shows.

To show that planar geometric realization of $(\overrightarrow{\Gamma},\prec)$ is well-defined, that is, its  planar geometric realizations associated with different decompositions are planar isotopic, we will  prove using induction on the number  of real vertices $n=|V_{re}(\Gamma)|$.

When $n=1$. $(\overrightarrow{\Gamma},\prec)$ is a prime or an essential prime planar graph, we have shown that its planar geometric realization $\parallel(\overrightarrow{\Gamma},\prec)\parallel$ is unique up to a planar isotopy.

When $n=k$ with $k\geq 2$. We will prove in two cases.

Case 1: If $(\overrightarrow{\Gamma},\prec)$ has an unique decomposition as composition of essential prime planar graphs, we have an unique geometric realization up to a planar isotopy.

Case 2: If $(\overrightarrow{\Gamma},\prec)$ has two decompositions $(\overrightarrow{\Gamma}_k,\prec_k)\circ\cdots\circ(\overrightarrow{\Gamma}_1,\prec_1)$ and $(\overrightarrow{\Gamma'}_k,\prec'_k)\circ\cdots\circ(\overrightarrow{\Gamma'}_1,\prec'_1)$, we want to show that $\parallel(\overrightarrow{\Gamma}_k,\prec_k)\parallel\circ_{JS}\cdots\circ_{JS}\parallel(\overrightarrow{\Gamma}_1,\prec_1)\parallel$ and $\parallel(\overrightarrow{\Gamma'}_k,\prec'_k)\parallel\circ_{JS}\cdots\circ_{JS}\parallel(\overrightarrow{\Gamma'}_1,\prec'_1)\parallel$ are planar isotopic.

Let $v_i\in V_{re}(\Gamma_i)\subseteq V_{re}(\Gamma)$ and $v'_i\in V_{re}(\Gamma'_i)\subseteq V_{re}(\Gamma)$ be only real vertices of $\Gamma_i$ and $\Gamma'_i$ for $1\leq i\leq k$.

case 2.1: If $v_k=v'_k$, then we must have $(\overrightarrow{\Gamma}_k,\prec_k)\cong(\overrightarrow{\Gamma'}_k,\prec'_k)$ and $(\overrightarrow{\Gamma}_{k-1},\prec_{k-1})\circ\cdots\circ(\overrightarrow{\Gamma}_1,\prec_1) \cong(\overrightarrow{\Gamma'}_{k-1},\prec'_{k-1})\circ\cdots\circ(\overrightarrow{\Gamma'}_1,\prec'_1)$. Thus $\parallel(\overrightarrow{\Gamma}_k,\prec_k)\parallel$ and $(\overrightarrow{\Gamma'}_k,\prec'_k)|$ are planar isotopic, and by induction hypothesis $\parallel(\overrightarrow{\Gamma}_{k-1},\prec_{k-1})\parallel\circ_{JS}\cdots\circ_{JS}\parallel(\overrightarrow{\Gamma}_1,\prec_1) \parallel$ and $\parallel(\overrightarrow{\Gamma'}_{k-1},\prec'_{k-1})\parallel\circ_{JS}\cdots\circ_{JS}\parallel(\overrightarrow{\Gamma'}_1,\prec'_1)\parallel$ are planar isotopic. Hence $\parallel(\overrightarrow{\Gamma}_k,\prec_k)\parallel\circ_{JS}\cdots\circ_{JS}\parallel(\overrightarrow{\Gamma}_1,\prec_1)\parallel$ and $\parallel(\overrightarrow{\Gamma'}_k,\prec'_k)\parallel\circ_{JS}\cdots\circ_{JS}\parallel(\overrightarrow{\Gamma'}_1,\prec'_1)\parallel$ are planar isotopic.

case 2.2: If $v_k\neq v'_k$, then there must exist $l\in \{1,\cdots, k\}$ such that $v'_l=v_k$. It is easy to see that for every $j\in \{l+1,k\}$, we have
$v'_j\nrightarrow v'_l$ and $v'_l\nrightarrow v'_j$. Now our strategy of proof is to construct a series of decomposition of $(\overrightarrow{\Gamma},\prec)$ step by step  and a series of  planar isotopies between  their successive planar geometric realizations to interchange $v'_l$ with $v'_{l+1},\cdots, v'_{k}$. So the proof is reduced to prove the following fact: for every combinatorial progressive elementary planar graph, all its planar geometric realizations associated with different decompositions are planar isotopic. We think this fact obvious and leave the proof to readers.

Plus the fact that planar geometric realizations of  isomorphic combinatorial progressive planar graphs are homeomorphic, we can see that the construction of planar geometric realization is indeed  a well-defined function  $\parallel\cdot\parallel:\mathsf{\Gamma}\rightarrow \mathsf{G}$.

From the definition of planar geometric realization, we can easily see that
$$\parallel(\overrightarrow{\Gamma}_1,\prec_1)\otimes (\overrightarrow{\Gamma}_2,\prec_2)\parallel=\parallel(\overrightarrow{\Gamma}_1,\prec_1)\parallel\otimes_{JS} \parallel(\overrightarrow{\Gamma}_2,\prec_2)\parallel,$$
and for every two composable elementary graphs $(\overrightarrow{\Gamma}_1,\prec_1)$ and $(\overrightarrow{\Gamma}_2,\prec_2)$, we have $$\parallel(\overrightarrow{\Gamma}_2,\prec_2)\circ (\overrightarrow{\Gamma}_1,\prec_1)\parallel=\parallel(\overrightarrow{\Gamma}_2,\prec_2)\parallel\circ_{JS} \parallel(\overrightarrow{\Gamma}_1,\prec_1)\parallel.$$

$\bullet$  Now we want to construct the inverse function $\parallel\cdot\parallel^{-1}:\mathsf{G}\rightarrow \mathsf{\Gamma}$ of planar geometric realization. That is, for any leveled progressive plane graph we want to  associate a combinatorial progressive planar graph. In fact, if $G$ is a leveled progressive planar graph, we can define a combinatorial directed graph $\parallel G\parallel^{-1}=(H, P, \sigma, sgn)$ as follows:

$(1)$   $H=\bigsqcup_{x\in Inn_V(G)}\{x\}\times (In(x)\sqcup Out(x))\bigsqcup_{e\in E_{unitary}(G)}\{e^+,e^-\} $;

$(2)$   $P=\bigsqcup_{x\in Inn_V(G)}\{\{x\}\times (In(x)\sqcup Out(x))\}\bigsqcup_{e\in E_{unitary}(G)}\{\{e^+,e^-\}\} $;

$(3)$   $\sigma((x,e))=(s(e),e)$, if $(x,e)\in \{x\}\times In(x)$ and $e$ is an inner edge; $\sigma((x,e))=(t(e),e)$, if $(x,e)\in \{x\}\times Out(x)$ and $e$ is an inner edge; $\sigma((x,e))=(x,e)$, if $(x,e)\in \{x\}\times In(x)$ and $e$ is an external edge; $\sigma((x,e))=(x,e)$, if $(x,e)\in \{x\}\times Out(x)$ and $e$ is an external edge;  $\sigma(e^+)=e^-$ and $\sigma(e^-)=e^+$, if $e\in E_{unitary}(G)$;

$(4)$   $sgn((x,e))=+$, if $(x,e)\in \{x\}\times In(x)$; $sgn((x,e))=-$, if $(x,e)\in \{x\}\times Out(x)$; $sgn(e^+)=+$, $sgn(e^-)=-$, for $e\in E_{unitary}(G)$, where $Inn_V(G)$ denotes the set of inner nodes of $G$ and $E_{unitary}(G)$ denotes the set of unitary edges of $G$, for definitions of inner nodes and unitary edges for (topological)  see appendix 7.2.

Evidently, $\parallel G\parallel^{-1}$ is a progressive graph equipped with a polarization and an anchor structure and it is easy to see that $\parallel G\parallel^{-1}$ with these structures is invariant under any planar isotopy of $G$.  From the definition of $\parallel\cdot\parallel^{-1}$, the following facts can be directly checked:
$(1)$ for any two boxed/leveled progressive plane graphs $G^1, G^2$, $\parallel G^1\otimes_{JS} G^2\parallel^{-1}=\parallel G^1\parallel^{-1}\otimes \parallel G^2 \parallel^{-1}$ as progressive, polarized and anchored graphs; $(2)$ for any two composable boxed/leveled progressive plane graphs $G^1, G^2$, $\parallel G^2\circ_{JS} G^1\parallel^{-1}\cong\parallel G^2\parallel^{-1}\circ \parallel G^1 \parallel^{-1}$ as progressive, polarized and anchored graphs.

By proposition 3.2.8, we know that if possible there will be an unique planar structure compatible with the progressive, polarization and anchor structure, so our only task now is to show the existence of a planar structure on $\parallel G\parallel^{-1}$ which is  compatible with these structures. In fact, the existence of such a planar structure on $\parallel G\parallel^{-1}$ is a direct consequence of the following facts:

$(1)$ for any elementary boxed/progressive progressive plane graph $G'$, there is an unique planar  order on $\parallel G'\parallel^{-1}$ compatible with its progressive, polarization and anchor structure;

$(2)$ any boxed/progressive progressive plane graph $G'$ can be represented as a composition $G'^m\circ_{JS}\cdots \circ_{JS} G'^1$ with each $G'^i$ $(1\leq i\leq m)$ being elementary;

$(3)$ for any two composable boxed/leveled progressive plane graphs $G^1, G^2$, $\parallel G^2\circ_{JS} G^1\parallel^{-1}\cong\parallel G^2\parallel^{-1}\circ \parallel G^1 \parallel^{-1}$ as progressive, polarized and anchored graphs.

Thus $\parallel\cdot\parallel^{-1}:\mathsf{G}\rightarrow \mathsf{\Gamma}$ is well-defined. The fact that $\parallel\cdot\parallel^{-1}$ is indeed an inverse function of $\parallel\cdot\parallel$ can be deduced from the facts : $(1)$ $\parallel\cdot\parallel$ and $\parallel\cdot\parallel^{-1}$ are inverse functions of each other when restricted on classes of  prime and unitary graphs; $(2)$ both $\parallel\cdot\parallel$ and $\parallel\cdot\parallel^{-1}$ preserve tensor products and compositions.

\end{proof}

\begin{rem}
As a direct consequence, we see that for any planar set, its Hasse diagram is a plane graph naturally.
\end{rem}

\subsection{Fissus planar graphs and their coarse-graingings}
Here we want to introduce the notions of a fissus planar graph and its coarse-graining.

First recall that a set $X$ with a linear order $<$ is called a \textbf{linearly ordered set} or \textbf{linear set} for short, a partition of $X$ with each of its blocks being a segment of $(X,<)$ is called a \textbf{linear partition} or \textbf{$2$-nested linear set} and is denoted by $P(X,<)$ or $P$ for short.  Each block equipped with an induced linear order is called a linear block. The number of blocks of $P(X,<)$ is called the cardinality of $P(X,<)$, and we denote it as $|P(X,<)|$. The number of elements of $X$ is called length of $P(X,<)$ which is denoted by $||P(X,<)||.$ Product of linear sets $(X,<_1)=x_1<_1...<_1x_m$ and $(Y,<_2)=y_1<_2...<_2y_n$ is the linear set $(X,<_1)\otimes (Y,<_2)=x_1<...<x_m<y_1<...<y_n$.   Each linear partition can be equipped with a linear order (called\textbf{ block order}) in a natural way and we usually write $P(X,<)$ as $P_1<...<P_s$ with each $P_i$ $(1\leq i \leq s)$ being a linear block. A linear partition with only one block is called a trivial linear partition. A linear partition is called  finest if  each of its block is a one element set.

Let $P(X,<)=P_1<...<P_s$ and $Q(Y,<)=Q_1<...<Q_t$ be two linear partitions, we define their \textbf{product} to be $P(X,<)\otimes Q(Y,<)=P_1<...<P_s<Q_1<...<Q_t$ which is a linear partition of $(X,<)\otimes (Y,<)$. If  $P$ is a linear partition of $Q$, that is, $P_i=Q_{\mu_1+\cdots+\mu_{i-1}+1}<\cdots<Q_{\mu_1+\cdots+\mu_{i}}$ $(1\leq i\leq s)$, then we define their \textbf{composition} to be $P\triangleleft Q=\widetilde{P}_1<...<\widetilde{P}_s$ which is a linear partition such that each $\widetilde{P}_i=\underset{\mu_1+\cdots+\mu_{i-1}+1\leq \nu \leq \mu_1+\cdots+\mu_{i}}\bigotimes Q_{\nu}$ being the product of the  linear sets $Q_{\nu}$ $(\mu_1+\cdots+\mu_{i-1}+1\leq \nu \leq \mu_1+\cdots+\mu_{i}).$

We say they are \textbf{equivalent} if $s=t$ and  for each $i\in\{1,...,s\}$, the cardinals of $P_{i}$ and $Q_{i}$ are equal and we denote this fact as $P(X,<)\approx Q(Y,<)$. For any linear partition we call its equivalence class  a  "type" and to any linear partition we can associate a $2$-leveled planar rooted tree to denote its "type". If either $P$ or $Q$ is trivial, then $P\triangleleft Q$ is trivial; if $Q$ is finest, then $P\triangleleft Q\approx P$; if $P$ is finest, then $P\triangleleft Q\approx Q$.

\begin{ex}
The set $\{1,2,...,10\}$ ordered by the usual less than $<$ is a linear set. The partition $$\{\{1,2\}, \{3,4,5,6\}, \{7\}, \{8,9,10\}\}$$ is a linear partition with the linear order: $\{1,2\}<\{3,4,5,6\}<\{7\}<\{8,9,10\}$.   We also denote this linear partition as $(1 2)(3 4 5 6)(7)(8 9 10)$.Its associated $2$-leveled planar rooted tree is

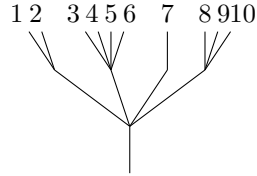
\begin{center}
\begin{tikzpicture}[scale=.5]
\node (v2) at (1,0) {};
\node (v6) at (1,-1.5) {};
\node (v1) at (-1,1.5) {};
\node (v3) at (0.5,1.5) {};
\node (v4) at (2,1.5) {};
\node (v5) at (3,1.5) {};
\draw  (-1,1.5) edge (1,0);
\draw (0.5,1.5) edge (1,0);
\draw  (2,1.5) edge (1,0);
\draw  (3,1.5) edge (1,0);
\draw  (v6) edge (1,0);
\node (v7) at (-2,3) {$1$};
\node (v8) at (-1.5,3) {$2$};
\node (v9) at (-0.5,3) {$3$};
\node (v10) at (0,3) {$4$};
\node (v11) at (0.5,3) {$5$};
\node (v12) at (1,3) {$6$};
\node (v13) at (2,3) {$7$};
\node (v14) at (3,3) {$8$};
\node (v15) at (3.5,3) {$9$};
\node (v16) at (4,3) {$10$};
\draw  (v7) edge  (-1,1.5);
\draw  (v8) edge  (-1,1.5);
\draw  (v9) edge (0.5,1.5);
\draw  (v10) edge (0.5,1.5);
\draw  (v11) edge (0.5,1.5);
\draw  (v12) edge (0.5,1.5);
\draw  (v13) edge (2,1.5);
\draw  (v14) edge (3,1.5);
\draw  (v15) edge (3,1.5);
\draw  (v16) edge (3,1.5);
\end{tikzpicture}.
\end{center}
\end{ex}

Using induction, for any $n\geq 2$ we can introduce the notions of a $n$-nested linear set,an equivalence of two $n$-nested linear sets.  We can also associate each $n$-nested linear set a $n$-levelled planar rooted tree as its "type".

Let $(\overrightarrow{\Gamma},\prec)$ be a $(m,n)$-planar graph with anchor structure $i_1<...<i_m$ and $o_1<...<o_n$. A \textbf{fission structure} of $(\overrightarrow{\Gamma},\prec)$ is a linear partition of its anchor structure, more precisely, a linear partition $P_{in}$ of $i_1<...<i_m$ and a linear partition $P_{out}$ of $o_1<...<o_n$. A planar graph $(\overrightarrow{\Gamma},\prec)$ equipped with a  fission structure $\mathcal{P}=(P_{in}, P_{out})$ is called a \textbf{fissus planar graph} and we denote it as $(\overrightarrow{\Gamma},\prec,P_{in}, P_{out})$ or $(\Gamma,\mathcal{P})$ for short. The input set and output set of a fissus planar graph equipped with the anchor structure and fission structure are linear partitions. If  both $P_{in}$ and $P_{out}$ are trivial linear partitions, we say $(\overrightarrow{\Gamma},\prec,P_{in}, P_{out})$ a trivial fissus planar graph; if both $P_{in}$ and $P_{out}$ are finest linear partitions, we say $(\overrightarrow{\Gamma},\prec,P_{in}, P_{out})$ a fully  fissus planar graph.

Two equivalent planar graphs equipped with equivalent fission structures are called equivalent fissus planar graphs, and we denote the set of their  equivalent classes by $\mathsf{\Gamma}_{F}.$

For any two fissus planar graphs $(\overrightarrow{\Gamma}_1,\prec_1,P_{in}, P_{out})$ and $(\overrightarrow{\Gamma}_2,\prec_2,Q_{in}, Q_{out})$, we define their tensor product  to be the fissus planar graphs $(\overrightarrow{\Gamma}_1\otimes \overrightarrow{\Gamma}_2,\prec_1\otimes \prec_2,P_{in}\otimes Q_{in}, P_{out}\otimes Q_{out} )$. If $ P_{out}\approx Q_{in}$, we can define their composition to be a fissus planar graph $(\overrightarrow{\Gamma}_2\circ \overrightarrow{\Gamma}_1,\prec_2\circ \prec_1,P'_{in}, Q'_{out} )$ with $P'_{in}\approx P_{in} , Q'_{out}\approx Q'_{out}$ being linear partitions of $In(\overrightarrow{\Gamma}_2\circ \overrightarrow{\Gamma}_1)$ and $Out(\overrightarrow{\Gamma}_2\circ \overrightarrow{\Gamma}_1)$, respectively.

\begin{defn}
The \textbf{coarse-graining} or \textbf{residue} of $(\overrightarrow{\Gamma},\prec,P_{in}, P_{out})$ is the prime $(|P_{in}|,|P_{out}|)$-planar graph with blocks of $P_{in}$ and $P_{out}$ as half-edges, where $|P_{in}|,|P_{out}|$ denote the cardinals of $|P_{in}|,|P_{out}|$, respectively.
\end{defn}
More concretely, let $(\Gamma,\prec, P_{in},P_{out})$ be a fissus planar diagram with $P_{in}=(i_1<\cdots <i_{\mu_1})<\cdots <(i_{\mu_1+\cdots+\mu_{m-1}+1}<\cdots< i_{\mu_1+\cdots+\mu_m})$ and $P_{out}=(o_1<\cdots <o_{\nu_1})<\cdots <(o_{\nu_1+\cdots+\nu_{n-1}+1}<\cdots< o_{\nu_1+\cdots+\nu_n})$, then its coarse-graining is a prime planar graph with the set of half-edges being $\{I_1,\cdots, I_m, O_1,\cdots, O_n\}$ where $I_1=\{i_1,\cdots,i_{\mu_1}\},\cdots,I_m=\{i_{\mu_1+\cdots+\mu_{m-1}+1},\cdots, i_{\mu_1+\cdots+\mu_m}\}, O_1=\{o_1,\cdots,o_{\nu_1}\},\cdots, O_n=\{o_{\nu_1+\cdots+\nu_{n-1}+1},\cdots, o_{\nu_1+\cdots+\nu_n}\}$. The planar order is $\overline{I_1}\prec\cdots\prec\overline{I_m}\prec\overline{O_1}\prec\cdots\prec\overline{O_n}$. We denote the coarse-graining of $(\overrightarrow{\Gamma},\prec,P_{in}, P_{out})$ as $\underbrace{(\overrightarrow{\Gamma},\prec,P_{in}, P_{out})}$. The coarse-graining of a prime fissus planar graph is called a \textbf{fusion}.
\begin{ex}
One possible fission structure of the planar graph in example $3.1.3$ is $P_{in}=\{\{1,2\},\{3,4,10\},\{11\}\}$ and $P_{out}=\{\{7\},\{13\},\{14\}, \{17\}\}$, here we identify half-edges with the numbers which label their corresponding edges.  The coarse-graining with respect to this fission structure is a prime $(3,4)$-planar graph.
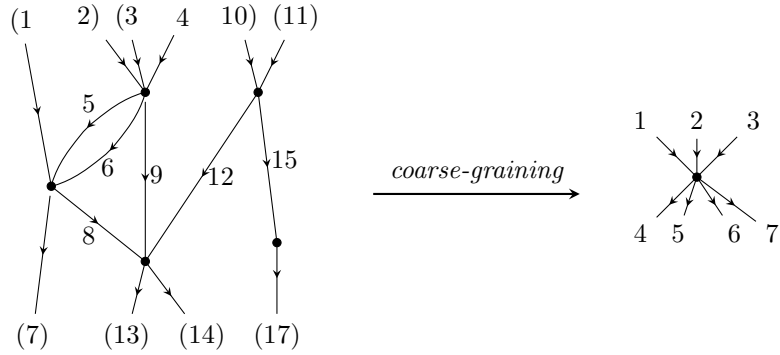
\begin{center}
$
\begin{matrix}
\begin{matrix}
\begin{tikzpicture}[scale=0.5]

\node (v2) at (-4,3) {};
\node (v1) at (-1.5,5.5) {};
\node (v7) at (-1.5,1) {};
\node (v9) at (1.5,5.5) {};
\node (v14) at (2,1.5) {};
\node (v3) at (-3,7.5) {$2)$};
\node (v4) at (-2,7.5) {$(3$};
\node (v5) at (-0.5,7.5) {$4$};
\node (v6) at (-4.8,7.4) {$(1$};
\node (v11) at (-4.5,-1) {$(7)$};
\node (v12) at (-2,-1) {$(13)$};
\node (v13) at (0,-1) {$(14)$};
\node (v15) at (2,-1) {$(17)$};
\node (v8) at (1,7.5) {$10)$};
\node (v10) at (2.5,7.5) {$(11)$};
\node  at (-2.5,3.5) {$6$};
\node  at (-3,5.2) {$5$};
\node  at (-1.2,3.3) {$9$};
\node  at (0.5,3.25) {$12$};
\node  at (2.2,3.7) {$15$};
\node  at (-3,1.7) {$8$};
\draw[fill] (-4,3) circle [radius=0.11];
\draw[fill] (v1) circle [radius=0.11];
\draw[fill] (v7) circle [radius=0.11];
\draw[fill] (v9) circle [radius=0.11];
\draw[fill] (v14) circle [radius=0.11];
\draw  plot[smooth, tension=1] coordinates {(v1) (-2.5,5)  (-3.5,4) (v2)}[postaction={decorate, decoration={markings,mark=at position .5 with {\arrow[black]{stealth}}}}];
\draw  plot[smooth, tension=1] coordinates {(v1) (-2,4.5)  (-3,3.5) (v2)}[postaction={decorate, decoration={markings,mark=at position .5 with {\arrow[black]{stealth}}}}];

\draw  (v3) -- (-1.5,5.5)[postaction={decorate, decoration={markings,mark=at position .5 with {\arrow[black]{stealth}}}}];
\draw  (v4) -- (-1.5,5.5)[postaction={decorate, decoration={markings,mark=at position .5 with {\arrow[black]{stealth}}}}];
\draw  (v5) -- (-1.5,5.5)[postaction={decorate, decoration={markings,mark=at position .5 with {\arrow[black]{stealth}}}}];

\draw  (v6) -- (-4,3)[postaction={decorate, decoration={markings,mark=at position .5 with {\arrow[black]{stealth}}}}];
\draw  (v1) -- (-1.5,1)[postaction={decorate, decoration={markings,mark=at position .5 with {\arrow[black]{stealth}}}}];
\draw  (-4,3) -- (-1.5,1)[postaction={decorate, decoration={markings,mark=at position .5 with {\arrow[black]{stealth}}}}];

\draw  (v8)--(1.5,5.5)[postaction={decorate, decoration={markings,mark=at position .5 with {\arrow[black]{stealth}}}}];
\draw  (v10) -- (1.5,5.5)[postaction={decorate, decoration={markings,mark=at position .5 with {\arrow[black]{stealth}}}}];
\draw  (1.5,5.5) -- (-1.5,1)[postaction={decorate, decoration={markings,mark=at position .5 with {\arrow[black]{stealth}}}}];
\draw  (v2) -- (v11)[postaction={decorate, decoration={markings,mark=at position .5 with {\arrow[black]{stealth}}}}];
\draw  (-1.5,1) -- (v12)[postaction={decorate, decoration={markings,mark=at position .65 with {\arrow[black]{stealth}}}}];
\draw  (v13) -- (-1.5,1)[postaction={decorate, decoration={markings,mark=at position .5 with {\arrowreversed[black]{stealth}}}}];
\draw  (1.5,5.5) -- (2,1.5)[postaction={decorate, decoration={markings,mark=at position .5 with {\arrow[black]{stealth}}}}];
\draw  (2,1.5) -- (v15)[postaction={decorate, decoration={markings,mark=at position .5 with {\arrow[black]{stealth}}}}];

\end{tikzpicture}
\end{matrix}
&\begin{matrix}
\begin{tikzpicture}
\node (v1) at (-2,0.5) {};
\node (v2) at (1,0.5) {};
\draw [thick,->,>=stealth] (v1) -- (v2);
\node at (-0.5,0.8) {coarse-graining};
\end{tikzpicture}
\end{matrix}&
\begin{matrix}
\begin{tikzpicture}[scale=.5]
\draw[fill] (0,0.5) circle [radius=0.11];
\node (v2) at (0,0.5) {};
\node (v1) at (-1.5,2) {$1$};
\node (v3) at (0,2) {$2$};
\node (v4) at (1.5,2) {$3$};
\node (v5) at (-1.5,-1) {$4$};
\node (v6) at (-0.5,-1) {$5$};
\node (v7) at (1,-1) {$6$};
\node (v8) at (2,-1) {$7$};
\draw  (v1) --  (0,0.5)[postaction={decorate, decoration={markings,mark=at position .5 with {\arrow[black]{stealth}}}}];
\draw  (v3) --  (0,0.5)[postaction={decorate, decoration={markings,mark=at position .5 with {\arrow[black]{stealth}}}}];
\draw  (v4) --  (0,0.5)[postaction={decorate, decoration={markings,mark=at position .5 with {\arrow[black]{stealth}}}}];
\draw  (v5) --  (0,0.5)[postaction={decorate, decoration={markings,mark=at position .5 with {\arrowreversed[black]{stealth}}}}];
\draw  (v6) --  (0,0.5)[postaction={decorate, decoration={markings,mark=at position .5 with {\arrowreversed[black]{stealth}}}}];
\draw  (v7) --  (0,0.5)[postaction={decorate, decoration={markings,mark=at position .5 with {\arrowreversed[black]{stealth}}}}];
\draw  (v8) --  (0,0.5)[postaction={decorate, decoration={markings,mark=at position .5 with {\arrowreversed[black]{stealth}}}}];
\end{tikzpicture}
\end{matrix}
\end{matrix}
$
\end{center}
\end{ex}

\begin{ex}
One possible fission structure of the planar graph in example $3.1.4$ is $P_{in}=\{\{1,3,5\},\{6,7,10,12\},\{15\}\}$ and $P_{out}=\{\{2,4\},\{8,9,11\},\{13,14,16\}\}$, here we identify half-edges with the numbers which they are labelled.  The coarse-graining with respect to this fission structure is a prime $(3,3)$-planar graph.

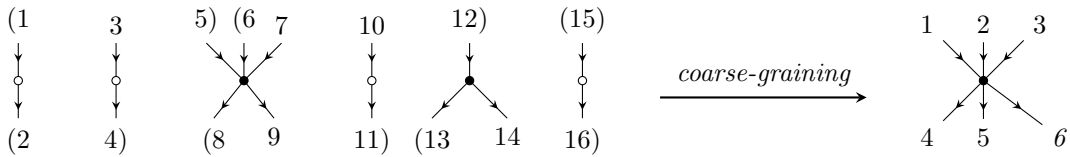
\begin{center}
$
\begin{matrix}
\begin{matrix}
\begin{tikzpicture}

\draw (0,0) circle [radius=0.055];
\draw (0,-0.055)-- (0,-0.5)[postaction={decorate, decoration={markings,mark=at position .75 with {\arrow[black]{stealth}}}}];
\draw (0,0.055)-- (0,0.5)[postaction={decorate, decoration={markings,mark=at position .75 with {\arrowreversed[black]{stealth}}}}];
\draw (1.3,0) circle [radius=0.055];
\draw (1.3,-0.055)-- (1.3,-0.5)[postaction={decorate, decoration={markings,mark=at position .75 with {\arrow[black]{stealth}}}}];
\draw (1.3,0.055)-- (1.3,0.5)[postaction={decorate, decoration={markings,mark=at position .75 with {\arrowreversed[black]{stealth}}}}];
\draw[fill] (3,0) circle [radius=0.055];
\draw (3,0)-- (3.5,0.5)[postaction={decorate, decoration={markings,mark=at position .75 with {\arrowreversed[black]{stealth}}}}];
\draw (3,0)-- (2.5,0.5)[postaction={decorate, decoration={markings,mark=at position .75 with {\arrowreversed[black]{stealth}}}}];
\draw (3,0)-- (3,0.5)[postaction={decorate, decoration={markings,mark=at position .75 with {\arrowreversed[black]{stealth}}}}];
\draw (3,0)-- (2.6,-0.5)[postaction={decorate, decoration={markings,mark=at position .75 with {\arrow[black]{stealth}}}}];
\draw (3,0)-- (3.4,-0.5)[postaction={decorate, decoration={markings,mark=at position .75 with {\arrow[black]{stealth}}}}];
\draw (4.7,0) circle [radius=0.055];
\draw (4.7,-0.055)-- (4.7,-0.5)[postaction={decorate, decoration={markings,mark=at position .75 with {\arrow[black]{stealth}}}}];
\draw (4.7,0.055)-- (4.7,0.5)[postaction={decorate, decoration={markings,mark=at position .75 with {\arrowreversed[black]{stealth}}}}];
\draw[fill] (6,0) circle [radius=0.055];
\draw (6,0)-- (6,0.5)[postaction={decorate, decoration={markings,mark=at position .75 with {\arrowreversed[black]{stealth}}}}];
\draw (6,0)-- (6.5,-0.5)[postaction={decorate, decoration={markings,mark=at position .75 with {\arrow[black]{stealth}}}}];
\draw (6,0)-- (5.5,-0.5)[postaction={decorate, decoration={markings,mark=at position .75 with {\arrow[black]{stealth}}}}];
\draw (7.5,0) circle [radius=0.055];
\draw (7.5,-0.055)-- (7.5,-0.5)[postaction={decorate, decoration={markings,mark=at position .75 with {\arrow[black]{stealth}}}}];
\draw (7.5,0.055)-- (7.5,0.5)[postaction={decorate, decoration={markings,mark=at position .75 with {\arrowreversed[black]{stealth}}}}];

\node [below] at (0,-0.5) {$( 2$};
\node [above] at (0,0.5) {$( 1$};
\node [below] at (1.3,-0.5) {$4 )$};
\node [above] at (1.3,0.5) {$3$};
\node [above] at (3.5,0.5) {$7$};
\node [above] at (3,0.5) {$(6$};
\node [above] at (2.5,0.5) {$5 )$};
\node [below] at (2.6,-0.5) {$( 8$};
\node [below] at (3.4,-0.5) {$9$};
\node [below] at (4.7,-0.5) {$11 )$};
\node [above] at (4.7,0.5) {$10$};
\node [above] at (6,0.5) {$12 )$};
\node [below] at (6.5,-0.5) {$14$};
\node [below] at (5.5,-0.5) {$(13$};
\node [below] at (7.5,-0.5) {$16 )$};
\node [above] at (7.5,0.5) {$( 15 )$};

\end{tikzpicture}
\end{matrix}
&\begin{matrix}
\begin{tikzpicture}
\node (v1) at (-2,0.5) {};
\node (v2) at (1,0.5) {};
\draw [thick,->,>=stealth] (v1) -- (v2);
\node at (-0.5,0.8) {coarse-graining};
\end{tikzpicture}
\end{matrix}&
\begin{matrix}
\begin{tikzpicture}[scale=.5]
\draw[fill] (0,0.5) circle [radius=0.11];
\node (v2) at (0,0.5) {};
\node (v1) at (-1.5,2) {$1$};
\node (v3) at (0,2) {$2$};
\node (v4) at (1.5,2) {$3$};
\node (v5) at (-1.5,-1) {$4$};
\node (v6) at (0,-1) {$5$};

\node (v8) at (2,-1) {6};
\draw  (v1) --  (0,0.5)[postaction={decorate, decoration={markings,mark=at position .5 with {\arrow[black]{stealth}}}}];
\draw  (v3) --  (0,0.5)[postaction={decorate, decoration={markings,mark=at position .5 with {\arrow[black]{stealth}}}}];
\draw  (v4) --  (0,0.5)[postaction={decorate, decoration={markings,mark=at position .5 with {\arrow[black]{stealth}}}}];
\draw  (v5) --  (0,0.5)[postaction={decorate, decoration={markings,mark=at position .5 with {\arrowreversed[black]{stealth}}}}];
\draw  (v6) --  (0,0.5)[postaction={decorate, decoration={markings,mark=at position .5 with {\arrowreversed[black]{stealth}}}}];

\draw  (v8) --  (0,0.5)[postaction={decorate, decoration={markings,mark=at position .5 with {\arrowreversed[black]{stealth}}}}];
\end{tikzpicture}
\end{matrix}
\end{matrix}
$
\end{center}

\end{ex}

Let $(\overrightarrow{\Gamma},\prec)$ be a planar graph, then there will be a progressive structure and a planar structure $\underline{\prec}$ on the quotient graph $\Gamma/\Gamma$ naturally induced from $\prec$. Hence $\Gamma/\Gamma$ equipped these structures is a planar graph which is called  \textbf{contraction} of $(\overrightarrow{\Gamma},\prec)$ and is denoted by $(\underline{\overrightarrow{\Gamma}},\underline{\prec})$ or $\underline{(\overrightarrow{\Gamma},\prec)}$. In fact, $H(\Gamma/\Gamma)=In(\Gamma)\sqcup Out(\Gamma)$, $P(\Gamma/\Gamma)=\{H(\Gamma/\Gamma)\}$ and $\sigma=id_{H(\Gamma/\Gamma)}$. If $In(\Gamma)=i_1<...< i_m$ and $Out(\Gamma)=o_1<...< o_n$, then the planar structure is given as $\overline{i_1}\underline{\prec}...\underline{\prec} \overline{i_m}\underline{\prec} \overline{o_1}\underline{\prec}...\underline{\prec} \overline{o_n}$. From the definition, we see that contraction of an unitary graph will be a prime $(1,1)$-planar graph.

\begin{ex}
Contraction of the planar graph in example $3.1.3$ is a prime $(6,4)$-planar graph.

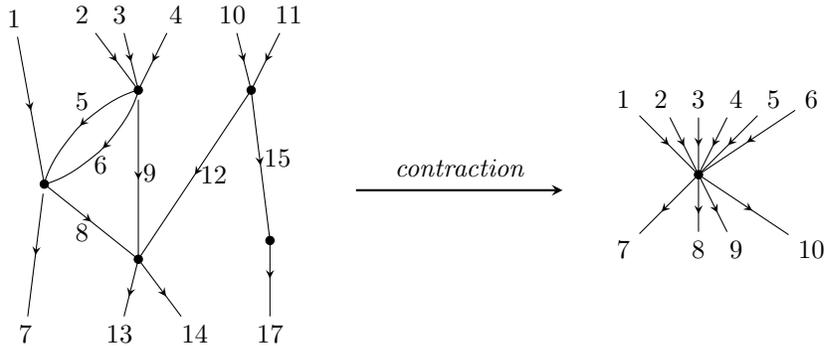
\begin{center}
$
\begin{matrix}
\begin{matrix}
\begin{tikzpicture}[scale=0.5]

\node (v2) at (-4,3) {};
\node (v1) at (-1.5,5.5) {};
\node (v7) at (-1.5,1) {};
\node (v9) at (1.5,5.5) {};
\node (v14) at (2,1.5) {};
\node (v3) at (-3,7.5) {$2$};
\node (v4) at (-2,7.5) {$3$};
\node (v5) at (-0.5,7.5) {$4$};
\node (v6) at (-4.8,7.4) {$1$};
\node (v11) at (-4.5,-1) {$7$};
\node (v12) at (-2,-1) {$13$};
\node (v13) at (0,-1) {$14$};
\node (v15) at (2,-1) {$17$};
\node (v8) at (1,7.5) {$10$};
\node (v10) at (2.5,7.5) {$11$};
\node  at (-2.5,3.5) {$6$};
\node  at (-3,5.2) {$5$};
\node  at (-1.2,3.3) {$9$};
\node  at (0.5,3.25) {$12$};
\node  at (2.2,3.7) {$15$};
\node  at (-3,1.7) {$8$};
\draw[fill] (-4,3) circle [radius=0.11];
\draw[fill] (v1) circle [radius=0.11];
\draw[fill] (v7) circle [radius=0.11];
\draw[fill] (v9) circle [radius=0.11];
\draw[fill] (v14) circle [radius=0.11];
\draw  plot[smooth, tension=1] coordinates {(v1) (-2.5,5)  (-3.5,4) (v2)}[postaction={decorate, decoration={markings,mark=at position .5 with {\arrow[black]{stealth}}}}];
\draw  plot[smooth, tension=1] coordinates {(v1) (-2,4.5)  (-3,3.5) (v2)}[postaction={decorate, decoration={markings,mark=at position .5 with {\arrow[black]{stealth}}}}];

\draw  (v3) -- (-1.5,5.5)[postaction={decorate, decoration={markings,mark=at position .5 with {\arrow[black]{stealth}}}}];
\draw  (v4) -- (-1.5,5.5)[postaction={decorate, decoration={markings,mark=at position .5 with {\arrow[black]{stealth}}}}];
\draw  (v5) -- (-1.5,5.5)[postaction={decorate, decoration={markings,mark=at position .5 with {\arrow[black]{stealth}}}}];

\draw  (v6) -- (-4,3)[postaction={decorate, decoration={markings,mark=at position .5 with {\arrow[black]{stealth}}}}];
\draw  (v1) -- (-1.5,1)[postaction={decorate, decoration={markings,mark=at position .5 with {\arrow[black]{stealth}}}}];
\draw  (-4,3) -- (-1.5,1)[postaction={decorate, decoration={markings,mark=at position .5 with {\arrow[black]{stealth}}}}];

\draw  (v8)--(1.5,5.5)[postaction={decorate, decoration={markings,mark=at position .5 with {\arrow[black]{stealth}}}}];
\draw  (v10) -- (1.5,5.5)[postaction={decorate, decoration={markings,mark=at position .5 with {\arrow[black]{stealth}}}}];
\draw  (1.5,5.5) -- (-1.5,1)[postaction={decorate, decoration={markings,mark=at position .5 with {\arrow[black]{stealth}}}}];
\draw  (v2) -- (v11)[postaction={decorate, decoration={markings,mark=at position .5 with {\arrow[black]{stealth}}}}];
\draw  (-1.5,1) -- (v12)[postaction={decorate, decoration={markings,mark=at position .65 with {\arrow[black]{stealth}}}}];
\draw  (v13) -- (-1.5,1)[postaction={decorate, decoration={markings,mark=at position .5 with {\arrowreversed[black]{stealth}}}}];
\draw  (1.5,5.5) -- (2,1.5)[postaction={decorate, decoration={markings,mark=at position .5 with {\arrow[black]{stealth}}}}];
\draw  (2,1.5) -- (v15)[postaction={decorate, decoration={markings,mark=at position .5 with {\arrow[black]{stealth}}}}];

\end{tikzpicture}
\end{matrix}
&\begin{matrix}
\begin{tikzpicture}
\node (v1) at (-2,0.5) {};
\node (v2) at (1,0.5) {};
\draw [thick,->,>=stealth] (v1) -- (v2);
\node at (-0.5,0.8) {contraction};
\end{tikzpicture}
\end{matrix}&
\begin{matrix}
\begin{tikzpicture}[scale=1]
\draw[fill] (0,0) circle [radius=0.055];

\node (v2) at (0,0) {};
\node (v1) at (-1,1) {$1$};
\node (v3) at (-0.5,1) {$2$};
\node (v4) at (0,1) {$3$};
\node (v5) at (0.5,1) {$4$};
\node (v6) at (1,1) {$5$};
\node (v7) at (1.5,1) {$6$};
\node (v8) at (-1,-1) {$7$};
\node (v9) at (0,-1) {$8$};
\node (v10) at (0.5,-1) {$9$};
\node (v11) at (1.5,-1) {$10$};
\draw  (v1) -- (0,0)[postaction={decorate, decoration={markings,mark=at position .5 with {\arrow[black]{stealth}}}}];
\draw  (v3) -- (0,0)[postaction={decorate, decoration={markings,mark=at position .5 with {\arrow[black]{stealth}}}}];
\draw  (v4) -- (0,0)[postaction={decorate, decoration={markings,mark=at position .5 with {\arrow[black]{stealth}}}}];
\draw  (v5) -- (0,0)[postaction={decorate, decoration={markings,mark=at position .5 with {\arrow[black]{stealth}}}}];
\draw  (v6) -- (0,0)[postaction={decorate, decoration={markings,mark=at position .5 with {\arrow[black]{stealth}}}}];
\draw  (v7) -- (0,0)[postaction={decorate, decoration={markings,mark=at position .5 with {\arrow[black]{stealth}}}}];
\draw  (v8) -- (0,0)[postaction={decorate, decoration={markings,mark=at position .5 with {\arrowreversed[black]{stealth}}}}];
\draw  (v9) -- (0,0)[postaction={decorate, decoration={markings,mark=at position .5 with {\arrowreversed[black]{stealth}}}}];
\draw  (v10)-- (0,0)[postaction={decorate, decoration={markings,mark=at position .5 with {\arrowreversed[black]{stealth}}}}];
\draw  (v11)-- (0,0)[postaction={decorate, decoration={markings,mark=at position .5 with {\arrowreversed[black]{stealth}}}}];
\end{tikzpicture}
\end{matrix}
\end{matrix}
$
\end{center}

\end{ex}

Any fission structure on a fissus planar graph can transfer to be a fission structure on its contraction. In fact, if $(\overrightarrow{\Gamma},\prec)$ is equipped with a fission structure $(P_{in}, P_{out})$, then $(P_{in}, P_{out})$ is also a fission structure of the contraction $(\underline{\overrightarrow{\Gamma}},\underline{\prec})$. With respect to this fission structure, the fusion of $(\underline{\overrightarrow{\Gamma}},\underline{\prec})$ is equivalent to $\underbrace{(\overrightarrow{\Gamma},\prec,P_{in}, P_{out})}$, thus we can see that coarse-graining can be decomposed as  composition of  fusion and contraction.

\begin{ex}
For the fissus planar graph in example $3.7.4$, we have

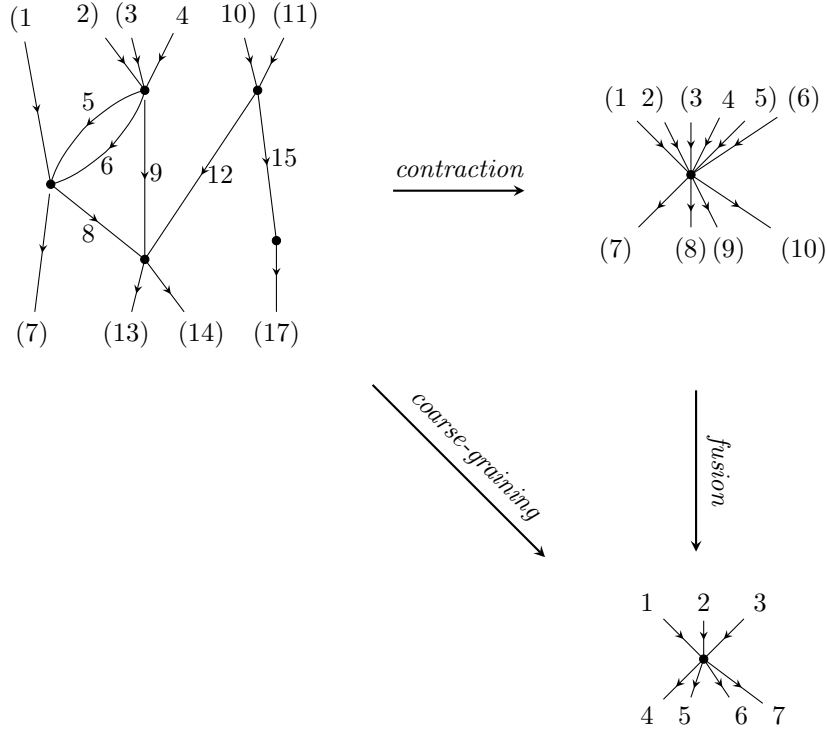
\begin{center}
$
\begin{matrix}
\begin{matrix}
\begin{tikzpicture}[scale=0.5]

\node (v2) at (-4,3) {};
\node (v1) at (-1.5,5.5) {};
\node (v7) at (-1.5,1) {};
\node (v9) at (1.5,5.5) {};
\node (v14) at (2,1.5) {};
\node (v3) at (-3,7.5) {$2)$};
\node (v4) at (-2,7.5) {$(3$};
\node (v5) at (-0.5,7.5) {$4$};
\node (v6) at (-4.8,7.4) {$(1$};
\node (v11) at (-4.5,-1) {$(7)$};
\node (v12) at (-2,-1) {$(13)$};
\node (v13) at (0,-1) {$(14)$};
\node (v15) at (2,-1) {$(17)$};
\node (v8) at (1,7.5) {$10)$};
\node (v10) at (2.5,7.5) {$(11)$};
\node  at (-2.5,3.5) {$6$};
\node  at (-3,5.2) {$5$};
\node  at (-1.2,3.3) {$9$};
\node  at (0.5,3.25) {$12$};
\node  at (2.2,3.7) {$15$};
\node  at (-3,1.7) {$8$};
\draw[fill] (-4,3) circle [radius=0.11];
\draw[fill] (v1) circle [radius=0.11];
\draw[fill] (v7) circle [radius=0.11];
\draw[fill] (v9) circle [radius=0.11];
\draw[fill] (v14) circle [radius=0.11];
\draw  plot[smooth, tension=1] coordinates {(v1) (-2.5,5)  (-3.5,4) (v2)}[postaction={decorate, decoration={markings,mark=at position .5 with {\arrow[black]{stealth}}}}];
\draw  plot[smooth, tension=1] coordinates {(v1) (-2,4.5)  (-3,3.5) (v2)}[postaction={decorate, decoration={markings,mark=at position .5 with {\arrow[black]{stealth}}}}];

\draw  (v3) -- (-1.5,5.5)[postaction={decorate, decoration={markings,mark=at position .5 with {\arrow[black]{stealth}}}}];
\draw  (v4) -- (-1.5,5.5)[postaction={decorate, decoration={markings,mark=at position .5 with {\arrow[black]{stealth}}}}];
\draw  (v5) -- (-1.5,5.5)[postaction={decorate, decoration={markings,mark=at position .5 with {\arrow[black]{stealth}}}}];

\draw  (v6) -- (-4,3)[postaction={decorate, decoration={markings,mark=at position .5 with {\arrow[black]{stealth}}}}];
\draw  (v1) -- (-1.5,1)[postaction={decorate, decoration={markings,mark=at position .5 with {\arrow[black]{stealth}}}}];
\draw  (-4,3) -- (-1.5,1)[postaction={decorate, decoration={markings,mark=at position .5 with {\arrow[black]{stealth}}}}];

\draw  (v8)--(1.5,5.5)[postaction={decorate, decoration={markings,mark=at position .5 with {\arrow[black]{stealth}}}}];
\draw  (v10) -- (1.5,5.5)[postaction={decorate, decoration={markings,mark=at position .5 with {\arrow[black]{stealth}}}}];
\draw  (1.5,5.5) -- (-1.5,1)[postaction={decorate, decoration={markings,mark=at position .5 with {\arrow[black]{stealth}}}}];
\draw  (v2) -- (v11)[postaction={decorate, decoration={markings,mark=at position .5 with {\arrow[black]{stealth}}}}];
\draw  (-1.5,1) -- (v12)[postaction={decorate, decoration={markings,mark=at position .65 with {\arrow[black]{stealth}}}}];
\draw  (v13) -- (-1.5,1)[postaction={decorate, decoration={markings,mark=at position .5 with {\arrowreversed[black]{stealth}}}}];
\draw  (1.5,5.5) -- (2,1.5)[postaction={decorate, decoration={markings,mark=at position .5 with {\arrow[black]{stealth}}}}];
\draw  (2,1.5) -- (v15)[postaction={decorate, decoration={markings,mark=at position .5 with {\arrow[black]{stealth}}}}];

\end{tikzpicture}
\end{matrix}
&\begin{matrix}
\begin{tikzpicture}
\node (v1) at (-1.5,0.5) {};
\node (v2) at (0.5,0.5) {};
\draw [thick,->,>=stealth] (v1) -- (v2);
\node at (-0.5,0.8) {contraction};
\end{tikzpicture}
\end{matrix}&
\begin{matrix}
\begin{tikzpicture}[scale=1]
\draw[fill] (0,0) circle [radius=0.055];

\node (v2) at (0,0) {};
\node (v1) at (-1,1) {$(1$};
\node (v3) at (-0.5,1) {$2)$};
\node (v4) at (0,1) {$(3$};
\node (v5) at (0.5,1) {$4$};
\node (v6) at (1,1) {$5)$};
\node (v7) at (1.5,1) {$(6)$};
\node (v8) at (-1,-1) {$(7)$};
\node (v9) at (0,-1) {$(8)$};
\node (v10) at (0.5,-1) {$(9)$};
\node (v11) at (1.5,-1) {$(10)$};
\draw  (v1) -- (0,0)[postaction={decorate, decoration={markings,mark=at position .5 with {\arrow[black]{stealth}}}}];
\draw  (v3) -- (0,0)[postaction={decorate, decoration={markings,mark=at position .5 with {\arrow[black]{stealth}}}}];
\draw  (v4) -- (0,0)[postaction={decorate, decoration={markings,mark=at position .5 with {\arrow[black]{stealth}}}}];
\draw  (v5) -- (0,0)[postaction={decorate, decoration={markings,mark=at position .5 with {\arrow[black]{stealth}}}}];
\draw  (v6) -- (0,0)[postaction={decorate, decoration={markings,mark=at position .5 with {\arrow[black]{stealth}}}}];
\draw  (v7) -- (0,0)[postaction={decorate, decoration={markings,mark=at position .5 with {\arrow[black]{stealth}}}}];
\draw  (v8) -- (0,0)[postaction={decorate, decoration={markings,mark=at position .5 with {\arrowreversed[black]{stealth}}}}];
\draw  (v9) -- (0,0)[postaction={decorate, decoration={markings,mark=at position .5 with {\arrowreversed[black]{stealth}}}}];
\draw  (v10)-- (0,0)[postaction={decorate, decoration={markings,mark=at position .5 with {\arrowreversed[black]{stealth}}}}];
\draw  (v11)-- (0,0)[postaction={decorate, decoration={markings,mark=at position .5 with {\arrowreversed[black]{stealth}}}}];
\end{tikzpicture}
\end{matrix}
\\
&
\begin{matrix}
\begin{tikzpicture}[rotate=-45]
\node (v1) at (-2.3,0.5) {};
\node (v2) at (1.3,0.5) {};
\draw [thick,->,>=stealth] (v1) -- (v2);
\node [rotate=-45] at (-0.5,0.8) {coarse-graining};
\end{tikzpicture}
\end{matrix}
&
\begin{matrix}
\begin{tikzpicture}[rotate=-90]
\node (v1) at (-1.7,0.5) {};
\node (v2) at (0.7,0.5) {};
\draw [thick,->,>=stealth] (v1) -- (v2);
\node [rotate=-90] at (-0.5,0.8) {fusion};

\end{tikzpicture}
\end{matrix}
\\&&
\begin{matrix}
\begin{tikzpicture}[scale=.5]
\draw[fill] (0,0.5) circle [radius=0.11];
\node (v2) at (0,0.5) {};
\node (v1) at (-1.5,2) {$1$};
\node (v3) at (0,2) {$2$};
\node (v4) at (1.5,2) {$3$};
\node (v5) at (-1.5,-1) {$4$};
\node (v6) at (-0.5,-1) {$5$};
\node (v7) at (1,-1) {$6$};
\node (v8) at (2,-1) {$7$};
\draw  (v1) --  (0,0.5)[postaction={decorate, decoration={markings,mark=at position .5 with {\arrow[black]{stealth}}}}];
\draw  (v3) --  (0,0.5)[postaction={decorate, decoration={markings,mark=at position .5 with {\arrow[black]{stealth}}}}];
\draw  (v4) --  (0,0.5)[postaction={decorate, decoration={markings,mark=at position .5 with {\arrow[black]{stealth}}}}];
\draw  (v5) --  (0,0.5)[postaction={decorate, decoration={markings,mark=at position .5 with {\arrowreversed[black]{stealth}}}}];
\draw  (v6) --  (0,0.5)[postaction={decorate, decoration={markings,mark=at position .5 with {\arrowreversed[black]{stealth}}}}];
\draw  (v7) --  (0,0.5)[postaction={decorate, decoration={markings,mark=at position .5 with {\arrowreversed[black]{stealth}}}}];
\draw  (v8) --  (0,0.5)[postaction={decorate, decoration={markings,mark=at position .5 with {\arrowreversed[black]{stealth}}}}];
\end{tikzpicture}
\end{matrix}

\end{matrix}
$
\end{center}

\end{ex}

\section{Evaluation and coarse-graining of diagrams}
In this section, we will introduce the category of tensor schemes, diagrams in a tensor scheme or a strict tensor category closely following \cite{[JS91]}. Finally, we will prove that there is a well-defined value of each diagram in a strict tensor category (Theorem 4.4.2).
\subsection{The category of tensor schemes}
In this section, we introduce the category of tensor schemes. First let us recall from \cite{[JS91]} that a \textbf{tensor scheme} $\mathcal{D}$ consists of two sets $Ob(\mathcal{D})$ and $Mor(\mathcal{D})$ together with two functions from $Mor(\mathcal{D})$ to the set of words $W(Ob(\mathcal{D}))$ in the elements of $Ob(\mathcal{D})$ $$s,t:Mor(\mathcal{D})\rightarrow W(Ob(\mathcal{D})),$$ which are called source and target maps.

A tensor scheme is a way to present the generator of a strict tensor category, and elements of $Ob(\mathcal{D})$ and $Mor(\mathcal{D})$ are called objects and morphisms, respectively.
To every two words $x_1\cdots x_m,\  y_1\cdots y_n\in W(Ob(\mathcal{D}))$, we associate a set $Mor_{\mathcal{D}}(x_1\cdots x_m,y_1\cdots y_n)$ of morphisms defined as $$\{f\in Mor(\mathcal{D})|s(f)=x_1\cdots x_m, t(f)=y_1\cdots y_n\},$$
whose elements are called \textbf{morphisms} from  $x_1\cdots x_m$ to $y_1\cdots y_n$ and usually written as $f:x_1\cdots x_m\rightarrow y_1\cdots y_n$.
We usually write a tensor scheme as $\mathcal{D}=\{Ob(\mathcal{D}),Mor(\mathcal{D}),s, t\}$ or $\xymatrix{Mor(\mathcal{D})\ar@{-->}@<1mm>[r]^{s}\ar@{-->}@<-1mm>[r]_{t}&Ob(\mathcal{D})}.$  There is no constraint in the definition, thus the construction of a tensor scheme is an easy thing.

\begin{ex}
Any quiver $(s,t:E\rightarrow V)$ is a tensor scheme.
\end{ex}

\begin{ex}
Let $\mathsf{\Gamma}$ be the set of isomorphic classes of planar graphs, $\{x\}$ be a set with one element $x$. For a $(m,n)$-planar graph $(\overrightarrow{\Gamma},\prec)$, we define $s((\overrightarrow{\Gamma},\prec))=\overset{m\ times}{\overbrace{x\cdots x}}$ and $t((\overrightarrow{\Gamma},\prec))=\overset{n\ times}{\overbrace{x\cdots x}}$. Then $\mathbf{\Gamma}=(\mathsf{\Gamma},\{x\},s,t)$ forms a tensor scheme.
\end{ex}

\begin{ex}
Replace $\mathsf{\Gamma}$ by any subset $S\subseteq \mathsf{\Gamma}$ in above example, we can get a tensor scheme. Especially, taking the set $\mathsf{Prim}$ of isomorphic classes of prime planar graphs, we get a tensor scheme $\mathbf{Prim}=(\mathsf{Prim},\{x\},s,t)$.
\end{ex}

\begin{ex}
Replace $\{x\}$ by the set $W(\{x\})$ of words of $x$, we can get two new tensor scheme $(\mathsf{\Gamma},W(\{x\}),s,t)$ and $(\mathsf{Prim},W(\{x\}),s,t)$.
\end{ex}

\begin{ex}
 For any planar graph $(\overrightarrow{\Gamma},\prec)$, we set $s'((\overrightarrow{\Gamma},\prec))=t'((\overrightarrow{\Gamma},\prec))=x$, then  $(\mathsf{\Gamma},\{x\},s',t')$, $(\mathsf{Prim},\{x\},s',t')$, $(\mathsf{\Gamma},W(\{x\}),s',t')$ and $(\mathsf{Prim},W(\{x\}),s',t')$ are four tensor schemes.
\end{ex}

\begin{ex}
Let $([\Gamma,\gamma], P_{in},P_{out})$ be a fissus planar diagram with $P_{in}=(i_1<\cdots <i_{\mu_1})<\cdots <(i_{\mu_1+\cdots+\mu_{m-1}+1}<\cdots< i_{\mu_1+\cdots+\mu_m})$ and $P_{out}=(o_1<\cdots <o_{\nu_1})<\cdots <(o_{\nu_1+\cdots+\nu_{n-1}+1}<\cdots< o_{\nu_1+\cdots+\nu_n})$, we define the domain $dom((\Gamma,\gamma), P_{in},P_{out})=(\overset{\mu_1\ times}{\overbrace{x\cdots x}})\cdots (\overset{\mu_m\ times}{\overbrace{x\cdots x}})$ and $cod((\Gamma,\gamma), P_{in},P_{out})=(\overset{\nu_1\ times}{\overbrace{x\cdots x}})\cdots (\overset{\nu_n\ times}{\overbrace{x\cdots x}})$ being words in $W(\{x\})$. Then $\mathbf{\Gamma}_F=(\mathsf{\Gamma}_F,W(\{x\}), dom, cod )$ is a tensor scheme. Take $\mathsf{Prim}_F\subset\mathsf{\Gamma}_F$ being the set of prime fissus planar graphs, then $\mathbf{Prim}_F=(\mathsf{Prim}_F,W(\{x\}), dom, cod )$ is an example of tensor scheme.
\end{ex}

A morphism $\varphi:\mathcal{D}_1\rightarrow\mathcal{D}_2$ of tensor schemes consists of two functions $\varphi_o:Ob(\mathcal{D}_1)\rightarrow Ob(\mathcal{D}_2)$ and $\varphi_m:Mor(\mathcal{D}_1)\rightarrow Mor(\mathcal{D}_2)$ such that the diagram
$$\xymatrix{W(Ob(\mathcal{D}_1))\ar[d]_{\widehat{\varphi_o}}&\ar[d]^{\varphi_m}Mor(\mathcal{D}_1)\ar[r]^{t_1}\ar[l]_{s_1}&\ar[d]^{\widehat{\varphi_o}}W(Ob(\mathcal{D}_1))\\W(Ob(\mathcal{D}_2))&\ar[l]_{s_2}\ar[r]^{t_2}Mor(\mathcal{D}_2)&W(Ob(\mathcal{D}_2))}$$
commutes, where $\widehat{\varphi_o}:W(Ob(\mathcal{D}_1))\rightarrow W(Ob(\mathcal{D}_2))$ is the natural extension of $\varphi_o$ which sends a word $x_1\cdots x_n\in W(Ob(\mathcal{D}_1))$ to $\varphi_o(x_1)\cdots \varphi_o(x_n)\in W(Ob(\mathcal{D}_2))$. Here and thereafter, the subscript $"o"$ and $"m"$ will always stand for the words $"object"$ and $"morphism"$, respectively.

\begin{ex}
Take $\varphi_m:\mathsf{Prim}\rightarrow \mathsf{\Gamma}$ being the inclusion map and $\varphi_o:\{x\}\rightarrow \{x\} $ being the identity map, then the pair $(\varphi_0,\varphi_m):\mathbf{Prim}\rightarrow \mathbf{\Gamma}$ defines a morphism  of tensor schemes.
\end{ex}

\begin{ex}
Take $\varphi_m:\mathsf{\Gamma}\rightarrow \mathsf{Prim}$ being the contraction map and $\varphi_o:\{x\}\rightarrow \{x\} $ being the identity map, then the pair $(\varphi_0,\varphi_m):\mathbf{\Gamma}\rightarrow \mathbf{Prim}$ defines a morphism  of tensor schemes.
\end{ex}

\begin{ex}
Take $\varphi_m:\mathsf{Prim}\rightarrow \mathsf{\Gamma}$ being  the inclusion map and $\varphi_o:W(\{x\})\rightarrow \{x\} $ being the constant map, then the pair $(\varphi_0,\varphi_m):(\mathsf{Prim},W\{x\},s,t)\rightarrow (\mathsf{\Gamma},\{x\},s',t')$ defines a morphism  of tensor schemes.
\end{ex}

\begin{ex}
Take $\varphi_m:\mathsf{\Gamma}_F\rightarrow \mathsf{Prim}$ being  the coarse-graining map and $\varphi_o:W(\{x\})\rightarrow \{x\} $ being the constant map, then the pair $(\varphi_0,\varphi_m):(\mathsf{\Gamma}_F,W\{x\},dom,cod)\rightarrow (\mathsf{Prim},\{x\},s,t)$ defines a morphism  of tensor schemes.
\end{ex}

It is routine to check that all tensor schemes and their morphisms form a category, we denote it as $\textbf{T.Sch}$.

\subsection{Planar diagrams in  tensor schemes }
In this section, we will introduce the notion of a planar diagram in a tensor scheme. First, let us define valuation of a  planar graph in a tensor scheme using polarized structure of the planar graph.

A \textbf{valuation} $\gamma:\Gamma\rightarrow \mathcal{D}$ of a planar graph $(\overrightarrow{\Gamma},\prec)$  in a tensor scheme $\mathcal{D}$ is a pair of functions
$$\begin{matrix}\gamma_o: H(\Gamma)\rightarrow Ob(\mathcal{D}),&&\gamma_m:V_{re}(\Gamma)\rightarrow Mor(\mathcal{D}) \end{matrix}$$
such that

$\bullet$ $\gamma_o$ is $\sigma_{\Gamma}$ invariant, that is, for every $h\in H(\Gamma)$, $$\gamma_o(h)=\gamma_o(\sigma_{\Gamma}(h));$$

$\bullet$  for every \textbf{real vertex}  $v\in V_{re}(\Gamma)$, $$\gamma_m(v)\in Mor(\gamma_o(h_1)\cdots \gamma_o(h_m),\gamma_o(h'_1)\cdots \gamma_o(h'_n))$$
where $h_1\prec_H\cdots\prec_Hh_m$ and $h'_1\prec_H\cdots\prec_Hh'_n$  are the ordered lists of elements of $In(v)$ and $Out(v)$ respectively.

\begin{defn}
A \textbf{planar diagram} in $\mathcal{D}$ is a pair $[(\overrightarrow{\Gamma},\prec),\gamma]$ where $(\overrightarrow{\Gamma},\prec)$ is a planar graph and $\gamma$ is a valuation of $(\overrightarrow{\Gamma},\prec)$ in $\mathcal{D}$.
\end{defn}
A planar diagram in $\mathcal{D}$ is also called a \textbf{tensor network state} in $\mathcal{D}$ with type $\Gamma$ and is usually denoted by $[\Gamma,\gamma]$ or even only $\gamma$ if no confusion arises.

The \textbf{domain} and \textbf{codomain} of a diagram $[\Gamma,\gamma]$  are the words in $Ob(\mathcal{D})$
$$\begin{matrix}dom([\Gamma,\gamma])=\gamma_o(i_1)\cdots \gamma_o(i_k),&&cod([\Gamma,\gamma])=\gamma_o(o_1)\cdots \gamma_o(o_l) \end{matrix}$$
where $i_1\prec_H\cdots\prec_Hi_k$ and $o_1\prec_H\cdots\prec_Ho_l$ are the ordered lists of elements of $In(\overrightarrow{\Gamma})$ and $Out(\overrightarrow{\Gamma})$.

An \textbf{isomorphism} of  plane diagrams $$\underline{\phi}:[(\overrightarrow{\Gamma}_1,\prec_1),\gamma_1]\rightarrow [(\overrightarrow{\Gamma}_2,\prec_2),\gamma_2]$$ is an isomorphism $\phi:(\overrightarrow{\Gamma}_1,\prec_1)\rightarrow (\overrightarrow{\Gamma}_2,\prec_2)$ such that $\gamma_1=\gamma_2\circ \phi$. The fact that there is an isomorphism between  $[(\overrightarrow{\Gamma}_1,\prec_1),\gamma_1]$ and $[(\overrightarrow{\Gamma}_2,\prec_2),\gamma_2]$ is written as $[(\overrightarrow{\Gamma}_1,\prec_1),\gamma_1]\cong[(\overrightarrow{\Gamma}_2,\prec_2),\gamma_2].$

All the  diagrams in $\mathcal{D}$ and their isomorphisms form a groupoid, and the set of isomorphic classes is denoted  by  $\mathsf{Diag}(\mathcal{D})$ or $\mathsf{\Gamma}(\mathcal{D})$. The domain and codomain of diagrams define two functions from $\mathsf{\Gamma}(\mathcal{D})$ to $W(Ob(\mathcal{D}))$ and the quadruple $(\mathsf{\Gamma}(\mathcal{D}), Ob(\mathcal{D}), dom, cod)$ is an example of tensor scheme, and we denote it by $\mathbf{\Gamma}(\mathcal{D})$.

Let $[(\overrightarrow{\Gamma}_1,\prec_1),\gamma_1]$ and $[(\overrightarrow{\Gamma}_2,\prec_2),\gamma_2]$ be two planar diagrams, we define their \textbf{tensor product} $[(\overrightarrow{\Gamma},\prec),\gamma]$ to be the  planar diagram with $(\overrightarrow{\Gamma},\prec)=(\overrightarrow{\Gamma}_1,\prec_1)\otimes(\overrightarrow{\Gamma}_2,\prec_2)$ and $\gamma=\gamma_1\sqcup\gamma_2$ which is defined as

$\bullet$
\begin{equation*}\gamma_o(h)=
\begin{cases}
(\gamma_1)_o(h),&  \text{if}\ h\in H(\Gamma_1),\\
(\gamma_2)_o(h),&  \text{if}\ h\in H(\Gamma_2);
\end{cases}
\end{equation*}

$\bullet$
\begin{equation*}\gamma_m(v)=
\begin{cases}
(\gamma_1)_m(v),&  \text{if}\ v\in V_{re}(\Gamma_1),\\
(\gamma_2)_m(v),&  \text{if}\ v\in V_{re}(\Gamma_2).
\end{cases}
\end{equation*}

Assume $Out(\overrightarrow{\Gamma}_1)=\{o_1\prec_{H1}\cdots\prec_{H1}o_n\}$  and $In(\overrightarrow{\Gamma}_2)=\{i_1\prec_{H2}\cdots\prec_{H2}i_n\}$,  we say they are composable if $(\gamma_1)_o(o_k)=(\gamma_2)_o(i_k)$ for every $1\leq k\leq n$. Their \textbf{composition} $[(\overrightarrow{\Gamma},\prec),\gamma]$ is defined to be the  planar diagram with $(\overrightarrow{\Gamma},\prec)=(\overrightarrow{\Gamma}_2,\prec_2)\circ(\overrightarrow{\Gamma}_1,\prec_1)$ and $\gamma=\gamma_2\vee \gamma_1$ which is defined as

$\bullet$
\begin{equation*}\gamma_o(h)=
\begin{cases}
(\gamma_1)_o(h),&  \text{if}\ h\in H(\Gamma_1)\cap H(\Gamma),\\
(\gamma_2)_o(h),&  \text{if}\ h\in H(\Gamma_2)\cap H(\Gamma);
\end{cases}
\end{equation*}

$\bullet$
\begin{equation*}\gamma_m(v)=
\begin{cases}
(\gamma_1)_m(v),&  \text{if}\ v\in V_{re}(\Gamma_1),\\
(\gamma_2)_m(v),&  \text{if}\ v\in V_{re}(\Gamma_2).
\end{cases}
\end{equation*}

If $\gamma:\Gamma\rightarrow \mathcal{D}_1$ is a valuation of $\Gamma$ in $\mathcal{D}_1$ and $\varphi:\mathcal{D}_1\rightarrow\mathcal{D}_2$ is a morphism of tensor schemes, we can naturally get a valuation $\tilde{\gamma}=\varphi_{\ast}(\gamma):\Gamma\rightarrow \mathcal{D}_2$ of $\Gamma$ in $\mathcal{D}_2$ defined by
$$\begin{matrix}\tilde{\gamma}_o=\varphi_o\circ \gamma_o,&&\tilde{\gamma}_m=\varphi_m\circ \gamma_m.\end{matrix}$$

\begin{prop}
Every morphism $\varphi:\mathcal{D}_1\rightarrow\mathcal{D}_2$ of tensor schemes induces a function (push-forward) $$\varphi_{\ast}:\mathsf{Diag}(\mathcal{D}_1)\longrightarrow\mathsf{Diag}(\mathcal{D}_2),$$ which sends $[\Gamma,\gamma]$ to $[\Gamma,\varphi_{\ast}(\gamma)].$ Moreover, $\varphi_{\ast}$ is compatible with the tensor product and composition of diagrams.
\end{prop}
As a corollary,  we see that
\begin{cor}
The construction of $\mathbf{\Gamma}(\mathcal{D})$ defines a functor $\mathbf{\Gamma}:\mathbf{T.Sch}\rightarrow \mathbf{T.Sch}$.
\end{cor}
\begin{proof}
Only to do is to check the fact that for any morphism $\varphi:\mathcal{D}_1\rightarrow\mathcal{D}_2$ of tensor schemes, $\varphi_{\ast}$ is compatible with the source and target maps of tensor schemes and this can be directly checked by definitions.
\end{proof}
A planar diagram in  a tensor scheme $\mathcal{D}$ is called a \textbf{fissus planar diagram} if the underlying  planar  graph is equipped with a fission structure.  The notion of an isomorphisms of two fissus planar diagrams is obvious and we denote the set of isomorphic classes  by  $\mathsf{Diag}_F(\mathcal{D})$ or $\mathsf{\Gamma}_F(\mathcal{D})$.  As decorated version of a planar graph, a fissus planar diagram  can be viewed as a planar diagram with its domain and codomain bracketed. More precisely, let $([\Gamma,\gamma], P_{in},P_{out})$ be a fissus planar diagram with $P_{in}=(i_1<\cdots <i_{\mu_1})<\cdots <(i_{\mu_1+\cdots+\mu_{m-1}+1}<\cdots< i_{\mu_1+\cdots+\mu_m})$, $P_{out}=(o_1<\cdots <o_{\nu_1})<\cdots <(o_{\nu_1+\cdots+\nu_{n-1}+1}<\cdots< o_{\nu_1+\cdots+\nu_n})$ and assume $\gamma_o(i_k)=x_k$ for $1\leq k\leq \mu_1+\cdots+\mu_m$, $\gamma_o(o_l)=y_l$ for $1\leq l\leq \nu_1+\cdots+\nu_n$, then we define the domain $dom((\Gamma,\gamma), P_{in},P_{out})=(x_1\cdots x_{\mu_1})\cdots (x_{\mu_1+\cdots+\mu_{m-1}+1}\cdots x_{\mu_1+\cdots+\mu_m})$ and $cod((\Gamma,\gamma), P_{in},P_{out})=(y_1\cdots y_{\nu_1})\cdots (y_{\nu_1+\cdots+\nu_{n-1}+1}\cdots y_{\nu_1+\cdots+\nu_n})$ to be words in $W(Ob(\mathcal{D}))$. Thus $\mathbf{\Gamma}_F(\mathcal{D})=(\mathsf{\Gamma}_F(\mathcal{D}),W(Ob(\mathcal{D})), dom, cod)$ defines a tensor scheme.

The following proposition is evident.
\begin{prop}
Every morphism $\varphi:\mathcal{D}_1\rightarrow\mathcal{D}_2$ of tensor schemes induces a morphism of tensor schemes
$$(\widehat{\varphi_o},\varphi_{\ast F}):(\mathsf{\Gamma}_F(\mathcal{D}_1),W(Ob(\mathcal{D}_1)), dom, cod)\longrightarrow(\mathsf{\Gamma}_F(\mathcal{D}_2),W(Ob(\mathcal{D}_2)), dom, cod),$$
where $\widehat{\varphi_o}:W(Ob(\mathcal{D}_1))\rightarrow W(Ob(\mathcal{D}_2))$ is the natural extension of $\varphi_o$ which sends a word $x_1\cdots x_n\in W(Ob(\mathcal{D}_1))$ to $\varphi_o(x_1)\cdots \varphi_o(x_n)\in W(Ob(\mathcal{D}_2))$ and $\varphi_{\ast F}$ send a fissus diagram $([\Gamma,\gamma], P_{in},P_{out})$ in $\mathcal{D}_1$ to a fissus $([\Gamma,\varphi_{\ast}(\gamma)], P_{in},P_{out})$ in  $\mathcal{D}_2$ .
\end{prop}

\begin{cor}
The constructions of $\mathbf{\Gamma}_F(\mathcal{D})$ define a functor  $\mathbf{\Gamma}_F:\mathbf{T.Sch}\rightarrow \mathbf{T.Sch}$.
\end{cor}

\subsection{Planar diagrams in tensor categories }
Recall that a \textbf{valuation} $\gamma:\Gamma\rightarrow \mathcal{D}$ of a planar graph $(\overrightarrow{\Gamma},\prec)$  in a tensor category $\mathcal{V}$ is a pair of functions
$$\begin{matrix}\gamma_o: H(\Gamma)\rightarrow Ob(\mathcal{V}),&&\gamma_m:V_{re}(\Gamma)\rightarrow Mor(\mathcal{V}) \end{matrix}$$
such that

$\bullet$ $\gamma_o$ is $\sigma_{\Gamma}$ invariant, that is, for every $h\in H(\Gamma)$, $$\gamma_o(h)=\gamma_o(\sigma_{\Gamma}(h));$$

$\bullet$  for every \textbf{real vertex}  $v\in V_{re}(\Gamma)$, $$\gamma_m(v)\in Mor(\gamma_o(h_1)\otimes\cdots \otimes\gamma_o(h_m),\gamma_o(h'_1)\otimes\cdots \otimes \gamma_o(h'_n))$$
where $h_1\prec_H\cdots\prec_Hh_m$ and $h'_1\prec_H\cdots\prec_Hh'_n$  are the ordered lists of elements of $In(v)$ and $Out(v)$ respectively.

\begin{defn}
A \textbf{planar diagram} in $\mathcal{V}$ is a pair $[(\overrightarrow{\Gamma},\prec),\gamma]$ where $(\overrightarrow{\Gamma},\prec)$ is a planar graph and $\gamma$ is a valuation of $(\overrightarrow{\Gamma},\prec)$ in $\mathcal{V}$.
\end{defn}

A planar diagram in $\mathcal{V}$ is also called a \textbf{tensor network state} in $\mathcal{V}$ with type $\Gamma$ and is usually denoted by $[\Gamma,\gamma]$ or even only $\gamma$ if no confusion arises.

The \emph{domain} and \textbf{codomain} of a diagram $[\Gamma,\gamma]$  are the words in $Ob(\mathcal{V})$
$$\begin{matrix}dom([\Gamma,\gamma])=\gamma_o(i_1)\cdots \gamma_o(i_k),&&cod([\Gamma,\gamma])=\gamma_o(o_1)\cdots \gamma_o(o_l) \end{matrix}$$
where $i_1\prec_H\cdots\prec_Hi_k$ and $o_1\prec_H\cdots\prec_Ho_l$ are the ordered lists of elements of $In(\overrightarrow{\Gamma})$ and $Out(\overrightarrow{\Gamma})$.\\

The notions of an isomorphism, tensor product and composition of diagrams in $\mathcal{V}$ are same as those of diagrams in a tensor scheme. The set of isomorphic classes of diagrams in $\mathcal{V}$ is denoted by  $\mathsf{Diag}(\mathcal{V})$ or $\mathsf{\Gamma}(\mathcal{V})$.

If $\gamma:\Gamma\rightarrow \mathcal{V}_1$ is a valuation of $\Gamma$ in $\mathcal{V}_1$ and $K:\mathcal{V}_1\rightarrow\mathcal{V}_2$ is a strict tensor functor, we can naturally get a valuation $\tilde{\gamma}=K_{\ast}(\gamma):\Gamma\rightarrow \mathcal{V}_2$ of $\Gamma$ in $\mathcal{V}_2$ defined by
$$\begin{matrix}\tilde{\gamma}_o=K(\gamma_o),&&\tilde{\gamma}_m=K(\gamma_m).\end{matrix}$$

Similar to proposition $4.2.2$, we have
\begin{prop}
Every strict tensor functor $K:\mathcal{V}_1\rightarrow\mathcal{V}_2$ induces a function (push-forward) $$K_{\ast}:\mathsf{Diag}(\mathcal{V}_1)\longrightarrow\mathsf{Diag}(\mathcal{V}_2),$$ which sends $[\Gamma,\gamma]$ to $[\Gamma,K_{\ast}(\gamma)].$ Moreover, $K_{\ast}$ is compatible with the tensor product and composition of diagrams.
\end{prop}

As in previous section, $\mathbf{\Gamma}(\mathcal{V})=(\mathsf{\Gamma}(\mathcal{V}), Ob(\mathcal{V}), dom, cod)$ forms a tensor scheme. Similarly, a fissus planar diagram in $\mathcal{V}$ is a planar diagram in $\mathcal{V}$  equipped with a fission structure which makes the domain and codomain to be words in $W(Ob(\mathcal{V}))$.  The notion of an isomorphisms of two fissus planar diagrams are obvious and we denote the set of isomorphic classes  by $\mathsf{Diag}_F(\mathcal{V})$  or $\mathsf{\Gamma}_F(\mathcal{V})$.   Evidently $\mathbf{\Gamma}_F(\mathcal{V})=(\mathsf{\Gamma}_F(\mathcal{V}),W(Ob(\mathcal{V})), dom, cod)$ also defines a tensor scheme.

An analogue of proposition $4.2.3$ is as following:
\begin{prop}
Every strict tensor functor $K:\mathcal{V}_1\rightarrow\mathcal{V}_2$ induces a morphism of tensor schemes
$$(\widehat{K_o},K_{\ast F}):(\mathsf{\Gamma}_F(\mathcal{V}_1),W(Ob(\mathcal{V}_1)), dom, cod)\longrightarrow(\mathsf{\Gamma}_F(\mathcal{V}_2),W(Ob(\mathcal{V}_2)), dom, cod),$$
where $\widehat{K_o}:W(Ob(\mathcal{V}_1))\rightarrow W(Ob(\mathcal{V}_2))$ is the natural extension of $K_o$ which sends a word $x_1\cdots x_n\in W(Ob(\mathcal{V}_1))$ to $Kx_1\cdots Kx_n\in W(Ob(\mathcal{V}_2))$ and $K_{\ast F}$ send a fissus diagram $([\Gamma,\gamma], P_{in},P_{out})$ in $\mathcal{V}_1$ to a fissus $([\Gamma,K_{\ast}(\gamma)], P_{in},P_{out})$ in  $\mathcal{V}_2$ .
\end{prop}

\begin{rem}
All strict tensor categories and their strict tensor functors form a category, denoted by $\mathbf{Str.T}$.  The construction of $\mathbf{\Gamma}(\mathcal{V})$ and $\mathbf{\Gamma}_F(\mathcal{V})$ defines two functors from $\mathbf{Str.T}$ to $\mathbf{T.Sch}$ which are also denoted by $\mathbf{\Gamma}$ and $\mathbf{\Gamma}_F$, respectively.
\end{rem}
\subsection{Evaluation of  diagrams in a tensor category }
In this section we introduce an evaluation map $\varepsilon: \mathsf{Diag}(\mathcal{V})\rightarrow Mor(\mathcal{V})$ from the set of isomorphic classes of diagrams in $\mathcal{V}$ to $Mor(\mathcal{V})$. As in \cite{[JS91]}, we subdivide a diagram into simpler parts and "integrate" the result by composing and tensoring and prove that the result is independent of any special choice of subdivision and the ways or orders to integration. In our combinatorial theory, the simplest parts of a diagram are prime diagrams defined on corollas  and unitary diagrams defined on unitary graphs. We give their definitions:

\begin{defn}
A diagram $[\Gamma,\gamma]$ is called \textbf{prime} when $\Gamma$ is prime and is called \textbf{unitary} when $\Gamma$ is unitary.
A diagram $[\Gamma,\gamma]$ is called \textbf{essential prime}, \textbf{invertible}, \textbf{elementary} when $\Gamma$ is essential prime, invertible, elementary, respectively.
\end{defn}

The following theorem is equivalent to proposition $1.1$ and theorem $1.2$ in \cite{[JS91]}.
\begin{thm}
For any strict tensor category $\mathcal{V}$, there is an unique map $$\varepsilon: \mathsf{Diag}(\mathcal{V})\rightarrow Mor(\mathcal{V})$$ such that

$\bullet$ $\varepsilon(\varnothing)=Id_{1_{\mathcal{V}}}$, where $\varnothing$ is the empty diagram and $Id_{1_{\mathcal{V}}}$ is the identity morphism of the unit object $1_{\mathcal{V}}$;

$\bullet$ $\varepsilon([\Gamma,\gamma])=\gamma_m(v)$, if $\Gamma$ is a prime graph with real vertex $v$;

$\bullet$ $\varepsilon([\Gamma,\gamma])=Id_{\gamma_o(e)}$, if  $\Gamma$ is an unitary graph  with virtual edge $e$;

$\bullet$ $\varepsilon([\Gamma_2,\gamma_2]\circ[\Gamma_1,\gamma_1])=\varepsilon([\Gamma_2,\gamma_2])\circ\varepsilon([\Gamma_1,\gamma_1])$
for any two composable diagrams $[\Gamma_1,\gamma_1],\ [\Gamma_2,\gamma_2]$.

The morphism $\varepsilon([\Gamma,\gamma])$ is called the \textbf{value} of $[\Gamma,\gamma]$.
\end{thm}
\begin{proof}
$(1)$  We first construct a map $\varepsilon: \mathbf{Diag}(\mathcal{V})\rightarrow Mor(\mathcal{V})$ as follows:

$\bullet$ $\varepsilon(\varnothing)=Id_{1_{\mathcal{V}}}$, where $\varnothing$ is the empty diagram and $Id_{1_{\mathcal{V}}}$ is the identity morphism of the unit object $1_{\mathcal{V}}$;

$\bullet$  $\varepsilon([\Gamma,\gamma])=\gamma_m(v)$, if $\Gamma$ is a prime graph with real vertex $v$.

$\bullet$ $\varepsilon([\Gamma,\gamma])=Id_{\gamma_o(e)}$, if  $\Gamma$ is a unitary graph  with virtual edge $e$.

$\bullet$ if $[\Gamma,\gamma]$ is an essential elementary diagram, that is, $$[\Gamma,\gamma]=[\Gamma_1,\gamma_1]\otimes\cdots\otimes [\Gamma_n,\gamma_n]$$ with each $\Gamma_i$ $(1\leq i\leq n)$ either prime or unitary, we define $$\varepsilon([\Gamma,\gamma])=\varepsilon([\Gamma_1,\gamma_1])\otimes\cdots\otimes\varepsilon([\Gamma_n,\gamma_n]).$$

$\bullet$ for a general diagram $[\Gamma,\gamma]$,  we have proved in Theorem \ref{decomposition} that any planar graph has a decomposition as composition of several essential prime planar graphs, thus $[\Gamma,\gamma]$ has a decomposition as composition of several essential prime planar diagrams, that is, $$[\Gamma,\gamma]=[\Gamma_n,\gamma_n]\circ\cdots\circ[\Gamma_1,\gamma_1]$$
with $[\Gamma_k,\gamma_k]$ $(1\leq k\leq n)$ be essential prime planar diagram.  We define the value $\varepsilon([\Gamma,\gamma])$ as $$\varepsilon([\Gamma,\gamma])=\varepsilon([\Gamma_n,\gamma_n])\circ\cdots\circ\varepsilon([\Gamma_1,\gamma_1]).$$

$(2)$  Now we want to show that $\varepsilon$ is well-defined. In fact, we have:

$\bullet$ if $[\Gamma,\gamma]$ is a prime or unitary or essential prime, it has no nontrivial decomposition under tensor and composition and it is easy to see that $\varepsilon([\Gamma,\gamma])$ is unique, hence well-defined.\\

$\bullet$ if $[\Gamma,\gamma]$ is a general  diagram, we use induction on the number $n$ of its real vertices.

When $n=1$. In this case, $[\Gamma,\gamma]$ is an essential prime diagram, we have prove that $\varepsilon([\Gamma,\gamma])$ is well defined.

When $n=2$. In this case, let $v_1, v_2$ be two real vertices of $\Gamma$. We will prove in three cases:

Case $1$: $v_1\rightarrow v_2$. In this case, we have no nontrivial decomposition under tensor and only one decomposition $[\Gamma,\gamma]=[\Gamma_2,\gamma_2]\circ[\Gamma_1,\gamma_1]$ with $v_1\in V_{re}(\Gamma_1)$ and $v_2\in V_{re}(\Gamma_2)$. We have an unique value $\varepsilon([\Gamma,\gamma])=\varepsilon([\Gamma_2,\gamma_2])\circ\varepsilon([\Gamma_1,\gamma_1]).$

Case $2$: $v_2\rightarrow v_1$. In this case, we have  no nontrivial decomposition under tensor and only  one decomposition $[\Gamma,\gamma]=[\Gamma_2,\gamma_2]\circ[\Gamma_1,\gamma_1]$ with $v_1\in V_{re}(\Gamma_2)$ and $v_2\in V_{re}(\Gamma_1)$. We have an unique value $\varepsilon([\Gamma,\gamma])=\varepsilon([\Gamma_2,\gamma_2])\circ\varepsilon([\Gamma_1,\gamma_1]).$

Case $3$: $v_1\nrightarrow v_2$ and $v_2\nrightarrow v_1$. In this case, $[\Gamma,\gamma]$ is an elementary diagram, that is, $[\Gamma,\gamma]=[\Gamma_1,\gamma_1]\otimes[\Gamma_2,\gamma_2]$ with each $[\Gamma_i,\gamma_i]$ $(i=1,2)$ prime, we also have other two nontrivial decompositions $[\Gamma,\gamma]=[\widetilde{\Gamma_2},\widetilde{\gamma_2}]\circ[\widetilde{\Gamma_1},\widetilde{\gamma_1}]$ with $$[\widetilde{\Gamma_2},\widetilde{\gamma_2}]=[\Gamma_1,\gamma_1]\otimes [\Gamma_I,\gamma_{cod([\Gamma_2,\gamma_2])}],\  [\widetilde{\Gamma_1},\widetilde{\gamma_1}]=[\Gamma_I,\gamma_{dom([\Gamma_1,\gamma_1])}]\otimes[\Gamma_2,\gamma_2] $$  or $$[\widetilde{\Gamma_2},\widetilde{\gamma_2}]=[\Gamma_2,\gamma_2]\otimes [\Gamma_I,\gamma_{cod([\Gamma_1,\gamma_1])}],\  [\widetilde{\Gamma_1},\widetilde{\gamma_1}]=[\Gamma_I,\gamma_{dom([\Gamma_2,\gamma_2])}]\otimes[\Gamma_1,\gamma_1], $$ where those $\Gamma_I$s are invertible graphs such that those composition can be defined. In the first case we have
\begin{align*}
&\varepsilon([\widetilde{\Gamma_2},\widetilde{\gamma_2}])\circ\varepsilon([\widetilde{\Gamma_1},\widetilde{\gamma_1}])&\\
=&\varepsilon([\Gamma_1,\gamma_1]\otimes [\Gamma_I,\gamma_{cod([\Gamma_2,\gamma_2])}])\circ\varepsilon([\Gamma_I,\gamma_{dom([\Gamma_1,\gamma_1])}]\otimes[\Gamma_2,\gamma_2])&(1)\\
=&(\varepsilon([\Gamma_1,\gamma_1])\otimes I_{cod([\Gamma_2,\gamma_2])})\circ(I_{dom([\Gamma_1,\gamma_1])}\otimes\varepsilon([\Gamma_2,\gamma_2]))&(2)\\
=&(\varepsilon([\Gamma_1,\gamma_1])\circ I_{dom([\Gamma_1,\gamma_1])})\otimes(I_{cod([\Gamma_2,\gamma_2])}\circ \varepsilon([\Gamma_2,\gamma_2]) )&(3)\\
=&\varepsilon([\Gamma_1,\gamma_1])\otimes\varepsilon([\Gamma_2,\gamma_2]),&(4)
\end{align*}
where $(1)$= definition of $[\widetilde{\Gamma_2},\widetilde{\gamma_2}]$ and $[\widetilde{\Gamma_1},\widetilde{\gamma_1}]$, $(2)$= definition of $\varepsilon$ for essential prime diagrams, $(3)$=functorial property of $\otimes$, $(4)$=property of identity morphisms.

In a similar way we can prove in the second case that $$\varepsilon([\widetilde{\Gamma_2},\widetilde{\gamma_2}])\circ\varepsilon([\widetilde{\Gamma_1},\widetilde{\gamma_1}])=\varepsilon([\Gamma_1,\gamma_1])\otimes\varepsilon([\Gamma_2,\gamma_2]).$$ Thus we prove that $\varepsilon([\Gamma,\gamma])$ is well-defined.\\

When $[\Gamma,\gamma]$ has $k+1$ real vertices for $k\geq2$.

Case $1$:  If $[\Gamma,\gamma]$ has an unique decomposition  as composition of essential prime diagrams, we have an unique $\varepsilon([\Gamma,\gamma])$ by definition.

Case $2$:  If $[\Gamma,\gamma]$ has two decompositions: $$[\Gamma,\gamma]=[\Gamma_{k+1},\gamma_{k+1}]\circ\cdots\circ[\Gamma_1,\gamma_1]$$ and $$[\Gamma,\gamma]=[\Gamma'_{k+1},\gamma'_{k+1}]\circ\cdots\circ[\Gamma'_1,\gamma'_1],$$ we want to prove that  $$\varepsilon([\Gamma_{k+1},\gamma_{k+1}])\circ\cdots\circ\varepsilon([\Gamma_1,\gamma_1])=\varepsilon([\Gamma'_{k+1},\gamma'_{k+1}])\circ\cdots\circ \varepsilon([\Gamma'_1,\gamma'_1]).$$

Let $v_i\in V_{re}(\Gamma_{i})\subseteq V_{re}(\Gamma)$ and $v'_i\in V_{re}(\Gamma'_{i})\subseteq V_{re}(\Gamma)$ be only real vertices of $\Gamma_{i}$ and $\Gamma'_{i}$ for $1\leq i\leq k+1$.

case $2.1$:
If $v_{k+1}=v'_{k+1}$, then we must have $$[\Gamma_{k+1},\gamma_{k+1}]\cong[\Gamma'_{k+1},\gamma'_{k+1}]$$ and $$[\Gamma_{k},\gamma_{k}]\circ\cdots\circ[\Gamma_1,\gamma_1]\cong[\Gamma'_{k},\gamma'_{k}]\circ\cdots\circ[\Gamma'_1,\gamma'_1],$$ thus $$\varepsilon([\Gamma_{k+1},\gamma_{k+1}])=\varepsilon([\Gamma'_{k+1},\gamma'_{k+1}])$$ and by induction  hypothesis
\begin{align*}
&\varepsilon([\Gamma_{k},\gamma_{k}])\circ\cdots\circ\varepsilon([\Gamma_1,\gamma_1])\\
=&\varepsilon([\Gamma_{k},\gamma_{k}]\circ\cdots\circ[\Gamma_1,\gamma_1])\\
=&\varepsilon([\Gamma'_{k},\gamma'_{k}]\circ\cdots\circ[\Gamma'_1,\gamma'_1])\\
=&\varepsilon([\Gamma'_{k},\gamma'_{k}])\circ\cdots\circ \varepsilon([\Gamma'_1,\gamma'_1]).
\end{align*}
Hence $\varepsilon([\Gamma_{k+1},\gamma_{k+1}])\circ\cdots\circ\varepsilon([\Gamma_1,\gamma_1])=\varepsilon([\Gamma'_{k+1},\gamma'_{k+1}])\circ\cdots\circ \varepsilon([\Gamma'_1,\gamma'_1]).$

case $2.2$:
If $v_{k+1}\neq v'_{k+1}$, then there must exist $l\in\{1,\cdots, k \}$ such that $v'_l=v_{k+1}.$ It is easy to see that for every $l<j\leq k+1$, we have $$v'_l\nrightarrow v'_j,\  v'_j\nrightarrow v'_l.$$

Now our strategy of proof is to construct a series of decomposition of $[\Gamma,\gamma]$ by interchanging  $v'_l$ and $v'_{l+1},...,v'_{k+1}$ in the composition $[\Gamma,\gamma]=[\Gamma'_{k+1},\gamma'_{k+1}]\circ\cdots\circ[\Gamma'_1,\gamma'_1]$ step by step and to show that these values defined using these compositions are equal. The construct is similar as  the case of elementary diagrams with two real vertices. Let  $f=\varepsilon([\Gamma_{v'_{l}},\gamma_{v'_{l}}])$, $g=\varepsilon([\Gamma_{v'_{l+1}},\gamma_{v'_{l+1}}])$ and  $p,q, r,s $ be cardinalities  of $In(v'_l)$, $Out(v'_l)$, $In(v'_{l+1})$, $Out(v'_{l+1})$, respectively. The fact that $v'_l\nrightarrow v'_{l+1}$ implies that $[\Gamma'_{l+1},\gamma'_{l+1}]\circ[\Gamma'_{l},\gamma'_{l}]$ is an elementary diagram, thus we can assume that it is of the forms

$$\bigotimes_{w=1}^a[\Gamma_{Uw},\gamma^U_w]\otimes[\Gamma_{v'_{l}},\gamma_{v'_{l}}]\otimes \bigotimes_{w=a+1}^b[\Gamma_{Uw},\gamma^U_w]\otimes[\Gamma_{v'_{l+1}},\gamma_{v'_{l+1}}]\otimes\bigotimes_{w=b+1}^c[\Gamma_{Uw},\gamma^U_w]$$ or $$\bigotimes_{w=1}^a[\Gamma_{Uw},\gamma^U_w]\otimes[\Gamma_{v'_{l+1}},\gamma_{v'_{l+1}}]\otimes \bigotimes_{w=a+1}^b[\Gamma_{Uw},\gamma^U_w]\otimes[\Gamma_{v'_{l}},\gamma_{v'_{l}}]\otimes\bigotimes_{w=b+1}^c[\Gamma_{Uw},\gamma^U_w].$$
where $a\leq b \leq c$ are non-negative integers, $[\Gamma_{Uw},\gamma^U_w]$ $(1\leq w\leq c)$ are unitary subdiagrams if $c\geq 1$, and  $\bigotimes_{w=\mu+1}^\nu[\Gamma_{Uw},\gamma^U_w]$ is defined to be an empty diagram in case of $\mu=\nu\geq 0$.

In the first case, we can assume $[\Gamma'_{l},\gamma'_{l}]$ and $[\Gamma_{v'_{l+1}},\gamma_{v'_{l+1}}]$ as

$$\bigotimes_{w=1}^a[\Gamma_{Uw},\gamma^U_w]\otimes [\Gamma_{v'_{l}},\gamma_{v'_{l}}]\otimes\bigotimes_{w=a+1}^b[\Gamma_{Uw},\gamma^U_w]\otimes [\Gamma_I,\gamma_{dom (g)}]\otimes \bigotimes_{w=b+1}^c[\Gamma_{Uw},\gamma^U_w]$$
and
$$\bigotimes_{w=1}^a[\Gamma_{Uw},\gamma^U_w]\otimes[\Gamma_I,\gamma_{cod (f)}]\otimes\bigotimes_{w=a+1}^b[\Gamma_{Uw},\gamma^U_w] \otimes[\Gamma_{v'_{l+1}},\gamma_{v'_{l+1}}]\otimes  \bigotimes_{w=b+1}^c[\Gamma_{Uw},\gamma^U_w],$$ respectively.

Now we define two new essential prime diagrams $[\widehat{\Gamma},\widehat{\gamma}]$ and $[\widehat{\widehat{\Gamma}},\widehat{\widehat{\gamma}}]$ as

$$\bigotimes_{w=1}^a[\Gamma_{Uw},\gamma^U_w]\otimes[\Gamma_I,\gamma_{dom (f)}]\otimes\bigotimes_{w=a+1}^b[\Gamma_{Uw},\gamma^U_w] \otimes[\Gamma_{v'_{l+1}},\gamma_{v'_{l+1}}]\otimes  \bigotimes_{w=b+1}^c[\Gamma_{Uw},\gamma^U_w]$$
and
$$\bigotimes_{w=1}^a[\Gamma_{Uw},\gamma^U_w]\otimes [\Gamma_{v'_{l}},\gamma_{v'_{l}}]\otimes\bigotimes_{w=a+1}^b[\Gamma_{Uw},\gamma^U_w]\otimes [\Gamma_I,\gamma_{cod (g)}]\otimes \bigotimes_{w=b+1}^c[\Gamma_{Uw},\gamma^U_w],$$ respectively.

It can be directly checked  that \begin{align*}&[\widehat{\widehat{\Gamma}},\widehat{\widehat{\gamma}}]\circ[\widehat{\Gamma},\widehat{\gamma}]\\
=&[\Gamma'_{l+1},\gamma'_{l+1}]\circ[\Gamma'_{l},\gamma'_{l}]\\
=& \bigotimes_{w=1}^a[\Gamma_{Uw},\gamma^U_w]\otimes[\Gamma_{v'_{l}},\gamma_{v'_{l}}]\otimes \bigotimes_{w=a+1}^b[\Gamma_{Uw},\gamma^U_w]\otimes\\&[\Gamma_{v'_{l+1}},\gamma_{v'_{l+1}}]\otimes\bigotimes_{w=b+1}^c[\Gamma_{Uw},\gamma^U_w]
\end{align*}
and
\begin{align*}&\varepsilon([\widehat{\widehat{\Gamma}},\widehat{\widehat{\gamma}}])\circ\varepsilon([\widehat{\Gamma},\widehat{\gamma}])\\
=&\varepsilon([\Gamma'_{l+1},\gamma'_{l+1}])\circ\varepsilon([\Gamma'_{l},\gamma'_{l}])\\
=& \bigotimes_{w=1}^a\varepsilon([\Gamma_{Uw},\gamma^U_w])\otimes\varepsilon([\Gamma_{v'_{l}},\gamma_{v'_{l}}])\otimes \bigotimes_{w=a+1}^b\varepsilon([\Gamma_{Uw},\gamma^U_w])\otimes\\&\varepsilon([\Gamma_{v'_{l+1}},\gamma_{v'_{l+1}}])\otimes\bigotimes_{w=b+1}^c\varepsilon([\Gamma_{Uw},\gamma^U_w]).
\end{align*}

Thus we get a new decomposition of $[\Gamma,\gamma]$ as $$[\Gamma'_{k+1},\gamma'_{k+1}]\circ\cdots\circ[\Gamma'_{l+2},\gamma'_{l+2}]\circ[\widehat{\widehat{\Gamma}},\widehat{\widehat{\gamma}}]\circ[\widehat{\Gamma},\widehat{\gamma}]\circ[\Gamma'_{l-1},\gamma'_{l-1}]\circ\cdots\circ[\Gamma'_1,\gamma'_1],$$ and the value of this decomposition is same as that of
$$[\Gamma'_{k+1},\gamma'_{k+1}]\circ\cdots\circ[\Gamma'_{l+2},\gamma'_{l+2}]\circ[\Gamma'_{l+1},\gamma'_{l+1}]\circ[\Gamma'_{l},\gamma'_{l}]\circ[\Gamma'_{l-1},\gamma'_{l-1}]\circ\cdots\circ[\Gamma'_1,\gamma'_1].$$

In the second case, we can give  similar construction and prove similar results. Repeating the construction  $k+1-l$ times, we will get a new decomposition with $v_l$ appearing in the last level and the value equal to that of $[\Gamma'_{k+1},\gamma'_{k+1}]\circ\cdots\circ[\Gamma'_1,\gamma'_1].$ Using the result of case $2.1$ proved above, we prove that the value of $[\Gamma'_{k+1},\gamma'_{k+1}]\circ\cdots\circ[\Gamma'_1,\gamma'_1]$ and $[\Gamma_{k+1},\gamma_{k+1}]\circ\cdots\circ[\Gamma_1,\gamma_1].$ Hence we complete the proof of the fact that $\varepsilon$ is well-defined.\\

For any two composable diagrams $[\Gamma_1,\gamma_1],\ [\Gamma_2,\gamma_2]$,  the fact $\varepsilon([\Gamma_2,\gamma_2]\circ[\Gamma_1,\gamma_1])=\varepsilon([\Gamma_2,\gamma_2])\circ\varepsilon([\Gamma_1,\gamma_1])$ is a direct consequence of the definition of $\varepsilon$ and associative law of composition of morphisms in $\mathcal{V}$.\\

The uniqueness of $\varepsilon$ is obvious.
\end{proof}

The proof of  following corollaries are easy, we leave them to readers.
\begin{cor}
$\varepsilon([\Gamma_1,\gamma_1]\otimes[\Gamma_2,\gamma_2])=\varepsilon([\Gamma_1,\gamma_1])\otimes\varepsilon([\Gamma_2,\gamma_2])$ for any two diagrams $[\Gamma_1,\gamma_1],\ [\Gamma_2,\gamma_2]$.
\end{cor}

\begin{cor}
If $[\Gamma_1,\gamma_1],\ [\Gamma'_1,\gamma'_1]$ and $[\Gamma_2,\gamma_2],\  [\Gamma'_2,\gamma'_2]$ are composable, then
\begin{align*}
&\varepsilon([\Gamma'_1,\gamma'_1]\otimes[\Gamma'_2,\gamma'_2])\circ \varepsilon([\Gamma_1,\gamma_1]\otimes[\Gamma_2,\gamma_2])\\=&\varepsilon([\Gamma'_1,\gamma'_1]\circ[\Gamma_1,\gamma_1])\otimes\varepsilon([\Gamma'_2,\gamma'_2]\circ[\Gamma_2,\gamma_2]).
\end{align*}
\end{cor}

Similar as contraction of a planar graph and coarse-graining of a fissus planar graph, we introduce the following notions which are helpful for analysing the combinatorial nature of tensor calculus.

Let $[\Gamma, \gamma]$ be a planar diagram in $\mathcal{V}$, we define its \textbf{contraction} to be a prime diagram in $\mathcal{V}$ with underlying planar graph being the contraction of $\Gamma$, domain and codomain equal to those of $[\Gamma, \gamma]$ and vertex-decoration being $\varepsilon_{\mathcal{V}}([\Gamma, \gamma])$. The construction of contraction defines a map $\kappa:\mathsf{Diag}\rightarrow \mathsf{Prim}$.

For a fissus planar diagram $[\Gamma, P_{in},P_{out},\gamma]$ in $\mathcal{V}$, its \textbf{coarse-graining} is defined to be a prime diagram in $\mathcal{V}$ with underlying planar graph being the coarse-graining of $(\Gamma, P_{in},P_{out})$, domain and codomain equal to fusion of $[\Gamma, \gamma]$'s and vertex-decoration being $\varepsilon_{\mathcal{V}}([\Gamma, \gamma])$. Concretely, let $[\Gamma, P_{in},P_{out},\gamma]$ be a fissus planar diagram with $P_{in}=(i_1<\cdots <i_{\mu_1})<\cdots <(i_{\mu_1+\cdots+\mu_{m-1}+1}<\cdots< i_{\mu_1+\cdots+\mu_m})$, $P_{out}=(o_1<\cdots <o_{\nu_1})<\cdots <(o_{\nu_1+\cdots+\nu_{n-1}+1}<\cdots< o_{\nu_1+\cdots+\nu_n})$, and assume $\gamma_o(i_k)=x_k$ for $1\leq k\leq \mu_1+\cdots+\mu_m$, $\gamma_o(o_l)=y_l$ for $1\leq l\leq \nu_1+\cdots+\nu_n$, then the coarse-grained valuation $\widehat{\gamma}$ is defined as following:
$\widehat{\gamma}_o(\{i_1,\cdots,i_{\mu_1}\})=x_1\otimes \cdots\otimes x_{\mu_1},\cdots, \widehat{\gamma}_o(i_{\mu_1+\cdots+\mu_{m-1}+1},\cdots, i_{\mu_1+\cdots+\mu_m})=i_{\mu_1+\cdots+\mu_{m-1}+1}\otimes\cdots\otimes i_{\mu_1+\cdots+\mu_m},\widehat{\gamma}_o(\{o_1,\cdots,o_{\nu_1}\})=y_1\otimes \cdots\otimes y_{\nu_1},\cdots, \widehat{\gamma}_o(o_{\nu_1+\cdots+\nu_{n-1}+1},\cdots, o_{\nu_1+\cdots+\nu_n})=y_{\nu_1+\cdots+\nu_{n-1}+1}\otimes\cdots\otimes y_{\nu_1+\cdots+\nu_n}$ and $\widehat{\gamma}_m(v)=\varepsilon_{\mathcal{V}}([\Gamma, \gamma])$ for $v\in V(\underbrace{(\Gamma, P_{in},P_{out})})$.
 The construction of contraction defines a map $\zeta:\mathsf{Diag}\rightarrow \mathsf{Prim}$.

The following proposition shows that the constructions of evaluation, contraction and coarse-graining are "functorial".

\begin{prop}
Let $K:\mathcal{V}_1\rightarrow \mathcal{V}_2$ be a strict tensor functor, then we have
$$
\begin{matrix}
$$\xymatrix{\mathsf{Diag}(\mathcal{V}_1)\ar[d]_{K_{\ast}}\ar[r]^{\varepsilon_{\mathcal{V}_1}}&Mor(\mathcal{V}_1)\ar[d]^{K}\\\mathsf{Diag}(\mathcal{V}_2)\ar[r]^{\varepsilon_{\mathcal{V}_2}}&Mor(\mathcal{V}_2),}$$
&
$$\xymatrix{\mathsf{Diag}(\mathcal{V}_1)\ar[d]_{K_{\ast}}\ar[r]^{\kappa_{\mathcal{V}_1}}&\mathsf{Prim}(\mathcal{V}_1)\ar[d]^{K_{\ast}}\\\mathsf{Diag}(\mathcal{V}_2)\ar[r]^{\kappa_{\mathcal{V}_2}}&\mathsf{Prim}(\mathcal{V}_2),}$$
&
$$\xymatrix{\mathsf{Diag}(\mathcal{V}_1)\ar[d]_{K_{\ast}}\ar[r]^{\zeta_{\mathcal{V}_1}}&\mathsf{Prim}(\mathcal{V}_1)\ar[d]^{K_{\ast}}\\\mathsf{Diag}(\mathcal{V}_2)\ar[r]^{\zeta_{\mathcal{V}_2}}&\mathsf{Prim}(\mathcal{V}_2).}$$
\end{matrix}
$$
\end{prop}

\subsection{Coarse-graining of a compound planar graph}
Recall that $\mathbf{\Gamma}=(\mathsf{\Gamma},\{x\},s,t)$ in example $4.1.2$ and $\mathbf{\Gamma}_F=(\mathsf{\Gamma}_F,W(\{x\}), dom, cod )$ in example $4.1.6$ are tensor schemes, their diagrams are important for analyzing the combinatorics  of tensor calculus .
\begin{defn}
A planar diagram in $\mathbf{\Gamma}$ is called a \textbf{compound planar graph}, and a planar diagram in $\mathbf{\Gamma}_F$ is called a \textbf{compound fissus planar graph}.
\end{defn}

So a compound planar graph can be equivalently described as a pair $[\Gamma,\lambda]$ with $\Gamma$ being a planar graph  and $\lambda$ being a map $\lambda:V_{re}(\Gamma)\rightarrow \mathsf{\Gamma}$, such that for any $v\in V_{re}(\Gamma)$ there is an equivalence $\phi_v:\underline{\lambda(v)}\overset{\sim}\rightarrow (\overrightarrow{\Gamma_v},\prec_{\Gamma_v})$ from the contraction of value $\lambda(v)$ to the vertex subgraph  $(\overrightarrow{\Gamma_v},\prec_{\Gamma_v})$ (for precise definition see the proof of theorem $3.5.7$). We call $\Gamma$  the \textbf{contraction} of the compound planar graph, which is also denoted as $\underline{[\Gamma,\lambda]}$ and  $\lambda(v)$ is called the component map with each $\lambda(v)$  called a \textbf{component}. Two compound planar graphs are equivalent if both their contractions and their corresponding components are equivalent, and we denote the set of equivalence classes of compound planar graphs by $\mathsf{\Gamma}_C$ or $\mathsf{\Gamma}(\mathbf{\Gamma})$.

Now we introduce two strict tensor categories $\mathbf{\Gamma}^{\otimes}=(\mathsf{\Gamma}, W(\{x\}), s,t)$ and $\mathbf{\Gamma}^{\otimes}_F=(\mathsf{\Gamma}_F, W(W(\{x\})), dom, cod)$ where their set of morphisms, objects and sources, targets are same as tensor schemes in example $4.1.2$ and $4.1.6$. Their tensor product and composition of morphisms are tensor product and composition of planar graphs and fissus planar graphs, respectively. Their tensor products of objects are given by juxtaposition of words and their unit objects are words of length zero. For an object $\overset{m\ times}{\overbrace{x\cdots x}}$ in $\mathbf{\Gamma}^{\otimes}$, its identity morphism is the $(m,m)$-invertible planar graph, for an object $(\overset{\mu_1\ times}{\overbrace{x\cdots x}})\cdots (\overset{\mu_m\ times}{\overbrace{x\cdots x}})$ in $\mathbf{\Gamma}^{\otimes}_F$, its identity morphism is the $(\mu_1+\cdots +\mu_m,\mu_1+\cdots +\mu_m)$-invertible planar graph with fissus structure $P_{in}=(i_1<\cdots <i_{\mu_1})<\cdots <(i_{\mu_1+\cdots+\mu_{m-1}+1}<\cdots< i_{\mu_1+\cdots+\mu_m})$ and $P_{out}=(o_1<\cdots <o_{\mu_1})<\cdots <(o_{\mu_1+\cdots+\mu_{m-1}+1}<\cdots< o_{\mu_1+\cdots+\mu_m})$. The identity morphisms of unit objects are empty graphs.

A compound planar graph can also be viewed as a special kind of planar diagram in the strict tensor category $\mathbf{\Gamma}^{\otimes}$ with its flag-decorations being non-decomposable objects, so as an usual diagram in a strict tensor category we can get a value in $Mor(\mathbf{\Gamma}^{\otimes})$, that is, given a compound planar graph, we can get a planar graph by tensoring and composing its components in the way its contraction provided, and we call this planar graph the \textbf{value} of this compound planar graph which is given exactly by the  evaluation map of $\mathbf{\Gamma}^{\otimes}$ denoted as $$Z:\mathsf{\Gamma}(\mathbf{\Gamma})\rightarrow \mathsf{\Gamma}.$$

When all components of a compound planar graph are reduced, another more intuitive view is to take it as a planar graph with a planar division structure with each cell being a component, and if we contract all its components, we will get a planar graph which reflects the \textbf{planar division structure} (i.e., a finite set of "admissible" planar sub-graphs). So the contraction of a compound planar graph represents its planar division structure and the value of it can be obtained by substitution all its (reduced) components into its contraction, and in this sense, contraction and substitution are reciprocal constructions. Any reduced component of a compound planar graph will become an planar sub-graph of the value, so we also call it an admissible planar sub-graph of the compound planar graph.

As diagrams in a strict tensor scheme, there are  well-defined notions of tensor products and compositions of  compound planar graphs and of their equivalence classes. The contraction  of a compound planar graphs defines a map
$$\xi:\mathsf{\Gamma}(\mathbf{\Gamma})\rightarrow \mathsf{\Gamma}$$
preserving tensor product and composition, and for any compound planar graph the map $\xi$ just records its planar division structure and forgets its components.

A compound planar graph $[\Gamma,\lambda]$ is called prime if its contraction $\Gamma$ is prime. We denote the set of classes of prime compound planar graphs by $\mathsf{Prim}_C$ or $\mathsf{Prim}(\mathbf{\Gamma})$ and the restriction of $Z:\mathsf{Prim}(\mathbf{\Gamma})\rightarrow \mathsf{\Gamma}$ gives  a natural isomorphism of sets $Z:\mathsf{\Gamma}(\mathbf{\Gamma})\overset{\sim}\rightarrow \mathsf{\Gamma}$. In fact, $Z([\Gamma_v, \lambda])=\lambda(v).$

\begin{defn}
A \textbf{fission structure} of a compound fissus planar graphs is  a fission structure of its contraction.
\end{defn}

A compound planar graph is called fissus if it equips with a fission structure and we  denote the set of equivalent classes of fissus compound planar graph by $\mathsf{\Gamma}_{FC}$ or $(\mathsf{\Gamma}_C)_F$ or $\mathsf{\Gamma}_F(\mathbf{\Gamma})$. Any fission structure of a compound planar graph will induce a fission structure  on its value, so the evaluation map induces a map $$Z_F:\mathsf{\Gamma}_F(\mathbf{\Gamma})\rightarrow \Gamma_F.$$ The construction of taking contraction also defines a map $$\xi_F:\mathsf{\Gamma}_F(\mathbf{\Gamma})\rightarrow \Gamma_F$$
which just forgets the components of a fissus compound planar graph.\\

Similarly a compound fissus planar graph can be equivalently described as a pair $[\Gamma,\lambda]$ with $\Gamma$ being a planar graph  and $\lambda$ being a map $\lambda:V_{re}(\Gamma)\rightarrow \mathsf{\Gamma}_F$, such that for any $v\in V_{re}(\Gamma)$ there is an equivalence $\phi_v:\underbrace{\lambda(v)}\overset{\sim}\rightarrow (\overrightarrow{\Gamma_v},\prec_{\Gamma_v})$ from the coarse-graining of value $\lambda(v)$ to the vertex subgraph  $(\overrightarrow{\Gamma_v},\prec_{\Gamma_v})$. We call $\Gamma$  the \textbf{coarse-graining} of the compound planar graph, which is also denoted as $\underbrace{[\Gamma,\lambda]}$ and  $\lambda(v)$ is called the component map with each $\lambda(v)$  called a \textbf{component}. Two compound fissus planar graphs are equivalent if both their coarse-grainings and the corresponding components are equivalent, and we denote the set of equivalence classes of compound fissus planar graphs by $(\mathsf{\Gamma}_F)_C$ or $\mathsf{\Gamma}(\mathbf{\Gamma}_F)$. Given a compound fissus planar graph, we can get a fissus planar graph by tensoring and composing its components in the way its coarse-graining provided, and we call this fissus planar graph the \textbf{value} of this compound fissus planar graph which is given exactly by the  evaluation map of $\mathbf{\Gamma}_F^{\otimes}$ denoted as $$\widehat{Z}:\mathsf{\Gamma}(\mathbf{\Gamma}_F)\rightarrow \mathsf{\Gamma}_F.$$

When all components of a compound fissus planar graph are reduced, another more intuitive view is to take it as a fissus planar graph with a planar division structure with each cell being a component, and if we coarse-grain all its components, we will get a planar graph which reflects the planar division structure. So the coarse-graining of a compound fissus planar graph represents its planar division structure and the value of it can be obtained by "substitution" all its (reduced) components into its coarse-graining, and we call this procedure \textbf{fine-graining}. In a sense, coarse-graining and fine-graining are reciprocal constructions.

There are  also well-defined notions of tensor products and compositions of  compound fissus planar graphs and of their equivalence classes. The coarse-graining of a compound fissus planar graph defines a map
$$\widehat{\xi}:\mathsf{\Gamma}(\mathbf{\Gamma}_F)\rightarrow \mathsf{\Gamma}$$
preserving tensor product and composition, and for any compound planar graph the map $\widehat{\xi}$ just records its planar division structure and forgets its components.

A compound fissus planar graph $[\Gamma,\lambda]$ is called prime if its coarse-graining $\Gamma$ is prime. We denote the set of classes of prime compound fissus planar graphs by $\mathsf{Prim}_{CF}$ or $\mathsf{Prim}(\mathbf{\Gamma}_F)$ and the restriction of $\widehat{Z}:\mathsf{Prim}(\mathbf{\Gamma}_F)\rightarrow \mathsf{\Gamma}_F$ gives  a natural isomorphism of sets $\widehat{Z}:\mathsf{Prim}_{CF}\overset{\sim}\rightarrow \mathsf{\Gamma}_F$. In fact, $Z([\Gamma_v, \lambda])=\lambda(v).$

\begin{defn}
A \textbf{fission structure} of a compound fissus planar graphs is  a fission structure of its coarse-graining.
\end{defn}

A compound fissus planar graph is called fissus if it equips with a fission structure and we  denote the set of equivalent classes of fissus compound fissus planar graph by  $((\mathsf{\Gamma}_F)_C)_F$ or $(\mathsf{\Gamma}_{FCF})$ or $\mathsf{\Gamma}_F(\mathbf{\Gamma}_F)$. Any fission structure of a compound fissus planar graph will induce a second layer of fission structure  on its value, so the evaluation map induces a map $$\widehat{Z}_F:\mathsf{\Gamma}_F(\mathbf{\Gamma}_F)\rightarrow (\Gamma_F)_F$$ from $\mathsf{\Gamma}_F(\mathbf{\Gamma}_F)$ to the set $(\Gamma_F)_F$ of isomorphic classes of two-layer fissus planar graphs.   The composition of linear partitions naturally induces a map $$\sigma:(\Gamma_F)_F\rightarrow \Gamma_F$$ which sends a two-layer fissus planar graph $((\Gamma, P_{in},P_{out}),Q_{in},Q_{out})$ to a fissus planar graph $(\Gamma, O_{in}\triangleleft P_{in}, O_{out}\triangleleft P_{out}).$

The construction of taking coarse-graining also defines a map
$$\widehat{\xi}_F:\mathsf{\Gamma}_F(\mathbf{\Gamma}_F)\rightarrow \Gamma_F$$
which  forgets components of a fissus compound fissus planar graph.

Here in summary, we see two parallel stories:
\begin{center}
\begin{tabular}{|l|l|}
\hline
planar graph & fissus planar graph\\\hline
$\mathsf{\Gamma}$ &$\mathsf{\Gamma}_F$\\ \hline
$\mathsf{\Gamma}(\mathbf{\Gamma})$ &$\mathsf{\Gamma}(\mathbf{\Gamma}_F)$\\ \hline
$\mathsf{\Gamma}_F(\mathbf{\Gamma})$ &$\mathsf{\Gamma}_F(\mathbf{\Gamma}_F)$\\ \hline
contraction & coarse-graining \\\hline
substitution &fine-graining\\\hline
$Z :\mathsf{\Gamma}(\mathbf{\Gamma})\rightarrow \mathsf{\Gamma}$ &$\widehat{Z} :\mathsf{\Gamma}(\mathbf{\Gamma}_F)\rightarrow \mathsf{\Gamma}_F$\\\hline
$Z_F :\mathsf{\Gamma}_F(\mathbf{\Gamma})\rightarrow \mathsf{\Gamma}_F$ &$\widehat{Z}_F :\mathsf{\Gamma}_F(\mathbf{\Gamma}_F)\rightarrow \mathsf{\Gamma}_F$\\\hline
$Z:\mathsf{Prim}(\mathbf{\Gamma})\overset{\sim}\rightarrow \mathsf{\Gamma}$ &$\widehat{Z}:\mathsf{Prim}(\mathbf{\Gamma}_F)\overset{\sim}\rightarrow \mathsf{\Gamma}_F$\\\hline
$\xi :\mathsf{\Gamma}(\mathbf{\Gamma})\rightarrow \mathsf{\Gamma}$ &$\widehat{\xi} :\mathsf{\Gamma}(\mathbf{\Gamma}_F)\rightarrow \mathsf{\Gamma}$\\\hline
$\xi_F :\mathsf{\Gamma}_F(\mathbf{\Gamma})\rightarrow \mathsf{\Gamma}_F$ &$\widehat{\xi}_F :\mathsf{\Gamma}_F(\mathbf{\Gamma}_F)\rightarrow \mathsf{\Gamma}_F$\\\hline
\end{tabular}
\end{center}

\section{Adjunction of tensor calculus}
The main task of  this section is to  introduce the adjunction of tensor calculus between the category $\mathbf{T.Sch}$ of tensor schemes and the category $\mathbf{Str.T}$ of strict small tensor categories. Namely, we will introduce two functors $F,U$ as
$$\xymatrix{\mathbf{T.Sch}\ar@<0.7mm>[rr]^{F}&&\ar@<0.7mm>[ll]^{U}\mathbf{Str.T}}$$
such that for any tensor scheme $\mathcal{D}\in \mathbf{T.Sch}$ and strict tensor category $\mathcal{V}\in\mathbf{Str.T}$, we have a family of natural bijections $$\Theta_{\mathcal{D},\mathcal{V}}: Hom_{\mathbf{Str.T}}(F(\mathcal{D}),\mathcal{V})\longrightarrow Hom_{\mathbf{T.Sch}}(\mathcal{D},U(\mathcal{V})).$$
After giving a precise description of the unit and counit of this adjunction, we will also discuss the associated monad and comonad.

\subsection{The functor $F$}
The functor $F$ (called \textbf{aggregation functor}) is given by Joyal and Street's construction of a free tensor category.
For every tensor scheme $\mathcal{D}$, the functor $F$ will produce a strict tensor category $F(\mathcal{D})$ defined as follows: the objects are words in elements of $Ob(\mathcal{D})$; the morphisms are isomorphic classes of planar diagrams in $\mathcal{D}$; the domain, codomain, tensor product and composition of morphisms are induced on isomorphic classes by the corresponding operations for the diagrams; identity morphisms are isomorphic classes of diagrams with invertible graphs; the tensor product on objects is given by juxtaposition of words. In short, we write $F(\mathcal{D})=(W(Ob(\mathcal{D})), \mathsf{Diag}(\mathcal{D}),s=dom, t=cod)$ or  $\xymatrix{\mathsf{Diag}(\mathcal{D})\ar@<1mm>[r]^{s}\ar@<-1mm>[r]_{t}&W(Ob(\mathcal{D})).}$

The unit object $1_{F(\mathcal{D})}$ is the word of length zero (or null string) denoted by $\varnothing$ and the identity morphism $Id_{1_{F(\mathcal{D})}}$ of $1_{F(\mathcal{D})}$ is the empty graph. For any word $x_1\cdots x_n\in Ob(F(\mathcal{D}))$, its identity morphism $Id_{x_1\cdots x_n}=Id_{\langle x_1\rangle}\otimes\cdots \otimes Id_{\langle x_n\rangle}$ is given by the invertible diagram
$$
\begin{matrix}
\begin{tikzpicture}[scale=.5]
\node (v2) at (0,0) {};
\draw (0,0) circle [radius=0.1];
\node (v1) at (0,1.5) {$x_1$};
\node (v3) at (0,-1.5) {$x_1$};
\node at (1,0) {$\cdots$};
\node (v5) at (2,0) {};
\draw (2,0) circle [radius=0.1];

\node (v4) at (2,1.5) {$x_n$};
\node (v6) at (2,-1.5) {$x_n.$};
\draw  (v1) -- (0,0.1)[postaction={decorate, decoration={markings,mark=at position .50 with {\arrow[black]{stealth}}}}];
\draw  (v3) -- (0,-0.1)[postaction={decorate, decoration={markings,mark=at position .50 with {\arrowreversed[black]{stealth}}}}];
\draw  (v4) -- (2,0.1)[postaction={decorate, decoration={markings,mark=at position .50 with {\arrow[black]{stealth}}}}];
\draw  (v6) -- (2,-0.1)[postaction={decorate, decoration={markings,mark=at position .50 with {\arrowreversed[black]{stealth}}}}];
\end{tikzpicture}
\end{matrix}
=
\begin{matrix}
\begin{tikzpicture}[scale=.5]
\node (v2) at (0,0) {};
\draw (0,0) circle [radius=0.1];
\node (v1) at (0,1.5) {$\langle x_1\rangle$};
\node (v3) at (0,-1.5) {$\langle x_1\rangle$};
\draw  (v1) -- (0,0.1)[postaction={decorate, decoration={markings,mark=at position .50 with {\arrow[black]{stealth}}}}];
\draw  (v3) -- (0,-0.1)[postaction={decorate, decoration={markings,mark=at position .50 with {\arrowreversed[black]{stealth}}}}];
\end{tikzpicture}
\end{matrix}
\begin{matrix}\otimes\end{matrix} \begin{matrix}\cdots\end{matrix} \begin{matrix}\otimes\end{matrix}
\begin{matrix}
\begin{tikzpicture}[scale=.5]
\node (v2) at (0,0) {};
\draw (0,0) circle [radius=0.1];
\node (v1) at (0,1.5) {$\langle x_n\rangle$};
\node (v3) at (0,-1.5) {$\langle x_n\rangle$};
\draw  (v1) -- (0,0.1)[postaction={decorate, decoration={markings,mark=at position .50 with {\arrow[black]{stealth}}}}];
\draw  (v3) -- (0,-0.1)[postaction={decorate, decoration={markings,mark=at position .50 with {\arrowreversed[black]{stealth}}}}];
\end{tikzpicture}
\end{matrix},
$$
where we write $\langle x\rangle$ to distinguish the word $\langle x\rangle$ in $W(Ob(\mathcal{D}))$ from the element $x\in Ob(\mathcal{D})$.

\begin{ex}
Take $\mathcal{D}=\mathbf{Prim}$ be the tensor scheme in example $4.1.3$, then the evaluation map $Z :\mathsf{\Gamma}(\mathbf{\Gamma})\rightarrow \mathsf{\Gamma}$ can induce an isomorphism of strict tensor categories $$\widetilde{\zeta}:F(\mathbf{Prim})\overset{\sim}\rightarrow \mathbf{\Gamma}^{\otimes}.$$

In fact, $\mathbf{Prim}=(\mathsf{Prim}, \{x\},s, t)$, then $F(\mathbf{Prim})=(\mathsf{\Gamma}(\mathbf{Prim}),W(\{x\}), s,t)$. So a morphism in $F(\mathbf{Prim})$ is a compound planar graph with each non-empty component being a prime graph, hence $Z$ can induces an isomorphism  $$\zeta:\mathsf{\Gamma}(\mathbf{Prim})\rightarrow \mathsf{\Gamma}$$
due to the fact that prime graphs are reduced graphs.  Note that $\mathbf{\Gamma}^{\otimes}=(\mathsf{\Gamma}, W(\{x\}), s,t)$ $($defined in section $4.5$$)$, and  we define $\widetilde{\zeta}_o:W(\{x\})\rightarrow W(\{x\})$ to be the identity map and $\widetilde{\zeta}_m=\zeta$ being the restriction of $Z$ on $\mathsf{\Gamma}(\mathbf{Prim})$. Then it is easy to check that $\widetilde{\zeta}=(\widetilde{\zeta}_m,\widetilde{\zeta}_o)$ is an isomorphism of strict tensor categories.
\end{ex}

Further more, if $\varphi:\mathcal{D}_1\rightarrow \mathcal{D}_2$ is a morphism of tensor schemes, the functor $F$ will produce a "strict tensor functor" $F(\varphi):F(\mathcal{D}_1)\rightarrow F(\mathcal{D}_2)$ constituted by the function $\widehat{\varphi_{o}}:W(Ob(\mathcal{D}_1))\rightarrow W(Ob(\mathcal{D}_2))$  at the level of objects and  the function $\varphi_{\ast}:\mathsf{Diag}(\mathcal{D}_1)\rightarrow\mathsf{Diag}(\mathcal{D}_2)$  at the level of  morphisms.

By proposition $4.2.2$, the following proposition can be easily checked.
\begin{prop}
For every morphism  $\varphi:\mathcal{D}_1\rightarrow \mathcal{D}_2$ of tensor schemes, $F(\varphi):F(\mathcal{D}_1)\rightarrow F(\mathcal{D}_2)$ is a strict tensor functor. Moreover, if  $\varphi_1:\mathcal{D}_1\rightarrow \mathcal{D}_2$  and $\varphi_2:\mathcal{D}_2\rightarrow \mathcal{D}_3$  are two morphisms of tensor schemes, then  $$F(\varphi_2\circ\varphi_1)=F(\varphi_2)\circ F(\varphi_1)$$ as strict tensor functors from $F(\mathcal{D}_1)$ to $F(\mathcal{D}_3)$.
\end{prop}
Hence,  we can easily get the fact:
\begin{prop}
The construction $F$ is a functor from $\mathbf{T.Sch}$ to $\mathbf{Str.T}.$
\end{prop}

\subsection{The functor $U$}
The functor $U$ (called \textbf{fission functor}) is given essentially by the construction of prime diagrams in a strict tensor category.
For every strict small tensor category $\mathcal{V}$, we define a tensor scheme $U(\mathcal{V})$ as follows: the objects are the objects of $\mathcal{V}$; the morphisms are isomorphic classes of  prime diagrams in $\mathcal{V}$ ; the source and target maps are given by the domain and codomain of diagrams.  We denote the set of isomorphic classes of  prime diagrams in $\mathcal{V}$ by $\mathsf{Prim}(\mathcal{V})$ and write $U(\mathcal{V})=(Ob(\mathcal{V}), \mathsf{Prim}(\mathcal{V}),s=dom, t=cod)$ or $\xymatrix{ \mathsf{Prim}(\mathcal{V})\ar@{-->}@<1mm>[r]^{s}\ar@{-->}@<-1mm>[r]_{t}&Ob(\mathcal{V}).}$

\begin{ex}
Take $\mathcal{V}=\mathbf{\Gamma}^{\otimes}$, then there is an natural isomorphism from the tensor scheme $U(\mathbf{\Gamma}^{\otimes})$ to the tensor scheme $\mathbf{\Gamma}_F$ in example $4.1.6$. In fact, $\mathbf{\Gamma}^{\otimes}=(\mathsf{\Gamma}, W(\{x\}), s,t)$, so  $U(\mathbf{\Gamma}^{\otimes})=(\mathsf{Prim}(\mathbf{\Gamma}^{\otimes}),W(\{x\}), dom,cod)$.

Let $[\Gamma_v, \lambda]$ be a morphism of $U(\mathbf{\Gamma}^{\otimes})$ which is a prime diagram in $\mathbf{\Gamma}^{\otimes}$. Assume $H(\Gamma_v)=\{I_1,...,I_m, O_1,...,O_n\}$ and $\lambda(I_\alpha)=\overset{\mu_\alpha}{\overbrace{x\cdots x}}$ $(1\leq\alpha\leq m)$, $\lambda(O_\beta)=\overset{\nu_\beta}{\overbrace{x\cdots x}}$ $(1\leq\beta\leq n)$, then $\lambda(v)$ is a $(\mu_1+\cdots+\mu_m, \nu_1+\cdots+\nu_n)$-planar graph with domain $\overset{\mu_1+\cdots+\mu_m}{\overbrace{x\cdots x}}=\overset{\mu_1}{\overbrace{x\cdots x}}\otimes \cdots\otimes\overset{\mu_m}{\overbrace{x\cdots x}}$ and codomain $\overset{\nu_1+\cdots+\nu_n}{\overbrace{x\cdots x}}=\overset{\nu_1}{\overbrace{x\cdots x}}\otimes \cdots\otimes\overset{\nu_n}{\overbrace{x\cdots x}}$, so we can assume $H(\lambda(v))=\{i_1,...,i_{\mu_1+\cdots+\mu_m},o_1,...,o_{\nu_1+\cdots+\nu_n}\}$.

Now we define a map $$\omega:\mathsf{Prim}(\mathbf{\Gamma}^{\otimes})\rightarrow \mathsf{\Gamma}_F$$
which sends $[\Gamma_v, \lambda]$ to a fissus planar graph $[\lambda(v),P_{in},P_{out}]$ with $P_{in}=(i_1<\cdots<i_{\mu_1})<\cdots<(i_{\mu_1+\cdots+\mu_{m-1}+1}<\cdots<i_{\mu_1+\cdots+\mu_m})$, $P_{out}=(o_1<\cdots<o_{\nu_1})<\cdots<(o_{\nu_1+\cdots+\nu_{n-1}+1}<\cdots<o_{\nu_1+\cdots+\nu_n})$, hence the coarse-graining of $\omega([\Gamma_v, \lambda])$ is isomorphic to $\Gamma_v$.

Note that  $\mathbf{\Gamma}_F=(\mathsf{\Gamma}_F,W(\{x\}), dom, cod )$ and we can define a morphism of tensor schemes $$\widetilde{\omega}:U(\mathbf{\Gamma}^{\otimes})\rightarrow\mathbf{\Gamma}_F$$
with $\widetilde{\omega}_m=\omega:\mathsf{Prim}(\mathbf{\Gamma}^{\otimes})\rightarrow \mathsf{\Gamma}_F$ and $\widetilde{\omega}_o:W(\{x\})\rightarrow W(\{x\})$ being the identity map. The fact that $\widetilde{\omega}$ is an isomorphism can be directly checked.
\end{ex}

For a strict tensor functor $K:\mathcal{V}_1\rightarrow \mathcal{V}_2$  between strict tensor categories, the construction $U$ gives a morphism $U(K):U(\mathcal{V}_1)\rightarrow U(\mathcal{V}_2)$ of tensor schemes  constituted by the function $K_o:Ob(\mathcal{V}_1)\rightarrow Ob(\mathcal{V}_2)$ at the level of objects and the function $K_{\ast}:\mathsf{Prim}(\mathcal{V}_1)\rightarrow \mathsf{Prim}(\mathcal{V}_2)$ (i.e, the restriction of $K_{\ast}$ on prime diagrams ) at the level of morphisms.

The following propositions can be easily checked.
\begin{prop}
If $K_1:\mathcal{V}_1\rightarrow \mathcal{V}_2$ and $K_2:\mathcal{V}_2\rightarrow \mathcal{V}_3$ are two strict tensor functors, then $$U(K_2\circ K_1)=U(K_2)\circ U(K_1)$$ as morphisms of tensor schemes from $U(\mathcal{V}_1)$ to $U(\mathcal{V}_3)$.
\end{prop}

\begin{prop}
The construction $U$ is a functor from $\mathbf{Str.T}$ to $\mathbf{T.Sch}.$
\end{prop}
\begin{rem}
Please notice that the functor $U$ is different from usual forgetful functors which just forget some structures, rather it seems like a kind of resolution..
\end{rem}

\subsection{Adjointness}

Our intension in this section is to prove the following theorem:

\begin{thm}
The pair of functors $F,U$ form an adjunction.
\end{thm}
Before proving the theorem, we introduce four natural functions which are useful to make our proof clear:

$$
\begin{matrix}
\begin{split}i_o:&Ob(\mathcal{D})\hookrightarrow Ob(F(\mathcal{D}))=W(Ob(\mathcal{D}))\\&x\mapsto i_o(x)=\langle x\rangle, \end{split}
&\begin{split}i_m:&Mor(\mathcal{D})\hookrightarrow Mor(F(\mathcal{D}))\\&(f:x_1\cdots x_m\rightarrow y_1\cdots y_n)\mapsto i_m(f)=[v,\gamma_f], \end{split}
\end{matrix}
$$
where $[v,\gamma_{f}]$ is a  prime diagram in $\mathcal{D}$ with domain $x_1\cdots x_m$ codomain $y_1\cdots y_n$ with the unique vertex $v$ decorated by $f$, that is,

$$
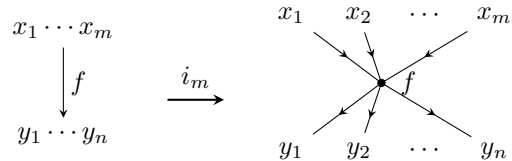
\begin{matrix}
\begin{matrix}
\begin{tikzpicture}
\node (v1) at (0,0.7) {$x_1\cdots x_m$};
\node (v2) at (0,-0.7) {$y_1\cdots y_n$};
\draw [->,>=stealth] (v1) -- (v2);
\node at (0.2,0) {$f$};
\end{tikzpicture}
\end{matrix}

&\begin{matrix}
\begin{tikzpicture}
\node (v1) at (-1,0.5) {};
\node (v2) at (0,0.5) {};
\draw [thick,->,>=stealth] (v1) -- (v2);
\node at (-0.5,0.8) {$i_m$};
\end{tikzpicture}
\end{matrix}&
\begin{matrix}\begin{tikzpicture}[scale=.6]
\node (v1) at (0,0) {};
\draw[fill] (0,0) circle [radius=0.08];
\node at (0.6,0){$f$};
\node (v2) at (-2,1.5) {$x_1$};
\node (v3) at (-0.5,1.5) {$x_2$};
\node (v4) at (1,1.5) {$\cdots$};
\node (v5) at (2.5,1.5) {$x_m$};
\node (v6) at (-2,-1.5) {$y_1$};
\node (v7) at (-0.5,-1.5) {$y_2$};
\node at (1,-1.5) {$\cdots$};
\node (v8) at (2.5,-1.5) {$y_n$};
\draw  (v2) -- (0,0)[postaction={decorate, decoration={markings,mark=at position .50 with {\arrow[black]{stealth}}}}];
\draw  (v3) -- (0,0)[postaction={decorate, decoration={markings,mark=at position .50 with {\arrow[black]{stealth}}}}];
\draw  (v5) -- (0,0)[postaction={decorate, decoration={markings,mark=at position .50 with {\arrow[black]{stealth}}}}];
\draw  (v6) -- (0,0)[postaction={decorate, decoration={markings,mark=at position .50 with {\arrowreversed[black]{stealth}}}}];
\draw  (v7) -- (0,0)[postaction={decorate, decoration={markings,mark=at position .50 with {\arrowreversed[black]{stealth}}}}];
\draw  (v8) -- (0,0)[postaction={decorate, decoration={markings,mark=at position .50 with {\arrowreversed[black]{stealth}}}}];
\end{tikzpicture}
\end{matrix}
\end{matrix}
$$
and
$$
\begin{matrix}
\begin{split}j_o:&Ob(\mathcal{V})\overset{\sim}\rightarrow Ob(U(\mathcal{V}))\\&v\mapsto j_o(v)=v, \end{split}&
\begin{split}j_m:&Mor(\mathcal{V})\hookrightarrow Mor(U(\mathcal{V}))=\mathsf{Prim}(\mathcal{V})\\&(g:x\rightarrow y)\mapsto j_m(g)=[v,\gamma_g], \end{split}
\end{matrix}
$$
where $[v,\gamma_g]$ is a prime diagram in $\mathcal{V}$ with domain $x$, codomain $y$ and an unique vertex $v$ decorated by $g$ if the source and target of $g$ is $x$ and $y$, that is,
$$
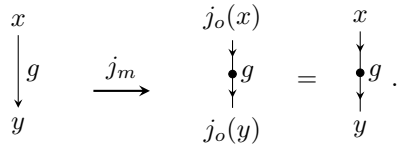
\begin{matrix}
\begin{matrix}
\begin{tikzpicture}
\node (v1) at (0,0.7) {$x$};
\node (v2) at (0,-0.7) {$y$};
\draw [->,>=stealth] (v1) -- (v2);
\node at (0.2,0) {$g$};
\end{tikzpicture}
\end{matrix}

&\begin{matrix}
\begin{tikzpicture}
\node (v1) at (-1,0.5) {};
\node (v2) at (0,0.5) {};
\draw [thick,->,>=stealth] (v1) -- (v2);
\node at (-0.5,0.8) {$j_m$};
\end{tikzpicture}
\end{matrix}&
\begin{matrix}
\begin{tikzpicture}[scale=.5]
\node (v1) at (0,1.5) {$j_o(x)$};
\node (v2) at (0,-1.5) {$j_o(y)$};
\draw[fill] (0,0) circle [radius=0.1];
\node at (0.4,0){$g$};
\draw  (v1) -- (0,0)[postaction={decorate, decoration={markings,mark=at position .50 with {\arrow[black]{stealth}}}}];
\draw  (0,0) -- (v2)[postaction={decorate, decoration={markings,mark=at position .70 with {\arrow[black]{stealth}}}}];
\end{tikzpicture}
\end{matrix}
&=&\begin{matrix}
\begin{tikzpicture}[scale=.5]
\node (v1) at (0,1.5) {$x$};
\node (v2) at (0,-1.5) {$y$};
\draw[fill] (0,0) circle [radius=0.1];
\node at (0.4,0){$g$};
\draw  (v1) -- (0,0)[postaction={decorate, decoration={markings,mark=at position .50 with {\arrow[black]{stealth}}}}];
\draw  (0,0) -- (v2)[postaction={decorate, decoration={markings,mark=at position .70 with {\arrow[black]{stealth}}}}];
\end{tikzpicture}
\end{matrix}.
\end{matrix}
$$

Obviously, $i_o,i_m$ and $j_m$ are injections and $j_o$ is a bijection. From their definitions, we have
$$
\begin{matrix}
i_o(s(f))=s(i_m(f)),
&i_o(t(f))=t(i_m(f))
\end{matrix}
$$
for $f\in Mor(\mathcal{D})$ and
$$
\begin{matrix}
j_o(s(g))=s(j_m(g)),
&j_o(t(g))=t(j_m(g))
\end{matrix}
$$
for $g\in Mor(\mathcal{V}).$
\begin{proof}
$(1)$  As the first step, we will define two functions $$\Theta:Hom(F(\mathcal{D}),\mathcal{V})\rightarrow Hom(\mathcal{D},U(\mathcal{V}))$$ and $$\Theta^{-1}:Hom(\mathcal{D},U(\mathcal{V}))\rightarrow Hom(F(\mathcal{D}),\mathcal{V})$$   for any tensor scheme $\mathcal{D}$ and strict tensor category $\mathcal{V}$,  and prove that they are  inverse functions of each other.

$\bullet$ Definition of $\Theta$.

 If $K:F(\mathcal{D})\rightarrow\mathcal{V}$ is a strict tensor functor, then we define $$\varphi_K=\Theta(K):\mathcal{D}\rightarrow U(\mathcal{V})$$ as follows: for an element $x\in Ob(\mathcal{D})$, we define $$(\varphi_K)_o(x)=j_o(K(i_o(x)))=K\langle x\rangle\in Ob(U(\mathcal{V}));$$

$$\xymatrix{F(\mathcal{D})\ar[r]^{K}&\ar[d]^{j_o,j_m}\mathcal{V}\\\mathcal{D}\ar[u]^{i_o,i_m}\ar[r]^{\varphi_K}&U(\mathcal{V})}$$
for a morphism
$f:x_1\cdots x_m\rightarrow y_1\cdots y_n$ in $Mor(\mathcal{D})$, we define $(\varphi_K)_m(f)$ to be a prime diagram $[v,\gamma_{K(i_m(f))}]$ with domain $$(\varphi_K)_o(x_1)\cdots(\varphi_K)_o(x_m)=K\langle x_1\rangle\cdots K\langle x_m\rangle$$ and codomain $$(\varphi_K)_o(y_1)\cdots(\varphi_K)_o(y_n)=K\langle y_1\rangle\cdots K\langle y_n\rangle$$ in  $W(Ob(U(\mathcal{V})))$ and with the unique real vertex decorated by $$K(i_m(f))\in Mor_{\mathcal{V}}(K\langle x_1\rangle\otimes\cdots\otimes K\langle x_m\rangle, K\langle y_1 \rangle\otimes\cdots\otimes K\langle y_n\rangle),$$ that is,
$$
\begin{matrix}
\begin{matrix}
\begin{tikzpicture}
\node (v1) at (0,0.7) {$x_1\cdots x_m$};
\node (v2) at (0,-0.7) {$y_1\cdots y_n$};
\draw [->,>=stealth] (v1) -- (v2);
\node at (0.2,0) {$f$};
\end{tikzpicture}
\end{matrix}

&\begin{matrix}
\begin{tikzpicture}
\node (v1) at (-1.5,0.5) {};
\node (v2) at (0.5,0.5) {};
\draw [thick,->,>=stealth] (v1) -- (v2);
\node at (-0.5,0.8) {$(\varphi_K)_m$};
\end{tikzpicture}
\end{matrix}&
\begin{matrix}
\begin{tikzpicture}[scale=.6]
\node (v1) at (0,0) {};
\draw[fill] (0,0) circle [radius=0.08];
\node at (2,0){$K(i_m(f))$};
\node (v2) at (-2,1.5) {$K\langle x_1\rangle$};
\node (v3) at (-0.5,1.5) {$K\langle x_2\rangle$};
\node (v4) at (1,1.5) {$\cdots$};
\node (v5) at (2.5,1.5) {$K\langle x_m\rangle$};
\node (v6) at (-2,-1.5) {$K\langle y_1\rangle$};
\node (v7) at (-0.5,-1.5) {$K\langle y_2\rangle$};
\node at (1,-1.5) {$\cdots$};
\node (v8) at (2.5,-1.5) {$K\langle y_n\rangle$};
\draw  (v2) -- (0,0)[postaction={decorate, decoration={markings,mark=at position .50 with {\arrow[black]{stealth}}}}];
\draw  (v3) -- (0,0)[postaction={decorate, decoration={markings,mark=at position .50 with {\arrow[black]{stealth}}}}];
\draw  (v5) -- (0,0)[postaction={decorate, decoration={markings,mark=at position .50 with {\arrow[black]{stealth}}}}];
\draw  (v6) -- (0,0)[postaction={decorate, decoration={markings,mark=at position .50 with {\arrowreversed[black]{stealth}}}}];
\draw  (v7) -- (0,0)[postaction={decorate, decoration={markings,mark=at position .50 with {\arrowreversed[black]{stealth}}}}];
\draw  (v8) -- (0,0)[postaction={decorate, decoration={markings,mark=at position .50 with {\arrowreversed[black]{stealth}}}}];
\end{tikzpicture}
\end{matrix}

\end{matrix}
$$

The fact  that  $\varphi_K$ is a morphism of tensor schemes, namely, $\varphi_K$ is compatible with the source and target maps can be easily checked from its definition.

$$\xymatrix{&&\\
Mor(\mathcal{D})\ar[d]_{\varphi}\ar@{-->}[r]^{i_m}_{\cong}&\ar[r]^{\subset}\mathsf{Prim}(\mathcal{D})&\mathsf{Diag}(\mathcal{D})\ar@{-->}[d]^{\varphi_{\sharp}}&Mor(F(\mathcal{D}))\ar@{=}[l]\ar[d]^{K}\\
\ar@{=}[r]Mor(U(\mathcal{V}))&\ar[r]^{\subset}\mathsf{Prim}(\mathcal{V})&\mathsf{Diag}(\mathcal{V})\ar@{-->}[r]^{\epsilon}&Mor(\mathcal{V})\ar@{=}[dl]\\
&&\ar[ul]^{j_m}Mor(\mathcal{V})}$$

$\bullet$ Definition of $\Theta^{-1}$.

If $\varphi:\mathcal{D}\rightarrow U(\mathcal{V})$ is a morphism of tensor schemes, we want to  define a strict tensor functor $$K_{\varphi}=\Theta^{-1}(K):F(\mathcal{D})\rightarrow \mathcal{V}.$$

Note that  any object $w \in Ob(F(D))$ can be uniquely written as $$w=i_o(x_1)\otimes_{F(\mathcal{D})}\cdots \otimes_{F(\mathcal{D})}i_o(x_m)\triangleq \langle x_1\rangle\otimes_{F(\mathcal{D})}\cdots \otimes_{F(\mathcal{D})}\langle x_m\rangle=x_1\cdots x_m,$$ where
$x_1,\cdots, x_m$ are elements of $Ob(\mathcal{D})$.
We define $K_{\varphi}$ on objects as follows:

$(1)$ If $w=1_{F(\mathcal{D})}$ be the unit object of $F(\mathcal{D})$, we define $K_{\varphi}(w)=1_{\mathcal{V}}$ to be the unit object of $\mathcal{V}$;

$(2)$ If $w=i_o(x)=\langle x\rangle$ for some $x\in Ob(\mathcal{D})$, we define
$$K_{\varphi}(w)=j_o^{-1}(\varphi(x))=\varphi(x)\in Ob(\mathcal{V});$$

$(3)$ If $w=i_o(x_1)\otimes_{F(\mathcal{D})}\cdots\otimes_{F(\mathcal{D})} i_o(x_m)$, we define
$$K_{\varphi}(w)=j_o^{-1}(\varphi(x_1))\otimes_{\mathcal{V}}\cdots\otimes_{\mathcal{V}} j_o^{-1}(\varphi(x_m))=\varphi(x_1)\otimes_{\mathcal{V}}\cdots\otimes_{\mathcal{V}} \varphi(x_m)\in Ob(\mathcal{V}).$$

$$\xymatrix{F(\mathcal{D})\ar[r]^{K_{\varphi}}&\ar[d]^{j_o,j_m}\mathcal{V}\\\mathcal{D}\ar[u]^{i_o,i_m}\ar[r]^{\varphi}&U(\mathcal{V})}$$

For any morphism $$[\Gamma,\gamma]:i_o(x_1)\otimes_{F(\mathcal{D})}\cdots\otimes_{F(\mathcal{D})} i_o(x_m)\rightarrow i_o(y_1)\otimes_{F(\mathcal{D})}\cdots \otimes_{F(\mathcal{D})}i_o(y_n)$$ in $Mor(F(\mathcal{D}))=\mathsf{Diag}(\mathcal{D})$, we define a diagram $\varphi_{\sharp}([\Gamma,\gamma]):=[\Gamma,\varphi_{\sharp}\gamma]$ in $\mathsf{Diag}(\mathcal{V})$ as follows:
$$(\varphi_{\sharp}\gamma)_m(v)=\varepsilon(\varphi_m (\gamma_m(v)))\in Mor(\mathcal{V})$$
for every real vertex  $v \in V_{re}(\Gamma)$ and $$(\varphi_{\sharp}\gamma)_o(h)=j^{-1}_o(\varphi_o (\gamma_o(h)))\in Ob(\mathcal{V})$$ for every half-edge $h\in H(\Gamma).$ In fact, from the definition, $\varphi_{\sharp}$  is an operation on diagrams in $\mathcal{D}$ by applying $\varphi_{\ast}$ on every prime sub-diagram in $\mathcal{D}$ which is naturally identified with $Mor(\mathcal{D})$. As $\varphi$ is a morphism of tensor schemes, it is easy to check that $\varphi_{\sharp}([\Gamma,\gamma])$ is indeed a diagram in $\mathcal{V}$.

From the definition of $\varphi_{\sharp},\otimes_{F(\mathcal{D})}$, $\circ_{F(\mathcal{D})}$, tensor product $\otimes_{\mathsf{Diag}(\mathcal{V})}$ and  composition $\circ_{\mathsf{Diag}(\mathcal{V})}$ in $\mathcal{V}$, it is easy to check that $$\varphi_{\sharp}([\Gamma_1,\gamma_1]\otimes_{F(\mathcal{D})}[\Gamma_2,\gamma_2])=\varphi_{\sharp}([\Gamma_1,\gamma_1])\otimes_{\mathsf{Diag}(\mathcal{V})} \varphi_{\sharp}([\Gamma_2,\gamma_2])$$ for any two morphisms in $F(\mathcal{D})$ and $$\varphi_{\sharp}([\Gamma_2,\gamma_2]\circ_{F(\mathcal{D})}[\Gamma_1,\gamma_1])=\varphi_{\sharp}([\Gamma_2,\gamma_2])\circ_{\mathsf{Diag}(\mathcal{V})} \varphi_{\sharp}([\Gamma_1,\gamma_1])$$
for any two composable morphisms in $F(\mathcal{D})$.

Now we define $K_{\varphi}([\Gamma,\gamma])$ to be $$K_{\varphi}([\Gamma,\gamma])=\varepsilon(\varphi_{\sharp}([\Gamma,\gamma]))\in Mor(\mathcal{V}),$$ that is, the value of  the diagram $[\Gamma,\varphi_{\sharp}\gamma]$.

From the definition, it is easy to check that $K_{\varphi}$ is compatible with source and target maps.

Now let us show that $K_{\varphi}$ is a strict tensor functor. For  any two morphisms $[\Gamma_1,\gamma_1],[\Gamma_2,\gamma_2]\in Mor(F(\mathcal{D}))$, we have
\begin{align*}
&K_{\varphi}([\Gamma_1,\gamma_1]\otimes_{F(\mathcal{D})}[\Gamma_2,\gamma_2])&\\
=&\varepsilon(\varphi_{\sharp}([\Gamma_1,\gamma_1]\otimes_{F(\mathcal{D})}[\Gamma_2,\gamma_2]))&(\text{definition of}\ K_{\varphi} )\\
=&\varepsilon(\varphi_{\sharp}([\Gamma_1,\gamma_1])\otimes_{\mathsf{Diag}(\mathcal{V})} \varphi_{\sharp}([\Gamma_2,\gamma_2]))&(\text{property of}\ \varphi_{\sharp}) \\
=&\varepsilon(\varphi_{\sharp}([\Gamma_1,\gamma_1]))\otimes_{\mathcal{V}}\varepsilon (\varphi_{\sharp}([\Gamma_2,\gamma_2]))&(\text{property of}\ \varepsilon)\\
=&K_{\varphi}([\Gamma_1,\gamma_1])\otimes_{\mathcal{V}}K_{\varphi}([\Gamma_2,\gamma_2])&(\text{definition of}\ K_{\varphi} )
\end{align*}

For any two composable morphisms $[\Gamma_1,\gamma_1]:x_1\cdots x_m\rightarrow y_1\cdots y_n$ and $[\Gamma_2,\gamma_2]:y_1\cdots y_n\rightarrow z_1\cdots z_l$ of $F(\mathcal{D})$, we have
\begin{align*}
&K_{\varphi}([\Gamma_1,\gamma_1]\circ_{F(\mathcal{D})}[\Gamma_2,\gamma_2])&\\
=&\varepsilon(\varphi_{\sharp}([\Gamma_1,\gamma_1])\circ_{F(\mathcal{D})}[\Gamma_2,\gamma_2]))&(\text{definition of}\ K_{\varphi} )\\
=&\varepsilon(\varphi_{\sharp}([\Gamma_1,\gamma_1])\circ_{\mathsf{Diag}(\mathcal{V})} \varphi_{\sharp}([\Gamma_2,\gamma_2]))&(\text{property of}\ \varphi_{\sharp}) \\
=&\varepsilon(\varphi_{\sharp}([\Gamma_1,\gamma_1]))\circ_{\mathcal{V}}\varepsilon (\varphi_{\sharp}([\Gamma_2,\gamma_2]))&(\text{property of}\ \varepsilon)\\
=&K_{\varphi}([\Gamma_1,\gamma_1])\circ_{\mathcal{V}}K_{\varphi}([\Gamma_2,\gamma_2])&(\text{definition of}\ K_{\varphi} )
\end{align*}
The fact that $K_{\varphi}$ preserves unit objects is followed  from the definition of $K_{\varphi}$.

$\bullet$ $\Theta$ and $\Theta^{-1}$ are inverse functions of each other.

We want to show that $\Theta(K_{\varphi})=\varphi$ and $\Theta^{-1}(\varphi_K)=K$
for every $\varphi:\mathcal{D}\rightarrow U(\mathcal{V})$ and $K:F(\mathcal{D})\rightarrow \mathcal{V}$ and this can be directly checked by the definitions of $\Theta$ and $\Theta^{-1}$.

First we prove that $\Theta(K_{\varphi})=\varphi$.
In fact, for every $x\in Ob(\mathcal{D})$,
\begin{align*}
&\Theta(K_{\varphi})(x)&\\
=&j_o(K_{\varphi}(i_o(x)))&(\text{definition of}\  \Theta(K_{\varphi}))&\\
=&j_o(j_o^{-1}(\varphi(x)))&(\text{definition of}\  K_{\varphi})\\
=&\varphi(x),
\end{align*}
thus at the level of objects $\Theta(K_{\varphi})=\varphi$.

For every morphism $f:x_1\cdots x_m\rightarrow y_1\cdots y_n\in Mor(\mathcal{D})$, $\varphi(f)$ is a morphism in $U(\mathcal{V})$ which by definition is a prime diagram in $\mathcal{V}$ with domain $\varphi(x_1)\cdots \varphi(x_m)$, codomain $\varphi(y_1)\cdots \varphi(y_n)$ and with the unique real vertex decorated by $\epsilon_{\mathcal{V}}(\varphi(f))\in Mor(\mathcal{V})$. Also by definition $\Theta(K_{\varphi})(f)$ is  a prime diagram in $\mathcal{V}$ with domain $\Theta(K_{\varphi})(x_1)\cdots \Theta(K_{\varphi})(x_m)$ and codomain $\Theta(K_{\varphi})(y_1)\cdots \Theta(K_{\varphi})(y_n)$.   Notice that we have proved that $\Theta(K_{\varphi})=\varphi$ on the level of objects, then the edge-decoration of $\varphi(f)$ and $\Theta(K_{\varphi})(f)$ are equal as prime diagrams in $\mathcal{V}$. Thus to show they are equal as prime diagrams in $\mathcal{V}$, we only need to  show their vertex-decorations are equal. In fact, we have
\begin{align*}
&\Theta(K_{\varphi})(f)&\\
=&[v,\gamma_{K_{\varphi}(i_m(f))}]&(\text{definition of }\Theta(K_{\varphi}))\\
=&[v,\gamma_{K_{\varphi}([v,\gamma_f])}]&(\text{definition of }i_m(f))\\
=&[v,\gamma_{\varepsilon(\varphi_{\sharp}([v,\gamma_f]))}]&(\text{definition of }K_{\varphi})\\
=&[v,\gamma_{\varepsilon([v,\varphi_{\ast}\gamma_f])}]&(\text{definition of }\varphi_{\sharp})\\
=&[v,\gamma_{\varepsilon(\varphi(f))}]&(\text{definition of value of prime diagram})\\
=&\varphi(f).&(\text{property of  prime diagrams in }\mathcal{V})
\end{align*}
Thus we prove that $\Theta(K_{\varphi})=\varphi$.

Now we want to  prove that $\Theta^{-1}(\varphi_K)=K$.

First, for any object $i_o(x_1)\otimes_{F(\mathcal{D})}\cdots \otimes_{F(\mathcal{D})}i_o(x_m)\in Ob(F(\mathcal{D}))$, we have
\begin{align*}
&\Theta^{-1}(\varphi_K)(i_o(x_1)\otimes_{F(\mathcal{D})}\cdots\otimes_{F(\mathcal{D})} i_o(x_m))&\\
=&j_o^{-1}\varphi_K(x_1)\otimes_{\mathcal{V}}\cdots\otimes_{\mathcal{V}}j_o^{-1}\varphi_K(x_m) &(\text{definition of}\  \Theta^{-1}(\varphi_K))&\\
=&j_o^{-1}j_oKi_o(x_1)\otimes_{\mathcal{V}}\cdots\otimes_{\mathcal{V}}j_o^{-1}j_oKi_o(x_m)& (\text{definition of}\  \varphi_K)\\
=&Ki_o(x_1)\otimes_{\mathcal{V}}\cdots\otimes_{\mathcal{V}}Ki_o(x_m)&(j_o^{-1}j_o=id)\\
=&K(i_o(x_1)\otimes_{F(\mathcal{D})}\cdots\otimes_{F(\mathcal{D})} i_o(x_m))&(K\ \text{is a strict tensor fucntor})
\end{align*}
thus $\Theta^{-1}(\varphi_K)=K$ on the level of objects. Secondly to prove they are equal on the level of morphisms, we will  prove that for any prime morphism $[v,\gamma]$ in $F(\mathcal{D})$, that is, for any prime diagram in $\mathcal{D}$, the equation    $$\Theta^{-1}(\varphi_K)([v,\gamma])=K([v,\gamma])$$ is true. In fact, if $[v,\gamma_f]:x_1\cdots x_m\rightarrow y_1\cdots y_n\in F(\mathcal{D})$ is a prime diagram with vertex-decoration $f\in Mor(D)$, we have

\begin{align*}
&\Theta^{-1}(\varphi_K)([v,\gamma_f])&\\
=&\varepsilon((\varphi_K)_{\sharp}([v,\gamma_f])&(\text{definition of}\  \Theta^{-1}(\varphi_K))\\
=&\varepsilon([v,\gamma_{\varepsilon(\varphi_K(f))}])&(\text{definition of }\ (\varphi_K)_{\sharp})\\
=&\varepsilon(\varphi_K(f))&(\text{definition of }\ \varepsilon)\\
=&\varepsilon([v,\gamma_{K(i_m(f))}])&(\text{definition of }\ \varphi_K)\\
=&\varepsilon([v,\gamma_{K([v,\gamma_f)}])&(\text{definition of }\ i_o)\\
=&K([v,\gamma_f]).&(\text{definition of }\ \varepsilon)
\end{align*}
At last, for any general morphism $[\Gamma,\gamma]$, the equation
$$\Theta^{-1}(\varphi_K)[\Gamma,\gamma]=K([\Gamma,\gamma])$$
is direct consequence of the facts  that both $\Theta^{-1}(\varphi_K)$ and $K$ are strict tensor functors and  every diagrams can be decomposed as composition and tensor product of prime diagrams and invertible diagrams.
So we prove that $\Theta$ is a bijection.\\

$(2)$ In this second step, we will show that $\Theta$ is natural both in $\mathcal{D}$s and $\mathcal{V}$s.

$\bullet$ Now we want to show $\Theta$ is natural in $\mathcal{D}$s.

If $\psi:\mathcal{D}_1\rightarrow\mathcal{D}_2$ is a morphism of tensor schemes, we need to show that the diagram $$\xymatrix{Hom(F(\mathcal{D}_1),\mathcal{V})\ar[rr]^{\Theta_{\mathcal{D}_1,\mathcal{V}}}&&Hom(\mathcal{D}_1,U(\mathcal{V}))\\
Hom(F(\mathcal{D}_2),\mathcal{V})\ar[u]^{(F\psi)^{\ast}}\ar[rr]^{\Theta_{\mathcal{D}_2,\mathcal{V}}}&&\ar[u]_{\psi^{\ast}}Hom(\mathcal{D}_2,U(\mathcal{V}))}$$
is commutative, that is,
for every $K\in Hom(F(\mathcal{D}_2),\mathcal{V})$, the equation $$\varphi_K\circ\psi=\varphi_{K\circ F(\psi)}$$ holds in $Hom(\mathcal{D}_1,U(\mathcal{V}))$.

In fact, for every object $x\in Ob(\mathcal{D}_1)$, using the following commutative diagram
$$\xymatrix{Ob(F(\mathcal{D}_1))\ar[r]^{F(\psi)}&Ob(F(\mathcal{D}_2))\\ \ar[u]^{i_o^{\mathcal{D}_1}}Ob(\mathcal{D}_1)\ar[r]^{\psi}&\ar[u]_{i_o^{\mathcal{D}_2}}Ob(\mathcal{D}_2)}$$we have
\begin{align*}
&LHS\\
=&\varphi_K\circ\psi(x)&\\
=&j_o^{\mathcal{V}}K(i_o^{\mathcal{D}_2}\psi(x))&(\text{definition of }\varphi_K)\\
=&j_o^{\mathcal{V}}K( F(\psi)(i_o^{\mathcal{D}_1}x))&( i_o^{\mathcal{D}_2}\psi=F(\psi)i_o^{\mathcal{D}_1})\\
=&j_o^{\mathcal{V}}(K\circ F(\psi))(i_o^{\mathcal{D}_1}x))&(\text{rewriting})\\
=&\varphi_{K\circ F(\psi)}(x)&(\text{definition of }\varphi_{K\circ F(\psi)})\\
=&RHS.
\end{align*}
Thus on the level of objects, $\varphi_K\circ\psi=\varphi_{K\circ F(\psi)}$ holds.

For any morphism $f:x_1\cdots x_m\rightarrow y_1\cdots y_n\in Mor(\mathcal{D}_1)$, using the following commutative diagram $$\xymatrix{Mor(F(\mathcal{D}_1))\ar[r]^{F(\psi)}&Mor(F(\mathcal{D}_2))\\ \ar[u]^{i_m^{\mathcal{D}_1}}Mor(\mathcal{D}_1)\ar[r]^{\psi}&\ar[u]_{i_m^{\mathcal{D}_2}}Mor(\mathcal{D}_2)}$$ we have
\begin{align*}
&LHS\\
=&\varphi_K\circ\psi(f)&\\
=&[v,\gamma_{K(i_m^{\mathcal{D}_2}\psi(f))}]&(\text{definition of }\varphi_K)\\
=&[v,\gamma_{K(F(\psi)i_m^{\mathcal{D}_1}(f))}]&(i_m^{\mathcal{D}_2}\psi=F(\psi)i_m^{\mathcal{D}_1})\\
=&[v,\gamma_{K\circ F(\psi)(i_m^{\mathcal{D}_1}(f))}]&(\text{rewriting })\\
=&\varphi_{K\circ F(\psi)}(f)&(\text{definition of }\varphi_{K\circ F(\psi)})\\
=&RHS.
\end{align*}
Thus on the level of morphisms, $\varphi_K\circ\psi=\varphi_{K\circ F(\psi)}$ holds.

$\bullet$ Now let us show $\Theta$ is natural in  $\mathcal{V}$s.

If $L:\mathcal{V}_1\rightarrow\mathcal{V}_2$ is a morphism of strict tensor categories, we need to show the following diagram
$$\xymatrix{Hom(F(\mathcal{D}),\mathcal{V}_1)\ar[d]_{L_{\ast}}\ar[rr]^{\Theta_{\mathcal{D},\mathcal{V}_1}}&&Hom(\mathcal{D},U(\mathcal{V}_1))\ar[d]^{U(L)_{\ast}}\\
Hom(F(\mathcal{D}),\mathcal{V}_2)\ar[rr]^{\Theta_{\mathcal{D},\mathcal{V}_2}}&&Hom(\mathcal{D},U(\mathcal{V}_2))}$$
is commutative. That is, for every $K\in Hom(F(\mathcal{D}),\mathcal{V}_1)$, the equation $$U(L)\circ\varphi_K=\varphi_{L\circ K}$$ holds in $Hom(\mathcal{D},U(\mathcal{V}_2))$.

First, for every object $x\in Ob(\mathcal{D})$, using the commutative diagram
$$\xymatrix{Ob(U(\mathcal{V}_1))\ar[r]^{U(L)}&Ob(U(\mathcal{V}_2))\\
\ar[u]^{j_o^{\mathcal{V}_1}}Ob(\mathcal{V}_1)\ar[r]^{L}&\ar[u]_{j_o^{\mathcal{V}_2}}Ob(\mathcal{V}_2)}$$
we have
\begin{align*}
&LHS\\
=&U(L)\circ\varphi_K(x)&\\
=&U(L)j_o^{\mathcal{V}_1}K(i_o^{\mathcal{D}}x)&(\text{definition of }\varphi_K)\\
=&j_o^{\mathcal{V}_2}L\circ K(i_o^{\mathcal{D}}x)&(U(L)j_o^{\mathcal{V}_1}=j_o^{\mathcal{V}_2}L)\\
=&j_o^{\mathcal{V}_2}(L\circ K)(i_o^{\mathcal{D}}x)&(\text{rewriting })\\
=&\varphi_{L\circ K}(x)&(\text{definition of }\varphi_{L\circ K})\\
=&RHS.
\end{align*}
Thus on the level of objects, $U(L)\circ\varphi_K=\varphi_{L\circ K}$ holds.

For any morphism  $f:x_1\cdots x_m\rightarrow y_1\cdots y_n\in Mor(\mathcal{D})$, we have
\begin{align*}
&LHS\\
=&U(L)\circ\varphi_K(f)&\\
=&U(L)([v,\gamma_{K(i_m^{\mathcal{D}}f)}])&(\text{definition of }\varphi_K)\\
=&[v,\gamma_{L(K(i_m^{\mathcal{D}}f))}]&(\text{definition of } U(L))\\
=&[v,\gamma_{L\circ K(i_m^{\mathcal{D}}f)}]&(\text{rewriting })\\
=&\varphi_{L\circ K}(f)&(\text{definition of }\varphi_{L\circ K})\\
=&RHS.
\end{align*}
Thus on the level of morphisms, $U(L)\circ\varphi_K=\varphi_{L\circ K}$ holds.
Hence we complete the proof of adjointness of $F,U$.
\end{proof}

\subsection{Two natural isomorphisms }

By the general theory of adjunctions \cite{[Mac71]}, the adjunction $\Theta:F\dashv U$ can be equivalently defined as the quadruple  $(F,U,\varepsilon, \eta)$ with unit $\eta:I\rightarrow UF$ and counit $\varepsilon:FU\rightarrow I$ being natural transformations, such that $\varepsilon_F\circ F(\eta)=1_F$ and $U(\varepsilon)\circ\eta_U=1_U$. We define $\mu=U\varepsilon F:UFUF\rightarrow UF$ and  $\delta=F\eta U:FU\rightarrow FUFU$, and take  $T=UF$, $G=FU$. Then $(T,\mu,\eta)$ defines a monad on $\mathbf{T.Sch}$ and  $(G,\delta,\varepsilon)$ defines a comonad on $\mathbf{Str.T}$.

Now we introduce a functor $$\mathbf{\Gamma}^{\otimes}: \mathbf{Str.T}\rightarrow \mathbf{Str.T}$$  which sends a strict tensor category $\mathcal{V}$ to a strict tensor category $\mathbf{\Gamma}^{\otimes}(\mathcal{V}) = ( \mathsf{\Gamma}(\mathcal{V}),  W(Ob(\mathcal{V})), dom, cod )$ with the set of morphisms being the set of diagrams in $\mathcal{V}$. The tensor product of morphisms is the tensor product of diagrams and composition of morphisms is the composition of diagrams. The tensor product of objects is given by juxtaposition of words.  For a strict tensor functor $K:\mathcal{V}_1\rightarrow \mathcal{V}_2$, the strict tensor functor  $\mathbf{\Gamma}^{\otimes}(K):\mathbf{\Gamma}^{\otimes}(\mathcal{V}_1)\rightarrow \mathbf{\Gamma}^{\otimes}(\mathcal{V}_2)$ is defined as $\mathbf{\Gamma}^{\otimes}(K)_o=\widehat{K_o}:W(Ob(\mathcal{V}_1))\rightarrow W(Ob(\mathcal{V}_2))$, $\mathbf{\Gamma}^{\otimes}(K)_m=K_{\ast}:\mathsf{\Gamma}(\mathcal{V}_1)\rightarrow\mathsf{\Gamma}(\mathcal{V}_2)$.

As a decorated version of example $5.1.1$, we have
\begin{prop}\label{first identification}
There is a natural isomorphism $\widetilde{\zeta}: G\rightarrow \mathbf{\Gamma}^{\otimes}$.
\end{prop}

\begin{proof}
Given any strict tensor functor $\mathcal{V}$,  there is an isomorphism $\widetilde{\zeta}_{\mathcal{V}}: G(\mathcal{V})\rightarrow \mathbf{\Gamma}^{\otimes}(\mathcal{V})$ of strict tensor categories defined as follows:

$\bullet$ for a morphism $[\Gamma,\lambda]$ in $G(\mathcal{V})$ which is a diagram in $U(\mathcal{V})$, the image $(\widetilde{\zeta}_{\mathcal{V}})_m([\Gamma,\lambda])$ is a diagram $[\Gamma,\widetilde{\lambda}]$ in $\mathcal{V}$, such that $\widetilde{\lambda}_o(h)=i_o(\lambda_o(h))$ for any $h\in H(\Gamma)$ and $\widetilde{\lambda}_m(v)=\varepsilon_{\mathcal{V}}(\lambda_m(v))$ for any $v\in V_{re}(\Gamma)$, where $i_o:Ob(U(\mathcal{V}))\rightarrow Ob(G(\mathcal{V}))$ sends an object $x$ of $U(\mathcal{V})$ to the object $\langle x\rangle$ of $G(\mathcal{V})$.

$\bullet$ for an object $x_1\cdots x_m$ of $G(\mathcal{V})$,  $(\widetilde{\zeta}_{\mathcal{V}})_o(x_1\cdots x_m)=x_1\cdots x_m$.

It is easy to see that $\widetilde{\zeta}_{\mathcal{V}}$ is a strict tensor functor and the fact that $\widetilde{\zeta}_{\mathcal{V}}$ is a natural transformation can be directly checked.
\end{proof}
According to this proposition, we will not  make a distinction between the functor $G$ and the functor $\mathbf{\Gamma}^{\otimes}$.

Recall in section $4.2$ we have introduced a functor $\mathbf{\Gamma}_F:\mathbf{T.Sch}\rightarrow \mathbf{T.Sch}$, the following proposition shows that we can identity $\mathbf{\Gamma}_F$ with the functor $T$.

\begin{prop}\label{second identification}
There is a natural isomorphism $\widetilde{\omega}: T\rightarrow \mathbf{\Gamma}_F$.
\end{prop}

\begin{proof}
For any tensor scheme $\mathcal{D}$, we can define an isomorphism $\widetilde{\omega}_{(\mathcal{D})}: T(\mathcal{D})\rightarrow \mathbf{\Gamma}_F(\mathcal{D})$ of tensor schemes as follows:

$\bullet$ similar to example $5.2.1$, let $[\Gamma_v,\lambda]$ be a morphism in $T(\mathcal{D})$ which is a prime diagram in $F(\mathcal{D})$, and assume $In(\Gamma_v)=\{I_1,...,I_m\}$, $Out(\Gamma_v)=\{O_1,...,O_n\}$, $\lambda_o(I_\alpha)=x_{\mu_{0}+\cdots+\mu_{\alpha-1}+1}\cdots x_{\mu_0+\cdots+\mu_{\alpha}}$, and $\lambda_o(O_\beta)=y_{\nu_{0}+\cdots+\nu_{\beta-1}+1}\cdots y_{\nu_0+\cdots+\nu_{\beta}}$, where  $x_{\mu}, y_{\nu}\in Ob(\mathcal{D})$,for $1\leq\mu\leq\mu_1+\cdots+\mu_m$, $1\leq\nu\leq\nu_1+\cdots+\nu_n$ and $\mu_{0}=\nu_{0}=0$, $\mu_{\alpha},\nu_{\beta}\geq 1$ for $1\leq\alpha\leq m$, $1\leq\beta\leq n$. Then the domain of $\lambda(v)$ should be  $x_1\cdots x_{\mu_1+\cdots+\mu_m}$ and the codomain of  $\lambda(v)$ should be $y_1\cdots y_{\nu_1+\cdots+\nu_n}$. Assume $In(\lambda(v))=\{i_1<\cdots<i_{\mu_1+\cdots+\mu_m}\}$ and $Out(\lambda(v))=\{o_1<\cdots<o_{\nu_1+\cdots+\nu_n}\}$, then we define the image $(\widetilde{\omega}_{\mathcal{D}})_m([\Gamma_v,\lambda])$ is the fissus diagram $[\lambda(v),P_{in}, P_{out}]$ in $\mathcal{D}$ with $P_{in}=(i_1<\cdots<i_{\mu_1})<\cdots<(i_{\mu_1+\cdots+\mu_{m-1}+1}<\cdots<i_{\mu_1+\cdots+\mu_m})$ and $P_{out}=(o_1<\cdots<o_{\nu_1})<\cdots<(o_{\nu_1+\cdots+\nu_{n-1}+1}<\cdots<o_{\nu_1+\cdots+\nu_n})$.

$\bullet$ for an object $x_1\cdots x_m$ of $T(\mathcal{D})$,  $(\widetilde{\omega}_{\mathcal{D}})_o(x_1\cdots x_m)=x_1\cdots x_m$.

Due to the fact that $F(\mathcal{D})$ is free, it is easy to see that $\widetilde{\omega}$ is an isomorphism of tensor schemes. The naturality of $\widetilde{\omega}$ can be directly checked.

\end{proof}

\subsection{The unit and counit of tensor calculus }
Here we want to give a detailed analysis of $\eta$ and $\varepsilon$. For a tensor scheme $\mathcal{D}$, the unit $\eta:I\rightarrow UF$ gives rise to a morphism $\eta_{\mathcal{D}}:\mathcal{D}\rightarrow UF(\mathcal{D})$ of tensor schemes. By definition, $\eta_{\mathcal{D}}$ is given by the equation $$\eta_{\mathcal{D}}=\Theta(Id_{F(\mathcal{D})}).$$

Notice that $Ob(UF(\mathcal{D}))=W(Ob(\mathcal{D}))$ is the set of  words in $Ob(\mathcal{D})$, $Mor(UF(\mathcal{D}))=\mathsf{Prim}(F(\mathcal{D}))$ is the set of prime diagrams in $F(\mathcal{D})$.

For any object $x\in Ob(\mathcal{D})$, by definition $$\eta_{\mathcal{D}}(x)=j^{F(\mathcal{D})}_o\circ Id_{F(\mathcal{D})}\circ i^{\mathcal{D}}_o(x)=j^{F(\mathcal{D})}_o\circ i^{\mathcal{D}}_o(x)=\langle x\rangle\in Ob(UF(\mathcal{D}))=W(Ob(\mathcal{D})).$$

For any morphism $f:x_1\cdots x_m\rightarrow y_1\cdots y_n\in Mor(\mathcal{D})$, by definition $$\eta_{\mathcal{D}}(f)=[v,\gamma_{Id_{F(\mathcal{D})}\circ i^{\mathcal{D}}_m(f)}]=[v,\gamma_{i^{\mathcal{D}}_m(f)}]=[v,\gamma_{[v,\gamma_f]}],$$ that is, a prime diagram with domain and codomain  $\langle x_1\rangle\cdots \langle x_m\rangle, \langle y_1\rangle\cdots \langle y_n\rangle\in W(W(Ob(\mathcal{D})))=W(F(\mathcal{D}))$, respectively, and with the unique vertex decorated by the prime diagram $[v,\gamma_f]=i^{\mathcal{D}}_m(f)$.

By proposition \ref{second identification} ,  we have a natural isomorphism $\widetilde{\omega}_{\mathcal{D}}:UF(\mathcal{D})\rightarrow \mathbf{\Gamma}_F(\mathcal{D})$, thus we can identify $[v,\gamma_{[v,\gamma_f]}]$  with a fissus prime diagram in $\mathcal{D}$.

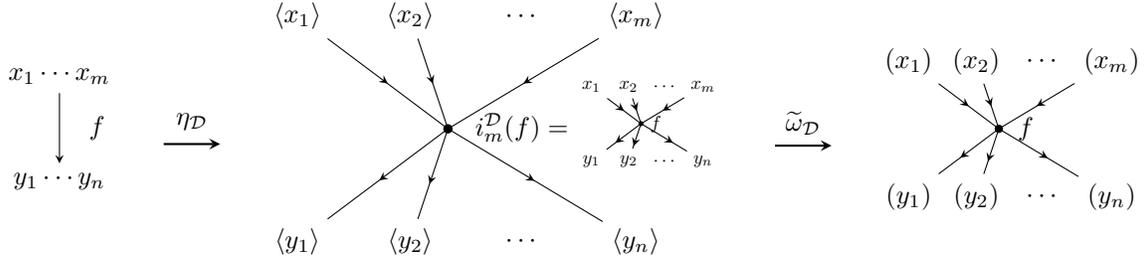
\begin{center}
$
\begin{matrix}

\begin{matrix}
\begin{tikzpicture}
\node (v1) at (0,0.7) {$x_1\cdots x_m$};
\node (v2) at (0,-0.7) {$y_1\cdots y_n$};
\draw [->,>=stealth] (v1) -- (v2);
\node at (0.5,0) {$f$};
\end{tikzpicture}
\end{matrix}
&
\begin{matrix}
\begin{tikzpicture}
\node (v1) at (-1,0.5) {};
\node (v2) at (0,0.5) {};
\draw [thick,->,>=stealth] (v1) -- (v2);
\node at (-0.5,0.8) {$\eta_{\mathcal{D}}$};
\end{tikzpicture}
\end{matrix}
&
\begin{matrix}
\begin{tikzpicture}[scale=1]
\node (v1) at (0,0) {};
\draw[fill] (0,0) circle [radius=0.055];
\node at (2,0){$\begin{matrix}i^{\mathcal{D}}_m(f)=\begin{matrix}\begin{tikzpicture}[scale=.33]
\node (v1) at (0,0) {};
\draw[fill] (0,0) circle [radius=0.08];
\node [scale=.7] at (0.6,0){$f$};
\node [scale=.7] (v2) at (-2,1.5) {$x_1$};
\node [scale=.7](v3) at (-0.5,1.5) {$x_2$};
\node [scale=.7](v4) at (1,1.5) {$\cdots$};
\node [scale=.7](v5) at (2.5,1.5) {$x_m$};
\node [scale=.7](v6) at (-2,-1.5) {$y_1$};
\node [scale=.7](v7) at (-0.5,-1.5) {$y_2$};
\node [scale=.7] at (1,-1.5) {$\cdots$};
\node [scale=.7](v8) at (2.5,-1.5) {$y_n$};
\draw  (v2) -- (0,0)[postaction={decorate, decoration={markings,mark=at position .50 with {\arrow[black]{stealth}}}}];
\draw  (v3) -- (0,0)[postaction={decorate, decoration={markings,mark=at position .50 with {\arrow[black]{stealth}}}}];
\draw  (v5) -- (0,0)[postaction={decorate, decoration={markings,mark=at position .50 with {\arrow[black]{stealth}}}}];
\draw  (v6) -- (0,0)[postaction={decorate, decoration={markings,mark=at position .50 with {\arrowreversed[black]{stealth}}}}];
\draw  (v7) -- (0,0)[postaction={decorate, decoration={markings,mark=at position .50 with {\arrowreversed[black]{stealth}}}}];
\draw  (v8) -- (0,0)[postaction={decorate, decoration={markings,mark=at position .50 with {\arrowreversed[black]{stealth}}}}];
\end{tikzpicture}\end{matrix}\end{matrix}$};
\node (v2) at (-2,1.5) {$\langle x_1\rangle$};
\node (v3) at (-0.5,1.5) {$\langle x_2\rangle$};
\node (v4) at (1,1.5) {$\cdots$};
\node (v5) at (2.5,1.5) {$\langle x_m\rangle$};
\node (v6) at (-2,-1.5) {$\langle y_1\rangle$};
\node (v7) at (-0.5,-1.5) {$\langle y_2\rangle$};
\node at (1,-1.5) {$\cdots$};
\node (v8) at (2.5,-1.5) {$\langle y_n\rangle$};
\draw  (v2) -- (0,0)[postaction={decorate, decoration={markings,mark=at position .50 with {\arrow[black]{stealth}}}}];
\draw  (v3) -- (0,0)[postaction={decorate, decoration={markings,mark=at position .50 with {\arrow[black]{stealth}}}}];
\draw  (v5) -- (0,0)[postaction={decorate, decoration={markings,mark=at position .50 with {\arrow[black]{stealth}}}}];
\draw  (v6) -- (0,0)[postaction={decorate, decoration={markings,mark=at position .50 with {\arrowreversed[black]{stealth}}}}];
\draw  (v7) -- (0,0)[postaction={decorate, decoration={markings,mark=at position .50 with {\arrowreversed[black]{stealth}}}}];
\draw  (v8) -- (0,0)[postaction={decorate, decoration={markings,mark=at position .50 with {\arrowreversed[black]{stealth}}}}];
\end{tikzpicture}
\end{matrix}
&
\begin{matrix}
\begin{tikzpicture}
\node (v1) at (-1,0.5) {};
\node (v2) at (0,0.5) {};
\draw [thick,->,>=stealth] (v1) -- (v2);
\node at (-0.5,0.8) {$\widetilde{\omega}_{\mathcal{D}}$};
\end{tikzpicture}
\end{matrix}
&
\begin{matrix}
\begin{tikzpicture}[scale=.6]
\node (v1) at (0,0) {};
\draw[fill] (0,0) circle [radius=0.08];
\node at (0.6,0){$f$};
\node (v2) at (-2,1.5) {$(x_1)$};
\node (v3) at (-0.5,1.5) {$(x_2)$};
\node (v4) at (1,1.5) {$\cdots$};
\node (v5) at (2.5,1.5) {$(x_m)$};
\node (v6) at (-2,-1.5) {$(y_1)$};
\node (v7) at (-0.5,-1.5) {$(y_2)$};
\node at (1,-1.5) {$\cdots$};
\node (v8) at (2.5,-1.5) {$(y_n)$};
\draw  (v2) -- (0,0)[postaction={decorate, decoration={markings,mark=at position .50 with {\arrow[black]{stealth}}}}];
\draw  (v3) -- (0,0)[postaction={decorate, decoration={markings,mark=at position .50 with {\arrow[black]{stealth}}}}];
\draw  (v5) -- (0,0)[postaction={decorate, decoration={markings,mark=at position .50 with {\arrow[black]{stealth}}}}];
\draw  (v6) -- (0,0)[postaction={decorate, decoration={markings,mark=at position .50 with {\arrowreversed[black]{stealth}}}}];
\draw  (v7) -- (0,0)[postaction={decorate, decoration={markings,mark=at position .50 with {\arrowreversed[black]{stealth}}}}];
\draw  (v8) -- (0,0)[postaction={decorate, decoration={markings,mark=at position .50 with {\arrowreversed[black]{stealth}}}}];
\end{tikzpicture}

\end{matrix}
\end{matrix}
$
\end{center}

We will not make a distinction between the natural transformation $\eta:I_{\mathbf{T.sch}}\rightarrow UF$ and the natural transformation $\widetilde{\eta}=\eta\circ\widetilde{\omega}:I_{\mathbf{T.sch}}\rightarrow\mathbf{\Gamma}_F$, and call both of  them \textbf{fission transformations}.

$$\xymatrix{&UF\ar[dd]^{\widetilde{\omega}}\\I_{\mathbf{T.sch}}\ar[ur]^{\eta}\ar[dr]_{\widetilde{\eta}}&\\&\mathbf{\Gamma}_F}$$

\begin{ex}
Take $\mathcal{D}=U(\mathbf{\Gamma}^\otimes)$. As example $5.2.1$ shows, $Ob(U(\mathbf{\Gamma}^\otimes))=W(\{x\})$ and $Mor(U(\mathbf{\Gamma}^\otimes))\cong\Gamma_F$ under  $\widetilde{\omega}_{U(\mathbf{\Gamma}^\otimes)}$. The map $(\eta_{U(\mathbf{\Gamma}^\otimes)})_o:Ob(U(\mathbf{\Gamma}^\otimes))\rightarrow Ob(UFU(\mathbf{\Gamma}^\otimes))$ sends a word $\overset{m\ times}{\overbrace{x\cdots x}}$ to a word of word $\langle\overset{m\ times}{\overbrace{x\cdots x}}\rangle$. Let $(\Gamma,P_{in},P_{out})$ be a fissus planar graph $P_{in}=(i_1<\cdots<i_{\mu_1})<\cdots<(i_{\mu_1+\cdots+\mu_{m-1}+1}<\cdots<i_{\mu_1+\cdots+\mu_m})$ and $P_{in}=(o_1<\cdots<o_{\nu_1})<\cdots<(o_{\nu_1+\cdots+\nu_{n-1}+1}<\cdots<o_{\nu_1+\cdots+\nu_n})$, then the map $(\eta_{U(\mathbf{\Gamma}^\otimes)})_m:Mor(U(\mathbf{\Gamma}^\otimes))\rightarrow Mor(UFU(\mathbf{\Gamma}^\otimes))$ sends $(\Gamma,P_{in},P_{out})$ to a fissus compound fissus planar graph $[\Gamma_v,P_{in}',P_{out}',\lambda]\in \mathsf{\Gamma}_{F}(\mathbf{\Gamma}_F)$ such that

$\bullet$  $P_{in}'=(1)<\cdots<(m)$ and $P_{out}'=(1)<\cdots<(n)$ are the finest linear partitions,

$\bullet$  $\lambda(v)=(\Gamma,P_{in},P_{out})$, where $v$ is the unique vertex of $\Gamma_v$.

$\bullet$  $\Gamma_v$ is isomorphic to the coarse-graining $\underbrace{(\Gamma,P_{in},P_{out})}$,

$\bullet$  $\lambda(I_\alpha)=\overset{\mu_\alpha}{\overbrace{x\cdots x}}$, $(1\leq \alpha\leq m)$ and $\lambda(O_\beta)=\overset{\nu_\beta}{\overbrace{x\cdots x}}$, $(1\leq \beta\leq n)$ where $I_\alpha,O_\beta$ are inputs and outputs of $\Gamma_v$, respectively,
\end{ex}

Now let us  study the counit $\varepsilon:FU\rightarrow I$.
For a strict tensor category $\mathcal{V}$, let $\mathcal{D}=U(\mathcal{V})$ be the tensor scheme associated to $\mathcal{V}$. Then by definition $\varepsilon_{\mathcal{V}}$ is  the strict tensor functor given by the equation $$\varepsilon_{\mathcal{V}}=\Theta^{-1}(Id_{U(\mathcal{V})}):F(U(\mathcal{V}))\rightarrow \mathcal{V}.$$

We define the natural injection $\iota_o=i^{U(\mathcal{V})}_o\circ j^{\mathcal{V}}_o:Ob(\mathcal{V})\rightarrow Ob(F(U(\mathcal{V})))$ which sends $x\in Ob(\mathcal{V})$ to $\langle x\rangle\in Ob(FU(\mathcal{V}))$.
$$\xymatrix{Ob(\mathcal{V})\ar[r]^{j^{\mathcal{V}}_o}\ar[dr]^{\iota_o}&\ar[d]^{i^{U(\mathcal{V})}_o}Ob(U(\mathcal{V}))\\&Ob(F(U(\mathcal{V})))}$$

Every non-unit object $w\in Ob(F(U(\mathcal{V})))$ can be uniquely written as $w=\iota_0(x_1)\otimes\cdots \otimes\iota_0(x_m)=\langle x_1\rangle\otimes\cdots\otimes\langle x_m\rangle=x_1\cdots x_m$ for some $x_1,\cdots, x_m\in Ob(\mathcal{V})$. Then by definition
\begin{align*}
&\varepsilon_{\mathcal{V}}w\\
=&\varepsilon_{\mathcal{V}}(\iota_o(x_1)\otimes\cdots \otimes\iota_o(x_m))\\
=&(j_o^{\mathcal{V}})^{-1}\circ Id_{U(\mathcal{V})}\circ j^{\mathcal{V}}_o(x_1)\otimes_{\mathcal{V}}\cdots \otimes_{\mathcal{V}}(j_o^{\mathcal{V}})^{-1}\circ Id_{U(\mathcal{V})}\circ j^{\mathcal{V}}_o(x_m)\\
=&x_1\otimes_{\mathcal{V}}\cdots\otimes_{\mathcal{V}}x_m .
\end{align*}

For any morphism $[\Gamma,\lambda]:\iota_o(x_1)\otimes\cdots \otimes\iota_o(x_m)\rightarrow \iota_o(y_1)\otimes\cdots \otimes \iota_o(y_n)$ in $F(U(\mathcal{V})),$ by definition $$\varepsilon_{\mathcal{V}}([\Gamma,\lambda])=\varepsilon((Id_{U(\mathcal{V})})_{\sharp}([\Gamma,\lambda]).$$

By proposition \ref{first identification}, we have a natural isomorphism $\widetilde{\zeta}_{\mathcal{V}}:FU(\mathcal{V})\rightarrow \mathbf{\Gamma}^{\otimes}(\mathcal{V})$. From the definition of $(Id_{U(\mathcal{V})})_{\sharp}$, we see that $(Id_{U(\mathcal{V})})_{\sharp}([\Gamma,\lambda]) $ as a diagram in $\mathcal{V}$ is equal to $(\widetilde{\zeta}_{\mathcal{V}})_m([\Gamma,\lambda])$, thus under the identification of $FU(\mathcal{V})$ and $\mathbf{\Gamma}^{\otimes}(\mathcal{V})$  the counit $\varepsilon_{\mathcal{V}}$ can be identified with the evaluation map $\varepsilon$ of $\mathcal{V}$, which is the reason why we use the same notation.  We call both the natural transformation $\zeta:FU\rightarrow I_{\mathbf{Str.T}}$ and the natural transformation  $\widetilde{\varepsilon}=\varepsilon\circ\widetilde{\zeta}^{-1}:\mathbf{\Gamma}^{\otimes}\rightarrow I_{\mathbf{Str.T}}$ \textbf{evaluation transformations} and will not make a distinction between them.

$$\xymatrix{\mathbf{\Gamma}^{\otimes}\ar[dd]_{\widetilde{\zeta}^{-1}}\ar[dr]_{\widetilde{\varepsilon}}&\\&I_{\mathbf{Str.T}}\\FU\ar[ur]_{\varepsilon}&}$$

\subsection{The multiplication and comultiplication}

Recall that the multiplication $\mu:T\circ T\rightarrow T$ is given by $\mu=U\varepsilon F:UFUF\rightarrow UF$.
For any tensor scheme $\mathcal{D}$, let us consider the map $\varepsilon_{F(\mathcal{D})}:FU(F(\mathcal{D}))\rightarrow F(\mathcal{D})$.  By definition, for an object $(x_1\cdot\cdot\cdot x_{i_1})\cdot\cdot\cdot(x_{i_1+\cdot\cdot\cdot+i_{m-1}+1}\cdots x_{i_1+\cdot\cdot\cdot+i_{m}})$ in $Ob(FUF(\mathcal{D}))$ which is a word of words in $Ob(\mathcal{D})$,

$$(\varepsilon_{F(\mathcal{D})})_o((x_1\cdot\cdot\cdot x_{i_1})\cdot\cdot\cdot(x_{i_1+\cdot\cdot\cdot+i_{m-1}+1}\cdots x_{i_1+\cdot\cdot\cdot+i_{m}}))=x_1\cdots x_{i_1+\cdot\cdot\cdot+i_{m}},$$
which just forgets the brackets.  That is, $(\varepsilon_{F(\mathcal{D})})_o$ is given by juxtaposition of words.

By proposition \ref{first identification}, we have an natural isomorphism $\widetilde{\zeta}_{F(\mathcal{D})}:FU(F(\mathcal{D}))\overset{\sim}\rightarrow \mathbf{\Gamma}^{\otimes}(F(\mathcal{D}))$. Under this identification, a morphism $[\Gamma,\gamma]$ in $FU(F(\mathcal{D}))$ is a diagram in $F(\mathcal{D})$. Thus $(\widetilde{\varepsilon}_{F(\mathcal{D})})_m([\Gamma,\gamma])$ is given exactly by the evaluation map of diagrams in $F(\mathcal{D})$. In fact, the map $\widetilde{\varepsilon}_{F(\mathcal{D})}:\mathbf{\Gamma}^{\otimes}(F(\mathcal{D}))\rightarrow F(\mathcal{D})$ is a strict tensor functor. So by definition, $$\mu_{F(\mathcal{D})}=U(\widetilde{\varepsilon}_{F(\mathcal{D})}).$$
That is, for an object $(x_1\cdot\cdot\cdot x_{i_1})\cdot\cdot\cdot(x_{i_1+\cdot\cdot\cdot+i_{m-1}+1}\cdots x_{i_1+\cdot\cdot\cdot+i_{m}})$ in $Ob(UFUF(\mathcal{D}))$,
$$(\mu_{F(\mathcal{D})})_o((x_1\cdot\cdot\cdot x_{i_1})\cdot\cdot\cdot(x_{i_1+\cdot\cdot\cdot+i_{m-1}+1}\cdots x_{i_1+\cdot\cdot\cdot+i_{m}}))=x_1\cdots x_{i_1+\cdot\cdot\cdot+i_{m}};$$

For a morphism $[\Gamma_v,\lambda]$ in $U\mathbf{\Gamma}^{\otimes}(F(\mathcal{D}))$ which is a prime diagram in $\mathbf{\Gamma}^{\otimes}(F(\mathcal{D}))$, $(\mu_{F(\mathcal{D})})_m([\Gamma_v,\lambda])$ is a prime diagram $[\Gamma_v,\gamma_{\widetilde{\varepsilon}_{F(\mathcal{D})}(\lambda(v))}]$ in $F(\mathcal{D})$ with its unique  vertex $v$ decorated by $\widetilde{\varepsilon}_{F(\mathcal{D})}(\lambda(v))$. For each half-edge $h\in v$, $$\gamma_o(h)=(\mu_{F(\mathcal{D})})_o\circ \lambda_o(h).$$

On the other hand, under the natural identification $\widetilde{\omega}_{\mathcal{D}}:UF(\mathcal{D})\overset{\sim}\rightarrow \mathbf{\Gamma}_F(\mathcal{D})$ in proposition \ref{second identification} , we can view a morphism of $UF(\mathcal{D})$ as a fissus diagram in $\mathcal{D}$. Then on the level of morphisms, the multiplication can be presented  by $\sigma\circ \widehat{Z}_F:\Gamma_F(\mathbf{\Gamma}_F)\rightarrow \Gamma_F.$ In summary, we have the following commutative diagram:

$$\xymatrix{\mathbf{\Gamma}_F\circ\mathbf{\Gamma}_F\ar[rr]^{(\sigma \widehat{Z}_F,\mu_o)}&&\mathbf{\Gamma}_F\\ \ar[u]^{\widetilde{\omega}^2}UFUF\ar[d]_{U\widetilde{\zeta}F}\ar[rr]^{\mu}&&UF\ar@{=}[d]\ar[u]_{\widetilde{\omega}}\\U\mathbf{\Gamma}^{\otimes}F\ar[rr]^{U\widetilde{\varepsilon} F}&&UF.}$$

The comultiplication $\delta:G\rightarrow G\circ G$ is given by $\delta=F\eta U:FU\rightarrow FUFU.$  For any strict tensor category $\mathcal{V}$, let us consider the strict tensor functor $\delta_{\mathcal{V}}:FU(\mathcal{V})\rightarrow FUFU(\mathcal{V}).$  First of all, for any object $x_1\cdots x_m$ of $FU(\mathcal{V})$ with each $x_i\in Ob(\mathcal{V})$, $1\leq i\leq m$, $$(\delta_{\mathcal{V}})_o(x_1\cdots x_m)=\langle x_1\cdots x_m\rangle.$$

Secondly, let consider the morphism $\widetilde{\eta}_{U(\mathcal{V})}:U(\mathcal{V})\rightarrow \mathbf{\Gamma}_F(U(\mathcal{V}))$. For a morphism $[\Gamma_v,\gamma_f]$ in $U(\mathcal{V})$ which is a prime diagram in $\mathcal{V}$, $\widetilde{\eta}_{U(\mathcal{V})}([\Gamma_v,\gamma_f])=([\Gamma_v, \lambda],P_{in}^{full},P_{out}^{full})$ with $\lambda_m(v)=[\Gamma_v,\gamma_f]$ and $\lambda_o(h)=\gamma_o(h)$ for any $h\in v.$  So for any morphism $[\Gamma,\lambda]$ in $FU(\mathcal{V})$, $F(\widetilde{\eta}_{U(\mathcal{V})})$ maps it to be a diagram $[\Gamma,\widetilde{\lambda}]\in \mathsf{\Gamma}(\mathbf{\Gamma}_F(U({\mathcal{V}})))$ such that $\widetilde{\lambda}_m(v)=([\Gamma_v, \gamma_{\lambda_m(v)}],P_{in}^{full},P_{out}^{full})$ for any $v\in V_{re}(\Gamma)$ and $\widetilde{\lambda}_o(h)=\langle\lambda_o(h)\rangle$ for any $h\in H(\Gamma).$

On the other hand, by proposition \ref{first identification} we have a natural isomorphism $\widetilde{\zeta}_{\mathcal{V}}:FU(\mathcal{V})\rightarrow \mathbf{\Gamma}^{\otimes}(\mathcal{V})$. Now we define a functor $\chi_{\mathcal{V}}:\mathbf{\Gamma}^{\otimes}(\mathcal{V})\rightarrow \mathbf{\Gamma}^{\otimes}\circ \mathbf{\Gamma}^{\otimes}(\mathcal{V})$ as follows:

$\bullet$ for any object $x_1\cdots x_m$ of $\mathbf{\Gamma}^{\otimes}(\mathcal{V})$,  $(\chi_{\mathcal{V}})_o(x_1\cdots x_m)=\langle x_1\cdots x_m \rangle. $

$\bullet$ for any morphism $[\Gamma,\lambda]$ of $\mathbf{\Gamma}^{\otimes}(\mathcal{V})$,  $(\chi_{\mathcal{V}})_m([\Gamma,\lambda])=[\Gamma,\widetilde{\lambda}]$ with $\widetilde{\lambda}_m(v)=[\Gamma_v,\gamma_{\lambda(v)}]$ for any $v\in V_{re}(\Gamma)$ and $\widetilde{\lambda}_o(h)=\langle\lambda_o(h)\rangle$ for any $h\in H(\Gamma).$

In summary, we have the following commutative diagram which can be carefully checked.

$$\xymatrix{FU\ar@{=}[d]\ar[rr]^{F\widetilde{\eta}U}&&F\mathbf{\Gamma}_FU\\\ar[d]_{\widetilde{\zeta}}FU\ar[rr]^{\delta}&&\ar[u]_{F\widetilde{\omega}U}\ar[d]^{\widetilde{\zeta}^2}FUFU\\\mathbf{\Gamma}^{\otimes}\ar[rr]^{\chi}&&\mathbf{\Gamma}^{\otimes}\circ \mathbf{\Gamma}^{\otimes}.}$$

\section{Algebras of tensor calculus}

\subsection{Tensor manifold}
\begin{defn}
An algebra of $(T,\mu,\eta)$ is a tensor scheme $\mathcal{D}$ equipped with a morphism $\epsilon: T(\mathcal{D})\rightarrow \mathcal{D}$ of tensor schemes such that
$$
\begin{matrix}
\xymatrix{T\circ T(\mathcal{D})\ar[r]^{T(\epsilon)}\ar[d]_{\mu_{\mathcal{D}}}&T(\mathcal{D})\ar[d]^{\epsilon}\\T(\mathcal{D})\ar[r]^{\epsilon}&\mathcal{D}}
&
\xymatrix{\mathcal{D}\ar[r]^{\eta_{\mathcal{D}}}\ar@{=}[dr]&T(\mathcal{D})\ar[d]^{\epsilon}\\&\mathcal{D}}
\end{matrix}
$$
The morphism $\epsilon$ is called the structure map of algebra $(\mathcal{D},\epsilon)$. We also  call $(\mathcal{D},\epsilon)$ a \textbf{tensor manifold} and view it  as a categorical noncommutative space.
\end{defn}

Next we will give some examples of tensor manifolds.
\begin{ex}
To any tensor scheme $\mathcal{D}$, we can associate a free tensor manifold $(T(\mathcal{D}),\mu_{\mathcal{D}})$ functorially with

$\bullet$ $Ob(T(\mathcal{D}))=W(Ob(\mathcal{D}))$ being the set of words in $Ob(\mathcal{D})$;

$\bullet$ $Mor(T(\mathcal{D}))=\mathsf{\Gamma}_F(\mathcal{D})$ being the set of fissus planar diagrams in $\mathcal{D}$;

$\bullet$  $s=dom$ and $t=cod$ which send each fissus planar diagram to  words in $Ob(T(\mathcal{D}))$;

$\bullet$ the map $(\mu_{\mathcal{D}})_o: Ob(T\circ T(\mathcal{D}))\rightarrow Ob(T(\mathcal{D}))$ sends a word $(x_1\cdot\cdot\cdot x_{i_1})\cdot\cdot\cdot(x_{i_1+\cdot\cdot\cdot+i_{m-1}+1}\cdots x_{i_1+\cdot\cdot\cdot+i_{m}})$ in $W(Ob(\mathcal{D}))$ to the word $x_1\cdot\cdot\cdot x_{i_1}\cdot\cdot\cdot x_{i_1+\cdot\cdot\cdot+i_{m-1}+1}\cdots x_{i_1+\cdot\cdot\cdot+i_{m}}$ in $Ob(\mathcal{D})$;

$\bullet$ the map $(\mu_{\mathcal{D}})_m: Mor(T\circ T(\mathcal{D}))\rightarrow Mor(T(\mathcal{D}))$ is equal to the map $\sigma\circ \widehat{Z}_F:\mathsf{\Gamma}_{F}(\mathsf{\Gamma}_F(\mathcal{D})\rightarrow \mathsf{\Gamma}_F(\mathcal{\mathcal{D}})$, i.e, evaluation map of compound fissus planar diagrams.
\end{ex}

\begin{ex}
We define a tensor manifold $(\mathbf{Prim},\varepsilon^{C-G})$ as follows:

$\bullet$ Set $Ob(\textbf{Prim})=\{x\}$, $Mor(\textbf{Prim})=\mathsf{Prim}$ be the set of isomorphic classes of prime planar graphs.

$\bullet$ For a prime $(m,n)$-planar graph, the source map $s$ sends it to the string of $x$ with length $m$ and the target map $t$ sends it to the string of $x$'s with length $n$.

$\bullet$ On the level of objects, the structure map $\varepsilon_o^{C-G}$ sends any  string of $x$ to $x$. On the level of morphism, the structure map $\varepsilon_m^{C-G}$ is equal to the coarse-graining of fissus planar planar graphs.

The fact that $(\mathsf{Prim},\{x\},s, t, \varepsilon_o^{C-G},\varepsilon_m^{C-G} )$ defines a tensor manifold can be directly checked.
\end{ex}

\begin{ex}
We define a tensor manifold $(\mathbf{\Gamma},\varepsilon^{C-G})$ as follows:

$\bullet$ Set $Ob(\mathbf{\Gamma})=\{x\}$, $Mor(\mathbf{\Gamma})=\mathsf{\Gamma}$ be the set of isomorphic classes of  planar graphs.

$\bullet$ For a $(m,n)$-planar graph, the source map $s$ sends it to the string of $x$'s with length $m$ and the target map $t$ sends it to the string of $x$'s with length $n$.

$\bullet$ On the level of objects, the structure map $\varepsilon_o^{C-G}$ sends any  string of $x$'s to $x$. On the level of morphism, the structure map $\varepsilon_m^{C-G}$ is equal to the coarse-graining of fissus planar planar graphs.

The fact that $(\mathsf{\Gamma},\{x\}, s, t, \varepsilon_o^{C-G},\varepsilon_m^{C-G} )$ defines a tensor manifold can be directly checked.
\end{ex}

\begin{ex}\label{U(V)}
For any  small strict tensor category $\mathcal{V}$, we can associate a tensor manifold $(U(\mathcal{V}),\epsilon)$ in the following way.

$\bullet$ Set $Ob(U(\mathcal{V}))=Ob(\mathcal{V})$, $Mor(U(\mathcal{V}))=\mathsf{Prim}(\mathcal{V})$,  $s=dom$ and $t=cod$.

$\bullet$ For any object $x_1\cdot\cdot\cdot x_m\in Ob(T(U(\mathcal{V})))$, $$\epsilon_o(x_1\cdot\cdot\cdot x_m)=x_1\otimes\cdot\cdot\cdot\otimes x_m\in Ob(\mathcal{V}).$$

$\bullet$ For any morphism $(\overrightarrow{\Gamma},\prec,P_{in}, P_{out}, \gamma)\in Mor(T(U(\mathcal{V})))$ with bracketed domain $(x_1\cdot\cdot\cdot x_{i_1})\cdot\cdot\cdot(x_{i_1+\cdot\cdot\cdot+i_{m-1}+1}\cdots x_{i_1+\cdot\cdot\cdot+i_{m}})$ and codomain $(y_1\cdot\cdot\cdot y_{j_1})\cdot\cdot\cdot(y_{j_1+\cdot\cdot\cdot+j_{n-1}+1}\cdots y_{j_1+\cdot\cdot\cdot+j_{n}})$, $$\epsilon_m((\overrightarrow{\Gamma},\prec,P_{in}, P_{out}, \gamma))\in Mor(U(\mathcal{V}))$$ is the coase-graining of $(\overrightarrow{\Gamma},\prec,P_{in}, P_{out}, \gamma)$, that is,
 a prime $(m,n)$-planar diagram in $\mathcal{V}$ with domain $(x_1\otimes \cdot\cdot\cdot\otimes x_{i_1})\cdot\cdot\cdot(x_{i_1+\cdot\cdot\cdot+i_{m-1}+1}\otimes\cdots\otimes x_{i_1+\cdot\cdot\cdot+i_{m}}),$
codomain $(y_1\otimes\cdot\cdot\cdot \otimes y_{j_1})\cdot\cdot\cdot(y_{j_1+\cdot\cdot\cdot+j_{n-1}+1}\otimes\cdots \otimes y_{j_1+\cdot\cdot\cdot+j_{n}})$
and its unique vertex decorated by the value of $[\Gamma,\gamma]$  $$\varepsilon_{\mathcal{V}}([\Gamma,\gamma]):x_1\otimes \cdots\otimes x_{i_1+\cdot\cdot\cdot+i_{m}}\rightarrow y_1\otimes \cdots\otimes y_{j_1+\cdot\cdot\cdot+j_{n}}.$$

The axioms of a tensor manifold can be directly checked.
\end{ex}

\begin{ex}
If we only replace $\mathsf{Prim}(\mathcal{V})$ by $\mathsf{\Gamma}(\mathcal{V})$ in the definition of $(U(\mathcal{V}),\epsilon)$ above, we get a new tensor manifold $(\mathbf{\Gamma}(\mathcal{V}),\epsilon)$.  In fact, for any subset $\mathsf{Prim}(\mathcal{V})\subseteq \mathsf{S}(\mathcal{V}) \subseteq \mathsf{\Gamma}(\mathcal{V})$, we can get a tensor manifold $(\mathbf{S}(\mathcal{V}),\epsilon)$ defined in the same way as  $(U(\mathcal{V}),\epsilon)$ and $(\mathbf{\Gamma}(\mathcal{V}),\epsilon)$.
\end{ex}

\begin{defn}
Let $(\mathcal{D}_1,\epsilon_{1})$ and $(\mathcal{D}_2,\epsilon_{2})$ be two algebras over $(T,\mu,\eta)$, a morphism of them is a morphism $\varphi:\mathcal{D}_1\rightarrow \mathcal{D}_2$ of tensor schemes such that
$$
\xymatrix{T(\mathcal{D}_1)\ar[r]^{\epsilon_1}\ar[d]_{T(\varphi)}&\ar[d]^{\varphi}\mathcal{D}_1\\T(\mathcal{D}_1)\ar[r]^{\epsilon_2}&\mathcal{D}_2}
$$
\end{defn}

All $T$-algebras and their morphisms form a category denoted by $\mathbf{T.Sch}^{T}$ which is called Eilenberg-Moore category of $T$.

The following lemma shows some properties of tensor manifolds.
\begin{lem}
Let $(\mathcal{D},\epsilon)$ be an algebra of $(T,\mu,\eta)$, then

$(1)$   $(Ob(\mathcal{D}),\epsilon_o)$ is a monoid;

$(2)$   for each fissus prime planar diagram $(\overrightarrow{\Gamma},\prec,P_{in}, P_{out}, \gamma)\in Mor(T(\mathcal{D}))$ with domain $(x_1)(x_2)\cdots (x_m)$, codomain $(y_1)(y_2)\cdots (y_n)$ and its unique real vertex decorated by $f:x_1\cdots x_m\rightarrow y_1\cdots y_n$, $\epsilon_m((\overrightarrow{\Gamma},\prec,P_{in}, P_{out}, \gamma))=f;$
\begin{center}
$
\begin{matrix}
\begin{matrix}

\begin{tikzpicture}[scale=.6]
\node (v1) at (0,0) {};
\draw[fill] (0,0) circle [radius=0.08];
\node at (0.6,0){$f$};
\node (v2) at (-2,1.5) {$(x_1)$};
\node (v3) at (-0.5,1.5) {$(x_2)$};
\node (v4) at (1,1.5) {$\cdots$};
\node (v5) at (2.5,1.5) {$(x_m)$};
\node (v6) at (-2,-1.5) {$(y_1)$};
\node (v7) at (-0.5,-1.5) {$(y_2)$};
\node at (1,-1.5) {$\cdots$};
\node (v8) at (2.5,-1.5) {$(y_n)$};
\draw  (v2) -- (0,0)[postaction={decorate, decoration={markings,mark=at position .50 with {\arrow[black]{stealth}}}}];
\draw  (v3) -- (0,0)[postaction={decorate, decoration={markings,mark=at position .50 with {\arrow[black]{stealth}}}}];
\draw  (v5) -- (0,0)[postaction={decorate, decoration={markings,mark=at position .50 with {\arrow[black]{stealth}}}}];
\draw  (v6) -- (0,0)[postaction={decorate, decoration={markings,mark=at position .50 with {\arrowreversed[black]{stealth}}}}];
\draw  (v7) -- (0,0)[postaction={decorate, decoration={markings,mark=at position .50 with {\arrowreversed[black]{stealth}}}}];
\draw  (v8) -- (0,0)[postaction={decorate, decoration={markings,mark=at position .50 with {\arrowreversed[black]{stealth}}}}];
\end{tikzpicture}

\end{matrix}
&\begin{matrix}
\begin{tikzpicture}
\node (v1) at (-1,0.5) {};
\node (v2) at (0,0.5) {};
\draw [thick,->,>=stealth] (v1) -- (v2);
\node at (-0.5,0.8) {$\epsilon_m$};
\end{tikzpicture}
\end{matrix}&
\begin{matrix}
\begin{tikzpicture}
\node (v1) at (0,0.7) {$x_1\cdots x_m$};
\node (v2) at (0,-0.7) {$y_1\cdots y_n$};
\draw [->,>=stealth] (v1) -- (v2);
\node at (0.5,0) {$f$};
\end{tikzpicture}
\end{matrix}
\end{matrix}
$
\end{center}

$(3)$  for any $x\in Ob(\mathcal{D})$, the space of morphisms $Mor_{\mathcal{D}}(x,x)$ is not an empty set;

$(4)$  for $m,n \geq 1$, the set  $Mor(\mathcal{D})(m,n)=\{f\in Mor(\mathcal{D})| length(s(f))=m, length(t(f))=n \}$ is not an empty set.

\end{lem}
\begin{proof}
$\bullet$ The fact that $(Ob(\mathcal{D}),\epsilon_o)$ is a monoid is evident from the definition of a tensor manifold and we call the object $\epsilon_o(\varnothing)$ the \textbf{unit object} of $(\mathcal{D},\epsilon)$, where $\varnothing$ is the null string. We denote the unit object by $1_{\mathcal{D}}.$

$\bullet$  The second fact is a direct consequence of the fact  $\eta\circ \epsilon=Id$ in the definition of a $T$-algebra.

$\bullet$ To prove the third statement, notice that for every object $x\in Ob(\mathcal{D})$, the fissus planar diagram
$$[v,\gamma_{\varnothing:(x)\rightarrow (x)}]=
\begin{matrix}
\begin{tikzpicture}[scale=.5]
\node (v2) at (-0.5,0.5) {};
\node (v1) at (-0.5,2) {$(x)$};
\node (v3) at (-0.5,-1) {$(x)$};
\draw (v2) circle [radius=0.1];
\draw  (v1) -- (-0.5,0.6)[postaction={decorate, decoration={markings,mark=at position .70 with {\arrow[black]{stealth}}}}];
\draw  (v3) -- (-0.5,0.4)[postaction={decorate, decoration={markings,mark=at position .70 with {\arrowreversed[black]{stealth}}}}];
\end{tikzpicture}
\end{matrix}
$$
is a morphism in $T(\mathcal{D})$. Thus $ \epsilon_m([v,\gamma_{\varnothing:(x)\rightarrow (x)}])\in Mor_{\mathcal{D}}(x,x).$ We call it the identity morphism of $x\in Ob(\mathcal{D})$, and denote it by $Id_x.$

$\bullet$  To prove the fourth statement, notice that for any $m,n\geq 1$ and  any object $x\in Ob(\mathcal{D})$, there exist at least one fissus invertible planar diagram
$([\Gamma,\gamma],P_{in},P_{out})$ such that
$[\Gamma,\gamma]=\begin{pmatrix}
\begin{tikzpicture}[scale=.5]
\node (v2) at (-0.5,0.5) {};
\node (v1) at (-0.5,2) {$x$};
\node (v3) at (-0.5,-1) {$x$};
\draw (v2) circle [radius=0.1];
\draw  (v1) -- (-0.5,0.6)[postaction={decorate, decoration={markings,mark=at position .70 with {\arrow[black]{stealth}}}}];
\draw  (v3) -- (-0.5,0.4)[postaction={decorate, decoration={markings,mark=at position .70 with {\arrowreversed[black]{stealth}}}}];
\end{tikzpicture}
\end{pmatrix}^{\otimes k}$ $(k\geq max\{m,n\})$ and $P_{in},P_{out}$ have $m, n$ blocks, respectively. Thus $ \epsilon_m(([\Gamma,\gamma],P_{in},P_{out}))\in Mor_{\mathcal{D}}(m,n).$

\end{proof}

\subsection{Operations on a tensor manifold}
In this section, we want to define some operations on a tensor manifold, and give another characteristic   of a tensor manifold.

Let $f:x_1\cdots x_m\rightarrow y_1\cdots y_n$  and $g:u_1\cdots u_k\rightarrow v_1\cdots v_l$ be two morphisms of  $T$-algebra$(\mathcal{D},\epsilon)$,  we define their tensor product $$f\otimes_{\epsilon}g\triangleq\epsilon_m([\Gamma_f,P_{in}^{full},P_{out}^{full}]\otimes_{fissus}[\Gamma_g,Q_{in}^{full},Q_{out}^{full}])$$ with $[\Gamma_f,P_{in}^{full},P_{out}^{full}]$ and $[\Gamma_g,Q_{in}^{full},Q_{out}^{full}]$ being fully fissus prime diagrams with their vertex decorated by $f$ and $g$ respectively, and the tensor product $\otimes_{fissus}$ being the tensor product of fissus diagrams.  That is,
$$
\begin{matrix}
\begin{matrix}f\otimes_{\epsilon}g\end{matrix}&=&\epsilon_m\begin{pmatrix}\begin{tikzpicture}[scale=0.5]

\node (v2) at (-0.5,0) {};
\draw [fill] (-0.5,0) circle [radius=0.1];
\node at (0.5,0) {$f$};
\node (v1) at (-2,1.5) {$(x_1)$};
\node (v4) at (1,1.5) {$(x_m)$};
\node (v3) at (-0.5,1.5) {$\cdots$};
\node (v5) at (-2,-1.5) {$(y_1)$};
\node (v7) at (1,-1.5) {$(y_n)$};
\node (v6) at (-0.5,-1.5) {$\cdots$};

\node (v8) at (2.5,1.5) {$(u_1)$};
\node (v10) at (4,1.5) {$\cdots$};
\node (v11) at (5.5,1.5) {$(u_k)$};
\node (v9) at (4,0) {};
\draw [fill] (4,0) circle [radius=0.1];
\node at (4.8,0) {$g$};
\node (v12) at (2.5,-1.5) {$(v_1)$};
\node (v13) at (4,-1.5) {$\cdots$};
\node (v14) at (5.5,-1.5) {$(v_l)$};
\draw  (v1) -- (-0.5,0)[postaction={decorate, decoration={markings,mark=at position .5 with {\arrow[black]{stealth}}}}];

\draw  (v4) -- (-0.5,0)[postaction={decorate, decoration={markings,mark=at position .5 with {\arrow[black]{stealth}}}}];
\draw  (-0.5,0) -- (v5)[postaction={decorate, decoration={markings,mark=at position .6 with {\arrow[black]{stealth}}}}];

\draw  (-0.5,0) -- (v7)[postaction={decorate, decoration={markings,mark=at position .6 with {\arrow[black]{stealth}}}}];
\draw  (v8) -- (4,0)[postaction={decorate, decoration={markings,mark=at position .5 with {\arrow[black]{stealth}}}}];

\draw  (v11) -- (4,0)[postaction={decorate, decoration={markings,mark=at position .5 with {\arrow[black]{stealth}}}}];
\draw  (4,0) -- (v12)[postaction={decorate, decoration={markings,mark=at position .6 with {\arrow[black]{stealth}}}}];

\draw  (4,0) -- (v14)[postaction={decorate, decoration={markings,mark=at position .6 with {\arrow[black]{stealth}}}}];
\end{tikzpicture}\end{pmatrix}.
\end{matrix}
$$

\begin{prop}\label{tensor product}
The tensor product $\otimes_{\epsilon}$ is associative.
\end{prop}

\begin{proof}
Let $f:x_1\cdots x_l\rightarrow u_1\cdots u_p$, $g:y_1\cdots y_m\rightarrow v_1\cdots v_q$ and $h:z_1\cdots z_n\rightarrow w_1\cdots w_r$ be three morphisms in $\mathcal{D}$. Then the following commutative diagram shows that $f\otimes_{\epsilon}(g\otimes_{\epsilon}h)=(f\otimes_{\epsilon}g)\otimes_{\epsilon} h$.
$$
\begin{matrix}

\begin{matrix}
[\Gamma_1,\lambda_1, P_{in}^{full},P_{out}^{full}]
\end{matrix}

&\begin{matrix}
\begin{tikzpicture}
\node at (0.5,0.2) {$\mu_{\mathcal{D}}$};
\draw [->] (0,0)--(1,0);
\end{tikzpicture}
\end{matrix}
&
\begin{matrix}
\begin{tikzpicture}[scale=0.5]

\node (v2) at (-0.5,0) {};
\draw [fill] (-0.5,0) circle [radius=0.1];
\node at (0,0) {$f$};
\node (v1) at (-2,1.5) {$(x_1)$};
\node (v4) at (1,1.5) {$(x_l)$};
\node (v3) at (-0.5,1.5) {$\cdots$};
\node (v5) at (-2,-1.5) {$(u_1)$};
\node (v7) at (1,-1.5) {$(u_p)$};
\node (v6) at (-0.5,-1.5) {$\cdots$};

\node (v8) at (2.5,1.5) {$(y_1)$};
\node (v10) at (4,1.5) {$\cdots$};
\node (v11) at (5.5,1.5) {$(y_m)$};
\node (v9) at (4,0) {};
\draw [fill] (4,0) circle [radius=0.1];
\node at (4.5,0) {$g$};
\node (v12) at (2.5,-1.5) {$(v_1)$};
\node (v13) at (4,-1.5) {$\cdots$};
\node (v14) at (5.5,-1.5) {$(v_q)$};
\draw  (v1) -- (-0.5,0)[postaction={decorate, decoration={markings,mark=at position .5 with {\arrow[black]{stealth}}}}];

\draw  (v4) -- (-0.5,0)[postaction={decorate, decoration={markings,mark=at position .5 with {\arrow[black]{stealth}}}}];
\draw  (-0.5,0) -- (v5)[postaction={decorate, decoration={markings,mark=at position .6 with {\arrow[black]{stealth}}}}];

\draw  (-0.5,0) -- (v7)[postaction={decorate, decoration={markings,mark=at position .6 with {\arrow[black]{stealth}}}}];
\draw  (v8) -- (4,0)[postaction={decorate, decoration={markings,mark=at position .5 with {\arrow[black]{stealth}}}}];

\draw  (v11) -- (4,0)[postaction={decorate, decoration={markings,mark=at position .5 with {\arrow[black]{stealth}}}}];
\draw  (4,0) -- (v12)[postaction={decorate, decoration={markings,mark=at position .6 with {\arrow[black]{stealth}}}}];

\draw  (4,0) -- (v14)[postaction={decorate, decoration={markings,mark=at position .6 with {\arrow[black]{stealth}}}}];
\node (v16) at (8.5,0) {};
\draw [fill] (8.5,0) circle [radius=0.1];
\node at (9,0) {$h$};
\node (v15) at (7,1.5) {$(z_1)$};
\node (v17) at (8.5,1.5) {$\cdots$};
\node (v18) at (10,1.5) {$(z_n)$};
\node (v19) at (7,-1.5) {$(w_1)$};
\node at (8.5,-1.5) {$\cdots$};
\node (v20) at (10,-1.5) {$(w_r)$};
\draw  (8.5,0) -- (v15)[postaction={decorate, decoration={markings,mark=at position .6 with {\arrowreversed[black]{stealth}}}}];
\draw  (8.5,0) -- (v18)[postaction={decorate, decoration={markings,mark=at position .6 with {\arrowreversed[black]{stealth}}}}];
\draw  (8.5,0) -- (v19)[postaction={decorate, decoration={markings,mark=at position .6 with {\arrow[black]{stealth}}}}];
\draw  (8.5,0) -- (v20)[postaction={decorate, decoration={markings,mark=at position .6 with {\arrow[black]{stealth}}}}];
\end{tikzpicture}
\end{matrix}
&
\begin{matrix}
\begin{tikzpicture}
\node at (0.5,0.2) {$\mu_{\mathcal{D}}$};
\draw [->] (1,0)--(0,0);
\end{tikzpicture}
\end{matrix}
&
\begin{matrix}
[\Gamma_2,\lambda_2, P_{in}^{full},P_{out}^{full}]
\end{matrix}\\

\begin{matrix}
\begin{tikzpicture}
\node at (0.7,-0.5) {$T(\epsilon)$};
\draw [->] (0,0)--(0,-1);
\end{tikzpicture}
\end{matrix}

&\begin{matrix}

\end{matrix}
&
\begin{matrix}
\begin{tikzpicture}
\node at (0.2,-0.5) {$\epsilon$};
\draw [->] (0,0)--(0,-1);
\end{tikzpicture}
\end{matrix}
&
\begin{matrix}

\end{matrix}
&
\begin{matrix}
\begin{tikzpicture}
\node at (0.7,-0.5) {$T(\epsilon)$};
\draw [->] (0,0)--(0,-1);
\end{tikzpicture}
\end{matrix}\\

\begin{matrix}
[\Gamma_3,\lambda_3, P_{in}^{full},P_{out}^{full}]
\end{matrix}
&
\begin{matrix}
\begin{tikzpicture}
\node at (0.5,0.2) {$\epsilon$};
\draw [->] (0,0)--(1,0);
\end{tikzpicture}
\end{matrix}
&
\begin{matrix}
f\otimes_{\epsilon}g \otimes_{\epsilon}h

\end{matrix}
&
\begin{matrix}
\begin{tikzpicture}
\node at (0.5,0.2) {$\epsilon$};
\draw [->] (1,0)--(0,0);
\end{tikzpicture}
\end{matrix}
&
\begin{matrix}
[\Gamma_4,\lambda_4, P_{in}^{full},P_{out}^{full}]
\end{matrix}

\end{matrix}
$$

where $\begin{matrix}(\Gamma_1, P_{in}^{full},P_{out}^{full})&=&\begin{matrix}\begin{tikzpicture}[scale=0.5]

\node (v2) at (-0.5,0) {};
\draw [fill] (-0.5,0) circle [radius=0.1];
\node at (0.3,0) {$p_1$};
\node (v1) at (-2,1.5) {$(1)$};
\node (v4) at (1,1.5) {$(l)$};
\node (v3) at (-0.5,1.5) {$\cdots$};
\node (v5) at (-2,-1.5) {$(1)$};
\node (v7) at (1,-1.5) {$(p)$};
\node (v6) at (-0.5,-1.5) {$\cdots$};

\node (v8) at (2.5,1.5) {$(1)$};
\node (v10) at (4,1.5) {$\cdots$};
\node (v11) at (5.5,1.5) {$(m+n)$};
\node (v9) at (4,0) {};
\draw [fill] (4,0) circle [radius=0.1];
\node  at (4.8,0) {$p_2$};
\node (v12) at (2.5,-1.5) {$(1)$};
\node (v13) at (4,-1.5) {$\cdots$};
\node (v14) at (5.5,-1.5) {$(q+r)$};
\draw  (v1) -- (-0.5,0)[postaction={decorate, decoration={markings,mark=at position .5 with {\arrow[black]{stealth}}}}];

\draw  (v4) -- (-0.5,0)[postaction={decorate, decoration={markings,mark=at position .5 with {\arrow[black]{stealth}}}}];
\draw  (-0.5,0) -- (v5)[postaction={decorate, decoration={markings,mark=at position .6 with {\arrow[black]{stealth}}}}];

\draw  (-0.5,0) -- (v7)[postaction={decorate, decoration={markings,mark=at position .6 with {\arrow[black]{stealth}}}}];
\draw  (v8) -- (4,0)[postaction={decorate, decoration={markings,mark=at position .5 with {\arrow[black]{stealth}}}}];

\draw  (v11) -- (4,0)[postaction={decorate, decoration={markings,mark=at position .5 with {\arrow[black]{stealth}}}}];
\draw  (4,0) -- (v12)[postaction={decorate, decoration={markings,mark=at position .6 with {\arrow[black]{stealth}}}}];

\draw  (4,0) -- (v14)[postaction={decorate, decoration={markings,mark=at position .6 with {\arrow[black]{stealth}}}}];
\end{tikzpicture}\end{matrix}\end{matrix}$
and $$\begin{matrix}(\lambda_1)_m(p_1)&=&\begin{matrix}\begin{tikzpicture}[scale=0.5]

\node (v2) at (-0.5,0) {};
\draw [fill] (-0.5,0) circle [radius=0.1];
\node at (0,0) {$f$};
\node (v1) at (-2,1.5) {$(x_1)$};
\node (v4) at (1,1.5) {$(x_l)$};
\node (v3) at (-0.5,1.5) {$\cdots$};
\node (v5) at (-2,-1.5) {$(u_1)$};
\node (v7) at (1,-1.5) {$(u_p),$};
\node (v6) at (-0.5,-1.5) {$\cdots$};

\draw  (v1) -- (-0.5,0)[postaction={decorate, decoration={markings,mark=at position .5 with {\arrow[black]{stealth}}}}];

\draw  (v4) -- (-0.5,0)[postaction={decorate, decoration={markings,mark=at position .5 with {\arrow[black]{stealth}}}}];
\draw  (-0.5,0) -- (v5)[postaction={decorate, decoration={markings,mark=at position .6 with {\arrow[black]{stealth}}}}];
\draw  (-0.5,0) -- (v7)[postaction={decorate, decoration={markings,mark=at position .6 with {\arrow[black]{stealth}}}}];
\end{tikzpicture}\end{matrix}\end{matrix}$$

$$\begin{matrix}(\lambda_1)_m(p_2)&=&\begin{matrix}\begin{tikzpicture}[scale=0.5]

\node (v2) at (-0.5,0) {};
\draw [fill] (-0.5,0) circle [radius=0.1];
\node at (0,0) {$g$};
\node (v1) at (-2,1.5) {$(y_1)$};
\node (v4) at (1,1.5) {$(y_m)$};
\node (v3) at (-0.5,1.5) {$\cdots$};
\node (v5) at (-2,-1.5) {$(v_1)$};
\node (v7) at (1,-1.5) {$(v_q)$};
\node (v6) at (-0.5,-1.5) {$\cdots$};

\node (v8) at (2.5,1.5) {$(z_1)$};
\node (v10) at (4,1.5) {$\cdots$};
\node (v11) at (5.5,1.5) {$(z_n)$};
\node (v9) at (4,0) {};
\draw [fill] (4,0) circle [radius=0.1];
\node at (4.5,0) {$h$};
\node (v12) at (2.5,-1.5) {$(w_1)$};
\node (v13) at (4,-1.5) {$\cdots$};
\node (v14) at (5.5,-1.5) {$(w_r);$};
\draw  (v1) -- (-0.5,0)[postaction={decorate, decoration={markings,mark=at position .5 with {\arrow[black]{stealth}}}}];

\draw  (v4) -- (-0.5,0)[postaction={decorate, decoration={markings,mark=at position .5 with {\arrow[black]{stealth}}}}];
\draw  (-0.5,0) -- (v5)[postaction={decorate, decoration={markings,mark=at position .6 with {\arrow[black]{stealth}}}}];

\draw  (-0.5,0) -- (v7)[postaction={decorate, decoration={markings,mark=at position .6 with {\arrow[black]{stealth}}}}];
\draw  (v8) -- (4,0)[postaction={decorate, decoration={markings,mark=at position .5 with {\arrow[black]{stealth}}}}];

\draw  (v11) -- (4,0)[postaction={decorate, decoration={markings,mark=at position .5 with {\arrow[black]{stealth}}}}];
\draw  (4,0) -- (v12)[postaction={decorate, decoration={markings,mark=at position .6 with {\arrow[black]{stealth}}}}];

\draw  (4,0) -- (v14)[postaction={decorate, decoration={markings,mark=at position .6 with {\arrow[black]{stealth}}}}];
\end{tikzpicture}\end{matrix}\end{matrix}$$

$\begin{matrix}(\Gamma_2, P_{in}^{full},P_{out}^{full})&=&\begin{matrix}\begin{tikzpicture}[scale=0.5]

\node (v2) at (-0.5,0) {};
\draw [fill] (-0.5,0) circle [radius=0.1];
\node at (0.3,0) {$p_3$};
\node (v1) at (-2,1.5) {$(1)$};
\node (v4) at (1,1.5) {$(l+m)$};
\node (v3) at (-0.5,1.5) {$\cdots$};
\node (v5) at (-2,-1.5) {$(1)$};
\node (v7) at (1,-1.5) {$(p+q)$};
\node (v6) at (-0.5,-1.5) {$\cdots$};

\node (v8) at (2.5,1.5) {$(1)$};
\node (v10) at (4,1.5) {$\cdots$};
\node (v11) at (5.5,1.5) {$(n)$};
\node (v9) at (4,0) {};
\draw [fill] (4,0) circle [radius=0.1];
\node  at (4.8,0) {$p_4$};
\node (v12) at (2.5,-1.5) {$(1)$};
\node (v13) at (4,-1.5) {$\cdots$};
\node (v14) at (5.5,-1.5) {$(r)$};
\draw  (v1) -- (-0.5,0)[postaction={decorate, decoration={markings,mark=at position .5 with {\arrow[black]{stealth}}}}];

\draw  (v4) -- (-0.5,0)[postaction={decorate, decoration={markings,mark=at position .5 with {\arrow[black]{stealth}}}}];
\draw  (-0.5,0) -- (v5)[postaction={decorate, decoration={markings,mark=at position .6 with {\arrow[black]{stealth}}}}];

\draw  (-0.5,0) -- (v7)[postaction={decorate, decoration={markings,mark=at position .6 with {\arrow[black]{stealth}}}}];
\draw  (v8) -- (4,0)[postaction={decorate, decoration={markings,mark=at position .5 with {\arrow[black]{stealth}}}}];

\draw  (v11) -- (4,0)[postaction={decorate, decoration={markings,mark=at position .5 with {\arrow[black]{stealth}}}}];
\draw  (4,0) -- (v12)[postaction={decorate, decoration={markings,mark=at position .6 with {\arrow[black]{stealth}}}}];

\draw  (4,0) -- (v14)[postaction={decorate, decoration={markings,mark=at position .6 with {\arrow[black]{stealth}}}}];
\end{tikzpicture}\end{matrix}\end{matrix}$

and $$\begin{matrix}(\lambda_2)_m(p_3)&=&\begin{matrix}\begin{tikzpicture}[scale=0.5]

\node (v2) at (-0.5,0) {};
\draw [fill] (-0.5,0) circle [radius=0.1];
\node at (0,0) {$f$};
\node (v1) at (-2,1.5) {$(x_1)$};
\node (v4) at (1,1.5) {$(x_l)$};
\node (v3) at (-0.5,1.5) {$\cdots$};
\node (v5) at (-2,-1.5) {$(u_1)$};
\node (v7) at (1,-1.5) {$(u_p)$};
\node (v6) at (-0.5,-1.5) {$\cdots$};

\node (v8) at (2.5,1.5) {$(y_1)$};
\node (v10) at (4,1.5) {$\cdots$};
\node (v11) at (5.5,1.5) {$(y_m)$};
\node (v9) at (4,0) {};
\draw [fill] (4,0) circle [radius=0.1];
\node at (4.5,0) {$g$};
\node (v12) at (2.5,-1.5) {$(v_1)$};
\node (v13) at (4,-1.5) {$\cdots$};
\node (v14) at (5.5,-1.5) {$(v_q),$};
\draw  (v1) -- (-0.5,0)[postaction={decorate, decoration={markings,mark=at position .5 with {\arrow[black]{stealth}}}}];

\draw  (v4) -- (-0.5,0)[postaction={decorate, decoration={markings,mark=at position .5 with {\arrow[black]{stealth}}}}];
\draw  (-0.5,0) -- (v5)[postaction={decorate, decoration={markings,mark=at position .6 with {\arrow[black]{stealth}}}}];

\draw  (-0.5,0) -- (v7)[postaction={decorate, decoration={markings,mark=at position .6 with {\arrow[black]{stealth}}}}];
\draw  (v8) -- (4,0)[postaction={decorate, decoration={markings,mark=at position .5 with {\arrow[black]{stealth}}}}];

\draw  (v11) -- (4,0)[postaction={decorate, decoration={markings,mark=at position .5 with {\arrow[black]{stealth}}}}];
\draw  (4,0) -- (v12)[postaction={decorate, decoration={markings,mark=at position .6 with {\arrow[black]{stealth}}}}];

\draw  (4,0) -- (v14)[postaction={decorate, decoration={markings,mark=at position .6 with {\arrow[black]{stealth}}}}];
\end{tikzpicture}\end{matrix}\end{matrix}$$

$$\begin{matrix}(\lambda_2)_m(p_4)&=&\begin{matrix}\begin{tikzpicture}[scale=0.5]

\node (v2) at (-0.5,0) {};
\draw [fill] (-0.5,0) circle [radius=0.1];
\node at (0,0) {$h$};
\node (v1) at (-2,1.5) {$(z_1)$};
\node (v4) at (1,1.5) {$(z_n)$};
\node (v3) at (-0.5,1.5) {$\cdots$};
\node (v5) at (-2,-1.5) {$(w_1)$};
\node (v7) at (1,-1.5) {$(w_r);$};
\node (v6) at (-0.5,-1.5) {$\cdots$};

\draw  (v1) -- (-0.5,0)[postaction={decorate, decoration={markings,mark=at position .5 with {\arrow[black]{stealth}}}}];

\draw  (v4) -- (-0.5,0)[postaction={decorate, decoration={markings,mark=at position .5 with {\arrow[black]{stealth}}}}];
\draw  (-0.5,0) -- (v5)[postaction={decorate, decoration={markings,mark=at position .6 with {\arrow[black]{stealth}}}}];
\draw  (-0.5,0) -- (v7)[postaction={decorate, decoration={markings,mark=at position .6 with {\arrow[black]{stealth}}}}];
\end{tikzpicture}\end{matrix}\end{matrix}$$

$\begin{matrix}(\Gamma_3, P_{in}^{full},P_{out}^{full})&=&\begin{matrix}\begin{tikzpicture}[scale=0.5]

\node (v2) at (-0.5,0) {};
\draw [fill] (-0.5,0) circle [radius=0.1];
\node at (0.3,0) {$f$};
\node (v1) at (-2,1.5) {$(1)$};
\node (v4) at (1,1.5) {$(l)$};
\node (v3) at (-0.5,1.5) {$\cdots$};
\node (v5) at (-2,-1.5) {$(1)$};
\node (v7) at (1,-1.5) {$(p)$};
\node (v6) at (-0.5,-1.5) {$\cdots$};

\node (v8) at (2.5,1.5) {$(1)$};
\node (v10) at (4,1.5) {$\cdots$};
\node (v11) at (5.5,1.5) {$(m+n)$};
\node (v9) at (4,0) {};
\draw [fill] (4,0) circle [radius=0.1];
\node  at (5.8,0) {$g\otimes_{\epsilon}h$};
\node (v12) at (2.5,-1.5) {$(1)$};
\node (v13) at (4,-1.5) {$\cdots$};
\node (v14) at (5.5,-1.5) {$(q+r)$};
\draw  (v1) -- (-0.5,0)[postaction={decorate, decoration={markings,mark=at position .5 with {\arrow[black]{stealth}}}}];

\draw  (v4) -- (-0.5,0)[postaction={decorate, decoration={markings,mark=at position .5 with {\arrow[black]{stealth}}}}];
\draw  (-0.5,0) -- (v5)[postaction={decorate, decoration={markings,mark=at position .6 with {\arrow[black]{stealth}}}}];

\draw  (-0.5,0) -- (v7)[postaction={decorate, decoration={markings,mark=at position .6 with {\arrow[black]{stealth}}}}];
\draw  (v8) -- (4,0)[postaction={decorate, decoration={markings,mark=at position .5 with {\arrow[black]{stealth}}}}];

\draw  (v11) -- (4,0)[postaction={decorate, decoration={markings,mark=at position .5 with {\arrow[black]{stealth}}}}];
\draw  (4,0) -- (v12)[postaction={decorate, decoration={markings,mark=at position .6 with {\arrow[black]{stealth}}}}];

\draw  (4,0) -- (v14)[postaction={decorate, decoration={markings,mark=at position .6 with {\arrow[black]{stealth}}}}];
\end{tikzpicture}\end{matrix}\end{matrix}$
and

$\begin{matrix}(\Gamma_4, P_{in}^{full},P_{out}^{full})&=&\begin{matrix}\begin{tikzpicture}[scale=0.5]

\node (v2) at (-0.5,0) {};
\draw [fill] (-0.5,0) circle [radius=0.1];
\node at (1.3,0) {$f\otimes_{\epsilon}g$};
\node (v1) at (-2,1.5) {$(x_1)$};
\node (v4) at (1,1.5) {$(y_m)$};
\node (v3) at (-0.5,1.5) {$\cdots$};
\node (v5) at (-2,-1.5) {$(u_1)$};
\node (v7) at (1,-1.5) {$(v_q)$};
\node (v6) at (-0.5,-1.5) {$\cdots$};

\node (v8) at (2.5,1.5) {$(z_1)$};
\node (v10) at (4,1.5) {$\cdots$};
\node (v11) at (5.5,1.5) {$(z_n)$};
\node (v9) at (4,0) {};
\draw [fill] (4,0) circle [radius=0.1];
\node  at (4.8,0) {$h$};
\node (v12) at (2.5,-1.5) {$(w_1)$};
\node (v13) at (4,-1.5) {$\cdots$};
\node (v14) at (5.5,-1.5) {$(w_r).$};
\draw  (v1) -- (-0.5,0)[postaction={decorate, decoration={markings,mark=at position .5 with {\arrow[black]{stealth}}}}];

\draw  (v4) -- (-0.5,0)[postaction={decorate, decoration={markings,mark=at position .5 with {\arrow[black]{stealth}}}}];
\draw  (-0.5,0) -- (v5)[postaction={decorate, decoration={markings,mark=at position .6 with {\arrow[black]{stealth}}}}];

\draw  (-0.5,0) -- (v7)[postaction={decorate, decoration={markings,mark=at position .6 with {\arrow[black]{stealth}}}}];
\draw  (v8) -- (4,0)[postaction={decorate, decoration={markings,mark=at position .5 with {\arrow[black]{stealth}}}}];

\draw  (v11) -- (4,0)[postaction={decorate, decoration={markings,mark=at position .5 with {\arrow[black]{stealth}}}}];
\draw  (4,0) -- (v12)[postaction={decorate, decoration={markings,mark=at position .6 with {\arrow[black]{stealth}}}}];

\draw  (4,0) -- (v14)[postaction={decorate, decoration={markings,mark=at position .6 with {\arrow[black]{stealth}}}}];
\end{tikzpicture}\end{matrix}\end{matrix}$

\end{proof}


Let $f:x_1\cdots x_l\rightarrow y_1\cdots y_m$  and $g:y_1\cdots y_m\rightarrow z_1\cdots z_n$ be two morphisms of  $T$-algebra$(\mathcal{D},\epsilon)$,  we define their composition $$g\circ_{\epsilon}f\triangleq\epsilon_m([\Gamma_g,Q_{in}^{full},Q_{out}^{full}]\circ_{fissus}[\Gamma_f,P_{in}^{full},P_{out}^{full}])$$ with $[\Gamma_f,P_{in}^{full},P_{out}^{full}]$ and $[\Gamma_g,Q_{in}^{full},Q_{out}^{full}]$ being fully fissus prime diagrams with their vertices decorated by $f$ and $g$ respectively, and the tensor product $\circ_{fissus}$ being the composition of fissus diagrams.  That is,
$$
\begin{matrix}
\begin{matrix}g\circ_{\epsilon}f\end{matrix}&=&\epsilon_m\begin{pmatrix}\begin{tikzpicture}

\node (v7) at (0,2) {$(x_1)$};
\node (v8) at (1,2) {$\cdots$};
\node (v9) at (2,2) {$(x_l)$};
\node (v1) at (1,1) {};
\draw [fill] (1,1) circle [radius=0.055];
\node  at (1.5,1) {$f$};
\node (v2) at (0.3,0) {};
\node (v3) at (0.3,-1) {};
\node  at (-0.2,-0.5) {$y_1$};
\node (v5) at (1.7,0) {};
\node (v6) at (1.7,-1) {};
\node  at (2.2,-0.5) {$y_m$};
\node (v4) at (1,-2) {};
\draw [fill] (1,-2) circle [radius=0.055];
\node  at (1.5,-2) {$g$};

\node (v10) at (0,-3) {$(z_1)$};
\node (v11) at (1,-3) {$\cdots$};
\node (v12) at (2,-3) {$(z_n)$};
\node at (1,-0.5) {$\cdots$};

\draw  (v7) -- (1,1)[postaction={decorate, decoration={markings,mark=at position .6 with {\arrow[black]{stealth}}}}];

\draw  (v9) -- ((1,1)[postaction={decorate, decoration={markings,mark=at position .6 with {\arrow[black]{stealth}}}}];
\draw  (1,-2) -- (v10)[postaction={decorate, decoration={markings,mark=at position .6 with {\arrow[black]{stealth}}}}];

\draw  (1,-2) -- (v12)[postaction={decorate, decoration={markings,mark=at position .6 with {\arrow[black]{stealth}}}}];
\draw  plot[smooth, tension=.4] coordinates {(v1) (v2) (v3) (v4)}[postaction={decorate, decoration={markings,mark=at position .5 with {\arrow[black]{stealth}}}}];
\draw  plot[smooth, tension=.4] coordinates {(1,1) (v5) (v6) (1,-2)}[postaction={decorate, decoration={markings,mark=at position .5 with {\arrow[black]{stealth}}}}];
\end{tikzpicture}\end{pmatrix}.
\end{matrix}
$$

\begin{prop}
The composition $\circ_{\epsilon}$ is associative.
\end{prop}

\begin{proof}
Let $f:x_1\cdots x_k\rightarrow y_1\cdots y_l$, $g:y_1\cdots y_l\rightarrow z_1\cdots z_m$ and $h:z_1\cdots z_m\rightarrow w_1\cdots w_n$ be three morphisms in $\mathcal{D}$. Then the following commutative diagram shows that $h\circ_{\epsilon}(g\circ_{\epsilon}f)=(h\circ_{\epsilon}g)\circ_{\epsilon} f$.
$$
\begin{matrix}

\begin{matrix}
[\Gamma_1,\lambda_1, P_{in}^{full},P_{out}^{full}]
\end{matrix}

&\begin{matrix}
\begin{tikzpicture}
\node at (0.5,0.2) {$\mu_{\mathcal{D}}$};
\draw [->] (0,0)--(1,0);
\end{tikzpicture}
\end{matrix}
&
\begin{matrix}
\begin{tikzpicture}[scale=0.7]

\node (v2) at (0.5,2.5) {};
\draw [fill] (0.5,2.5) circle [radius=0.055];
\node [scale=0.5] at (1,2.5) {$f$};
\node [scale=0.5] at (0.5,1.5) {$\cdots$};
\node [scale=0.6] at (-0.4,1.5) {$y_1$};
\node [scale=0.6] at (1.4,1.5) {$y_l$};
\node (v9) at (0.5,0.5) {};
\draw [fill] (0.5,0.5) circle [radius=0.055];
\node [scale=0.5] at (1,0.5) {$g$};
\node [scale=0.5] at (0.5,-0.5) {$\cdots$};
\node [scale=0.6] at (-0.4,-0.5) {$z_1$};
\node [scale=0.6] at (1.4,-0.5) {$z_m$};
\node (v4) at (0.5,-1.5) {};
\draw [fill] (0.5,-1.5) circle [radius=0.055];
\node [scale=0.5] at (1,-1.5) {$h$};
\node [scale=0.5] at (0.5,-2.2) {$\cdots$};
\node (v7) at (0,2) {};
\node (v8) at (0,1) {};
\node (v10) at (1,2) {};
\node (v11) at (1,1) {};
\node (v12) at (0,0) {};
\node (v13) at (0,-1) {};
\node (v14) at (1,0) {};
\node (v15) at (1,-1) {};
\node [scale=0.6] at (0,-2.2) {$(w_1)$};
\node [scale=0.6] at (1,-2.2) {$(w_n)$};
\node [scale=0.6] at (0,3.2) {$(x_1)$};
\node [scale=0.6] at (1,3.2) {$(x_k)$};
\node [scale=0.5] at (0.5,3.2) {$\cdots$};
\draw  (0,3) -- (0.5,2.5)[postaction={decorate, decoration={markings,mark=at position .6 with {\arrow[black]{stealth}}}}];
\draw  (1,3) -- (0.5,2.5)[postaction={decorate, decoration={markings,mark=at position .6 with {\arrow[black]{stealth}}}}];
\draw  (0.5,-1.5) -- (0,-2)[postaction={decorate, decoration={markings,mark=at position .6 with {\arrow[black]{stealth}}}}];
\draw  (0.5,-1.5) -- (1,-2)[postaction={decorate, decoration={markings,mark=at position .6 with {\arrow[black]{stealth}}}}];
\draw  plot[smooth, tension=.5] coordinates {(v2) (v7) (v8) (v9)}[postaction={decorate, decoration={markings,mark=at position .51 with {\arrow[black]{stealth}}}}];
\draw  plot[smooth, tension=.5] coordinates {(0.5,2.5) (v10) (v11) (0.5,0.5)}[postaction={decorate, decoration={markings,mark=at position .51 with {\arrow[black]{stealth}}}}];
\draw  plot[smooth, tension=.5] coordinates {(0.5,0.5) (v12) (v13) (v4)}[postaction={decorate, decoration={markings,mark=at position .51 with {\arrow[black]{stealth}}}}];
\draw  plot[smooth, tension=.5] coordinates {(0.5,0.5) (v14) (v15) (0.5,-1.5)}[postaction={decorate, decoration={markings,mark=at position .51 with {\arrow[black]{stealth}}}}];
\end{tikzpicture}
\end{matrix}
&
\begin{matrix}
\begin{tikzpicture}
\node at (0.5,0.2) {$\mu_{\mathcal{D}}$};
\draw [->] (1,0)--(0,0);
\end{tikzpicture}
\end{matrix}
&
\begin{matrix}
[\Gamma_2,\lambda_2, P_{in}^{full},P_{out}^{full}]
\end{matrix}\\

\begin{matrix}
\begin{tikzpicture}
\node at (0.7,-0.5) {$T(\epsilon)$};
\draw [->] (0,0)--(0,-1);
\end{tikzpicture}
\end{matrix}

&\begin{matrix}

\end{matrix}
&
\begin{matrix}
\begin{tikzpicture}
\node at (0.2,-0.5) {$\epsilon$};
\draw [->] (0,0)--(0,-1);
\end{tikzpicture}
\end{matrix}
&
\begin{matrix}

\end{matrix}
&
\begin{matrix}
\begin{tikzpicture}
\node at (0.7,-0.5) {$T(\epsilon)$};
\draw [->] (0,0)--(0,-1);
\end{tikzpicture}
\end{matrix}\\

\begin{matrix}
[\Gamma_3,\lambda_3, P_{in}^{full},P_{out}^{full}]
\end{matrix}
&
\begin{matrix}
\begin{tikzpicture}
\node at (0.5,0.2) {$\epsilon$};
\draw [->] (0,0)--(1,0);
\end{tikzpicture}
\end{matrix}
&
\begin{matrix}
h\circ_{\epsilon}g \circ_{\epsilon}f

\end{matrix}
&
\begin{matrix}
\begin{tikzpicture}
\node at (0.5,0.2) {$\epsilon$};
\draw [->] (1,0)--(0,0);
\end{tikzpicture}
\end{matrix}
&
\begin{matrix}
[\Gamma_4,\lambda_4, P_{in}^{full},P_{out}^{full}]
\end{matrix}

\end{matrix}
$$

where $\begin{matrix}(\Gamma_1, P_{in}^{full},P_{out}^{full})&=&\begin{matrix}\begin{tikzpicture}[scale=0.5]

\node  at (0,2.2) {$(1)$};
\node  at (1,2.2) {$\cdots$};
\node  at (2,2.2) {$(k)$};
\node (v1) at (1,1) {};
\draw [fill] (1,1) circle [radius=0.055];
\node  at (1.5,1) {$v_1$};
\node (v2) at (0.3,0) {};
\node (v3) at (0.3,-1) {};
\node  at (-0.2,-0.5) {$1$};
\node (v5) at (1.7,0) {};
\node (v6) at (1.7,-1) {};
\node  at (2.2,-0.5) {$l$};
\node (v4) at (1,-2) {};
\draw [fill] (1,-2) circle [radius=0.055];
\node  at (1.5,-2) {$v_2$};

\node  at (0,-3.2) {$(1)$};
\node (v11) at (1,-3) {$\cdots$};
\node  at (2,-3.2) {$(n)$};
\node at (1,-0.5) {$\cdots$};

\draw  (0,2) -- (1,1)[postaction={decorate, decoration={markings,mark=at position .6 with {\arrow[black]{stealth}}}}];

\draw  (2,2) -- ((1,1)[postaction={decorate, decoration={markings,mark=at position .6 with {\arrow[black]{stealth}}}}];
\draw  (1,-2) -- (0,-3)[postaction={decorate, decoration={markings,mark=at position .6 with {\arrow[black]{stealth}}}}];

\draw  (1,-2) -- (2,-3)[postaction={decorate, decoration={markings,mark=at position .6 with {\arrow[black]{stealth}}}}];
\draw  plot[smooth, tension=.4] coordinates {(v1) (v2) (v3) (v4)}[postaction={decorate, decoration={markings,mark=at position .5 with {\arrow[black]{stealth}}}}];
\draw  plot[smooth, tension=.4] coordinates {(1,1) (v5) (v6) (1,-2)}[postaction={decorate, decoration={markings,mark=at position .5 with {\arrow[black]{stealth}}}}];
\end{tikzpicture}\end{matrix}\end{matrix}$
and $$\begin{matrix}(\lambda_1)_m(v_1)&=&\begin{matrix}\begin{tikzpicture}[scale=0.5]

\node (v2) at (-0.5,0) {};
\draw [fill] (-0.5,0) circle [radius=0.1];
\node at (0,0) {$f$};
\node (v1) at (-2,1.5) {$(x_1)$};
\node (v4) at (1,1.5) {$(x_k)$};
\node (v3) at (-0.5,1.5) {$\cdots$};
\node (v5) at (-2,-1.5) {$(y_1)$};
\node (v7) at (1,-1.5) {$(y_l),$};
\node (v6) at (-0.5,-1.5) {$\cdots$};

\draw  (v1) -- (-0.5,0)[postaction={decorate, decoration={markings,mark=at position .5 with {\arrow[black]{stealth}}}}];

\draw  (v4) -- (-0.5,0)[postaction={decorate, decoration={markings,mark=at position .5 with {\arrow[black]{stealth}}}}];
\draw  (-0.5,0) -- (v5)[postaction={decorate, decoration={markings,mark=at position .6 with {\arrow[black]{stealth}}}}];
\draw  (-0.5,0) -- (v7)[postaction={decorate, decoration={markings,mark=at position .6 with {\arrow[black]{stealth}}}}];
\end{tikzpicture}\end{matrix}\end{matrix}$$

$$\begin{matrix}(\lambda_1)_m(v_2)&=&\begin{matrix}\begin{tikzpicture}[scale=0.5]

\node  at (0,2.2) {$(y_1)$};
\node  at (1,2.2) {$\cdots$};
\node  at (2,2.2) {$(y_l)$};
\node (v1) at (1,1) {};
\draw [fill] (1,1) circle [radius=0.055];
\node  at (1.5,1) {$g$};
\node (v2) at (0.3,0) {};
\node (v3) at (0.3,-1) {};
\node  at (-0.2,-0.5) {$z_1$};
\node (v5) at (1.7,0) {};
\node (v6) at (1.7,-1) {};
\node  at (2.2,-0.5) {$z_m$};
\node (v4) at (1,-2) {};
\draw [fill] (1,-2) circle [radius=0.055];
\node  at (1.5,-2) {$h$};

\node  at (0,-3.2) {$(w_1)$};
\node (v11) at (1,-3) {$\cdots$};
\node  at (2,-3.2) {$(w_n);$};
\node at (1,-0.5) {$\cdots$};

\draw  (0,2) -- (1,1)[postaction={decorate, decoration={markings,mark=at position .6 with {\arrow[black]{stealth}}}}];

\draw  (2,2) -- ((1,1)[postaction={decorate, decoration={markings,mark=at position .6 with {\arrow[black]{stealth}}}}];
\draw  (1,-2) -- (0,-3)[postaction={decorate, decoration={markings,mark=at position .6 with {\arrow[black]{stealth}}}}];

\draw  (1,-2) -- (2,-3)[postaction={decorate, decoration={markings,mark=at position .6 with {\arrow[black]{stealth}}}}];
\draw  plot[smooth, tension=.4] coordinates {(v1) (v2) (v3) (v4)}[postaction={decorate, decoration={markings,mark=at position .5 with {\arrow[black]{stealth}}}}];
\draw  plot[smooth, tension=.4] coordinates {(1,1) (v5) (v6) (1,-2)}[postaction={decorate, decoration={markings,mark=at position .5 with {\arrow[black]{stealth}}}}];
\end{tikzpicture}\end{matrix}\end{matrix}$$

$\begin{matrix}(\Gamma_2, P_{in}^{full},P_{out}^{full})&=&\begin{matrix}\begin{tikzpicture}[scale=0.5]

\node  at (0,2.2) {$(1)$};
\node  at (1,2.2) {$\cdots$};
\node  at (2,2.2) {$(k)$};
\node (v1) at (1,1) {};
\draw [fill] (1,1) circle [radius=0.055];
\node  at (1.5,1) {$v_3$};
\node (v2) at (0.3,0) {};
\node (v3) at (0.3,-1) {};
\node  at (-0.2,-0.5) {$1$};
\node (v5) at (1.7,0) {};
\node (v6) at (1.7,-1) {};
\node  at (2.2,-0.5) {$m$};
\node (v4) at (1,-2) {};
\draw [fill] (1,-2) circle [radius=0.055];
\node  at (1.5,-2) {$v_4$};

\node  at (0,-3.2) {$(1)$};
\node (v11) at (1,-3) {$\cdots$};
\node  at (2,-3.2) {$(n)$};
\node at (1,-0.5) {$\cdots$};

\draw  (0,2) -- (1,1)[postaction={decorate, decoration={markings,mark=at position .6 with {\arrow[black]{stealth}}}}];

\draw  (2,2) -- ((1,1)[postaction={decorate, decoration={markings,mark=at position .6 with {\arrow[black]{stealth}}}}];
\draw  (1,-2) -- (0,-3)[postaction={decorate, decoration={markings,mark=at position .6 with {\arrow[black]{stealth}}}}];

\draw  (1,-2) -- (2,-3)[postaction={decorate, decoration={markings,mark=at position .6 with {\arrow[black]{stealth}}}}];
\draw  plot[smooth, tension=.4] coordinates {(v1) (v2) (v3) (v4)}[postaction={decorate, decoration={markings,mark=at position .5 with {\arrow[black]{stealth}}}}];
\draw  plot[smooth, tension=.4] coordinates {(1,1) (v5) (v6) (1,-2)}[postaction={decorate, decoration={markings,mark=at position .5 with {\arrow[black]{stealth}}}}];
\end{tikzpicture}\end{matrix}\end{matrix}$

and $$\begin{matrix}(\lambda_2)_m(v_3)&=&\begin{matrix}\begin{tikzpicture}[scale=0.5]

\node  at (0,2.2) {$(x_1)$};
\node  at (1,2.2) {$\cdots$};
\node  at (2,2.2) {$(x_k)$};
\node (v1) at (1,1) {};
\draw [fill] (1,1) circle [radius=0.055];
\node  at (1.5,1) {$f$};
\node (v2) at (0.3,0) {};
\node (v3) at (0.3,-1) {};
\node  at (-0.2,-0.5) {$y_1$};
\node (v5) at (1.7,0) {};
\node (v6) at (1.7,-1) {};
\node  at (2.2,-0.5) {$y_l$};
\node (v4) at (1,-2) {};
\draw [fill] (1,-2) circle [radius=0.055];
\node  at (1.5,-2) {$g$};

\node  at (0,-3.2) {$(z_1)$};
\node (v11) at (1,-3) {$\cdots$};
\node  at (2,-3.2) {$(z_m),$};
\node at (1,-0.5) {$\cdots$};

\draw  (0,2) -- (1,1)[postaction={decorate, decoration={markings,mark=at position .6 with {\arrow[black]{stealth}}}}];

\draw  (2,2) -- ((1,1)[postaction={decorate, decoration={markings,mark=at position .6 with {\arrow[black]{stealth}}}}];
\draw  (1,-2) -- (0,-3)[postaction={decorate, decoration={markings,mark=at position .6 with {\arrow[black]{stealth}}}}];

\draw  (1,-2) -- (2,-3)[postaction={decorate, decoration={markings,mark=at position .6 with {\arrow[black]{stealth}}}}];
\draw  plot[smooth, tension=.4] coordinates {(v1) (v2) (v3) (v4)}[postaction={decorate, decoration={markings,mark=at position .5 with {\arrow[black]{stealth}}}}];
\draw  plot[smooth, tension=.4] coordinates {(1,1) (v5) (v6) (1,-2)}[postaction={decorate, decoration={markings,mark=at position .5 with {\arrow[black]{stealth}}}}];
\end{tikzpicture}\end{matrix}\end{matrix}$$

$$\begin{matrix}(\lambda_2)_m(v_4)&=&\begin{matrix}\begin{tikzpicture}[scale=0.5]

\node (v2) at (-0.5,0) {};
\draw [fill] (-0.5,0) circle [radius=0.1];
\node at (0,0) {$h$};
\node (v1) at (-2,1.5) {$(z_1)$};
\node (v4) at (1,1.5) {$(z_m)$};
\node (v3) at (-0.5,1.5) {$\cdots$};
\node (v5) at (-2,-1.5) {$(w_1)$};
\node (v7) at (1,-1.5) {$(w_n);$};
\node (v6) at (-0.5,-1.5) {$\cdots$};

\draw  (v1) -- (-0.5,0)[postaction={decorate, decoration={markings,mark=at position .5 with {\arrow[black]{stealth}}}}];

\draw  (v4) -- (-0.5,0)[postaction={decorate, decoration={markings,mark=at position .5 with {\arrow[black]{stealth}}}}];
\draw  (-0.5,0) -- (v5)[postaction={decorate, decoration={markings,mark=at position .6 with {\arrow[black]{stealth}}}}];
\draw  (-0.5,0) -- (v7)[postaction={decorate, decoration={markings,mark=at position .6 with {\arrow[black]{stealth}}}}];
\end{tikzpicture}\end{matrix}\end{matrix}$$

$\begin{matrix}(\Gamma_3,\lambda_3 P_{in}^{full},P_{out}^{full})&=&\begin{matrix}\begin{tikzpicture}[scale=0.5]

\node  at (0,2.2) {$(x_1)$};
\node  at (1,2.2) {$\cdots$};
\node  at (2,2.2) {$(x_k)$};
\node (v1) at (1,1) {};
\draw [fill] (1,1) circle [radius=0.055];
\node  at (1.5,1) {$f$};
\node (v2) at (0.3,0) {};
\node (v3) at (0.3,-1) {};
\node  at (-0.2,-0.5) {$y_1$};
\node (v5) at (1.7,0) {};
\node (v6) at (1.7,-1) {};
\node  at (2.2,-0.5) {$y_l$};
\node (v4) at (1,-2) {};
\draw [fill] (1,-2) circle [radius=0.055];
\node  at (2.5,-2) {$h\circ_{\epsilon}f$};

\node  at (0,-3.2) {$(w_1)$};
\node (v11) at (1,-3) {$\cdots$};
\node  at (2,-3.2) {$(w_n)$};
\node at (1,-0.5) {$\cdots$};

\draw  (0,2) -- (1,1)[postaction={decorate, decoration={markings,mark=at position .6 with {\arrow[black]{stealth}}}}];

\draw  (2,2) -- ((1,1)[postaction={decorate, decoration={markings,mark=at position .6 with {\arrow[black]{stealth}}}}];
\draw  (1,-2) -- (0,-3)[postaction={decorate, decoration={markings,mark=at position .6 with {\arrow[black]{stealth}}}}];

\draw  (1,-2) -- (2,-3)[postaction={decorate, decoration={markings,mark=at position .6 with {\arrow[black]{stealth}}}}];
\draw  plot[smooth, tension=.4] coordinates {(v1) (v2) (v3) (v4)}[postaction={decorate, decoration={markings,mark=at position .5 with {\arrow[black]{stealth}}}}];
\draw  plot[smooth, tension=.4] coordinates {(1,1) (v5) (v6) (1,-2)}[postaction={decorate, decoration={markings,mark=at position .5 with {\arrow[black]{stealth}}}}];
\end{tikzpicture}\end{matrix}\end{matrix}$
and
$\begin{matrix}(\Gamma_4,\lambda_4, P_{in}^{full},P_{out}^{full})&=&\begin{matrix}\begin{tikzpicture}[scale=0.5]

\node  at (0,2.2) {$(x_1)$};
\node  at (1,2.2) {$\cdots$};
\node  at (2,2.2) {$(x_k)$};
\node (v1) at (1,1) {};
\draw [fill] (1,1) circle [radius=0.055];
\node  at (2.5,1) {$g\circ_{\epsilon}f$};
\node (v2) at (0.3,0) {};
\node (v3) at (0.3,-1) {};
\node  at (-0.2,-0.5) {$z_1$};
\node (v5) at (1.7,0) {};
\node (v6) at (1.7,-1) {};
\node  at (2.2,-0.5) {$z_m$};
\node (v4) at (1,-2) {};
\draw [fill] (1,-2) circle [radius=0.055];
\node  at (1.5,-2) {$h$};

\node  at (0,-3.2) {$(w_1)$};
\node (v11) at (1,-3) {$\cdots$};
\node  at (2,-3.2) {$(w_n).$};
\node at (1,-0.5) {$\cdots$};
\draw  (0,2) -- (1,1)[postaction={decorate, decoration={markings,mark=at position .6 with {\arrow[black]{stealth}}}}];
\draw  (2,2) -- ((1,1)[postaction={decorate, decoration={markings,mark=at position .6 with {\arrow[black]{stealth}}}}];
\draw  (1,-2) -- (0,-3)[postaction={decorate, decoration={markings,mark=at position .6 with {\arrow[black]{stealth}}}}];
\draw  (1,-2) -- (2,-3)[postaction={decorate, decoration={markings,mark=at position .6 with {\arrow[black]{stealth}}}}];
\draw  plot[smooth, tension=.4] coordinates {(v1) (v2) (v3) (v4)}[postaction={decorate, decoration={markings,mark=at position .5 with {\arrow[black]{stealth}}}}];
\draw  plot[smooth, tension=.4] coordinates {(1,1) (v5) (v6) (1,-2)}[postaction={decorate, decoration={markings,mark=at position .5 with {\arrow[black]{stealth}}}}];
\end{tikzpicture}\end{matrix}\end{matrix}$

\end{proof}


\begin{prop}
For $\otimes_{\epsilon}$ and $\circ_{\epsilon}$, the middle-four-interchange law holds.
\end{prop}

\begin{proof}
Let $f_1:x_1\cdots x_l\rightarrow y_1\cdots y_m$, $g_1:y_1\cdots y_m\rightarrow z_1\cdots z_n$ and $f_2:u_1\cdots u_p\rightarrow v_1\cdots v_q$, $g_2:v_1\cdots v_q\rightarrow w_1\cdots w_r$ be four morphisms in $\mathcal{D}$. We want to shows that $(g_1\otimes_{\epsilon} g_2)\circ_{\epsilon}(f_1\otimes_{\epsilon} f_2)=(g_1\circ_{\epsilon} f_1)\otimes_{\epsilon}(g_2\circ_{\epsilon} f_2)$ which is a direct consequence of the following commutative diagram.
$$
\begin{matrix}

\begin{matrix}
[\Gamma_1,\lambda_1, P_{in}^{full},P_{out}^{full}]
\end{matrix}

&\begin{matrix}
\begin{tikzpicture}
\node at (0.5,0.2) {$\mu_{\mathcal{D}}$};
\draw [->] (0,0)--(1,0);
\end{tikzpicture}
\end{matrix}
&
\begin{matrix}
\begin{tikzpicture}[scale=0.5]

\node  at (0,2.2) {$(x_1)$};
\node  at (1,2.2) {$\cdots$};
\node  at (2,2.2) {$(x_l)$};
\node (v1) at (1,1) {};
\draw [fill] (1,1) circle [radius=0.055];
\node  at (1.5,1) {$f_1$};
\node (v2) at (0.3,0) {};
\node (v3) at (0.3,-1) {};
\node  at (-0.2,-0.5) {$y_1$};
\node (v5) at (1.7,0) {};
\node (v6) at (1.7,-1) {};
\node  at (2.4,-0.5) {$y_m$};
\node (v4) at (1,-2) {};
\draw [fill] (1,-2) circle [radius=0.055];
\node  at (1.5,-2) {$g_1$};

\node  at (0,-3.2) {$(z_1)$};
\node (v11) at (1,-3) {$\cdots$};
\node  at (2,-3.2) {$(z_n)$};
\node at (1,-0.5) {$\cdots$};

\draw  (0,2) -- (1,1)[postaction={decorate, decoration={markings,mark=at position .6 with {\arrow[black]{stealth}}}}];

\draw  (2,2) -- ((1,1)[postaction={decorate, decoration={markings,mark=at position .6 with {\arrow[black]{stealth}}}}];
\draw  (1,-2) -- (0,-3)[postaction={decorate, decoration={markings,mark=at position .6 with {\arrow[black]{stealth}}}}];

\draw  (1,-2) -- (2,-3)[postaction={decorate, decoration={markings,mark=at position .6 with {\arrow[black]{stealth}}}}];
\draw  plot[smooth, tension=.4] coordinates {(v1) (v2) (v3) (v4)}[postaction={decorate, decoration={markings,mark=at position .5 with {\arrow[black]{stealth}}}}];
\draw  plot[smooth, tension=.4] coordinates {(1,1) (v5) (v6) (1,-2)}[postaction={decorate, decoration={markings,mark=at position .5 with {\arrow[black]{stealth}}}}];
\end{tikzpicture}
\begin{tikzpicture}[scale=0.5]
\node  at (0,2.2) {$(u_1)$};
\node  at (1,2.2) {$\cdots$};
\node  at (2,2.2) {$(u_p)$};
\node (v1) at (1,1) {};
\draw [fill] (1,1) circle [radius=0.055];
\node  at (1.5,1) {$f_2$};
\node (v2) at (0.3,0) {};
\node (v3) at (0.3,-1) {};
\node  at (-0.2,-0.5) {$v_1$};
\node (v5) at (1.7,0) {};
\node (v6) at (1.7,-1) {};
\node  at (2.2,-0.5) {$v_q$};
\node (v4) at (1,-2) {};
\draw [fill] (1,-2) circle [radius=0.055];
\node  at (1.5,-2) {$g_2$};

\node  at (0,-3.2) {$(w_1)$};
\node (v11) at (1,-3) {$\cdots$};
\node  at (2,-3.2) {$(w_r)$};
\node at (1,-0.5) {$\cdots$};

\draw  (0,2) -- (1,1)[postaction={decorate, decoration={markings,mark=at position .6 with {\arrow[black]{stealth}}}}];

\draw  (2,2) -- ((1,1)[postaction={decorate, decoration={markings,mark=at position .6 with {\arrow[black]{stealth}}}}];
\draw  (1,-2) -- (0,-3)[postaction={decorate, decoration={markings,mark=at position .6 with {\arrow[black]{stealth}}}}];

\draw  (1,-2) -- (2,-3)[postaction={decorate, decoration={markings,mark=at position .6 with {\arrow[black]{stealth}}}}];
\draw  plot[smooth, tension=.4] coordinates {(v1) (v2) (v3) (v4)}[postaction={decorate, decoration={markings,mark=at position .5 with {\arrow[black]{stealth}}}}];
\draw  plot[smooth, tension=.4] coordinates {(1,1) (v5) (v6) (1,-2)}[postaction={decorate, decoration={markings,mark=at position .5 with {\arrow[black]{stealth}}}}];
\end{tikzpicture}
\end{matrix}
&
\begin{matrix}
\begin{tikzpicture}
\node at (0.5,0.2) {$\mu_{\mathcal{D}}$};
\draw [->] (1,0)--(0,0);
\end{tikzpicture}
\end{matrix}
&
\begin{matrix}
[\Gamma_2,\lambda_2, P_{in}^{full},P_{out}^{full}]
\end{matrix}\\

\begin{matrix}
\begin{tikzpicture}
\node at (0.7,-0.5) {$T(\epsilon)$};
\draw [->] (0,0)--(0,-1);
\end{tikzpicture}
\end{matrix}

&\begin{matrix}

\end{matrix}
&
\begin{matrix}
\begin{tikzpicture}
\node at (0.2,-0.5) {$\epsilon$};
\draw [->] (0,0)--(0,-1);
\end{tikzpicture}
\end{matrix}
&
\begin{matrix}

\end{matrix}
&
\begin{matrix}
\begin{tikzpicture}
\node at (0.7,-0.5) {$T(\epsilon)$};
\draw [->] (0,0)--(0,-1);
\end{tikzpicture}
\end{matrix}\\

\begin{matrix}
[\Gamma_3,\lambda_3, P_{in}^{full},P_{out}^{full}]
\end{matrix}
&
\begin{matrix}
\begin{tikzpicture}
\node at (0.5,0.2) {$\epsilon$};
\draw [->] (0,0)--(1,0);
\end{tikzpicture}
\end{matrix}
&
\begin{matrix}
\epsilon(f_1,g_1;f_2,g_2)
\end{matrix}
&
\begin{matrix}
\begin{tikzpicture}
\node at (0.5,0.2) {$\epsilon$};
\draw [->] (1,0)--(0,0);
\end{tikzpicture}
\end{matrix}
&
\begin{matrix}
[\Gamma_4,\lambda_4, P_{in}^{full},P_{out}^{full}]
\end{matrix}

\end{matrix}
$$

where $\begin{matrix}(\Gamma_1, P_{in}^{full},P_{out}^{full})&=&\begin{matrix}\begin{tikzpicture}[scale=0.5]

\node [scale=0.7] at (0,2.2) {$(1)$};
\node [scale=0.7] at (1,2.2) {$\cdots$};
\node [scale=0.7] at (2,2.2) {$(l+p)$};
\node (v1) at (1,1) {};
\draw [fill] (1,1) circle [radius=0.055];
\node  at (1.5,1) {$p_1$};
\node (v2) at (0.3,0) {};
\node (v3) at (0.3,-1) {};
\node  [scale=0.7]at (-0.2,-0.5) {$1$};
\node (v5) at (1.7,0) {};
\node (v6) at (1.7,-1) {};
\node  [scale=0.7]at (2.5,-0.5) {$m+q$};
\node (v4) at (1,-2) {};
\draw [fill] (1,-2) circle [radius=0.055];
\node  at (1.5,-2) {$p_2$};

\node [scale=0.7] at (0,-3.2) {$(1)$};
\node [scale=0.7](v11) at (1,-3) {$\cdots$};
\node  [scale=0.7]at (2,-3.2) {$(n+r)$};
\node [scale=0.7]at (1,-0.5) {$\cdots$};

\draw  (0,2) -- (1,1)[postaction={decorate, decoration={markings,mark=at position .6 with {\arrow[black]{stealth}}}}];

\draw  (2,2) -- ((1,1)[postaction={decorate, decoration={markings,mark=at position .6 with {\arrow[black]{stealth}}}}];
\draw  (1,-2) -- (0,-3)[postaction={decorate, decoration={markings,mark=at position .6 with {\arrow[black]{stealth}}}}];

\draw  (1,-2) -- (2,-3)[postaction={decorate, decoration={markings,mark=at position .6 with {\arrow[black]{stealth}}}}];
\draw  plot[smooth, tension=.4] coordinates {(v1) (v2) (v3) (v4)}[postaction={decorate, decoration={markings,mark=at position .5 with {\arrow[black]{stealth}}}}];
\draw  plot[smooth, tension=.4] coordinates {(1,1) (v5) (v6) (1,-2)}[postaction={decorate, decoration={markings,mark=at position .5 with {\arrow[black]{stealth}}}}];
\end{tikzpicture}\end{matrix}\end{matrix}$
and $$\begin{matrix}(\lambda_1)_m(p_1)&=&\begin{matrix}\begin{tikzpicture}[scale=0.5]

\node (v2) at (-0.5,0) {};
\draw [fill] (-0.5,0) circle [radius=0.1];
\node at (0.5,0) {$f_1$};
\node (v1) at (-2,1.5) {$(x_1)$};
\node (v4) at (1,1.5) {$(x_l)$};
\node (v3) at (-0.5,1.5) {$\cdots$};
\node (v5) at (-2,-1.5) {$(y_1)$};
\node (v7) at (1,-1.5) {$(y_m)$};
\node (v6) at (-0.5,-1.5) {$\cdots$};

\node (v8) at (2.5,1.5) {$(u_1)$};
\node (v10) at (4,1.5) {$\cdots$};
\node (v11) at (5.5,1.5) {$(u_p)$};
\node (v9) at (4,0) {};
\draw [fill] (4,0) circle [radius=0.1];
\node at (4.8,0) {$g_1$};
\node (v12) at (2.5,-1.5) {$(v_1)$};
\node (v13) at (4,-1.5) {$\cdots$};
\node (v14) at (5.5,-1.5) {$(v_q),$};
\draw  (v1) -- (-0.5,0)[postaction={decorate, decoration={markings,mark=at position .5 with {\arrow[black]{stealth}}}}];

\draw  (v4) -- (-0.5,0)[postaction={decorate, decoration={markings,mark=at position .5 with {\arrow[black]{stealth}}}}];
\draw  (-0.5,0) -- (v5)[postaction={decorate, decoration={markings,mark=at position .6 with {\arrow[black]{stealth}}}}];

\draw  (-0.5,0) -- (v7)[postaction={decorate, decoration={markings,mark=at position .6 with {\arrow[black]{stealth}}}}];
\draw  (v8) -- (4,0)[postaction={decorate, decoration={markings,mark=at position .5 with {\arrow[black]{stealth}}}}];

\draw  (v11) -- (4,0)[postaction={decorate, decoration={markings,mark=at position .5 with {\arrow[black]{stealth}}}}];
\draw  (4,0) -- (v12)[postaction={decorate, decoration={markings,mark=at position .6 with {\arrow[black]{stealth}}}}];

\draw  (4,0) -- (v14)[postaction={decorate, decoration={markings,mark=at position .6 with {\arrow[black]{stealth}}}}];
\end{tikzpicture}\end{matrix}\end{matrix}$$

$$\begin{matrix}(\lambda_1)_m(p_2)&=&\begin{matrix}\begin{tikzpicture}[scale=0.5]

\node (v2) at (-0.5,0) {};
\draw [fill] (-0.5,0) circle [radius=0.1];
\node at (0.5,0) {$f_2$};
\node (v1) at (-2,1.5) {$(y_1)$};
\node (v4) at (1,1.5) {$(y_m)$};
\node (v3) at (-0.5,1.5) {$\cdots$};
\node (v5) at (-2,-1.5) {$(z_1)$};
\node (v7) at (1,-1.5) {$(z_n)$};
\node (v6) at (-0.5,-1.5) {$\cdots$};

\node (v8) at (2.5,1.5) {$(v_1)$};
\node (v10) at (4,1.5) {$\cdots$};
\node (v11) at (5.5,1.5) {$(v_q)$};
\node (v9) at (4,0) {};
\draw [fill] (4,0) circle [radius=0.1];
\node at (4.8,0) {$g_2$};
\node (v12) at (2.5,-1.5) {$(w_1)$};
\node (v13) at (4,-1.5) {$\cdots$};
\node (v14) at (5.5,-1.5) {$(w_r);$};
\draw  (v1) -- (-0.5,0)[postaction={decorate, decoration={markings,mark=at position .5 with {\arrow[black]{stealth}}}}];

\draw  (v4) -- (-0.5,0)[postaction={decorate, decoration={markings,mark=at position .5 with {\arrow[black]{stealth}}}}];
\draw  (-0.5,0) -- (v5)[postaction={decorate, decoration={markings,mark=at position .6 with {\arrow[black]{stealth}}}}];

\draw  (-0.5,0) -- (v7)[postaction={decorate, decoration={markings,mark=at position .6 with {\arrow[black]{stealth}}}}];
\draw  (v8) -- (4,0)[postaction={decorate, decoration={markings,mark=at position .5 with {\arrow[black]{stealth}}}}];

\draw  (v11) -- (4,0)[postaction={decorate, decoration={markings,mark=at position .5 with {\arrow[black]{stealth}}}}];
\draw  (4,0) -- (v12)[postaction={decorate, decoration={markings,mark=at position .6 with {\arrow[black]{stealth}}}}];

\draw  (4,0) -- (v14)[postaction={decorate, decoration={markings,mark=at position .6 with {\arrow[black]{stealth}}}}];
\end{tikzpicture}\end{matrix}\end{matrix}$$

$\begin{matrix}(\Gamma_2, P_{in}^{full},P_{out}^{full})&=&\begin{matrix}\begin{tikzpicture}[scale=0.5]

\node (v2) at (-0.5,0) {};
\draw [fill] (-0.5,0) circle [radius=0.1];
\node at (0.5,0) {$p_3$};
\node (v1) at (-2,1.5) {$(1)$};
\node (v4) at (1,1.5) {$(l)$};
\node (v3) at (-0.5,1.5) {$\cdots$};
\node (v5) at (-2,-1.5) {$(1)$};
\node (v7) at (1,-1.5) {$(n)$};
\node (v6) at (-0.5,-1.5) {$\cdots$};

\node (v8) at (2.5,1.5) {$(1)$};
\node (v10) at (4,1.5) {$\cdots$};
\node (v11) at (5.5,1.5) {$(p)$};
\node (v9) at (4,0) {};
\draw [fill] (4,0) circle [radius=0.1];
\node at (4.8,0) {$p_4$};
\node (v12) at (2.5,-1.5) {$(1)$};
\node (v13) at (4,-1.5) {$\cdots$};
\node (v14) at (5.5,-1.5) {$(r)$};
\draw  (v1) -- (-0.5,0)[postaction={decorate, decoration={markings,mark=at position .5 with {\arrow[black]{stealth}}}}];

\draw  (v4) -- (-0.5,0)[postaction={decorate, decoration={markings,mark=at position .5 with {\arrow[black]{stealth}}}}];
\draw  (-0.5,0) -- (v5)[postaction={decorate, decoration={markings,mark=at position .6 with {\arrow[black]{stealth}}}}];

\draw  (-0.5,0) -- (v7)[postaction={decorate, decoration={markings,mark=at position .6 with {\arrow[black]{stealth}}}}];
\draw  (v8) -- (4,0)[postaction={decorate, decoration={markings,mark=at position .5 with {\arrow[black]{stealth}}}}];

\draw  (v11) -- (4,0)[postaction={decorate, decoration={markings,mark=at position .5 with {\arrow[black]{stealth}}}}];
\draw  (4,0) -- (v12)[postaction={decorate, decoration={markings,mark=at position .6 with {\arrow[black]{stealth}}}}];

\draw  (4,0) -- (v14)[postaction={decorate, decoration={markings,mark=at position .6 with {\arrow[black]{stealth}}}}];
\end{tikzpicture}\end{matrix}\end{matrix}$

and $$\begin{matrix}(\lambda_2)_m(p_3)&=&\begin{matrix}\begin{tikzpicture}[scale=0.5]

\node  at (0,2.2) {$(x_1)$};
\node  at (1,2.2) {$\cdots$};
\node  at (2,2.2) {$(x_l)$};
\node (v1) at (1,1) {};
\draw [fill] (1,1) circle [radius=0.055];
\node  at (1.5,1) {$f_1$};
\node (v2) at (0.3,0) {};
\node (v3) at (0.3,-1) {};
\node  at (-0.2,-0.5) {$y_1$};
\node (v5) at (1.7,0) {};
\node (v6) at (1.7,-1) {};
\node  at (2.5,-0.5) {$y_m$};
\node (v4) at (1,-2) {};
\draw [fill] (1,-2) circle [radius=0.055];
\node  at (1.5,-2) {$g_1$};

\node  at (0,-3.2) {$(z_1)$};
\node (v11) at (1,-3) {$\cdots$};
\node  at (2,-3.2) {$(z_n),$};
\node at (1,-0.5) {$\cdots$};

\draw  (0,2) -- (1,1)[postaction={decorate, decoration={markings,mark=at position .6 with {\arrow[black]{stealth}}}}];

\draw  (2,2) -- ((1,1)[postaction={decorate, decoration={markings,mark=at position .6 with {\arrow[black]{stealth}}}}];
\draw  (1,-2) -- (0,-3)[postaction={decorate, decoration={markings,mark=at position .6 with {\arrow[black]{stealth}}}}];

\draw  (1,-2) -- (2,-3)[postaction={decorate, decoration={markings,mark=at position .6 with {\arrow[black]{stealth}}}}];
\draw  plot[smooth, tension=.4] coordinates {(v1) (v2) (v3) (v4)}[postaction={decorate, decoration={markings,mark=at position .5 with {\arrow[black]{stealth}}}}];
\draw  plot[smooth, tension=.4] coordinates {(1,1) (v5) (v6) (1,-2)}[postaction={decorate, decoration={markings,mark=at position .5 with {\arrow[black]{stealth}}}}];
\end{tikzpicture}\end{matrix}\end{matrix}$$

$$\begin{matrix}(\lambda_2)_m(p_4)&=&\begin{matrix}\begin{tikzpicture}[scale=0.5]

\node  at (0,2.2) {$(u_1)$};
\node  at (1,2.2) {$\cdots$};
\node  at (2,2.2) {$(u_p)$};
\node (v1) at (1,1) {};
\draw [fill] (1,1) circle [radius=0.055];
\node  at (1.5,1) {$f_2$};
\node (v2) at (0.3,0) {};
\node (v3) at (0.3,-1) {};
\node  at (-0.2,-0.5) {$v_1$};
\node (v5) at (1.7,0) {};
\node (v6) at (1.7,-1) {};
\node  at (2.5,-0.5) {$v_q$};
\node (v4) at (1,-2) {};
\draw [fill] (1,-2) circle [radius=0.055];
\node  at (1.5,-2) {$g_2$};

\node  at (0,-3.2) {$(w_1)$};
\node (v11) at (1,-3) {$\cdots$};
\node  at (2,-3.2) {$(w_r);$};
\node at (1,-0.5) {$\cdots$};

\draw  (0,2) -- (1,1)[postaction={decorate, decoration={markings,mark=at position .6 with {\arrow[black]{stealth}}}}];

\draw  (2,2) -- ((1,1)[postaction={decorate, decoration={markings,mark=at position .6 with {\arrow[black]{stealth}}}}];
\draw  (1,-2) -- (0,-3)[postaction={decorate, decoration={markings,mark=at position .6 with {\arrow[black]{stealth}}}}];

\draw  (1,-2) -- (2,-3)[postaction={decorate, decoration={markings,mark=at position .6 with {\arrow[black]{stealth}}}}];
\draw  plot[smooth, tension=.4] coordinates {(v1) (v2) (v3) (v4)}[postaction={decorate, decoration={markings,mark=at position .5 with {\arrow[black]{stealth}}}}];
\draw  plot[smooth, tension=.4] coordinates {(1,1) (v5) (v6) (1,-2)}[postaction={decorate, decoration={markings,mark=at position .5 with {\arrow[black]{stealth}}}}];
\end{tikzpicture}\end{matrix}\end{matrix}$$

$\begin{matrix}(\Gamma_3, \lambda_3,P_{in}^{full},P_{out}^{full})&=&\begin{matrix}\begin{tikzpicture}[scale=0.5]

\node [scale=0.7] at (0,2.2) {$(x_1)$};
\node [scale=0.7] at (1,2.2) {$\cdots$};
\node [scale=0.7]  at (2,2.2) {$(u_p)$};
\node (v1) at (1,1) {};
\draw [fill] (1,1) circle [radius=0.055];
\node  at (2.5,1) {$f_1\otimes_{\epsilon}f_2$};
\node (v2) at (0.3,0) {};
\node (v3) at (0.3,-1) {};
\node [scale=0.7] at (-0.2,-0.5) {$y_1$};
\node (v5) at (1.7,0) {};
\node (v6) at (1.7,-1) {};
\node [scale=0.7] at (2.2,-0.5) {$v_q$};
\node (v4) at (1,-2) {};
\draw [fill] (1,-2) circle [radius=0.055];
\node  at (2.5,-2) {$g_1\otimes_{\epsilon}g_2$};

\node [scale=0.7] at (0,-3.2) {$(z_1)$};
\node [scale=0.7] at (1,-3) {$\cdots$};
\node [scale=0.7] at (2,-3.2) {$(w_r)$};
\node [scale=0.7] at (1,-0.5) {$\cdots$};

\draw  (0,2) -- (1,1)[postaction={decorate, decoration={markings,mark=at position .6 with {\arrow[black]{stealth}}}}];

\draw  (2,2) -- ((1,1)[postaction={decorate, decoration={markings,mark=at position .6 with {\arrow[black]{stealth}}}}];
\draw  (1,-2) -- (0,-3)[postaction={decorate, decoration={markings,mark=at position .6 with {\arrow[black]{stealth}}}}];

\draw  (1,-2) -- (2,-3)[postaction={decorate, decoration={markings,mark=at position .6 with {\arrow[black]{stealth}}}}];
\draw  plot[smooth, tension=.4] coordinates {(v1) (v2) (v3) (v4)}[postaction={decorate, decoration={markings,mark=at position .5 with {\arrow[black]{stealth}}}}];
\draw  plot[smooth, tension=.4] coordinates {(1,1) (v5) (v6) (1,-2)}[postaction={decorate, decoration={markings,mark=at position .5 with {\arrow[black]{stealth}}}}];
\end{tikzpicture}\end{matrix}\end{matrix}$
and
$\begin{matrix}(\Gamma_4,\lambda_4, P_{in}^{full},P_{out}^{full})&=&\begin{matrix}\begin{tikzpicture}[scale=0.5]

\node (v2) at (-0.5,0) {};
\draw [fill] (-0.5,0) circle [radius=0.1];
\node at (1.5,0) {$g_1\circ_{\epsilon}f_1$};
\node (v1) at (-2,1.5) {$(x_1)$};
\node (v4) at (1,1.5) {$(x_l)$};
\node (v3) at (-0.5,1.5) {$\cdots$};
\node (v5) at (-2,-1.5) {$(z_1)$};
\node (v7) at (1,-1.5) {$(z_n)$};
\node (v6) at (-0.5,-1.5) {$\cdots$};

\node (v8) at (2.5,1.5) {$(u_1)$};
\node (v10) at (4,1.5) {$\cdots$};
\node (v11) at (5.5,1.5) {$(u_p)$};
\node (v9) at (4,0) {};
\draw [fill] (4,0) circle [radius=0.1];
\node at (5.8,0) {$g_2\circ_{\epsilon}f_2$};
\node (v12) at (2.5,-1.5) {$(w_1)$};
\node (v13) at (4,-1.5) {$\cdots$};
\node (v14) at (5.5,-1.5) {$(w_r).$};
\draw  (v1) -- (-0.5,0)[postaction={decorate, decoration={markings,mark=at position .5 with {\arrow[black]{stealth}}}}];

\draw  (v4) -- (-0.5,0)[postaction={decorate, decoration={markings,mark=at position .5 with {\arrow[black]{stealth}}}}];
\draw  (-0.5,0) -- (v5)[postaction={decorate, decoration={markings,mark=at position .6 with {\arrow[black]{stealth}}}}];

\draw  (-0.5,0) -- (v7)[postaction={decorate, decoration={markings,mark=at position .6 with {\arrow[black]{stealth}}}}];
\draw  (v8) -- (4,0)[postaction={decorate, decoration={markings,mark=at position .5 with {\arrow[black]{stealth}}}}];

\draw  (v11) -- (4,0)[postaction={decorate, decoration={markings,mark=at position .5 with {\arrow[black]{stealth}}}}];
\draw  (4,0) -- (v12)[postaction={decorate, decoration={markings,mark=at position .6 with {\arrow[black]{stealth}}}}];

\draw  (4,0) -- (v14)[postaction={decorate, decoration={markings,mark=at position .6 with {\arrow[black]{stealth}}}}];
\end{tikzpicture}\end{matrix}\end{matrix}$

\end{proof}

Now for any pair of linear partitions $I,O$, we define a morphism $\ast^I_O:Mor(\mathcal{D})\rightarrow Mor(\mathcal{D})$ (called \textbf{fusion}) as following:

$\bullet$ if $|dom(f)|=||I||, |cod(f)|=||O||$ then $$\ast^I_O(f)\triangleq\epsilon_m([\Gamma_f,I,O]),$$ where $\Gamma_f$ is a prime diagram with the vertex decorated by $f$ and its fissus structure is equivalent to $I$ and $O$;

$\bullet$  if $|dom(f)|\neq||I||$ or $|cod(f)|\neq||O||$, then $\ast^I_O(f)\triangleq\epsilon_m(\emptyset)$, i.e, the image of empty diagram under $\epsilon_m$.

\begin{prop}\label{fusion}
Let $I_1$ and $O_1$ are linear partitions of $I_2$ and $O_2$, respectively, then $$\ast^{I_1}_{O_1}\circ \ast^{I_2}_{O_2}=\ast^{I_1\triangleleft I_2}_{O_1\triangleleft O_2}.$$
\end{prop}
\begin{proof}
Let $f:x_1\cdots x_m\rightarrow y_1\cdots y_n$ be a morphism in $\mathcal{D}$. If $||I_2||\neq m$ or $||O_2||\neq n$, then $$\ast^{I_1}_{O_1}\circ \ast^{I_2}_{O_2}(f)=\ast^{I_1\triangleleft I_2}_{O_1\triangleleft O_2}(f)=\epsilon_m(\emptyset).$$

Now we assume $||I_2||= m$ and $||O_2||= n$, then the fact $\ast^{I_1}_{O_1}\circ \ast^{I_2}_{O_2}(f)=\ast^{I_1\triangleleft I_2}_{O_1\triangleleft O_2}(f)$ can be deduced from the following commutative diagram:

$$
\begin{matrix}

\begin{matrix}
[\Gamma_1,\lambda_1, I_1,O_1]
\end{matrix}

&\begin{matrix}
\begin{tikzpicture}
\node at (0.5,0.2) {$\mu_{\mathcal{D}}$};
\draw [->] (0,0)--(1,0);
\end{tikzpicture}
\end{matrix}
&
\begin{matrix}
[\Gamma_2,\gamma_f, I_1\triangleleft I_2,O_1\triangleleft O_2]
\end{matrix}
\\

\begin{matrix}
\begin{tikzpicture}
\node at (0.7,-0.5) {$T(\epsilon)$};
\draw [->] (0,0)--(0,-1);
\end{tikzpicture}
\end{matrix}

&\begin{matrix}

\end{matrix}
&
\begin{matrix}
\begin{tikzpicture}
\node at (0.2,-0.5) {$\epsilon$};
\draw [->] (0,0)--(0,-1);
\end{tikzpicture}
\end{matrix}

\\

\begin{matrix}
[\Gamma_1,\gamma_{\epsilon_m([\Gamma_2,\gamma_f,I_2,O_2])}, I_1,O_1]
\end{matrix}
&
\begin{matrix}
\begin{tikzpicture}
\node at (0.5,0.2) {$\epsilon$};
\draw [->] (0,0)--(1,0);
\end{tikzpicture}
\end{matrix}
&
\begin{matrix}
\widetilde{f},
\end{matrix}

\end{matrix}
$$
where $[\Gamma,\gamma_f]$ is a $(m,n)$-prime diagram with vertex decorated by $f$, $\Gamma_1=\underbrace{[\Gamma_2,I_2,O_2]}$ being the coarse-graining of $[\Gamma_2,I_2,O_2]$ and with its unique vertex deocorated by the fissus prime diagram $[\Gamma_2,\gamma_f,I_2,O_2]$.

\end{proof}

\begin{prop}\label{fusion-tensor product}
The fusion operation is compatible with the tensor product. More precisely, if $|dom(f)|=||I_1||$, $|cod(f)|=||O_1||$, $|dom(g)|=||I_2||$ and $|cod(g)|=||O_2||$, then
$$\ast^{I_1\otimes I_2}_{O_1\otimes O_2 }(f\otimes_{\epsilon} g)=\ast^{I_1}_{O_1}(f)\otimes_{\epsilon}\ast^{I_2}_{O_2}(g).$$
\end{prop}
\begin{proof}
Let $f:x_1\cdots x_m\rightarrow y_1\cdots y_n$, $g:z_1\cdots z_k\rightarrow w_1\cdots w_l$ be two morphisms in $\mathcal{D}$. We assume $|I_1|=\mu_1$, $|O_1|=\nu_1$, $|I_2|=\mu_2$, $|O_2|=\nu_2$. Then the fact that $\ast^{I_1\otimes I_2}_{O_1\otimes O_2 }(f\otimes_{\epsilon} g)=\ast^{I_1}_{O_1}(f)\otimes_{\epsilon}\ast^{I_2}_{O_2}(g)$ is the direct consequence of the following commutative diagram:

$$
\begin{matrix}

\begin{matrix}
[\Gamma_1,\lambda_1, I^{full},O^{full}]
\end{matrix}

&\begin{matrix}
\begin{tikzpicture}
\node at (0.5,0.2) {$\mu_{\mathcal{D}}$};
\draw [->] (0,0)--(1,0);
\end{tikzpicture}
\end{matrix}
&
\begin{matrix}
[\Gamma_2,\lambda_2, I_1\otimes I_2,O_1\otimes O_2]
\end{matrix}
\\

\begin{matrix}
\begin{tikzpicture}
\node at (0.7,-0.5) {$T(\epsilon)$};
\draw [->] (0,0)--(0,-1);
\end{tikzpicture}
\end{matrix}

&\begin{matrix}

\end{matrix}
&
\begin{matrix}
\begin{tikzpicture}
\node at (0.2,-0.5) {$\epsilon$};
\draw [->] (0,0)--(0,-1);
\end{tikzpicture}
\end{matrix}

\\

\begin{matrix}
[\Gamma_3,\lambda_3, I^{full},O^{full}]
\end{matrix}
&
\begin{matrix}
\begin{tikzpicture}
\node at (0.5,0.2) {$\epsilon$};
\draw [->] (0,0)--(1,0);
\end{tikzpicture}
\end{matrix}
&
\begin{matrix}
h,
\end{matrix}

\end{matrix}
$$

where $\begin{matrix}\Gamma_1&=&\begin{matrix}\begin{tikzpicture}[scale=0.5]

\node (v2) at (-0.5,0) {};
\draw [fill] (-0.5,0) circle [radius=0.1];
\node at (0.3,0) {$p_1$};
\node (v1) at (-2,1.5) {$1$};
\node (v4) at (1,1.5) {$\mu_1$};
\node (v3) at (-0.5,1.5) {$\cdots$};
\node (v5) at (-2,-1.5) {$1$};
\node (v7) at (1,-1.5) {$\nu_1$};
\node (v6) at (-0.5,-1.5) {$\cdots$};

\node (v8) at (2.5,1.5) {$1$};
\node (v10) at (4,1.5) {$\cdots$};
\node (v11) at (5.5,1.5) {$\mu_2$};
\node (v9) at (4,0) {};
\draw [fill] (4,0) circle [radius=0.1];
\node  at (4.8,0) {$p_2$};
\node (v12) at (2.5,-1.5) {$1$};
\node (v13) at (4,-1.5) {$\cdots$};
\node (v14) at (5.5,-1.5) {$\nu_2$};
\draw  (v1) -- (-0.5,0)[postaction={decorate, decoration={markings,mark=at position .5 with {\arrow[black]{stealth}}}}];

\draw  (v4) -- (-0.5,0)[postaction={decorate, decoration={markings,mark=at position .5 with {\arrow[black]{stealth}}}}];
\draw  (-0.5,0) -- (v5)[postaction={decorate, decoration={markings,mark=at position .6 with {\arrow[black]{stealth}}}}];

\draw  (-0.5,0) -- (v7)[postaction={decorate, decoration={markings,mark=at position .6 with {\arrow[black]{stealth}}}}];
\draw  (v8) -- (4,0)[postaction={decorate, decoration={markings,mark=at position .5 with {\arrow[black]{stealth}}}}];

\draw  (v11) -- (4,0)[postaction={decorate, decoration={markings,mark=at position .5 with {\arrow[black]{stealth}}}}];
\draw  (4,0) -- (v12)[postaction={decorate, decoration={markings,mark=at position .6 with {\arrow[black]{stealth}}}}];

\draw  (4,0) -- (v14)[postaction={decorate, decoration={markings,mark=at position .6 with {\arrow[black]{stealth}}}}];
\end{tikzpicture}\end{matrix}\end{matrix}$
and $$\begin{matrix}(\lambda_1)_m(p_1)&=&[\begin{matrix}\begin{tikzpicture}[scale=0.5]

\node (v2) at (-0.5,0) {};
\draw [fill] (-0.5,0) circle [radius=0.1];
\node at (0,0) {$f$};
\node (v1) at (-2,1.5) {$x_1$};
\node (v4) at (1,1.5) {$x_m$};
\node (v3) at (-0.5,1.5) {$\cdots$};
\node (v5) at (-2,-1.5) {$y_1$};
\node (v7) at (1,-1.5) {$y_n,$};
\node (v6) at (-0.5,-1.5) {$\cdots$};

\draw  (v1) -- (-0.5,0)[postaction={decorate, decoration={markings,mark=at position .5 with {\arrow[black]{stealth}}}}];

\draw  (v4) -- (-0.5,0)[postaction={decorate, decoration={markings,mark=at position .5 with {\arrow[black]{stealth}}}}];
\draw  (-0.5,0) -- (v5)[postaction={decorate, decoration={markings,mark=at position .6 with {\arrow[black]{stealth}}}}];
\draw  (-0.5,0) -- (v7)[postaction={decorate, decoration={markings,mark=at position .6 with {\arrow[black]{stealth}}}}];
\end{tikzpicture}\end{matrix}, I_1,O_1],\end{matrix}$$

$$\begin{matrix}(\lambda_1)_m(p_2)&=&[\begin{matrix}\begin{tikzpicture}[scale=0.5]

\node (v2) at (-0.5,0) {};
\draw [fill] (-0.5,0) circle [radius=0.1];
\node at (0,0) {$g$};
\node (v1) at (-2,1.5) {$z_1$};
\node (v4) at (1,1.5) {$z_k$};
\node (v3) at (-0.5,1.5) {$\cdots$};
\node (v5) at (-2,-1.5) {$w_1$};
\node (v7) at (1,-1.5) {$w_l,$};
\node (v6) at (-0.5,-1.5) {$\cdots$};

\draw  (v1) -- (-0.5,0)[postaction={decorate, decoration={markings,mark=at position .5 with {\arrow[black]{stealth}}}}];

\draw  (v4) -- (-0.5,0)[postaction={decorate, decoration={markings,mark=at position .5 with {\arrow[black]{stealth}}}}];
\draw  (-0.5,0) -- (v5)[postaction={decorate, decoration={markings,mark=at position .6 with {\arrow[black]{stealth}}}}];
\draw  (-0.5,0) -- (v7)[postaction={decorate, decoration={markings,mark=at position .6 with {\arrow[black]{stealth}}}}];
\end{tikzpicture}\end{matrix}, I_2,O_2];\end{matrix}$$

$\begin{matrix}[\Gamma_2,\lambda_2]&=&\begin{matrix}\begin{tikzpicture}[scale=0.5]

\node (v2) at (-0.5,0) {};
\draw [fill] (-0.5,0) circle [radius=0.1];
\node at (0.3,0) {$f$};
\node (v1) at (-2,1.5) {$x_1$};
\node (v4) at (1,1.5) {$x_m$};
\node (v3) at (-0.5,1.5) {$\cdots$};
\node (v5) at (-2,-1.5) {$y_1$};
\node (v7) at (1,-1.5) {$y_n$};
\node (v6) at (-0.5,-1.5) {$\cdots$};

\node (v8) at (2.5,1.5) {$z_1$};
\node (v10) at (4,1.5) {$\cdots$};
\node (v11) at (5.5,1.5) {$z_k$};
\node (v9) at (4,0) {};
\draw [fill] (4,0) circle [radius=0.1];
\node  at (4.8,0) {$g$};
\node (v12) at (2.5,-1.5) {$w_1$};
\node (v13) at (4,-1.5) {$\cdots$};
\node (v14) at (5.5,-1.5) {$w_l,$};
\draw  (v1) -- (-0.5,0)[postaction={decorate, decoration={markings,mark=at position .5 with {\arrow[black]{stealth}}}}];

\draw  (v4) -- (-0.5,0)[postaction={decorate, decoration={markings,mark=at position .5 with {\arrow[black]{stealth}}}}];
\draw  (-0.5,0) -- (v5)[postaction={decorate, decoration={markings,mark=at position .6 with {\arrow[black]{stealth}}}}];

\draw  (-0.5,0) -- (v7)[postaction={decorate, decoration={markings,mark=at position .6 with {\arrow[black]{stealth}}}}];
\draw  (v8) -- (4,0)[postaction={decorate, decoration={markings,mark=at position .5 with {\arrow[black]{stealth}}}}];

\draw  (v11) -- (4,0)[postaction={decorate, decoration={markings,mark=at position .5 with {\arrow[black]{stealth}}}}];
\draw  (4,0) -- (v12)[postaction={decorate, decoration={markings,mark=at position .6 with {\arrow[black]{stealth}}}}];

\draw  (4,0) -- (v14)[postaction={decorate, decoration={markings,mark=at position .6 with {\arrow[black]{stealth}}}}];
\end{tikzpicture}\end{matrix}\end{matrix}$

$\begin{matrix}\Gamma_3&=&\begin{matrix}\begin{tikzpicture}[scale=0.5]

\node (v2) at (-0.5,0) {};
\draw [fill] (-0.5,0) circle [radius=0.1];
\node at (0.3,0) {$p_3$};
\node (v1) at (-2,1.5) {$1$};
\node (v4) at (1,1.5) {$\mu_1$};
\node (v3) at (-0.5,1.5) {$\cdots$};
\node (v5) at (-2,-1.5) {$1$};
\node (v7) at (1,-1.5) {$\nu_1$};
\node (v6) at (-0.5,-1.5) {$\cdots$};

\node (v8) at (2.5,1.5) {$1$};
\node (v10) at (4,1.5) {$\cdots$};
\node (v11) at (5.5,1.5) {$\mu_2$};
\node (v9) at (4,0) {};
\draw [fill] (4,0) circle [radius=0.1];
\node  at (4.8,0) {$p_4$};
\node (v12) at (2.5,-1.5) {$1$};
\node (v13) at (4,-1.5) {$\cdots$};
\node (v14) at (5.5,-1.5) {$\nu_2$};
\draw  (v1) -- (-0.5,0)[postaction={decorate, decoration={markings,mark=at position .5 with {\arrow[black]{stealth}}}}];

\draw  (v4) -- (-0.5,0)[postaction={decorate, decoration={markings,mark=at position .5 with {\arrow[black]{stealth}}}}];
\draw  (-0.5,0) -- (v5)[postaction={decorate, decoration={markings,mark=at position .6 with {\arrow[black]{stealth}}}}];

\draw  (-0.5,0) -- (v7)[postaction={decorate, decoration={markings,mark=at position .6 with {\arrow[black]{stealth}}}}];
\draw  (v8) -- (4,0)[postaction={decorate, decoration={markings,mark=at position .5 with {\arrow[black]{stealth}}}}];

\draw  (v11) -- (4,0)[postaction={decorate, decoration={markings,mark=at position .5 with {\arrow[black]{stealth}}}}];
\draw  (4,0) -- (v12)[postaction={decorate, decoration={markings,mark=at position .6 with {\arrow[black]{stealth}}}}];

\draw  (4,0) -- (v14)[postaction={decorate, decoration={markings,mark=at position .6 with {\arrow[black]{stealth}}}}];
\end{tikzpicture}\end{matrix}\end{matrix}$
and $$(\lambda_3)_m(p_3)=\ast^{I_1}_{O_1}(f),$$

$$(\lambda_3)_m(p_4)=\ast^{I_2}_{O_2}(g).$$

\end{proof}

\begin{prop}\label{fusion-composition}
The fusion operation is compatible with the composition. More precisely, if

$\bullet$ $cod(f)=dom(g)$;

$\bullet$ $|dom(f)|=|I_1|$, $|cod(f)|=|O_1|$, $|dom(g)|=|I_2|$, $|cod(g)|=|O_2|$;

$\bullet$  $O_1\approx I_2$,

then
$$\ast^{I_1}_{ O_2}(g\circ_{\epsilon} f)=\ast^{I_2}_{O_2}(f)\circ_{\epsilon}\ast^{I_1}_{O_1}(f).$$
\end{prop}

\begin{proof}
Let $f:x_1\cdots x_l\rightarrow y_1\cdots y_m$, $g:y_1\cdots y_m\rightarrow z_1\cdots z_n$ be two morphisms in $\mathcal{D}$. We assume $|I_1|=p$, $|O_1|=q$, $|I_2|=q$, $|O_2|=r$. Then the fact that $\ast^{I_1}_{ O_2}(g\circ_{\epsilon} f)=\ast^{I_2}_{O_2}(f)\circ_{\epsilon}\ast^{I_1}_{O_1}(f)$ is the direct consequence of the following commutative diagram:

$$
$
and $$(\lambda_3)_m(v_3)=\ast^{I_1}_{O_1}(f),$$

$$(\lambda_3)_m(v_4)=\ast^{I_2}_{O_2}(g).$$

\end{proof}

Recall that  for any object $x\in Ob(\mathcal{D})$, we define its identity morphism $Id_x=\epsilon_m(%
).$
Now for any word $x_1\cdots x_m\in W(Ob(\mathcal{D}))$, we define its \textbf{identity morphism} to be the morphism $$Id_{x_1\cdots x_m}=\epsilon_m %
$$

\end{proof}

\begin{prop} \label{unit}
Let $f:x_1\cdots x_m\rightarrow y_1\cdots y_n$  be a morphism in $\mathcal{D}$, then $$f\circ_{\epsilon}Id_{x_1\cdots x_m}=f,\ \ Id_{y_1\cdots y_n}\circ_{\epsilon}f=f.$$
\end{prop}

\begin{proof}
The unitary law is a direct consequence of the following commutative diagram:
$$
%
$

\end{proof}

Note that for the unit object $1_{\mathcal{D}}=\epsilon_o(\varnothing)$, $Id_{1_{\mathcal{D}}}=\epsilon_m(%
)=\epsilon_m(\varnothing),$ i.e., the image of empty diagram/graph.

Now we can define a strict tensor category $\widetilde{\mathcal{D}}$ with $Mor(\widetilde{\mathcal{D}})=Mor(\mathcal{D}),$ $Ob(\widetilde{\mathcal{D}})=W(Ob(\mathcal{D}))$. The source and target are same as those of $\mathcal{D}$. The tensor product and composition of morphisms are given by $\otimes_{\epsilon}$ and $\circ_{\epsilon}$, respectively. The tensor product of objects are given by the juxtaposition of words. The null string $\varnothing$ is the unit object, and for any word $x_1\cdots x_m$ its identity morphism is $Id_{x_1\cdots x_m}.$  So $\widetilde{\mathcal{D}}$ is free on the level of objects.

\begin{thm}
Let $\mathcal{D}$ be a tensor scheme equipped with a family of identity morphisms one for each words in $Ob(\mathcal{D})$ and a family of operations $\otimes$, $\circ$ , $\ast_O^I$ for any pair of linear partitions $(I,O)$ such that their properties in proposition \ref{tensor product}---\ref{unit} are satisfied, then there exists an unique $T$-algebra structure inducing the same family of operations.
\end{thm}
\begin{proof}[Sketch of proof]
First of all, we want to define an evaluation map $\varepsilon:\mathsf{\Gamma}(\mathcal{D})\rightarrow Mor(\mathcal{D})$ which is exactly the evaluation map of the strict tensor category $\widetilde{\mathcal{D}}$.  Secondly, for any fissus diagram $([\Gamma, \gamma],I,O)\in \mathsf{\Gamma}_F$, we define $$\epsilon_m(([\Gamma, \gamma],I,O))=\ast_O^I\circ \varepsilon([\Gamma,\gamma]).$$

Thirdly, for any word $x_1\cdots x_m$, we define $\epsilon_o(x_1\cdots x_m)=dom(\ast^{I^{trivial}}_{O^{trivial}}(Id_{x_1\cdots x_m})),$ where $I^{trivial}, O^{trivial}$ are trivial linear partitions with $||I||=||O||=m.$

Then it is routine to check that the pair $\epsilon=(\epsilon_m,\epsilon_o): T(\mathcal{D})\rightarrow \mathcal{D}$ defines a $T$-algebra structure on $\mathcal{D}.$

The uniqueness of $\epsilon$ is evident.

\end{proof}

As a direct consequence, we have
\begin{prop}\label{morphism}
Let $(\mathcal{D}_1,\epsilon_1)$ and $(\mathcal{D}_2,\epsilon_2)$ be two $T$-algebras, then a morphism $\varphi:\mathcal{D}_1\rightarrow \mathcal{D}_2$ of tensor schemes is a morphism of $T$-algebras if and only if $\varphi$ preserves their identity morphisms and their operations of tensor products, compositions and fusions.
\end{prop}

The following corollary is  a direct consequence of proposition \ref{fusion} and \ref{fusion-tensor product}.
\begin{cor}
If $|dom(f)|=||I_1||$, $|cod(f)|=||O_1||$, $|dom(g)|=||I_2||$, $|cod(g)|=||O_2||$,  $I$ is a linear partition of  $I_1\otimes I_2$ and $O$ is a linear partition of $O_1\otimes O_2$.  Then $$\ast^{I\triangleleft(I_1\otimes I_2)}_{O\triangleleft(O_1\otimes O_2)}(f\otimes_{\epsilon}g)=\ast^I_O(\ast^{I_1}_{O_1}(f)\otimes_{\epsilon}\ast^{I_2}_{O_2}(g)).$$
\end{cor}

The following corollary is  a direct consequence of proposition \ref{fusion} and \ref{fusion-composition}.
\begin{cor}
If $cod(f)=dom(g)$, $|dom(f)|=||I_1||$, $|cod(f)=||O_2||$, $|dom(g)|=||I_2||$, $|cod(g)|=||O_2||$, $O_1\approx I_2$,  $I$ is a linear partition of  $I_1$ and $O$ is a linear partition of $O_2$.  Then $$\ast^{I\triangleleft I_1}_{O\triangleleft O_2}(f\circ_{\epsilon}g)=\ast^I_O(\ast^{I_1}_{O_1}(f)\circ_{\epsilon}\ast^{I_2}_{O_2}(g)).$$
\end{cor}

\subsection{Left inverse of the comparison and non-monadic property }

Now we want to investigate the relation between $\mathbf{Str.T}$ and $\mathbf{T.Sch}^{T}$. Given the monad $(T,\mu,\eta)$ in $\mathbf{T.Sch}$, there is an adjunction $$(F^T,U^T, \varepsilon^T, \eta^T):\mathbf{T.Sch}\rightharpoonup \mathbf{T.Sch}^T$$
where the functors $U^T$ and $F^T$ are given by the respective assignments

$$
\begin{matrix}
U^T\begin{pmatrix}
\xymatrix{(\mathcal{D}_1,\epsilon_1)\ar[d]^{\varphi}\\ (\mathcal{D}_2,\epsilon_2)}
\end{pmatrix}
=
\begin{matrix}
\xymatrix{\mathcal{D}_1\ar[d]^{\varphi}\\ \mathcal{D}_2,}
\end{matrix}
&
F^T\begin{pmatrix}
\xymatrix{\mathcal{D}_1\ar[d]^{\varphi}\\ \mathcal{D}_2}
\end{pmatrix}
=
\begin{matrix}
\xymatrix{(T\mathcal{D}_1,\mu_{\mathcal{D}_1})\ar[d]^{T\varphi}\\ (T\mathcal{D}_2,\mu_{\mathcal{D}_2}),}
\end{matrix}
\end{matrix}
$$
while $\eta^T=\eta$ and $\varepsilon^T_{(\mathcal{D},\epsilon)}=\epsilon$ for each  $T$-algebra $(\mathcal{D},\epsilon)$.

By the general theory there is an unique comparison functor $$\Phi:\mathbf{Str.T}\rightarrow \mathbf{T.Sch}^{T},$$ such that $F^T=\Phi\circ F$ and $U=U^T\circ \Phi$, that is, in the following diagram both the $F$-square and the $U$-square commute
$$\xymatrix{\mathbf{Str.T}\ar@{-->}[rr]^{\Phi}\ar@<0.8mm>[d]^{U}&&\ar@<0.8mm>[d]^{U^T}\mathbf{T.Sch}^{T}\\\mathbf{T.Sch}\ar@<0.8mm>[u]^{F}\ar@{=}[rr]&&\ar@<0.8mm>[u]^{F^T}\mathbf{T.Sch}.}$$
The functor $\Phi$ sends every strict tensor category $\mathcal{V}$ to a $T$-algebra $(U(\mathcal{V}),\epsilon)$ with structure map $\epsilon=U\varepsilon_{\mathcal{V}}:T(U(\mathcal{V}))\rightarrow U(\mathcal{V}).$ In example \ref{U(V)}, we have given a precise description of $(U(\mathcal{V}),\epsilon)$.

\begin{thm}
$\Phi$ is an embedding, that is,  fully  faithful on the level of morphisms and faithful on the level of objects.
\end{thm}

\begin{proof}[Sketch of proof ] We prove this theorem by constructing a functor $$\Psi:\mathbf{T.Sch}^{T}\rightarrow \mathbf{Str.T}$$ and showing that it is a right inverse of $\Phi$, that is, $$\Psi\circ \Phi\cong I_{\mathbf{Str.T}}.$$

Let $(\mathcal{D},\epsilon)$ be an $T$-algebra,  we define a strict tensor category $$\Psi((\mathcal{D},\epsilon))=(\mathcal{V},\otimes_{\mathcal{V}},\circ_{\mathcal{V}}, 1_{\mathcal{V}})$$  as follows:

$\bullet$ $Ob(\mathcal{V})=Ob(\mathcal{D});$

$\bullet$ $Mor_{\mathcal{V}}(x,y)=Mor_{\mathcal{D}}(x,y)$ for $x,y\in Ob(\mathcal{V});$

$\bullet$ the unit object $1_{\mathcal{V}}$ is defined to be $\epsilon_o(\varnothing)$ with $\varnothing\in Ob(T(\mathcal{D}))$ being the null string;

$\bullet$ the source and target maps of $\mathcal{V}$
$$\xymatrix{Mor(\mathcal{V})\ar@<1mm>[r]^{s}\ar@<-1mm>[r]_{t}&Ob(\mathcal{V})}$$
are induced from those of $\mathcal{D}$;

$\bullet$ for every $x,y\in Ob(\mathcal{V})$, $x\otimes_{\mathcal{V}} y=\epsilon_o(xy)$ with $xy\in Ob(T(\mathcal{D}))$;

$\bullet$ for $f:x\rightarrow y,\  g:z\rightarrow w\in Mor(\mathcal{V})$,

$$f\otimes_{\mathcal{V}} g=\epsilon_m(\begin{matrix}\begin{tikzpicture}[scale=.5]
\node (v2) at (-0.5,0.5) {};
\node (v1) at (-0.5,2) {$(x$};
\node (v3) at (-0.5,-1) {$(y$};
\node (v4) at (1.5,2) {$z)$};
\node (v6) at (1.5,0.5) {};
\node (v5) at (1.5,-1) {$w)$};
\draw[fill] (v2) circle [radius=0.1];
\draw[fill] (v6) circle [radius=0.1];
\draw  (v1) -- (v3)[postaction={decorate, decoration={markings,mark=at position .30 with {\arrow[black]{stealth}}}}][postaction={decorate, decoration={markings,mark=at position .85 with {\arrow[black]{stealth}}}}];
\draw  (v4) -- (v5)[postaction={decorate, decoration={markings,mark=at position .30 with {\arrow[black]{stealth}}}}][postaction={decorate, decoration={markings,mark=at position .85 with {\arrow[black]{stealth}}}}];

\node  at (0,0.5) {$f$};
\node  at (2,0.5) {$g$};
\end{tikzpicture}\end{matrix});$$

$\bullet$  for $f:x\rightarrow y,\  g:y\rightarrow z\in Mor(\mathcal{V})$,

$$g\circ_{\mathcal{V}} f=\epsilon_m(\begin{matrix}
\begin{tikzpicture}[scale=.5]
\node (v1) at (0,1.5) {$(x)$};
\node (v2) at (0,-3.5) {$(z)$};
\draw  (v1)-- (v2)[postaction={decorate, decoration={markings,mark=at position .93 with {\arrow[black]{stealth}}}}][postaction={decorate, decoration={markings,mark=at position .15 with {\arrow[black]{stealth}}}}][postaction={decorate, decoration={markings,mark=at position .55 with {\arrow[black]{stealth}}}}];
\draw[fill] (0,0) circle [radius=0.1];
\draw[fill] (0,-2) circle [radius=0.1];
\node at (0.5,-1) {$y$};
\node at (0.5,0) {$f$};
\node at (0.5,-2) {$g$};
\end{tikzpicture}
\end{matrix});$$

$\bullet$ for every object $x\in Ob(\mathcal{V})$, the identity morphism $Id_x$ is given by $\epsilon_m(\begin{matrix}
\begin{tikzpicture}[scale=.5]
\node (v2) at (-0.5,0.5) {};
\node (v1) at (-0.5,2) {$(x)$};
\node (v3) at (-0.5,-1) {$(x)$};
\draw (v2) circle [radius=0.1];
\draw  (v1) -- (-0.5,0.6)[postaction={decorate, decoration={markings,mark=at position .70 with {\arrow[black]{stealth}}}}];
\draw  (v3) -- (-0.5,0.4)[postaction={decorate, decoration={markings,mark=at position .70 with {\arrowreversed[black]{stealth}}}}];
\end{tikzpicture}
\end{matrix})$.

Using the fact that $(\mathcal{D},\epsilon)$ be a $T$-algebra, we can show that $\mathcal{V}$ is a well-defined  strict tensor category.  In fact,

$\bullet$ the  associative law of $\otimes$ is implied by the following commutative diagram:

$$
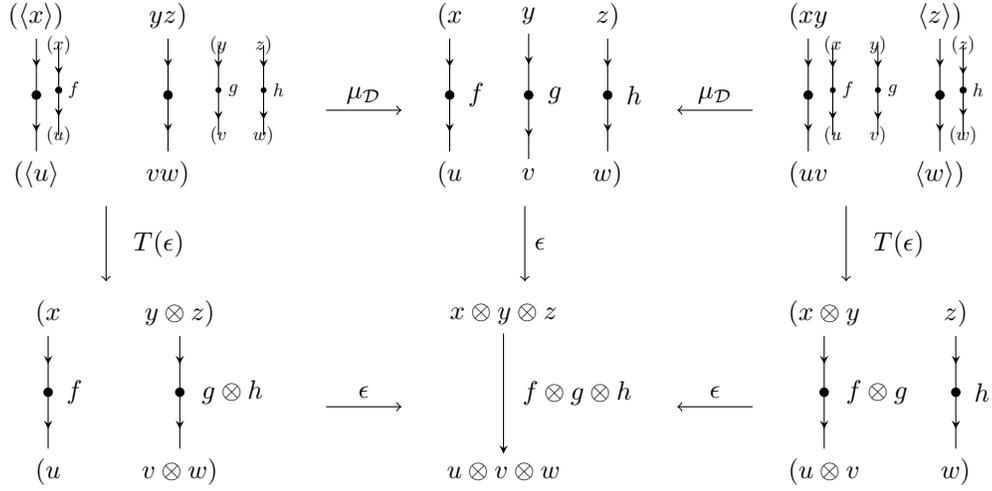
\begin{matrix}

\begin{matrix}
\begin{tikzpicture}[scale=.7]
\node (v2) at (0,0.5) {};
\node (v1) at (0,2) {$(\langle x\rangle)$};
\node (v3) at (0,-1) {$(\langle u\rangle$};
\node (v5) at (2.5,0.5) {};
\node (v4) at (2.5,2) {$yz)$};
\node (v6) at (2.5,-1) {$vw)$};

\node  at (0.5,0.5) {$\begin{matrix}\begin{tikzpicture}[scale=.4]

\node [scale=.7] at (-0.5,2) {$(x)$};
\node [scale=.7] at (-0.5,-1) {$(u)$};
\node [scale=.7] at (0,0.5) {$f$};
\draw [fill](-0.5,0.5) circle [radius=0.1];
\draw  (-0.5,2) -- (-0.5,0.6)[postaction={decorate, decoration={markings,mark=at position .70 with {\arrow[black]{stealth}}}}];
\draw  (-0.5,-1) -- (-0.5,0.4)[postaction={decorate, decoration={markings,mark=at position .70 with {\arrowreversed[black]{stealth}}}}];
\end{tikzpicture}\end{matrix}$};

\node  at (4,0.5) {$\begin{matrix}
\begin{tikzpicture}[scale=.4]
\node  at (0,0.5) {};
\node [scale=.7] at (0,2) {$(y$};
\node [scale=.7] at (0,-1) {$(v$};
\node at (1.5,0.5) {};
\node [scale=.7] at (1.5,2) {$z)$};
\node [scale=.7] at (1.5,-1) {$w)$};

\node [scale=.7] at (0.5,0.5) {$g$};
\node  [scale=.7] at (2,0.5) {$h$};
\draw [fill] (0,0.5) circle [radius=0.08];
\draw [fill] (1.5,0.5) circle [radius=0.08];
\draw (0,2) -- (0,0.5)[postaction={decorate, decoration={markings,mark=at position .50 with {\arrow[black]{stealth}}}}];
\draw  (0,0.5) -- (0,-1)[postaction={decorate, decoration={markings,mark=at position .65 with {\arrow[black]{stealth}}}}];
\draw  (1.5,2) -- (1.5,0.5)[postaction={decorate, decoration={markings,mark=at position .50 with {\arrow[black]{stealth}}}}];
\draw  (1.5,0.5) -- (1.5,-1)[postaction={decorate, decoration={markings,mark=at position .65 with {\arrow[black]{stealth}}}}];

\end{tikzpicture}\end{matrix}$};

\draw [fill] (0,0.5) circle [radius=0.08];
\draw [fill] (2.5,0.5) circle [radius=0.08];
\draw  (v1) -- (0,0.5)[postaction={decorate, decoration={markings,mark=at position .50 with {\arrow[black]{stealth}}}}];
\draw  (0,0.5) -- (v3)[postaction={decorate, decoration={markings,mark=at position .65 with {\arrow[black]{stealth}}}}];
\draw  (v4) -- (2.5,0.5)[postaction={decorate, decoration={markings,mark=at position .50 with {\arrow[black]{stealth}}}}];
\draw  (2.5,0.5) -- (v6)[postaction={decorate, decoration={markings,mark=at position .65 with {\arrow[black]{stealth}}}}];
\end{tikzpicture}
\end{matrix}

&\begin{matrix}
\begin{tikzpicture}
\node at (0.5,0.2) {$\mu_{\mathcal{D}}$};
\draw [->] (0,0)--(1,0);
\end{tikzpicture}
\end{matrix}
&
\begin{matrix}
\begin{tikzpicture}[scale=.7]
\node (v2) at (0,0.5) {};
\node (v1) at (0,2) {$(x$};
\node (v3) at (0,-1) {$(u$};
\node (v5) at (1.5,0.5) {};
\node (v4) at (1.5,2) {$y$};
\node (v6) at (1.5,-1) {$v$};
\node (v7) at (3,2) {$z)$};
\node (v8) at (3,0.5) {};
\node (v9) at (3,-1) {$w)$};
\node  at (0.5,0.5) {$f$};
\node  at (2,0.5) {$g$};
\node  at (3.5,0.5) {$h$};

\draw [fill] (0,0.5) circle [radius=0.08];
\draw [fill] (1.5,0.5) circle [radius=0.08];
\draw [fill] (3,0.5) circle [radius=0.08];
\draw  (v1) -- (0,0.5)[postaction={decorate, decoration={markings,mark=at position .50 with {\arrow[black]{stealth}}}}];
\draw  (0,0.5) -- (v3)[postaction={decorate, decoration={markings,mark=at position .65 with {\arrow[black]{stealth}}}}];
\draw  (v4) -- (1.5,0.5)[postaction={decorate, decoration={markings,mark=at position .50 with {\arrow[black]{stealth}}}}];
\draw  (1.5,0.5) -- (v6)[postaction={decorate, decoration={markings,mark=at position .65 with {\arrow[black]{stealth}}}}];
\draw  (v7) -- (3,0.5)[postaction={decorate, decoration={markings,mark=at position .50 with {\arrow[black]{stealth}}}}];
\draw  (3,0.5) -- (v9)[postaction={decorate, decoration={markings,mark=at position .65 with {\arrow[black]{stealth}}}}];
\end{tikzpicture}
\end{matrix}
&
\begin{matrix}
\begin{tikzpicture}
\node at (0.5,0.2) {$\mu_{\mathcal{D}}$};
\draw [->] (1,0)--(0,0);
\end{tikzpicture}
\end{matrix}
&
\begin{matrix}
\begin{tikzpicture}[scale=.7]
\node (v2) at (0,0.5) {};
\node (v1) at (0,2) {$(xy$};
\node (v3) at (0,-1) {$(uv$};
\node (v5) at (2.5,0.5) {};
\node (v4) at (2.5,2) {$\langle z\rangle)$};
\node (v6) at (2.5,-1) {$\langle w\rangle)$};

\node  at (1,0.5) {$\begin{matrix}
\begin{tikzpicture}[scale=.4]
\node  at (0,0.5) {};
\node [scale=.7] at (0,2) {$(x$};
\node [scale=.7] at (0,-1) {$(u$};
\node at (1.5,0.5) {};
\node [scale=.7] at (1.5,2) {$y)$};
\node [scale=.7] at (1.5,-1) {$v)$};

\node [scale=.7] at (0.5,0.5) {$f$};
\node [scale=.7] at (2,0.5) {$g$};
\draw [fill] (0,0.5) circle [radius=0.08];
\draw [fill] (1.5,0.5) circle [radius=0.08];
\draw (0,2) -- (0,0.5)[postaction={decorate, decoration={markings,mark=at position .50 with {\arrow[black]{stealth}}}}];
\draw  (0,0.5) -- (0,-1)[postaction={decorate, decoration={markings,mark=at position .65 with {\arrow[black]{stealth}}}}];
\draw  (1.5,2) -- (1.5,0.5)[postaction={decorate, decoration={markings,mark=at position .50 with {\arrow[black]{stealth}}}}];
\draw  (1.5,0.5) -- (1.5,-1)[postaction={decorate, decoration={markings,mark=at position .65 with {\arrow[black]{stealth}}}}];

\end{tikzpicture}\end{matrix}$};

\node  at (3,0.5) {$\begin{matrix}\begin{tikzpicture}[scale=.4]

\node [scale=.7] at (-0.5,2) {$(z)$};
\node [scale=.7] at (-0.5,-1) {$(w)$};
\node [scale=.7]  at (0,0.5) {$h$};
\draw [fill](-0.5,0.5) circle [radius=0.1];
\draw  (-0.5,2) -- (-0.5,0.6)[postaction={decorate, decoration={markings,mark=at position .70 with {\arrow[black]{stealth}}}}];
\draw  (-0.5,-1) -- (-0.5,0.4)[postaction={decorate, decoration={markings,mark=at position .70 with {\arrowreversed[black]{stealth}}}}];
\end{tikzpicture}\end{matrix}$};

\draw [fill] (0,0.5) circle [radius=0.08];
\draw [fill] (2.5,0.5) circle [radius=0.08];
\draw  (v1) -- (0,0.5)[postaction={decorate, decoration={markings,mark=at position .50 with {\arrow[black]{stealth}}}}];
\draw  (0,0.5) -- (v3)[postaction={decorate, decoration={markings,mark=at position .65 with {\arrow[black]{stealth}}}}];
\draw  (v4) -- (2.5,0.5)[postaction={decorate, decoration={markings,mark=at position .50 with {\arrow[black]{stealth}}}}];
\draw  (2.5,0.5) -- (v6)[postaction={decorate, decoration={markings,mark=at position .65 with {\arrow[black]{stealth}}}}];
\end{tikzpicture}
\end{matrix}\\



\end{matrix}
$$

$\bullet$ the  associative law of $\circ$ is implied by the following commutative diagram:

$$
%
$$

$\bullet$ the unitary law for identity morphisms  is implied by the following commutative diagram:

$$
%
$$

$\bullet$ the middle-four-interchange law of $\otimes$ and $\circ$ is implied by the following commutative diagram:

$$
%
$};

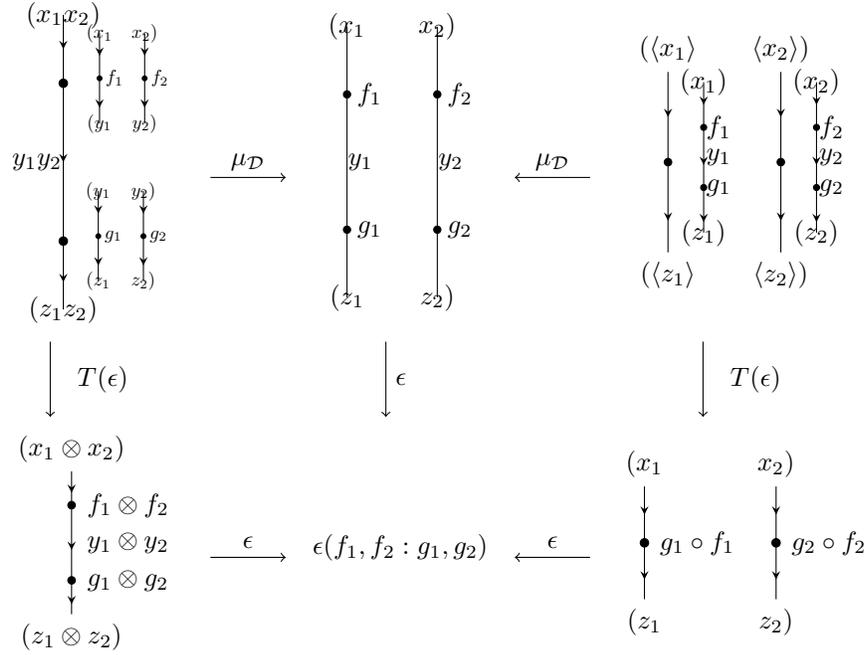
\draw [fill] (0,0.5) circle [radius=0.05];
\draw [fill] (1.5,0.5) circle [radius=0.05];
\draw  (v1) -- (0,0.5)[postaction={decorate, decoration={markings,mark=at position .50 with {\arrow[black]{stealth}}}}];
\draw  (0,0.5) -- (v3)[postaction={decorate, decoration={markings,mark=at position .65 with {\arrow[black]{stealth}}}}];
\draw  (v4) -- (1.5,0.5)[postaction={decorate, decoration={markings,mark=at position .50 with {\arrow[black]{stealth}}}}];
\draw  (1.5,0.5) -- (v6)[postaction={decorate, decoration={markings,mark=at position .65 with {\arrow[black]{stealth}}}}];
\end{tikzpicture}

\end{matrix}\\

\begin{matrix}
\begin{tikzpicture}
\node at (0.7,-0.5) {$T(\epsilon)$};
\draw [->] (0,0)--(0,-1);
\end{tikzpicture}
\end{matrix}

&\begin{matrix}

\end{matrix}
&
\begin{matrix}
\begin{tikzpicture}
\node at (0.2,-0.5) {$\epsilon$};
\draw [->] (0,0)--(0,-1);
\end{tikzpicture}
\end{matrix}
&
\begin{matrix}

\end{matrix}
&
\begin{matrix}
\begin{tikzpicture}
\node at (0.7,-0.5) {$T(\epsilon)$};
\draw [->] (0,0)--(0,-1);
\end{tikzpicture}
\end{matrix}\\

\begin{matrix}
\begin{tikzpicture}[scale=.5]
\node (v1) at (0,1.5) {$(x_1\otimes x_2)$};
\node (v2) at (0,-3.5) {$(z_1\otimes z_2)$};
\draw  (v1)-- (v2)[postaction={decorate, decoration={markings,mark=at position .93 with {\arrow[black]{stealth}}}}][postaction={decorate, decoration={markings,mark=at position .15 with {\arrow[black]{stealth}}}}][postaction={decorate, decoration={markings,mark=at position .55 with {\arrow[black]{stealth}}}}];
\draw[fill] (0,0) circle [radius=0.1];
\draw[fill] (0,-2) circle [radius=0.1];
\node at (1.5,-1) {$y_1\otimes y_2$};
\node at (1.5,0) {$f_1\otimes f_2$};
\node at (1.5,-2) {$g_1\otimes g_2$};
\end{tikzpicture}
\end{matrix}
&
\begin{matrix}
\begin{tikzpicture}
\node at (0.5,0.2) {$\epsilon$};
\draw [->] (0,0)--(1,0);
\end{tikzpicture}
\end{matrix}
&
\begin{matrix}
\epsilon(f_1,f_2: g_1,g_2)
\end{matrix}
&
\begin{matrix}
\begin{tikzpicture}
\node at (0.5,0.2) {$\epsilon$};
\draw [->] (1,0)--(0,0);
\end{tikzpicture}
\end{matrix}
&
\begin{matrix}
\begin{tikzpicture}[scale=.7]
\node (v2) at (0,0.5) {};
\node (v1) at (0,2) {$(x_1$};
\node (v3) at (0,-1) {$(z_1$};
\node (v5) at (2.5,0.5) {};
\node (v4) at (2.5,2) {$x_2)$};
\node (v6) at (2.5,-1) {$z_2)$};

\node  at (1,0.5) {$ g_1\circ f_1$};

\node  at (3.5,0.5) {$g_2\circ f_2$};

\draw [fill] (0,0.5) circle [radius=0.08];
\draw [fill] (2.5,0.5) circle [radius=0.08];
\draw  (v1) -- (0,0.5)[postaction={decorate, decoration={markings,mark=at position .50 with {\arrow[black]{stealth}}}}];
\draw  (0,0.5) -- (v3)[postaction={decorate, decoration={markings,mark=at position .65 with {\arrow[black]{stealth}}}}];
\draw  (v4) -- (2.5,0.5)[postaction={decorate, decoration={markings,mark=at position .50 with {\arrow[black]{stealth}}}}];
\draw  (2.5,0.5) -- (v6)[postaction={decorate, decoration={markings,mark=at position .65 with {\arrow[black]{stealth}}}}];
\end{tikzpicture}
\end{matrix}

\end{matrix}
$$

If  $\varphi:(\mathcal{D}_1,\epsilon_1)\rightarrow (\mathcal{D}_2,\epsilon_2)$ is a  morphism of $T$-algebras, it can naturally induces a strict tensor functor   $\Psi(\varphi):\Psi((\mathcal{D}_1,\epsilon_1))\rightarrow\Psi((\mathcal{D}_2,\epsilon_2))$ which can be directly checked using the fact that $\varphi$ is a morphism of $T$-algebras. In fact,  for any $f\in Mor_{\mathcal{D}_1}(x,y)$, $g\in Mor_{\mathcal{D}_1}(z,w)$ we have the following commutative diagram

$$
\begin{matrix}
\begin{matrix}
\begin{tikzpicture}[scale=.5]
\node (v2) at (-0.5,0.5) {};
\node (v1) at (-0.5,2) {$(x$};
\node (v3) at (-0.5,-1) {$(y$};
\node (v4) at (1.5,2) {$z)$};
\node (v6) at (1.5,0.5) {};
\node (v5) at (1.5,-1) {$w)$};
\draw[fill] (v2) circle [radius=0.1];
\draw[fill] (v6) circle [radius=0.1];
\draw  (v1) -- (v3)[postaction={decorate, decoration={markings,mark=at position .30 with {\arrow[black]{stealth}}}}][postaction={decorate, decoration={markings,mark=at position .85 with {\arrow[black]{stealth}}}}];
\draw  (v4) -- (v5)[postaction={decorate, decoration={markings,mark=at position .30 with {\arrow[black]{stealth}}}}][postaction={decorate, decoration={markings,mark=at position .85 with {\arrow[black]{stealth}}}}];

\node  at (0,0.5) {$f$};
\node  at (2,0.5) {$g$};
\end{tikzpicture}
\end{matrix}
&
\begin{matrix}
\begin{tikzpicture}
\node at (0.5,0.2) {$T(\varphi)$};
\draw [->] (0,0)--(1,0);
\end{tikzpicture}
\end{matrix}&
\begin{matrix}
\begin{tikzpicture}[scale=.5]
\node (v2) at (-0.5,0.5) {};
\node (v1) at (-0.5,2) {$(\varphi(x)$};
\node (v3) at (-0.5,-1) {$(\varphi(y)$};
\node (v4) at (1.5,2) {$\varphi(z))$};
\node (v6) at (1.5,0.5) {};
\node (v5) at (1.5,-1) {$\varphi(w))$};
\draw[fill] (v2) circle [radius=0.1];
\draw[fill] (v6) circle [radius=0.1];
\draw  (v1) -- (v3)[postaction={decorate, decoration={markings,mark=at position .30 with {\arrow[black]{stealth}}}}][postaction={decorate, decoration={markings,mark=at position .85 with {\arrow[black]{stealth}}}}];
\draw  (v4) -- (v5)[postaction={decorate, decoration={markings,mark=at position .30 with {\arrow[black]{stealth}}}}][postaction={decorate, decoration={markings,mark=at position .85 with {\arrow[black]{stealth}}}}];

\node  at (0.4,0.5) {$\varphi(f)$};
\node  at (2.4,0.5) {$\varphi(g)$};
\end{tikzpicture}
\end{matrix}\\

\begin{matrix}
\begin{tikzpicture}
\node at (0.2,-0.5) {$\epsilon_1$};
\draw [->] (0,0)--(0,-1);
\end{tikzpicture}
\end{matrix}&&
\begin{matrix}
\begin{tikzpicture}
\node at (0.2,-0.5) {$\epsilon_2$};
\draw [->] (0,0)--(0,-1);
\end{tikzpicture}
\end{matrix}
\\
\begin{matrix}
\xymatrix{x\otimes z\ar[d]^{f\otimes g}\\y\otimes w}
\end{matrix}&
\begin{matrix}
\begin{tikzpicture}
\node at (0.5,0.2) {$\varphi$};
\draw [->] (0,0)--(1,0);
\end{tikzpicture}
\end{matrix}&
\begin{matrix}
\xymatrix{\varphi(x)\otimes\varphi(z)\ar[d]^{\varphi(f)\otimes \varphi(g)}\\\varphi(y)\otimes \varphi(w)}
\end{matrix},
\end{matrix}
$$
which implies $\varphi(f\otimes g)=\varphi(f)\otimes \varphi(g)$; for any $f\in Mor_{\mathcal{D}_1}(x,y)$, $g\in Mor_{\mathcal{D}_1}(y,z)$ we have the following commutative diagram

$$
\begin{matrix}
\begin{matrix}
\begin{tikzpicture}[scale=.5]
\node (v1) at (0,1.5) {$(x)$};
\node (v2) at (0,-3.5) {$(z)$};
\draw  (v1)-- (v2)[postaction={decorate, decoration={markings,mark=at position .93 with {\arrow[black]{stealth}}}}][postaction={decorate, decoration={markings,mark=at position .15 with {\arrow[black]{stealth}}}}][postaction={decorate, decoration={markings,mark=at position .55 with {\arrow[black]{stealth}}}}];
\draw[fill] (0,0) circle [radius=0.1];
\draw[fill] (0,-2) circle [radius=0.1];
\node at (0.5,-1) {$y$};
\node at (0.5,0) {$f$};
\node at (0.5,-2) {$g$};
\end{tikzpicture}\ \ \
\end{matrix}
&
\begin{matrix}
\begin{tikzpicture}
\node at (0.5,0.2) {$T(\varphi)$};
\draw [->] (0,0)--(1,0);
\end{tikzpicture}
\end{matrix}&
\begin{matrix}
\begin{tikzpicture}[scale=.5]
\node (v1) at (0,1.5) {$(\varphi(x))$};
\node (v2) at (0,-3.5) {$(\varphi(z))$};
\draw  (v1)-- (v2)[postaction={decorate, decoration={markings,mark=at position .93 with {\arrow[black]{stealth}}}}][postaction={decorate, decoration={markings,mark=at position .15 with {\arrow[black]{stealth}}}}][postaction={decorate, decoration={markings,mark=at position .55 with {\arrow[black]{stealth}}}}];
\draw[fill] (0,0) circle [radius=0.1];
\draw[fill] (0,-2) circle [radius=0.1];
\node at (0.8,-1) {$\varphi(y)$};
\node at (0.8,0) {$\varphi(f)$};
\node at (0.8,-2) {$\varphi(g)$};
\end{tikzpicture}\ \ \
\end{matrix}\\

\begin{matrix}
\begin{tikzpicture}
\node at (0.2,-0.5) {$\epsilon_1$};
\draw [->] (0,0)--(0,-1);
\end{tikzpicture}
\end{matrix}&&
\begin{matrix}
\begin{tikzpicture}
\node at (0.2,-0.5) {$\epsilon_2$};
\draw [->] (0,0)--(0,-1);
\end{tikzpicture}
\end{matrix}
\\
\begin{matrix}
\xymatrix{x\ar[d]^{g\circ f}\\z}
\end{matrix}&
\begin{matrix}
\begin{tikzpicture}
\node at (0.5,0.2) {$\varphi$};
\draw [->] (0,0)--(1,0);
\end{tikzpicture}
\end{matrix}&
\begin{matrix}
\xymatrix{\varphi(x)\ar[d]^{ \varphi(g)\circ\varphi(f)}\\ \varphi(z)}
\end{matrix},
\end{matrix}
$$
which implies $\varphi(g\circ f)=\varphi(g)\circ \varphi(f)$.

Thus it is routine to check that $\Psi$ is a functor.  The last thing $\Psi\circ \Phi\cong I_{\mathbf{Str.T}}$ is evident from their definitions.
\end{proof}

\begin{prop}
The adjunction $\Theta:F\dashv U$ is not monadic.
\end{prop}

\begin{proof}
To see this, we take $\mathcal{V}=\textbf{1}$ be the strict tensor category with only one object $1$ and only one morphism $Id_1:1\rightarrow 1$, and notice that $\Phi(\textbf{1})=(U(\textbf{1}),\epsilon)\simeq (\textbf{Prim}(\textbf{1}),\varepsilon^{C-G})$ in example $6.1.6$. Take $\textbf{Prim}(\textbf{1})\subset\mathbf{S}(\textbf{1})\subset \mathbf{\Gamma}(\textbf{1})$ be a sub-T-algebra of $\mathbf{\Gamma}(\textbf{1})$ such that $Mor_{\mathbf{Prim}(\textbf{1})}(1,1)=Mor_{\mathbf{S}(\textbf{1})}(1,1)$ and $Mor_{\mathbf{Prim}(\textbf{1})}\neq Mor_{\mathbf{Prim}(\textbf{1})}.$ We see that $\mathbf{Prim}(\textbf{1})$ is not isomorphic to $\mathbf{S}(\textbf{1})$ and $\Psi(\mathbf{S}(\textbf{1}))=\textbf{1}$,
which implies that the "fiber" $\Psi^{-1}(\textbf{1})$ is not a groupoid, hence $\Theta:F\dashv U$ is not monadic.
\end{proof}

\subsection{Adjunction of comparison}

Now we define a map $\theta:Mor(D)\rightarrow Mor(\Psi((\mathcal{D},\epsilon)))$ which sends a morphism $f:x_1\cdots x_m\rightarrow y_1\cdots y_n$ in $\mathcal{D}$ to $\ast^{I^{trivial}}_{O^{trivial}}(f)$, that is, $\theta(f)=\epsilon_m\begin{pmatrix}\begin{tikzpicture}[scale=.3]
\node (v1) at (0,0) {};
\draw[fill] (0,0) circle [radius=0.08];
\node [scale=.7] at (0.8,0){$f$};
\node [scale=.7] (v2) at (-2,1.5) {$(x_1$};
\node [scale=.7](v3) at (-0.5,1.5) {$x_2$};
\node [scale=.7](v4) at (1,1.5) {$\cdots$};
\node [scale=.7](v5) at (2.5,1.5) {$x_m)$};
\node [scale=.7](v6) at (-2,-1.5) {$(y_1$};
\node [scale=.7](v7) at (-0.5,-1.5) {$y_2$};
\node [scale=.7] at (1,-1.5) {$\cdots$};
\node [scale=.7](v8) at (2.5,-1.5) {$y_n)$};
\draw  (v2) -- (0,0)[postaction={decorate, decoration={markings,mark=at position .50 with {\arrow[black]{stealth}}}}];
\draw  (v3) -- (0,0)[postaction={decorate, decoration={markings,mark=at position .50 with {\arrow[black]{stealth}}}}];
\draw  (v5) -- (0,0)[postaction={decorate, decoration={markings,mark=at position .50 with {\arrow[black]{stealth}}}}];
\draw  (v6) -- (0,0)[postaction={decorate, decoration={markings,mark=at position .50 with {\arrowreversed[black]{stealth}}}}];
\draw  (v7) -- (0,0)[postaction={decorate, decoration={markings,mark=at position .50 with {\arrowreversed[black]{stealth}}}}];
\draw  (v8) -- (0,0)[postaction={decorate, decoration={markings,mark=at position .50 with {\arrowreversed[black]{stealth}}}}];
\end{tikzpicture}\end{pmatrix}.$

\begin{prop}\label{preserve unit}
$\theta$ preserves identity morphisms, that is, for any object $x\in Ob(\mathcal{D})$, $$\theta(Id_x)=Id_x\in Mor(\Psi((\mathcal{D},\epsilon))).$$
\end{prop}
\begin{proof}
$$
\begin{matrix}
\theta(Id_x)&=&\epsilon_m\begin{pmatrix}
\begin{tikzpicture}[scale=.5]
\node (v2) at (-0.5,0.5) {};
\node (v1) at (-0.5,2) {$(x)$};
\node (v3) at (-0.5,-1) {$(x)$};
\draw [fill](-0.5,0.5) circle [radius=0.1];
\node at (0.5, 0.5) {$Id_x$};
\draw  (v1) -- (-0.5,0.6)[postaction={decorate, decoration={markings,mark=at position .70 with {\arrow[black]{stealth}}}}];
\draw  (v3) -- (-0.5,0.4)[postaction={decorate, decoration={markings,mark=at position .70 with {\arrowreversed[black]{stealth}}}}];
\end{tikzpicture}
 \end{pmatrix}\\
&=&\epsilon_m\circ T(\epsilon_m)\begin{pmatrix}
\begin{tikzpicture}[scale=.5]
\node (v2) at (-0.5,0.5) {};
\node (v1) at (-0.5,2) {$(x)$};
\node (v3) at (-0.5,-1) {$(x)$};
\draw [fill](-0.5,0.5) circle [radius=0.1];
\node [scale=.5]at (0.5, 0.5) {$\begin{matrix}
\begin{tikzpicture}[scale=.5]
\node (v2) at (-0.5,0.5) {};
\node (v1) at (-0.5,2) {$(x)$};
\node (v3) at (-0.5,-1) {$(x)$};
\draw (v2) circle [radius=0.1];
\draw  (v1) -- (-0.5,0.6)[postaction={decorate, decoration={markings,mark=at position .70 with {\arrow[black]{stealth}}}}];
\draw  (v3) -- (-0.5,0.4)[postaction={decorate, decoration={markings,mark=at position .70 with {\arrowreversed[black]{stealth}}}}];
\end{tikzpicture}
\end{matrix}$};
\draw  (v1) -- (-0.5,0.6)[postaction={decorate, decoration={markings,mark=at position .70 with {\arrow[black]{stealth}}}}];
\draw  (v3) -- (-0.5,0.4)[postaction={decorate, decoration={markings,mark=at position .70 with {\arrowreversed[black]{stealth}}}}];
\end{tikzpicture}
 \end{pmatrix}\\
&=&\epsilon_m\circ \mu_{\mathcal{D}}\begin{pmatrix}
\begin{tikzpicture}[scale=.5]
\node (v2) at (-0.5,0.5) {};
\node (v1) at (-0.5,2) {$(x)$};
\node (v3) at (-0.5,-1) {$(x)$};
\draw [fill](-0.5,0.5) circle [radius=0.1];
\node [scale=.5]at (0.5, 0.5) {$\begin{matrix}
\begin{tikzpicture}[scale=.5]
\node (v2) at (-0.5,0.5) {};
\node (v1) at (-0.5,2) {$(x)$};
\node (v3) at (-0.5,-1) {$(x)$};
\draw (v2) circle [radius=0.1];
\draw  (v1) -- (-0.5,0.6)[postaction={decorate, decoration={markings,mark=at position .70 with {\arrow[black]{stealth}}}}];
\draw  (v3) -- (-0.5,0.4)[postaction={decorate, decoration={markings,mark=at position .70 with {\arrowreversed[black]{stealth}}}}];
\end{tikzpicture}
\end{matrix}$};
\draw  (v1) -- (-0.5,0.6)[postaction={decorate, decoration={markings,mark=at position .70 with {\arrow[black]{stealth}}}}];
\draw  (v3) -- (-0.5,0.4)[postaction={decorate, decoration={markings,mark=at position .70 with {\arrowreversed[black]{stealth}}}}];
\end{tikzpicture}
 \end{pmatrix}\\
&=&\epsilon_m\begin{pmatrix}
\begin{tikzpicture}[scale=.5]
\node (v2) at (-0.5,0.5) {};
\node (v1) at (-0.5,2) {$(x)$};
\node (v3) at (-0.5,-1) {$(x)$};
\draw (v2) circle [radius=0.1];
\draw  (v1) -- (-0.5,0.6)[postaction={decorate, decoration={markings,mark=at position .70 with {\arrow[black]{stealth}}}}];
\draw  (v3) -- (-0.5,0.4)[postaction={decorate, decoration={markings,mark=at position .70 with {\arrowreversed[black]{stealth}}}}];
\end{tikzpicture}
\end{pmatrix}\\
&=&Id_x.

\end{matrix}
$$
\end{proof}

\begin{prop}\label{preserve tensor and composition}
$\theta$ preserves tensor products, compositions. More precisely, for any $f, g\in Mor(\mathcal{D})$, we have $$\theta(f\otimes_{\epsilon} g)=\theta(f)\otimes_{\mathcal{V}}\theta(g),$$  and for any $f, g\in Mor(\mathcal{D})$ with $cod(f)=dom(g)$, we have $$\theta(g\circ_{\epsilon} f)=\theta(g)\circ_{\mathcal{V}}\theta(f).$$
\end{prop}
\begin{proof}
$\bullet$   Let $f: x_1\cdots x_m\rightarrow y_1\cdots y_n$ and $g:u_1\cdots u_p\rightarrow v_1\cdots v_q$ be two morphism of $\mathcal{D}$. Take $I_1$, $I_2$, $O_1$, $O_2$, be trivial linear partitions with $||I_1||=m$, $||O_1||=n$, $||I_2||=p$, $||O_2||=q$, and $I$, $O$ be finest linear partitions with $||I||=m+p$, $||O||=n+q$.  We assume $\widetilde{x}=\epsilon_o(x_1\cdots x_m)$, $\widetilde{y}=\epsilon_o(y_1\cdots y_n)$, $\widetilde{u}=\epsilon_o(u_1\cdots u_p)$ and $\widetilde{v}=\epsilon_o(v_1\cdots v_q).$ On one hand, we have

$$
\begin{matrix}
\theta(f\otimes_{\epsilon}g)&=&\epsilon_m\begin{pmatrix}\begin{tikzpicture}[scale=0.5]

\node (v2) at (-0.5,0) {};
\draw [fill] (-0.5,0) circle [radius=0.1];
\node at (1,0) {$f\otimes_{\epsilon}g$};
\node (v1) at (-2,1.5) {$(x_1$};
\node (v4) at (1,1.5) {$u_p)$};
\node (v3) at (-0.5,1.5) {$\cdots$};
\node (v5) at (-2,-1.5) {$(y_1$};
\node (v7) at (1,-1.5) {$v_q)$};
\node (v6) at (-0.5,-1.5) {$\cdots$};

\draw  (v1) -- (-0.5,0)[postaction={decorate, decoration={markings,mark=at position .5 with {\arrow[black]{stealth}}}}];

\draw  (v4) -- (-0.5,0)[postaction={decorate, decoration={markings,mark=at position .5 with {\arrow[black]{stealth}}}}];
\draw  (-0.5,0) -- (v5)[postaction={decorate, decoration={markings,mark=at position .6 with {\arrow[black]{stealth}}}}];
\draw  (-0.5,0) -- (v7)[postaction={decorate, decoration={markings,mark=at position .6 with {\arrow[black]{stealth}}}}];
\end{tikzpicture}\end{pmatrix}\\
&=&\epsilon_m\begin{pmatrix}\begin{tikzpicture}[scale=0.5]

\node (v2) at (-0.5,0) {};
\draw [fill] (-0.5,0) circle [radius=0.1];
\node at (3,0) {$\epsilon_m\begin{pmatrix}
[\Gamma_f\otimes\Gamma_g, I,O]
\end{pmatrix}$};
\node (v1) at (-2,1.5) {$(x_1$};
\node (v4) at (1,1.5) {$u_p)$};
\node (v3) at (-0.5,1.5) {$\cdots$};
\node (v5) at (-2,-1.5) {$(y_1$};
\node (v7) at (1,-1.5) {$v_q)$};
\node (v6) at (-0.5,-1.5) {$\cdots$};

\draw  (v1) -- (-0.5,0)[postaction={decorate, decoration={markings,mark=at position .5 with {\arrow[black]{stealth}}}}];

\draw  (v4) -- (-0.5,0)[postaction={decorate, decoration={markings,mark=at position .5 with {\arrow[black]{stealth}}}}];
\draw  (-0.5,0) -- (v5)[postaction={decorate, decoration={markings,mark=at position .6 with {\arrow[black]{stealth}}}}];
\draw  (-0.5,0) -- (v7)[postaction={decorate, decoration={markings,mark=at position .6 with {\arrow[black]{stealth}}}}];
\end{tikzpicture}\end{pmatrix}\\
&=&\epsilon_m\circ T(\epsilon_m)\begin{pmatrix}\begin{tikzpicture}[scale=0.5]

\node (v2) at (-0.5,0) {};
\draw [fill] (-0.5,0) circle [radius=0.1];
\node at (3,0) {$[\Gamma_f\otimes\Gamma_g, I,O]$};
\node (v1) at (-2,1.5) {$(x_1$};
\node (v4) at (1,1.5) {$u_p)$};
\node (v3) at (-0.5,1.5) {$\cdots$};
\node (v5) at (-2,-1.5) {$(y_1$};
\node (v7) at (1,-1.5) {$v_q)$};
\node (v6) at (-0.5,-1.5) {$\cdots$};

\draw  (v1) -- (-0.5,0)[postaction={decorate, decoration={markings,mark=at position .5 with {\arrow[black]{stealth}}}}];

\draw  (v4) -- (-0.5,0)[postaction={decorate, decoration={markings,mark=at position .5 with {\arrow[black]{stealth}}}}];
\draw  (-0.5,0) -- (v5)[postaction={decorate, decoration={markings,mark=at position .6 with {\arrow[black]{stealth}}}}];
\draw  (-0.5,0) -- (v7)[postaction={decorate, decoration={markings,mark=at position .6 with {\arrow[black]{stealth}}}}];
\end{tikzpicture}\end{pmatrix}\\
&=&\epsilon_m\circ \mu_{\mathcal{D}}\begin{pmatrix}\begin{tikzpicture}[scale=0.5]

\node (v2) at (-0.5,0) {};
\draw [fill] (-0.5,0) circle [radius=0.1];
\node at (3,0) {$[\Gamma_f\otimes\Gamma_g, I,O]$};
\node (v1) at (-2,1.5) {$(x_1$};
\node (v4) at (1,1.5) {$u_p)$};
\node (v3) at (-0.5,1.5) {$\cdots$};
\node (v5) at (-2,-1.5) {$(y_1$};
\node (v7) at (1,-1.5) {$v_q)$};
\node (v6) at (-0.5,-1.5) {$\cdots$};

\draw  (v1) -- (-0.5,0)[postaction={decorate, decoration={markings,mark=at position .5 with {\arrow[black]{stealth}}}}];

\draw  (v4) -- (-0.5,0)[postaction={decorate, decoration={markings,mark=at position .5 with {\arrow[black]{stealth}}}}];
\draw  (-0.5,0) -- (v5)[postaction={decorate, decoration={markings,mark=at position .6 with {\arrow[black]{stealth}}}}];
\draw  (-0.5,0) -- (v7)[postaction={decorate, decoration={markings,mark=at position .6 with {\arrow[black]{stealth}}}}];
\end{tikzpicture}\end{pmatrix}\\
&=&\epsilon_m\begin{pmatrix} \begin{tikzpicture}[scale=0.5]

\node (v2) at (-0.5,0) {};
\draw [fill] (-0.5,0) circle [radius=0.1];
\node at (0,0) {$f$};
\node (v1) at (-2,1.5) {$(x_1$};
\node (v4) at (1,1.5) {$x_m$};
\node (v3) at (-0.5,1.5) {$\cdots$};
\node (v5) at (-2,-1.5) {$(y_1$};
\node (v7) at (1,-1.5) {$y_n$};
\node (v6) at (-0.5,-1.5) {$\cdots$};

\node (v8) at (2.5,1.5) {$u_1$};
\node (v10) at (4,1.5) {$\cdots$};
\node (v11) at (5.5,1.5) {$u_p)$};
\node (v9) at (4,0) {};
\draw [fill] (4,0) circle [radius=0.1];
\node at (4.5,0) {$g$};
\node (v12) at (2.5,-1.5) {$v_1$};
\node (v13) at (4,-1.5) {$\cdots$};
\node (v14) at (5.5,-1.5) {$v_q)$};
\draw  (v1) -- (-0.5,0)[postaction={decorate, decoration={markings,mark=at position .5 with {\arrow[black]{stealth}}}}];

\draw  (v4) -- (-0.5,0)[postaction={decorate, decoration={markings,mark=at position .5 with {\arrow[black]{stealth}}}}];
\draw  (-0.5,0) -- (v5)[postaction={decorate, decoration={markings,mark=at position .6 with {\arrow[black]{stealth}}}}];

\draw  (-0.5,0) -- (v7)[postaction={decorate, decoration={markings,mark=at position .6 with {\arrow[black]{stealth}}}}];
\draw  (v8) -- (4,0)[postaction={decorate, decoration={markings,mark=at position .5 with {\arrow[black]{stealth}}}}];

\draw  (v11) -- (4,0)[postaction={decorate, decoration={markings,mark=at position .5 with {\arrow[black]{stealth}}}}];
\draw  (4,0) -- (v12)[postaction={decorate, decoration={markings,mark=at position .6 with {\arrow[black]{stealth}}}}];

\draw  (4,0) -- (v14)[postaction={decorate, decoration={markings,mark=at position .6 with {\arrow[black]{stealth}}}}];
\end{tikzpicture}\end{pmatrix}

\end{matrix}
$$

where $[\Gamma_f\otimes\Gamma_g, I,O]=\begin{matrix}\begin{tikzpicture}[scale=0.5]

\node (v2) at (-0.5,0) {};
\draw [fill] (-0.5,0) circle [radius=0.1];
\node at (0,0) {$f$};
\node (v1) at (-2,1.5) {$(x_1)$};
\node (v4) at (1,1.5) {$(x_m)$};
\node (v3) at (-0.5,1.5) {$\cdots$};
\node (v5) at (-2,-1.5) {$(y_1)$};
\node (v7) at (1,-1.5) {$(y_n)$};
\node (v6) at (-0.5,-1.5) {$\cdots$};

\node (v8) at (2.5,1.5) {$(u_1)$};
\node (v10) at (4,1.5) {$\cdots$};
\node (v11) at (5.5,1.5) {$(u_p)$};
\node (v9) at (4,0) {};
\draw [fill] (4,0) circle [radius=0.1];
\node at (4.5,0) {$g$};
\node (v12) at (2.5,-1.5) {$(v_1)$};
\node (v13) at (4,-1.5) {$\cdots$};
\node (v14) at (5.5,-1.5) {$(v_q);$};
\draw  (v1) -- (-0.5,0)[postaction={decorate, decoration={markings,mark=at position .5 with {\arrow[black]{stealth}}}}];

\draw  (v4) -- (-0.5,0)[postaction={decorate, decoration={markings,mark=at position .5 with {\arrow[black]{stealth}}}}];
\draw  (-0.5,0) -- (v5)[postaction={decorate, decoration={markings,mark=at position .6 with {\arrow[black]{stealth}}}}];

\draw  (-0.5,0) -- (v7)[postaction={decorate, decoration={markings,mark=at position .6 with {\arrow[black]{stealth}}}}];
\draw  (v8) -- (4,0)[postaction={decorate, decoration={markings,mark=at position .5 with {\arrow[black]{stealth}}}}];

\draw  (v11) -- (4,0)[postaction={decorate, decoration={markings,mark=at position .5 with {\arrow[black]{stealth}}}}];
\draw  (4,0) -- (v12)[postaction={decorate, decoration={markings,mark=at position .6 with {\arrow[black]{stealth}}}}];

\draw  (4,0) -- (v14)[postaction={decorate, decoration={markings,mark=at position .6 with {\arrow[black]{stealth}}}}];
\end{tikzpicture}\end{matrix}$

on the other hand, we have
$$
\begin{matrix}
\theta(f)\otimes_{\mathcal{V}}\theta(g)&=&\epsilon_m\begin{pmatrix}\begin{tikzpicture}[scale=.5]
\node (v2) at (-0.5,0.5) {};
\node (v1) at (-0.5,2) {$(x$};
\node (v3) at (-0.5,-1) {$(y$};
\node (v4) at (7.5,2) {$z)$};
\node (v6) at (7.5,0.5) {};
\node (v5) at (7.5,-1) {$w)$};
\draw[fill] (v2) circle [radius=0.1];
\draw[fill] (v6) circle [radius=0.1];
\draw  (v1) -- (v3)[postaction={decorate, decoration={markings,mark=at position .30 with {\arrow[black]{stealth}}}}][postaction={decorate, decoration={markings,mark=at position .85 with {\arrow[black]{stealth}}}}];
\draw  (v4) -- (v5)[postaction={decorate, decoration={markings,mark=at position .30 with {\arrow[black]{stealth}}}}][postaction={decorate, decoration={markings,mark=at position .85 with {\arrow[black]{stealth}}}}];

\node  at (0.5,0.5) {$\theta(f)$};
\node  at (8.5,0.5) {$\theta(g)$};
\end{tikzpicture}\end{pmatrix}\\
&=&\epsilon_m\begin{pmatrix}\begin{tikzpicture}[scale=.5]
\node (v2) at (-0.5,0.5) {};
\node (v1) at (-0.5,2) {$(\widetilde{x}$};
\node (v3) at (-0.5,-1) {$(\widetilde{y}$};
\node (v4) at (7.5,2) {$\widetilde{u})$};
\node (v6) at (7.5,0.5) {};
\node (v5) at (7.5,-1) {$\widetilde{v})$};
\draw[fill] (v2) circle [radius=0.1];
\draw[fill] (v6) circle [radius=0.1];
\draw  (v1) -- (v3)[postaction={decorate, decoration={markings,mark=at position .30 with {\arrow[black]{stealth}}}}][postaction={decorate, decoration={markings,mark=at position .85 with {\arrow[black]{stealth}}}}];
\draw  (v4) -- (v5)[postaction={decorate, decoration={markings,mark=at position .30 with {\arrow[black]{stealth}}}}][postaction={decorate, decoration={markings,mark=at position .85 with {\arrow[black]{stealth}}}}];

\node  at (2.5,0.5) {$\epsilon_m([\Gamma_f, I_1,O_1])$};
\node  at (10.5,0.5) {$\epsilon_m([\Gamma_g, I_2,O_2])$};
\end{tikzpicture}\end{pmatrix}\\
&=&\epsilon_m\circ T(\epsilon_m)\begin{pmatrix}\begin{tikzpicture}[scale=.5]
\node (v2) at (-0.5,0.5) {};
\node (v1) at (-0.5,2) {$(x_1\cdots x_m$};
\node (v3) at (-0.5,-1) {$(y_1\cdots y_n$};
\node (v4) at (7.5,2) {$u_1\cdots u_p)$};
\node (v6) at (7.5,0.5) {};
\node (v5) at (7.5,-1) {$v_1\cdots v_q)$};
\draw[fill] (v2) circle [radius=0.1];
\draw[fill] (v6) circle [radius=0.1];
\draw  (v1) -- (v3)[postaction={decorate, decoration={markings,mark=at position .30 with {\arrow[black]{stealth}}}}][postaction={decorate, decoration={markings,mark=at position .85 with {\arrow[black]{stealth}}}}];
\draw  (v4) -- (v5)[postaction={decorate, decoration={markings,mark=at position .30 with {\arrow[black]{stealth}}}}][postaction={decorate, decoration={markings,mark=at position .85 with {\arrow[black]{stealth}}}}];

\node  at (2.5,0.5) {$[\Gamma_f, I_1,O_1]$};
\node  at (10.5,0.5) {$[\Gamma_g, I_2,O_2]$};
\end{tikzpicture}\end{pmatrix}\\
&=&\epsilon_m\circ \mu_{\mathcal{D}}\begin{pmatrix}\begin{tikzpicture}[scale=.5]
\node (v2) at (-0.5,0.5) {};
\node (v1) at (-0.5,2) {$(x_1\cdots x_m$};
\node (v3) at (-0.5,-1) {$(y_1\cdots y_n$};
\node (v4) at (7.5,2) {$u_1\cdots u_p)$};
\node (v6) at (7.5,0.5) {};
\node (v5) at (7.5,-1) {$v_1\cdots v_q)$};
\draw[fill] (v2) circle [radius=0.1];
\draw[fill] (v6) circle [radius=0.1];
\draw  (v1) -- (v3)[postaction={decorate, decoration={markings,mark=at position .30 with {\arrow[black]{stealth}}}}][postaction={decorate, decoration={markings,mark=at position .85 with {\arrow[black]{stealth}}}}];
\draw  (v4) -- (v5)[postaction={decorate, decoration={markings,mark=at position .30 with {\arrow[black]{stealth}}}}][postaction={decorate, decoration={markings,mark=at position .85 with {\arrow[black]{stealth}}}}];

\node  at (2.5,0.5) {$[\Gamma_f, I_1,O_1]$};
\node  at (10.5,0.5) {$[\Gamma_g, I_2,O_2]$};
\end{tikzpicture}\end{pmatrix}\\
&=&\epsilon_m\begin{pmatrix} \begin{tikzpicture}[scale=0.5]

\node (v2) at (-0.5,0) {};
\draw [fill] (-0.5,0) circle [radius=0.1];
\node at (0,0) {$f$};
\node (v1) at (-2,1.5) {$(x_1$};
\node (v4) at (1,1.5) {$x_m$};
\node (v3) at (-0.5,1.5) {$\cdots$};
\node (v5) at (-2,-1.5) {$(y_1$};
\node (v7) at (1,-1.5) {$y_n$};
\node (v6) at (-0.5,-1.5) {$\cdots$};

\node (v8) at (2.5,1.5) {$u_1$};
\node (v10) at (4,1.5) {$\cdots$};
\node (v11) at (5.5,1.5) {$u_p)$};
\node (v9) at (4,0) {};
\draw [fill] (4,0) circle [radius=0.1];
\node at (4.5,0) {$g$};
\node (v12) at (2.5,-1.5) {$v_1$};
\node (v13) at (4,-1.5) {$\cdots$};
\node (v14) at (5.5,-1.5) {$v_q)$};
\draw  (v1) -- (-0.5,0)[postaction={decorate, decoration={markings,mark=at position .5 with {\arrow[black]{stealth}}}}];

\draw  (v4) -- (-0.5,0)[postaction={decorate, decoration={markings,mark=at position .5 with {\arrow[black]{stealth}}}}];
\draw  (-0.5,0) -- (v5)[postaction={decorate, decoration={markings,mark=at position .6 with {\arrow[black]{stealth}}}}];

\draw  (-0.5,0) -- (v7)[postaction={decorate, decoration={markings,mark=at position .6 with {\arrow[black]{stealth}}}}];
\draw  (v8) -- (4,0)[postaction={decorate, decoration={markings,mark=at position .5 with {\arrow[black]{stealth}}}}];

\draw  (v11) -- (4,0)[postaction={decorate, decoration={markings,mark=at position .5 with {\arrow[black]{stealth}}}}];
\draw  (4,0) -- (v12)[postaction={decorate, decoration={markings,mark=at position .6 with {\arrow[black]{stealth}}}}];

\draw  (4,0) -- (v14)[postaction={decorate, decoration={markings,mark=at position .6 with {\arrow[black]{stealth}}}}];
\end{tikzpicture}\end{pmatrix}

\end{matrix}
$$
where  $[\Gamma_f, I_1,O_1]=\begin{matrix}\begin{tikzpicture}[scale=0.5]

\node (v2) at (-0.5,0) {};
\draw [fill] (-0.5,0) circle [radius=0.1];
\node at (0,0) {$f$};
\node (v1) at (-2,1.5) {$(x_1$};
\node (v4) at (1,1.5) {$x_m)$};
\node (v3) at (-0.5,1.5) {$\cdots$};
\node (v5) at (-2,-1.5) {$(y_1$};
\node (v7) at (1,-1.5) {$y_n)$};
\node (v6) at (-0.5,-1.5) {$\cdots$};

\draw  (v1) -- (-0.5,0)[postaction={decorate, decoration={markings,mark=at position .5 with {\arrow[black]{stealth}}}}];

\draw  (v4) -- (-0.5,0)[postaction={decorate, decoration={markings,mark=at position .5 with {\arrow[black]{stealth}}}}];
\draw  (-0.5,0) -- (v5)[postaction={decorate, decoration={markings,mark=at position .6 with {\arrow[black]{stealth}}}}];
\draw  (-0.5,0) -- (v7)[postaction={decorate, decoration={markings,mark=at position .6 with {\arrow[black]{stealth}}}}];
\end{tikzpicture}\end{matrix}$ and $[\Gamma_g, I_2,O_2]=\begin{matrix}\begin{tikzpicture}[scale=0.5]

\node (v2) at (-0.5,0) {};
\draw [fill] (-0.5,0) circle [radius=0.1];
\node at (0,0) {$g$};
\node (v1) at (-2,1.5) {$(u_1$};
\node (v4) at (1,1.5) {$u_p)$};
\node (v3) at (-0.5,1.5) {$\cdots$};
\node (v5) at (-2,-1.5) {$(v_1$};
\node (v7) at (1,-1.5) {$v_q).$};
\node (v6) at (-0.5,-1.5) {$\cdots$};

\draw  (v1) -- (-0.5,0)[postaction={decorate, decoration={markings,mark=at position .5 with {\arrow[black]{stealth}}}}];

\draw  (v4) -- (-0.5,0)[postaction={decorate, decoration={markings,mark=at position .5 with {\arrow[black]{stealth}}}}];
\draw  (-0.5,0) -- (v5)[postaction={decorate, decoration={markings,mark=at position .6 with {\arrow[black]{stealth}}}}];
\draw  (-0.5,0) -- (v7)[postaction={decorate, decoration={markings,mark=at position .6 with {\arrow[black]{stealth}}}}];
\end{tikzpicture}\end{matrix}$

Thus we get that $$\theta(f\otimes_{\epsilon} g)=\theta(f)\otimes_{\mathcal{V}}\theta(g).$$

$\bullet$   Let $f: x_1\cdots x_l\rightarrow y_1\cdots y_m$ and $g:y_1\cdots y_m\rightarrow z_1\cdots z_n$ be two morphism of $\mathcal{D}$. Take $I_1$, $I_2$, $O_1$, $O_2$, be trivial linear partitions with $||I_1||=l$, $||O_1||=m$, $||I_2||=m$, $||O_2||=n$,  and $I$, $O$ be finest linear partitions with $||I||=l$, $||O||=n$.  We assume $\widetilde{x}=\epsilon_o(x_1\cdots x_l)$, $\widetilde{y}=\epsilon_o(y_1\cdots y_m)$, $\widetilde{z}=\epsilon_o(z_1\cdots z_n).$ On one hand, we have

$$
\begin{matrix}
\theta(g\circ_{\epsilon}f)&=&\epsilon_m\begin{pmatrix}\begin{tikzpicture}[scale=0.5]

\node (v2) at (-0.5,0) {};
\draw [fill] (-0.5,0) circle [radius=0.1];
\node at (1,0) {$g\circ_{\epsilon}f$};
\node (v1) at (-2,1.5) {$(x_1$};
\node (v4) at (1,1.5) {$x_l)$};
\node (v3) at (-0.5,1.5) {$\cdots$};
\node (v5) at (-2,-1.5) {$(z_1$};
\node (v7) at (1,-1.5) {$z_n)$};
\node (v6) at (-0.5,-1.5) {$\cdots$};

\draw  (v1) -- (-0.5,0)[postaction={decorate, decoration={markings,mark=at position .5 with {\arrow[black]{stealth}}}}];

\draw  (v4) -- (-0.5,0)[postaction={decorate, decoration={markings,mark=at position .5 with {\arrow[black]{stealth}}}}];
\draw  (-0.5,0) -- (v5)[postaction={decorate, decoration={markings,mark=at position .6 with {\arrow[black]{stealth}}}}];
\draw  (-0.5,0) -- (v7)[postaction={decorate, decoration={markings,mark=at position .6 with {\arrow[black]{stealth}}}}];
\end{tikzpicture}\end{pmatrix}\\
&=&\epsilon_m\begin{pmatrix}\begin{tikzpicture}[scale=0.5]

\node (v2) at (-0.5,0) {};
\draw [fill] (-0.5,0) circle [radius=0.1];
\node at (3,0) {$\epsilon_m\begin{pmatrix}
[\Gamma_g\circ\Gamma_f, I,O]
\end{pmatrix}$};
\node (v1) at (-2,1.5) {$(x_1$};
\node (v4) at (1,1.5) {$x_l)$};
\node (v3) at (-0.5,1.5) {$\cdots$};
\node (v5) at (-2,-1.5) {$(z_1$};
\node (v7) at (1,-1.5) {$z_n)$};
\node (v6) at (-0.5,-1.5) {$\cdots$};

\draw  (v1) -- (-0.5,0)[postaction={decorate, decoration={markings,mark=at position .5 with {\arrow[black]{stealth}}}}];

\draw  (v4) -- (-0.5,0)[postaction={decorate, decoration={markings,mark=at position .5 with {\arrow[black]{stealth}}}}];
\draw  (-0.5,0) -- (v5)[postaction={decorate, decoration={markings,mark=at position .6 with {\arrow[black]{stealth}}}}];
\draw  (-0.5,0) -- (v7)[postaction={decorate, decoration={markings,mark=at position .6 with {\arrow[black]{stealth}}}}];
\end{tikzpicture}\end{pmatrix}\\
&=&\epsilon_m\circ T(\epsilon_m)\begin{pmatrix}\begin{tikzpicture}[scale=0.5]

\node (v2) at (-0.5,0) {};
\draw [fill] (-0.5,0) circle [radius=0.1];
\node at (3,0) {$[\Gamma_g\circ\Gamma_f, I,O]$};
\node (v1) at (-2,1.5) {$(x_1$};
\node (v4) at (1,1.5) {$x_l)$};
\node (v3) at (-0.5,1.5) {$\cdots$};
\node (v5) at (-2,-1.5) {$(z_1$};
\node (v7) at (1,-1.5) {$z_n)$};
\node (v6) at (-0.5,-1.5) {$\cdots$};

\draw  (v1) -- (-0.5,0)[postaction={decorate, decoration={markings,mark=at position .5 with {\arrow[black]{stealth}}}}];

\draw  (v4) -- (-0.5,0)[postaction={decorate, decoration={markings,mark=at position .5 with {\arrow[black]{stealth}}}}];
\draw  (-0.5,0) -- (v5)[postaction={decorate, decoration={markings,mark=at position .6 with {\arrow[black]{stealth}}}}];
\draw  (-0.5,0) -- (v7)[postaction={decorate, decoration={markings,mark=at position .6 with {\arrow[black]{stealth}}}}];
\end{tikzpicture}\end{pmatrix}\\
&=&\epsilon_m\circ \mu_{\mathcal{D}}\begin{pmatrix}\begin{tikzpicture}[scale=0.5]

\node (v2) at (-0.5,0) {};
\draw [fill] (-0.5,0) circle [radius=0.1];
\node at (3,0) {$[\Gamma_g\circ\Gamma_f, I,O]$};
\node (v1) at (-2,1.5) {$(x_1$};
\node (v4) at (1,1.5) {$x_l)$};
\node (v3) at (-0.5,1.5) {$\cdots$};
\node (v5) at (-2,-1.5) {$(z_1$};
\node (v7) at (1,-1.5) {$z_n)$};
\node (v6) at (-0.5,-1.5) {$\cdots$};

\draw  (v1) -- (-0.5,0)[postaction={decorate, decoration={markings,mark=at position .5 with {\arrow[black]{stealth}}}}];

\draw  (v4) -- (-0.5,0)[postaction={decorate, decoration={markings,mark=at position .5 with {\arrow[black]{stealth}}}}];
\draw  (-0.5,0) -- (v5)[postaction={decorate, decoration={markings,mark=at position .6 with {\arrow[black]{stealth}}}}];
\draw  (-0.5,0) -- (v7)[postaction={decorate, decoration={markings,mark=at position .6 with {\arrow[black]{stealth}}}}];
\end{tikzpicture}\end{pmatrix}\\
&=&\epsilon_m\begin{pmatrix} \begin{tikzpicture}[scale=0.5]

\node  at (0,2.2) {$(x_1$};
\node  at (1,2.2) {$\cdots$};
\node  at (2,2.2) {$x_l)$};
\node (v1) at (1,1) {};
\draw [fill] (1,1) circle [radius=0.055];
\node  at (1.5,1) {$f$};
\node (v2) at (0.3,0) {};
\node (v3) at (0.3,-1) {};
\node  at (-0.2,-0.5) {$y_1$};
\node (v5) at (1.7,0) {};
\node (v6) at (1.7,-1) {};
\node  at (2.5,-0.5) {$y_m$};
\node (v4) at (1,-2) {};
\draw [fill] (1,-2) circle [radius=0.055];
\node  at (1.5,-2) {$g$};

\node  at (0,-3.2) {$(z_1$};
\node (v11) at (1,-3) {$\cdots$};
\node  at (2,-3.2) {$z_n)$};
\node at (1,-0.5) {$\cdots$};

\draw  (0,2) -- (1,1)[postaction={decorate, decoration={markings,mark=at position .6 with {\arrow[black]{stealth}}}}];

\draw  (2,2) -- ((1,1)[postaction={decorate, decoration={markings,mark=at position .6 with {\arrow[black]{stealth}}}}];
\draw  (1,-2) -- (0,-3)[postaction={decorate, decoration={markings,mark=at position .6 with {\arrow[black]{stealth}}}}];

\draw  (1,-2) -- (2,-3)[postaction={decorate, decoration={markings,mark=at position .6 with {\arrow[black]{stealth}}}}];
\draw  plot[smooth, tension=.4] coordinates {(v1) (v2) (v3) (v4)}[postaction={decorate, decoration={markings,mark=at position .5 with {\arrow[black]{stealth}}}}];
\draw  plot[smooth, tension=.4] coordinates {(1,1) (v5) (v6) (1,-2)}[postaction={decorate, decoration={markings,mark=at position .5 with {\arrow[black]{stealth}}}}];
\end{tikzpicture}\end{pmatrix}

\end{matrix}
$$

where $[\Gamma_g\circ\Gamma_f, I,O]=\begin{matrix}\begin{tikzpicture}[scale=0.5]

\node  at (0,2.2) {$(x_1)$};
\node  at (1,2.2) {$\cdots$};
\node  at (2,2.2) {$(x_l)$};
\node (v1) at (1,1) {};
\draw [fill] (1,1) circle [radius=0.055];
\node  at (1.5,1) {$f$};
\node (v2) at (0.3,0) {};
\node (v3) at (0.3,-1) {};
\node  at (-0.2,-0.5) {$y_1$};
\node (v5) at (1.7,0) {};
\node (v6) at (1.7,-1) {};
\node  at (2.5,-0.5) {$y_m$};
\node (v4) at (1,-2) {};
\draw [fill] (1,-2) circle [radius=0.055];
\node  at (1.5,-2) {$g$};

\node  at (0,-3.2) {$(z_1)$};
\node (v11) at (1,-3) {$\cdots$};
\node  at (2,-3.2) {$(z_n);$};
\node at (1,-0.5) {$\cdots$};

\draw  (0,2) -- (1,1)[postaction={decorate, decoration={markings,mark=at position .6 with {\arrow[black]{stealth}}}}];

\draw  (2,2) -- ((1,1)[postaction={decorate, decoration={markings,mark=at position .6 with {\arrow[black]{stealth}}}}];
\draw  (1,-2) -- (0,-3)[postaction={decorate, decoration={markings,mark=at position .6 with {\arrow[black]{stealth}}}}];

\draw  (1,-2) -- (2,-3)[postaction={decorate, decoration={markings,mark=at position .6 with {\arrow[black]{stealth}}}}];
\draw  plot[smooth, tension=.4] coordinates {(v1) (v2) (v3) (v4)}[postaction={decorate, decoration={markings,mark=at position .5 with {\arrow[black]{stealth}}}}];
\draw  plot[smooth, tension=.4] coordinates {(1,1) (v5) (v6) (1,-2)}[postaction={decorate, decoration={markings,mark=at position .5 with {\arrow[black]{stealth}}}}];
\end{tikzpicture}\end{matrix}$

on the other hand, we have
$$
\begin{matrix}
\theta(g)\circ_{\mathcal{V}}\theta(f)&=&\epsilon_m\begin{pmatrix}\begin{tikzpicture}[scale=.5]
\node (v1) at (0,1.5) {$(\widetilde{x})$};
\node (v2) at (0,-3.5) {$(\widetilde{z})$};
\draw  (v1)-- (v2)[postaction={decorate, decoration={markings,mark=at position .93 with {\arrow[black]{stealth}}}}][postaction={decorate, decoration={markings,mark=at position .15 with {\arrow[black]{stealth}}}}][postaction={decorate, decoration={markings,mark=at position .55 with {\arrow[black]{stealth}}}}];
\draw[fill] (0,0) circle [radius=0.1];
\draw[fill] (0,-2) circle [radius=0.1];
\node at (0.5,-1) {$\widetilde{y}$};
\node at (1,0) {$\theta(f)$};
\node at (1,-2) {$\theta(g)$};
\end{tikzpicture}\end{pmatrix}\\
&=&\epsilon_m\begin{pmatrix}\begin{tikzpicture}[scale=.5]
\node (v1) at (0,1.5) {$(\widetilde{x})$};
\node (v2) at (0,-3.5) {$(\widetilde{z})$};
\draw  (v1)-- (v2)[postaction={decorate, decoration={markings,mark=at position .93 with {\arrow[black]{stealth}}}}][postaction={decorate, decoration={markings,mark=at position .15 with {\arrow[black]{stealth}}}}][postaction={decorate, decoration={markings,mark=at position .55 with {\arrow[black]{stealth}}}}];
\draw[fill] (0,0) circle [radius=0.1];
\draw[fill] (0,-2) circle [radius=0.1];
\node at (0.5,-1) {$\widetilde{y}$};
\node at (3,0) {$\epsilon_m([\Gamma_f, I_1,O_1])$};
\node at (3,-2) {$\epsilon_m([\Gamma_g, I_2,O_2])$};
\end{tikzpicture}\end{pmatrix}\\
&=&\epsilon_m\circ T(\epsilon_m)\begin{pmatrix}\begin{tikzpicture}[scale=.5]
\node (v1) at (0,1.5) {$(x_1\cdots x_l)$};
\node (v2) at (0,-3.5) {$(z_1\cdots z_n)$};
\draw  (v1)-- (v2)[postaction={decorate, decoration={markings,mark=at position .93 with {\arrow[black]{stealth}}}}][postaction={decorate, decoration={markings,mark=at position .15 with {\arrow[black]{stealth}}}}][postaction={decorate, decoration={markings,mark=at position .55 with {\arrow[black]{stealth}}}}];
\draw[fill] (0,0) circle [radius=0.1];
\draw[fill] (0,-2) circle [radius=0.1];
\node at (1.5,-1) {$y_1\cdots y_m$};
\node at (3,0) {$[\Gamma_f, I_1,O_1]$};
\node at (3,-2) {$[\Gamma_g, I_2,O_2]$};
\end{tikzpicture}\end{pmatrix}\\
&=&\epsilon_m\circ \mu_{\mathcal{D}}\begin{pmatrix}\begin{tikzpicture}[scale=.5]
\node (v1) at (0,1.5) {$(x_1\cdots x_l)$};
\node (v2) at (0,-3.5) {$(z_1\cdots z_n)$};
\draw  (v1)-- (v2)[postaction={decorate, decoration={markings,mark=at position .93 with {\arrow[black]{stealth}}}}][postaction={decorate, decoration={markings,mark=at position .15 with {\arrow[black]{stealth}}}}][postaction={decorate, decoration={markings,mark=at position .55 with {\arrow[black]{stealth}}}}];
\draw[fill] (0,0) circle [radius=0.1];
\draw[fill] (0,-2) circle [radius=0.1];
\node at (1.5,-1) {$y_1\cdots y_m$};
\node at (3,0) {$[\Gamma_f, I_1,O_1]$};
\node at (3,-2) {$[\Gamma_g, I_2,O_2]$};
\end{tikzpicture}\end{pmatrix}\\
&=&\epsilon_m\begin{pmatrix}\begin{tikzpicture}[scale=0.5]

\node  at (0,2.2) {$(x_1$};
\node  at (1,2.2) {$\cdots$};
\node  at (2,2.2) {$x_l)$};
\node (v1) at (1,1) {};
\draw [fill] (1,1) circle [radius=0.055];
\node  at (1.5,1) {$f$};
\node (v2) at (0.3,0) {};
\node (v3) at (0.3,-1) {};
\node  at (-0.2,-0.5) {$y_1$};
\node (v5) at (1.7,0) {};
\node (v6) at (1.7,-1) {};
\node  at (2.5,-0.5) {$y_m$};
\node (v4) at (1,-2) {};
\draw [fill] (1,-2) circle [radius=0.055];
\node  at (1.5,-2) {$g$};

\node  at (0,-3.2) {$(z_1$};
\node (v11) at (1,-3) {$\cdots$};
\node  at (2,-3.2) {$z_n)$};
\node at (1,-0.5) {$\cdots$};

\draw  (0,2) -- (1,1)[postaction={decorate, decoration={markings,mark=at position .6 with {\arrow[black]{stealth}}}}];

\draw  (2,2) -- ((1,1)[postaction={decorate, decoration={markings,mark=at position .6 with {\arrow[black]{stealth}}}}];
\draw  (1,-2) -- (0,-3)[postaction={decorate, decoration={markings,mark=at position .6 with {\arrow[black]{stealth}}}}];

\draw  (1,-2) -- (2,-3)[postaction={decorate, decoration={markings,mark=at position .6 with {\arrow[black]{stealth}}}}];
\draw  plot[smooth, tension=.4] coordinates {(v1) (v2) (v3) (v4)}[postaction={decorate, decoration={markings,mark=at position .5 with {\arrow[black]{stealth}}}}];
\draw  plot[smooth, tension=.4] coordinates {(1,1) (v5) (v6) (1,-2)}[postaction={decorate, decoration={markings,mark=at position .5 with {\arrow[black]{stealth}}}}];
\end{tikzpicture} \end{pmatrix}

\end{matrix}
$$
where  $[\Gamma_f, I_1,O_1]=\begin{matrix}\begin{tikzpicture}[scale=0.5]

\node (v2) at (-0.5,0) {};
\draw [fill] (-0.5,0) circle [radius=0.1];
\node at (0,0) {$f$};
\node (v1) at (-2,1.5) {$(x_1$};
\node (v4) at (1,1.5) {$x_l)$};
\node (v3) at (-0.5,1.5) {$\cdots$};
\node (v5) at (-2,-1.5) {$(y_1$};
\node (v7) at (1,-1.5) {$y_m)$};
\node (v6) at (-0.5,-1.5) {$\cdots$};

\draw  (v1) -- (-0.5,0)[postaction={decorate, decoration={markings,mark=at position .5 with {\arrow[black]{stealth}}}}];

\draw  (v4) -- (-0.5,0)[postaction={decorate, decoration={markings,mark=at position .5 with {\arrow[black]{stealth}}}}];
\draw  (-0.5,0) -- (v5)[postaction={decorate, decoration={markings,mark=at position .6 with {\arrow[black]{stealth}}}}];
\draw  (-0.5,0) -- (v7)[postaction={decorate, decoration={markings,mark=at position .6 with {\arrow[black]{stealth}}}}];
\end{tikzpicture}\end{matrix}$ and $[\Gamma_g, I_2,O_2]=\begin{matrix}\begin{tikzpicture}[scale=0.5]

\node (v2) at (-0.5,0) {};
\draw [fill] (-0.5,0) circle [radius=0.1];
\node at (0,0) {$g$};
\node (v1) at (-2,1.5) {$(y_1$};
\node (v4) at (1,1.5) {$y_m)$};
\node (v3) at (-0.5,1.5) {$\cdots$};
\node (v5) at (-2,-1.5) {$(z_1$};
\node (v7) at (1,-1.5) {$z_n).$};
\node (v6) at (-0.5,-1.5) {$\cdots$};

\draw  (v1) -- (-0.5,0)[postaction={decorate, decoration={markings,mark=at position .5 with {\arrow[black]{stealth}}}}];

\draw  (v4) -- (-0.5,0)[postaction={decorate, decoration={markings,mark=at position .5 with {\arrow[black]{stealth}}}}];
\draw  (-0.5,0) -- (v5)[postaction={decorate, decoration={markings,mark=at position .6 with {\arrow[black]{stealth}}}}];
\draw  (-0.5,0) -- (v7)[postaction={decorate, decoration={markings,mark=at position .6 with {\arrow[black]{stealth}}}}];
\end{tikzpicture}\end{matrix}$

Thus we get that $$\theta(g\circ_{\epsilon} f)=\theta(g)\circ_{\mathcal{V}}\theta(f).$$

\end{proof}
Note that the above proposition is not a direct consequence of corollary $6.2.11$ and $6.2.12$.

\begin{prop}\label{preserve fusions}
$\theta$ preserves fusions, that is, for any pair of linear partitions $I$ and $O$, $$\theta=\theta\circ \ast^I_O.$$
\end{prop}

\begin{proof}
Note that $\theta=\ast^{I^{trivial}}_{O^{trivial}},$ then the proposition is a spacial case of proposition \ref{fusion}.
\end{proof}

Let $(\mathcal{D},\epsilon)$ be a tensor manifold, then $\Phi\Psi((\mathcal{D},\epsilon))$ is a tensor manifold and we denote it by $(\widehat{\mathcal{D}},\widehat{\epsilon})$. Now we define a morphism $\epsilon_{\ast}:\mathcal{D}\rightarrow \widehat{\mathcal{D}}$  of tensor schemes in the following way:

$\bullet$  $(\epsilon_{\ast})_o$ is the identity map from $Ob(\mathcal{D})$ to $Ob(\widehat{\mathcal{D}})$;

$\bullet$ for a morphism $f:x_1\cdots x_m\rightarrow y_1\cdots y_n$ in $\mathcal{D}$, $(\epsilon_{\ast})_m(f):x_1\cdots x_m\rightarrow y_1\cdots y_n$ is a prime diagram in $\Psi((\mathcal{D},\epsilon))$ with domain $x_1\cdots x_m$, codomain $ y_1\cdots y_n$ and the unique vertex decorated by the morphism $\theta(f)$ in $\Psi((\mathcal{D},\epsilon))$ which is also a morphism in $\mathcal{D}$.
$$

$$

\begin{prop}
The morphism $\epsilon_{\ast}:\mathcal{D}\rightarrow \widehat{\mathcal{D}}$ is a morphism of tensor manifolds.
\end{prop}

\begin{proof}
By proposition \ref{morphism}, we only need to show that $\epsilon_{\ast}$ preserves the identity morphisms, tensor products, compositions and fusions.

$\bullet$ $\epsilon_{\ast}$ preserves the identity morphisms.
$$
%
$$

$\bullet$ $\epsilon_{\ast}$ preserves the tensor products.

Let $f:x_1\cdots x_m\rightarrow y_1\cdots y_n$ and $g:u_1\cdots u_p\rightarrow v_1\cdots v_q$ be two morphisms in $\mathcal{D}$.  On one hand, we have
$$
%
$$

On the other hand, we have
$$
%
$$
where the last step we use the definition of  $\widehat{\epsilon}$ (see example $6.1.5$), which is exactly given by the coarse-graining of fissus diagram in $\Psi((\mathcal{D},\epsilon)).$

$\bullet$ $\epsilon_{\ast}$ preserves the compositions.

Let $f:x_1\cdots x_l\rightarrow y_1\cdots y_m$ and $g:y_1\cdots y_m\rightarrow z_1\cdots z_n$ be two morphisms in $\mathcal{D}$.  On one hand, we have
$$
%
$$

On the other hand, we have
$$
%
$$
where the last step we use the definition of  $\widehat{\epsilon}$.

$\bullet$ $\epsilon_{\ast}$ preserves the fusions.

Let $f:x_1\cdots x_m\rightarrow y_1\cdots y_n$ be a morphism of $\mathcal{D},$  and $I, O$ be two linear partitions with $||I||=m,  ||O||=n$ and $|I|=p, |O|=q.$ Assume $\widetilde{f}=\ast^I_O(f):u_1\cdots u_p\rightarrow v_1\cdots v_q$.
Then we have
$$
%
&(\text{proposition \ref{preserve fusions}})\\
&=&(\epsilon_{\ast})_m(\ast^I_O(f)).&(\text{definition of $(\epsilon_{\ast})_m$})

\end{matrix}
$$

\end{proof}

The following proposition is not difficult to prove.

\begin{prop}
The construction of $\epsilon_{\ast}$ produces a natural transformation $\widehat{\eta}:I_{\mathbf{T.Sch}^T}\rightarrow\Phi\Psi$.
\end{prop}

For any tensor manifold $(\mathcal{D},\epsilon)$, we define the \textbf{core} of $(\mathcal{D},\epsilon)$ to be the image of $\epsilon_{\ast}:(\mathcal{D},\epsilon)\rightarrow (\widehat{\mathcal{D}},\widehat{\epsilon})$ which  is a sub-$T$-algebra of $(\widehat{\mathcal{D}},\widehat{\epsilon})$,  and we denote it by  $(\underline{\mathcal{D}},\underline{\epsilon})$. So for any tensor manifold $(\mathcal{D},\epsilon)$, there is a natural morphism of tensor manifolds $$\widetilde{\epsilon}:(\mathcal{D},\epsilon)\rightarrow (\underline{\mathcal{D}},\underline{\epsilon}).$$

\begin{thm}
$\Phi$ is right adjoint to $\Psi$, that is, for any strict tensor category $\mathcal{V}$ and tensor manifold $(\mathcal{D},\epsilon)$ we have a natural isomorphism $$\Xi:Hom_{\mathbf{Str.T}}(\Psi((\mathcal{D},\epsilon)),\mathcal{V})\overset{\sim}\rightarrow Hom_{\mathbf{T.Sch}^{T}}((\mathcal{D},\epsilon),\Phi(\mathcal{V})).$$
\end{thm}
\begin{proof}
$(1)$  For  any strict tensor functor  $K:\Psi((\mathcal{D},\epsilon))\rightarrow\mathcal{V}$,  we define a morphism $\varphi_K=\Xi(K):(\mathcal{D},\epsilon)\rightarrow \Phi(\mathcal{V})$ as following:

$\bullet$  notice that  $Ob(\Phi(\mathcal{V}))=Ob(\mathcal{V})$, $Ob(\mathcal{D})=Ob(\Psi((\mathcal{D},\epsilon)))$ as monoids, so we define $(\varphi_K)_o=K_o$ which is a morphism of monoids;

$\bullet$ for any morphism $f:x_1\cdots x_m\rightarrow y_1\cdots y_n$ in $(\mathcal{D},\epsilon)$, we define

$$(\varphi_K)_m(f)=\begin{matrix}
\begin{tikzpicture}[scale=.7]
\node (v1) at (0,0) {};
\draw[fill] (0,0) circle [radius=0.055];
\node at (1.5,0){$K_m\circ \theta(f),$};
\node (v2) at (-2,1.5) {$K_o x_1$};
\node (v3) at (-0.5,1.5) {$K_o x_2$};
\node (v4) at (1,1.5) {$\cdots$};
\node (v5) at (2.5,1.5) {$K_o x_m$};
\node (v6) at (-2,-1.5) {$K_o y_1$};
\node (v7) at (-0.5,-1.5) {$K_oy_2$};
\node at (1,-1.5) {$\cdots$};
\node (v8) at (2.5,-1.5) {$K_oy_n$};
\draw  (v2) -- (0,0)[postaction={decorate, decoration={markings,mark=at position .50 with {\arrow[black]{stealth}}}}];
\draw  (v3) -- (0,0)[postaction={decorate, decoration={markings,mark=at position .50 with {\arrow[black]{stealth}}}}];
\draw  (v5) -- (0,0)[postaction={decorate, decoration={markings,mark=at position .50 with {\arrow[black]{stealth}}}}];
\draw  (v6) -- (0,0)[postaction={decorate, decoration={markings,mark=at position .50 with {\arrowreversed[black]{stealth}}}}];
\draw  (v7) -- (0,0)[postaction={decorate, decoration={markings,mark=at position .50 with {\arrowreversed[black]{stealth}}}}];
\draw  (v8) -- (0,0)[postaction={decorate, decoration={markings,mark=at position .50 with {\arrowreversed[black]{stealth}}}}];
\end{tikzpicture}
\end{matrix}$$

In fact, it is not difficult to see that $\varphi_K=\Phi(K)\circ \widehat{\eta}_{(\mathcal{D},\epsilon)}$, hence it is a morphism of tensor manifolds.\\

$(2)$ For any morphism of $T$-algebras  $\varphi:(\mathcal{D},\epsilon)\rightarrow \Psi(\mathcal{V})$,  we define a strict tensor functor $K_{\varphi}=\Xi^{-1}(\varphi):\Phi((\mathcal{D},\epsilon))\rightarrow \mathcal{V}$ as following:

$\bullet$ on the level of objects, $(K_{\varphi})_o=\varphi_o;$

$\bullet$ on the level of morphisms, $(K_{\varphi})_m=\varepsilon_{\mathcal{V}}\circ\varphi_m$ which sends a morphism $f:x\rightarrow y$ in $\mathcal{D}$ to $\epsilon_{\mathcal{V}}(\begin{matrix}
\begin{tikzpicture}[scale=.5]
\node (v2) at (-0.5,0.5) {};
\node (v1) at (-0.5,2) {$\varphi_o (x)$};
\node (v3) at (-0.5,-1) {$\varphi_o (y)$};
\node  at (0.7,0.5) {$\varphi_m (f)$};
\draw [fill](v2) circle [radius=0.1];
\draw  (v1) -- (-0.5,0.6)[postaction={decorate, decoration={markings,mark=at position .70 with {\arrow[black]{stealth}}}}];
\draw  (v3) -- (-0.5,0.4)[postaction={decorate, decoration={markings,mark=at position .70 with {\arrowreversed[black]{stealth}}}}];
\end{tikzpicture}
\end{matrix})$.

In fact, we can identify $K_{\varphi}$ with the restriction of $\varphi$ under the pair of maps $j_o^{\mathcal{V}}$, $j_m^{\mathcal{V}}$.
The fact that $\Xi^{-1}(\varphi)$ is a well-defined strict tensor functor can be directly checked from the fact that $\varphi$ is a morphism of $T$-algebras.

$(3)$ By  a routine check, we can see that $\Xi$ and $\Xi^{-1}$ are indeed inverse functions of each other.

$(4)$ The naturality of  $\Xi$ is not difficult to check.

\end{proof}

A $T$-algebra $(\mathcal{D},\epsilon)$ is called \textbf{critical}, if $(\mathcal{D},\epsilon)$ is isomorphic to $\Phi(\mathcal{V})$ for some strict tensor category $\mathcal{V}$.

\begin{prop}
$(\mathcal{D},\epsilon)$ is critical if and only if the natural morphism $\epsilon_{\ast}:\mathcal{D}\rightarrow \widehat{\mathcal{D}}$ is an isomorphism.
\end{prop}

\begin{rem}\

$\bullet$ The pattern we get here

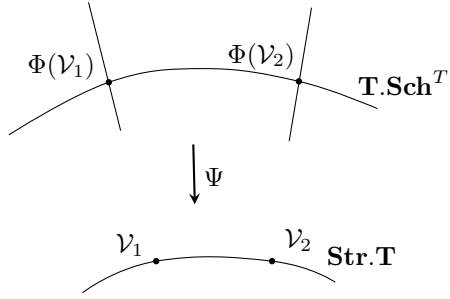
\begin{center}
\begin{tikzpicture}[scale=.7]

\node (v1) at (-3.6,-1) {};
\node (v4) at (1.2,-0.8) {};
\node (v5) at (-5,2) {};
\node (v7) at (-1.3,3.26) {};
\node (v9) at (2,2.5) {};
\node (v10) at (-3.53,4.7) {};
\node (v11) at (-2.83,1.9) {};
\node (v12) at (0.78,4.6) {};
\node (v13) at (0.295,1.7) {};
\draw  (v10) edge (v11);
\draw  (v12) edge (v13);
\node (v18) at (-3.1,3) {};
\node  at (-4,3.3) {$\Phi(\mathcal{V}_1)$};
\node (v19) at (0.5093,3.018) {};
\node  at (-0.2,3.5) {$\Phi(\mathcal{V}_2)$};
\node at (2.5,3) {$\mathbf{T.Sch}^{T}$};
\node at (1.7,-0.3) {$\mathbf{Str.T}$};
\node (v14) at (-1.5,2) {};
\node (v15) at (-1.4625,0.4996) {};
\draw [thick,->,>=stealth] (v14) -- (v15);
\node at (-1,1) {};
\node at (-1.1,1.2) {$\Psi$};
\node (v16) at (-2.2,-0.4) {};
\node  at (-2.7,-0.1) {$\mathcal{V}_1$};

\node (v17) at (0,-0.4) {};
\node  at (0.5,0) {$\mathcal{V}_2$};

\draw  plot[smooth, tension=0.7] coordinates {(v5) (v18) (v7) (v19) (v9)};
\draw  plot[smooth, tension=.8] coordinates {(v1) (v16) (v17) (v4)};

\draw[fill] (v18) circle [radius=0.05];
\draw[fill] (v19) circle [radius=0.05];
\draw[fill] (v16) circle [radius=0.05];
\draw[fill] (v17) circle [radius=0.05];
\end{tikzpicture}
\end{center}
is in some sense  similar to that of wave function renormalization \cite{[LW04],[LW05],[CGW10]} $($or entanglement renormalization \cite{[Vi06],[KRV08]}$)$ in the theory of  topological orders.

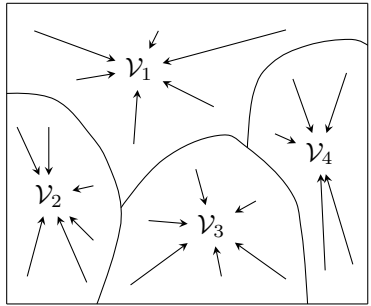
\begin{center}
\begin{tikzpicture}[scale=.8]

\node (v1) at (-3.5,4) {};
\node (v2) at (-3.5,-1) {};
\node (v4) at (2.5,4) {};
\node (v3) at (2.5,-1) {};

\draw  (-3.5,4) edge (-3.5,-1);
\draw  (-3.5,-1) edge (2.5,-1);
\draw  (-3.5,4) edge (2.5,4);
\draw  (2.5,-1) edge (2.5,4);

\draw  plot[smooth, tension=.7] coordinates {(-3.5,2.5) (-2.7,2.4) (-2.1,1.8) (-1.8,1.3) (-1.6,0.6) (-1.7,-0.1) (-2,-1)};
\draw  plot[smooth, tension=.7] coordinates {(-1.6,0.6) (-0.9,1.4) (0,1.8) (0.5,1.6) (1.2,0.9) (1.4,0.2) (1.5,-1) };
\draw  plot[smooth, tension=.7] coordinates {(0.5,1.6) (0.6,2.5) (0.9,3) (1.5,3.3) (2.2,3.4) (2.5,3.3) };
\node (v6) at (-1.3,2.9) {$\mathcal{V}_1$};
\node (v13) at (-2.8,0.8) {$\mathcal{V}_2$};
\node (v20) at (-0.1,0.3) {$\mathcal{V}_3$};
\node (v27) at (1.7,1.5) {$\mathcal{V}_4$};
\node (v5) at (-3.2,3.6) {};
\node (v7) at (-2.5,2.7) {};
\node (v11) at (-1.4,1.5) {};
\node (v10) at (0.1,2.2) {};
\node (v9) at (1.3,3.6) {};
\node (v8) at (-0.9,3.7) {};
\node (v28) at (1.2,2.9) {};
\node (v29) at (2.2,3) {};
\node (v26) at (0.8,1.9) {};
\node (v30) at (1.8,-0.6) {};
\node (v31) at (2.3,-0.5) {};
\node (v19) at (-1.6,-0.8) {};
\node (v21) at (-1.3,0.4) {};
\node (v22) at (-0.4,1.4) {};
\node (v23) at (0.8,0.8) {};
\node (v24) at (1.3,-0.7) {};
\node (v25) at (0.1,-0.7) {};
\node (v17) at (-2.1,-0.8) {};
\node (v18) at (-3.2,-0.7) {};
\node (v12) at (-3.4,2.1) {};
\node (v14) at (-2.8,2.1) {};
\node (v15) at (-1.9,1) {};
\node (v16) at (-1.9,-0.1) {};
\draw   [->,>=stealth] (v5)-- (v6);
\draw   [->,>=stealth]  (v7) edge (v6);
\draw   [->,>=stealth] (v8) edge (v6);
\draw   [->,>=stealth] (v9) edge (v6);
\draw   [->,>=stealth] (v10) edge (v6);

\draw   [->,>=stealth](v11) edge (v6);
\draw   [->,>=stealth] (v12) edge (v13);
\draw   [->,>=stealth](v14) edge (v13);
\draw  [->,>=stealth] (v15) edge (v13);
\draw  [->,>=stealth] (v16) edge (v13);
\draw   [->,>=stealth] (v17) edge (v13);
\draw   [->,>=stealth] (v18) edge (v13);
\draw  [->,>=stealth] (v19) edge (v20);
\draw  [->,>=stealth] (v21) edge (v20);
\draw  [->,>=stealth] (v22) edge (v20);
\draw  [->,>=stealth] (v23) edge (v20);
\draw   [->,>=stealth] (v24) edge (v20);
\draw   [->,>=stealth](v25) edge (v20);
\draw  [->,>=stealth] (v26) edge (v27);
\draw   [->,>=stealth] (v28) edge (v27);
\draw   [->,>=stealth](v29) edge (v27);
\draw   [->,>=stealth](v30) edge (v27);
\draw  [->,>=stealth] (v31) edge (v27);
\end{tikzpicture}
\end{center}

Our results may suggest that there is a monad behind the theory of entanglement renormalization.\\

$$
\begin{tabular}{|l|c|}
\hline
entanglement renormalization & tensor calculus\\ \hline
isometries & relations of generators/rewriting rules\\\hline
disentangler &coarse-graining\\\hline
scaling invariant & critical \\\hline
\end{tabular}
$$

$\bullet$ The two adjunction $(F,U, \varepsilon, \eta)$ and $(F^T,U^T, \varepsilon^T, \eta^T)$ produce two comonads: one is  $$G=FU:\mathbf{Str.T}\rightarrow \mathbf{Str.T}$$ and the other is  $$D=F^TU^T:\mathbf{T.Sch}^{T}\rightarrow\mathbf{T.Sch}^{T} .$$  The following commutative diagram
$$\xymatrix{\mathbf{Str.T}\ar[r]^{U}\ar[d]^{\Phi}&\ar[r]^{F}\mathbf{T.Sch}\ar@{=}[d]&\ar[d]^{\Phi}\mathbf{Str.T}\\\mathbf{T.Sch}^{T}\ar[r]^{U^T}&\ar[r]^{F^T}\mathbf{T.Sch}&\mathbf{T.Sch}^{T}}$$ shows that $\Phi$ can  induce a morphism of the two resolutions:
$$\xymatrix{\mathcal{V}\ar@{-->}[d]^{\Phi}&&\ar[ll]\ar@{-->}[d]^{\Phi^{\ast}}G^{\ast}(\mathcal{V})\\ \Phi(\mathcal{V})&&\ar[ll](D)^{\ast}(\Phi(\mathcal{V})).}$$
\end{rem}

\section{Appendix}

\subsection{Strict tensor category}
Here we give a brief review of strict tensor category, strict tensor functor and  strict tensor natural transformations.

Roughly speaking, a strict tensor category $\mathcal{C}$  is a category equipped with a bifunctor (called tensor product) $$\otimes:\mathcal{C}\times \mathcal{C}\rightarrow\mathcal{C}$$ and a distinguished object $1_{\mathcal{C}}\in ob(\mathcal{C})$ (called two-sided unit object) satisfying the associativity and unit axioms.

For objects $x,y\in Ob(\mathcal{C})$,
we write $\otimes$ on objects as $$(x,y)\longmapsto x\otimes y$$ and for $f\in Mor_{\mathcal{C}}(x,y)$ and $g\in Mor_{\mathcal{C}}(x',y')$ we write $\otimes$ on morphisms as
$$
\begin{pmatrix}
\xymatrix{x\ar[d]^{f}\\y}\xymatrix{\\,}\xymatrix{x'\ar[d]^{g}\\y'}
\end{pmatrix}\mapsto \begin{matrix} \xymatrix{x\otimes x'\ar[d]^{f\otimes g}\\y\otimes y'}\end{matrix}
$$
or briefly $$(f,g)\mapsto f\otimes g.$$
The functoriality of $\otimes$ tell us  that for arbitrary  $f\in Mor_{\mathcal{C}}(x,y)$, $g\in Mor_{\mathcal{C}}(y,z)$, $f'\in Mor_{\mathcal{C}}(x',y')$ and $g'\in Mor_{\mathcal{C}}(y',z')$, we have
$$
\begin{pmatrix}
\xymatrix{x\ar[d]^{f}\\y\ar[d]^{g}\\z}
\end{pmatrix}
\otimes
\begin{pmatrix}\xymatrix{x'\ar[d]^{f'}\\y'\ar[d]^{g'}\\z'}
\end{pmatrix}= \begin{pmatrix} \xymatrix{x\otimes x'\ar[d]^{f\otimes f'}\\y\otimes y'\ar[d]^{g\otimes g'}\\z\otimes z'}\end{pmatrix}
$$
or for short
$$(g\circ f)\otimes(g'\circ f')=(g\otimes g')\circ(f\otimes f'),$$ which is  called the middle-four-interchange law.
The equations $Id_x\otimes Id_y=Id_{x\otimes y}$ are hold for all $x,y\in Ob(\mathcal{C})$, which are also required by the functoriality of $\otimes$.

The associativity axiom means the following commutative diagram:
$$\xymatrix{\mathcal{C}\times\mathcal{C}\times\mathcal{C}\ar[r]^{I_{\mathcal{C}}\times\otimes}\ar[d]_{\otimes\times I_{\mathcal{C}}}&\mathcal{C}\times\mathcal{C}\ar[d]^{\otimes}\\ \mathcal{C}\times\mathcal{C}\ar[r]^{\otimes}&\mathcal{C}}$$

The unit axiom means the following diagram commutes:
$$\xymatrix{\textbf{1}\times \mathcal{C}\ar[dr]\ar[r]^{\iota\times I_{\mathcal{C}}}&\mathcal{C}\times\mathcal{C}\ar[d]^{\otimes}&\ar[l]_{I_{\mathcal{C}}\times \iota}\mathcal{C}\times \textbf{1}\ar[dl]\\&\mathcal{C}& }$$
where $\textbf{1}$ is the category with one object $1$ and one (identity) arrow $Id_{1}$, $\iota:\textbf{1}\rightarrow \mathcal{C}$ is the obvious functor determined by $1_{\mathcal{C}}\in Ob(\mathcal{C})$, that is, $\iota(1)=1_{\mathcal{C}}$ and $\iota(Id_1)=Id_{1_{\mathcal{C}}}.$

\begin{rem}
For every strict tensor category $(\mathcal{C},\otimes, \circ, 1_{\mathcal{C}})$, we have three associated strict tensor categories: its opposite tensor category $(\mathcal{C},\otimes^{op}, \circ, 1_{\mathcal{C}})$, tensor opposite category $(\mathcal{C},\otimes, \circ^{op}, 1_{\mathcal{C}})$ and opposite tensor opposite category $(\mathcal{C},\otimes^{op}, \circ^{op}, 1_{\mathcal{C}})$, where $X\otimes^{op}Y:=X\otimes Y$ for $X,Y\in Ob(\mathcal{C}^{op})$ and $f\otimes^{op}g:= g\otimes f$ for $f,g\in Mor(\mathcal{C})$ and $f\circ^{op}g=g\circ f$ for composable $f,g\in Mor(\mathcal{C})$.
\end{rem}

\begin{defn}
A strict tensor functor between two strict tensor categories $(\mathcal{C},\otimes_{\mathcal{C}},I_{\mathcal{C}})$ and  $(\mathcal{D},\otimes_{\mathcal{D}},I_{\mathcal{D}})$ is a functor $F:\mathcal{C}\rightarrow \mathcal{D}$ preserving strict tensor structures, that is, the diagrams
$$
\begin{matrix}\xymatrix{\mathcal{C}\times \mathcal{C}\ar[d]_{F\times F}\ar[r]^{\otimes_{\mathcal{C}}}&\ar[d]^{F}\mathcal{C}\\ \mathcal{D}\times \mathcal{D}\ar[r]^{\otimes_{D}}&\mathcal{D}}
&
\xymatrix{\mathcal{C}\ar[rr]^{F}&&\mathcal{D}\\&\textbf{1}\ar[lu]^{\iota_{\mathcal{C}}}\ar[ru]_{\iota_{\mathcal{D}}}&}
\end{matrix}
$$
should commute.
\end{defn}
More concretely, we have $F(1_{\mathcal{C}})=1_{\mathcal{D}}$ and for any $A,B\in Ob(\mathcal{C})$, we have $F(A\otimes_{\mathcal{C}}B)=F(A)\otimes_{\mathcal{D}}F(B).$ For any morphisms $f,g\in Mor(\mathcal{C})$, we have $F(f\otimes_{\mathcal{C}}g)=F(f)\otimes_{\mathcal{D}}F(g).$

\begin{defn}
A natural transformation  $\tau:F\Longrightarrow G$ of two strict tensor functors $F,G:\mathcal{C}\rightarrow\mathcal{D}$ is called a strict tensor natural transformation  if and only if $$\tau_{\otimes_{\mathcal{C}}}=\otimes_{\mathcal{D}}(\tau\times\tau).$$
\end{defn}

That is, $\tau:F\Longrightarrow G$ is a strict tensor natural transformation if and only if  for any $A,B\in Ob(\mathcal{C})$, we have $$\tau_{A\otimes_{\mathcal{C}}B}=\tau_{A}\otimes_{\mathcal{D}}\tau_{B}.$$

\subsection{Progressive plane graph}
Here we briefly review the theory of progressive plane graphs closely following \cite{[JS91]} but with some modification.

\begin{defn}
A \textbf{generalized (topological) graph} $G=G_1\coprod G_0$ is a Hausdorff space $G$  with a decomposition $G_0,G_1$ such that

$(1)$ $G_0 \subset G$ is a discrete closed subset and  not including isolated points of $G$,

$(2)$  the complement of  $G_1=G-G_0$ is a 1-dimensional manifold homeomorphism to an open interval. That is, $G_1$ is the topological sum of open intervals.

\end{defn}

An element $x\in G_0$ is called a \textbf{vertex} or \textbf{node}. The \textbf{degree} of a node $x$ is the number $deg(x)$ of connected components of $V-\{x\}$, where $V$ is a sufficiently small connected neighbourhood of $x$. A connected component $e\subset G_1$  is called an open edge. Each open edge $e$ can be compactified to a \textbf{closed edge} $\hat{e}$ by adjoining two end-points. An edge $e$ is called \textbf{pinned} when the inclusion $e\rightarrow G$ can be extended to a continuous map $\hat{e}\rightarrow G$ (called the \textbf{structure map}) and the image is not homeomorphic to a circle. When the inclusion $e\rightarrow G$ extends only to $\hat{e}$ minus one end-point, we call $e$ \textbf{half-loose} or \textbf{half-pinned}. An edge is $loose$ when it is neither pinned nor half-loose.

\begin{defn}
A \textbf{(topological) graph} is a generalized graph in which all the edges are pinned and every node is contained in the image of some closed edges $\hat{e}$  under the structure map. A \textbf{plane graph} $G$ is a graph with  $G\subset\mathbb{R}^2$.
\end{defn}
\begin{rem}
Our notion of generalized (topological) graph and graph are restricted version  of  Joyal and Street's. The only difference is that we exclude the existence of circles and isolated points in the definition of generalized (topological) graph. Our notion of graph is almost the same as their notion of ordinary graph, except we don't allow the existence of loops in their sense. In fact, our notion of (topological) graph can be identified with those which are geometric realization of combinational pre-graphs (see definition 2.1.1).
\end{rem}
A generalized (topological) graph is called \textbf{finite} if $G_0$ and the set $\pi_0(G_1)$ of connected components of $G_1$ are finite.
Obviously a finite generalized graph is a graph if and only if it is compact.

\begin{defn}
A graph with boundary $G=(G,\partial G)$ is a compact graph $G$ together with a distinguished subset of nodes $\partial G\subset G_0$ such that each $x\in \partial G$ is of degree one.
\end{defn}
The elements of $\partial G$ are called the \textbf{outer nodes} of $(G,\partial G)$ and the elements of $G_0-\partial G$ are called  \textbf{inner nodes} of $(G,\partial G)$. Obviously, the set of inner nodes of $(G,\partial G)$ is same as the set of nodes of $G-\partial G$. When an inner node $x$ is in the image of the structure map $\widehat{e}\rightarrow G$ of  an edge $e$, we call $e$ is incident to $x$ and denote this fact as $e\rightarrow x$. An \textbf{isomorphism} or \textbf{equivalence} $f:(G^1,\partial G^1)\rightarrow (G^2,\partial G^2)$ of graphs with boundary  is a homeomorphism $f:G^1\rightarrow G^2$ inducing bijections on the inner nodes and on the outer nodes.

An \textbf{oriented edge} of $G$ is an edge $e$ equipped with an orientation; or equivalently, with a linear order on $\partial \hat{e}$. The \textbf{source} $e(0)$ of an oriented edge $e$ is the image of the first element of $\partial \hat{e}$ under the structure map $\hat{e}\rightarrow \Gamma$; the \textbf{target} $e(1)$ is the image of the last element. An \textbf{oriented graph} is a graph together with a choice of orientation for each of its edges. For an oriented graph $G$, the \textbf{input} $In(x)$ of an inner node $x\in G_0$ is defined to be the set of oriented edges with target $x$; the \textbf{output}  $Out(x)$ of $x$ is the set of those with source $x$.

A \textbf{polarized graph} is an oriented graph together with a choice of linear order on each $In(x)$ and $Out(x)$.
A \textbf{progressive  graph} is an oriented  graph  with no circuits. The  \textbf{domain} $dom(G)$ of a progressive graph $G$ consists of the edges which have outer nodes as sources; the \textbf{codomain} $cod(G)$ consists of the edges which have outer nodes as targets. An \textbf{anchored graph} is an oriented graph with a choice of linear orders on its domain and codomain.

\begin{defn}
Let $a<b$ be real numbers. A leveled  progressive plane graph between the levels $a$ and $b$ is a graph $G=(G_1,\partial G\subset G_0)$ embedded in  $\mathbb{R}\times [a,b]$ such that

$(1)$ $\partial G=G\cap \mathbb{R}\times \{a,b\}$, and

$(2)$ the second projection $pr_2:\mathbb{R}\times[a,b]\rightarrow [a,b]$ is injective on each connected component of $G_1$.

A progressive plane graph $G$ between levels $a$ and $b$ is written  as $G[a,b]$.
\end{defn}

Each \textbf{leveled progressive plane graph} $G$ is both progressive and polarized in the above subsection. In fact, each edge $e\subset\Gamma_1$ is given the orientation with $pr_2e(0)<pr_2e(1)$. Condition $(2)$ excludes circuits. Also, $In(x)$ and $Out(x)$ can be linearly ordered as follows. Choose $u\in [a,b]$ smaller than but close enough to $pr_2(x)$. Then each edge $e\in In(x)$ intersects the line $\mathbb{R}\times \{u\}$ in one point which is different for different edges. This defines a bijection between $In(x)$ and a subset of $\mathbb{R}\times \{u\}(\cong \mathbb{R})$, and so induces a linear order on $In(x)$. The order on $Out(x)$ is defined similarly by intersecting with $\mathbb{R}\times \{u\}$ for $u$ larger than but close to $pr_2(x)$. The $dom(G)$ and $cod(G)$ are naturally linearly ordered as subsets of  $\mathbb{R}\times \{a\}$ and $\mathbb{R}\times\{b\}$, respectively.

\begin{defn}
For a leveled progressive plane graph $G[a,b]$, a number $u\in[a,b]$ is called a \textbf{regular level} for $G$ if the line $\mathbb{R}\times \{u\}$ contains no inner nodes of $G$, and in this case we write $G[a,b]=G[u,b]\circ G[a,u]$.
\end{defn}
If $c<d$ are regular levels of $G$, the graph $G \cap (\mathbb{R}\times [c,d])$ is written as $G[c,d]$, whose set of inner nodes is $(G_0-\partial \Gamma)\cap(\mathbb{R}\times[c,d])$ and whose set of outer nodes is $G\cap(\mathbb{R}\times\{c,d\})$. The graph $G[c,d]$ is a progressive plane graph between the levels $c$ and $d$; it is called a \textbf{layer} of $G[a,b]$.

When $\Gamma[a,b]$ is the disjoint union of two subgraphs $G^1[a,b]$ and $G^2[a,b]$, we say that the pair $(G^1[a,b],G^2[a,b])$ is a \textbf{tensor decomposition} of $G[a,b]$. Moreover, if there exists a number $\xi$ such that $G^1\subseteq(-\infty,\xi)\times [a,b]$ and $G^2\subseteq(\xi,\infty)\times [a,b]$, we write $G[a,b]=G^1[a,b]\otimes G^2[a,b]$. This notion extends in the obvious way to $n$-fold tensor decompositions $$G=G^1\otimes\cdots\otimes G^n.$$

\begin{prop}
For any progressive plane diagram $G$ between levels $a$ and $b$, there exist regular levels  $a=u_0<u_1<\cdots< u_n=b$ such that each layer $G[u_{i-1}, u_i]$ is elementary for $1\leq i\leq n$.
\end{prop}

\begin{defn}
Let $(G,\partial G)$ denote a graph with boundary. A deformation of leveled progressive plane graphs (between levels a and b) is a continuous function function $$h:G\times [0,1]\rightarrow \mathbb{R}^2$$ such that, for all $t\in[a,b]$, the function $$h(-,t):G\rightarrow \mathbb{R}\times [a(t),b(t)]$$ is an embedding whose image is a leveled progressive plane graph $(G(t),\partial G(t))$ between the levels $a(t)$ and $b(t)$.
\end{defn}
Any deformation of leveled progressive plane graphs gives rise to a unique isomorphism of leveled progressive plane graphs $G(t_1)$ and $G(t_2)$, for $t_1,t_2\in [0,1]$.

\begin{defn}
An isomorphism $f:G[a_1,b_1]\rightarrow G[a_2,b_2]$ of leveled progressive plane graphs is an isomorphism of graphs with boundary such that it can be lifted to a  deformation of leveled progressive plane graphs.
\end{defn}
It is easy to see that an isomorphism  $f$ preserves the orientation of progressive plane graphs, the linear order of every nodes and the linear order of domains and codomains.
We denote the set of isomorphic classes of leveled progressive plane graph by $\mathsf{G}$, and denote the isomorphic class of $G$ by $[G].$

A plane graph $G$ is called  boxed if it is between levels $-1$ and $+1$, and  is contained in $(-1,1)\times[-1,1]$. Write $G:m\rightarrow n$ when $m$, $n$ are the cardinalities of $dom(G)$, $cod(G)$, respectively.
In defining operations on boxed graphs, we use the functions $\gamma,\tau:\mathbb{R}^2\rightarrow \mathbb{R}^2$ defined by  $$\gamma(x,t)=(x,\frac{1}{3}t),\ \ \ \tau(x,t)=(\frac{1}{2}x,t)$$
and the points $e_1=(1,0),\ e_2=(0,1)\in \mathbb{R}^2.$ Notation such as $\gamma(S+e_2),$ for $S\subset \mathbb{R}^2,$ denotes the set $$\{(x,\frac{1}{3}(t+1))\in \mathbb{R}^2|(x,t)\in S\}.$$

The \textbf{tensor product} $G^1\underline{\otimes} G^2$ of two boxed plane graphs $G^1$, $G^2$ is the space $$\tau((G^1-e_1)\sqcup(G^2+e_1))$$ with  $\tau((G^1_0-e_1)\sqcup(G^2_0+e_1))$ as the set of nodes. Ignoring translations, we depict this as following:

$$
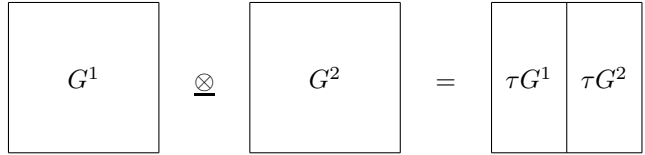
\begin{matrix}\begin{matrix}\begin{tikzpicture}
\node (v1) at (-1,1) {};
\node (v4) at (-1,-1) {};
\node (v2) at (1,1) {};
\node (v3) at (1,-1) {};
\draw  (-1,1)--(1,1);
\draw  (1,1) -- (1,-1);
\draw  (-1,1) -- (-1,-1);
\draw  (-1,-1) -- (1,-1);
\node at (0,0) {$G^1$};
\end{tikzpicture}\end{matrix}&\underline{\otimes}&\begin{matrix}\begin{tikzpicture}
\node (v1) at (-1,1) {};
\node (v4) at (-1,-1) {};
\node (v2) at (1,1) {};
\node (v3) at (1,-1) {};
\draw  (-1,1)--(1,1);
\draw  (1,1) -- (1,-1);
\draw  (-1,1) -- (-1,-1);
\draw  (-1,-1) -- (1,-1);
\node at (0,0) {$G^2$};
\end{tikzpicture}\end{matrix}&=&\begin{matrix}\begin{tikzpicture}
\node (v1) at (-1,1) {};
\node (v4) at (-1,-1) {};
\node (v2) at (1,1) {};
\node (v3) at (1,-1) {};
\draw  (-1,1)--(1,1);
\draw  (1,1) -- (1,-1);
\draw  (-1,1) -- (-1,-1);
\draw  (-1,-1) -- (1,-1);

\node (v5) at (0,1) {};
\node (v6) at (0,-1) {};
\draw  (0,1) edge (0,-1);
\node at (-0.5,0) {$\tau G^1$};
\node at (0.5,0) {$\tau G^2$};
\end{tikzpicture}\end{matrix} \end{matrix}$$

Suppose $G^1:l\rightarrow m$, $G^2:m\rightarrow n$ are boxed plane graphs. Let $a_1<\cdots<a_m$ be the elements of the codomain of $G^1$, and let $I$ be the set of inner nodes of the graph $\gamma(G^1-2e_2).$ Let $b_1<\cdots<b_m$ be the elements of the domain, and let $J$ be the set of inner nodes of the graph $\gamma(G^2+2e_2).$ The \textbf{composition} $G^2\underline{\circ} G^1:l\rightarrow n$ is the plane graph consisting of the space $$G^2\underline{\circ }G^1=\gamma((G^1-2e_2)\sqcup[a_1,b_1]\sqcup\cdots\sqcup[a_m,b_m]\sqcup(G^2+2e_2))$$
with $I\sqcup J$ as the set of inner nodes, where $[a,b]\subset\mathbb{R}^2$ is the segment between the point $a$ and $b$. We depict this as following:

$$
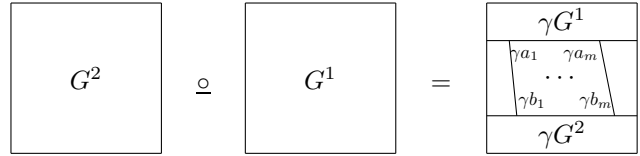
\begin{matrix}\begin{matrix}\begin{tikzpicture}
\node (v1) at (-1,1) {};
\node (v4) at (-1,-1) {};
\node (v2) at (1,1) {};
\node (v3) at (1,-1) {};
\draw  (-1,1)--(1,1);
\draw  (1,1) -- (1,-1);
\draw  (-1,1) -- (-1,-1);
\draw  (-1,-1) -- (1,-1);
\node at (0,0) {$G^2$};
\end{tikzpicture}\end{matrix}&\underline{\circ}&\begin{matrix}\begin{tikzpicture}
\node (v1) at (-1,1) {};
\node (v4) at (-1,-1) {};
\node (v2) at (1,1) {};
\node (v3) at (1,-1) {};
\draw  (-1,1)--(1,1);
\draw  (1,1) -- (1,-1);
\draw  (-1,1) -- (-1,-1);
\draw  (-1,-1) -- (1,-1);
\node at (0,0) {$G^1$};
\end{tikzpicture}\end{matrix}&=&\begin{matrix}\begin{tikzpicture}
\node (v1) at (-1,1) {};
\node (v4) at (-1,-1) {};
\node (v2) at (1,1) {};
\node (v3) at (1,-1) {};
\draw  (-1,1)--(1,1);
\draw  (1,1) -- (1,-1);
\draw  (-1,1) -- (-1,-1);
\draw  (-1,-1) -- (1,-1);

\node (v5) at (-1,0.5) {};
\node (v6) at (1,0.5) {};
\node (v7) at (-1,-0.5) {};
\node (v8) at (1,-0.5) {};
\draw  (-1,0.5) --(1,0.5);
\draw  (-1,-0.5) --(1,-0.5);
\node at (0,0.75) {$\gamma G^1$};
\node at (0,-0.75) {$\gamma G^2$};
\node (v9) at (-0.5,0.5) {};
\node (v10) at (-0.5,-0.5) {};
\node (v11) at (0.5,0.5) {};
\node (v12) at (0,-0.5) {};
\node [scale=0.7] at (-0.5,0.3) {$\gamma a_1$};
\node [scale=0.7]at (-0.4,-0.3) {$\gamma b_1$};
\node [scale=0.7]at (0.25,0.3) {$\gamma a_m$};
\node [scale=0.7]at (0.45,-0.3) {$\gamma b_m$};
\draw  (-0.7,0.5) -- (-0.6,-0.5);
\draw  (0.5,0.5) -- (0.7,-0.5);
\node at (0,0) {$\cdots$};
\end{tikzpicture}\end{matrix} \end{matrix}$$
Note that the layer $(G^2\underline{\circ} G^1)[-\frac{1}{3},\frac{1}{3}]$ has no inner nodes.

Now we introduce two operations $\otimes_{JS}$ and $\circ_{JS}$ on the set $\mathsf{G}$. Any element in $\mathsf{G}$ can be realized as a boxed progressive plane graph. For two elements $[G^1], [G^2]$ in $G$ with $G^1$, $G^2$ being two boxed progressive plane graphs, we define their tensor product $$[G^1]\otimes_{JS}[G^2]\triangleq [G^1\underline{\otimes} G^2].$$
For two elements $[G^1], [G^2]$ in $G$ with $G^1:l\rightarrow m$, $G^2:m\rightarrow n$ being two boxed progressive plane graphs, we define their composition $$[G^2]\circ_{JS}[G^1]\triangleq [G^2\underline{\circ} G^1].$$

For simplicity, we will not make a distinction between a boxed progressive plane graph $G$ and its isomorphic class $[G].$

\subsection{Adjunction, (co)monad and their (co)algebras}

\begin{defn}
For two categories $\mathcal{C}$ and $\mathcal{D}$, an adjunction from $\mathcal{C}$ to $\mathcal{D}$ is a triple $(F, G,\Phi):\mathcal{C}\rightharpoonup\mathcal{D}$, where $F, G$ are functors
$$F:\mathcal{C}\rightleftharpoons\mathcal{D}: G,$$ and $\Phi$ is a natural isomorphism of bifunctors $$\Psi_{-,-}:Mor_{\mathcal{D}}(F(-),-)\rightarrow Mor_{\mathcal{C}}(-, G(-)).$$
\end{defn}

There are also other useful equivalent definitions of an adjunction. Given an adjunction above, we can define two natural transformations $$\begin{matrix}\varepsilon:F\circ G\rightarrow  I_{\mathcal{D}},&&\eta: I_{\mathcal{C}}\rightarrow G\circ F\end{matrix}.$$

Substituting $ G$ into the first variable of $\Psi_{-,-}$, we get
$$\Psi_{ G(-),-}:Mor_{\mathcal{D}}(F\circ G(-),-)\rightarrow Mor_{\mathcal{C}}( G(-), G(-)),$$
then the value of $\varepsilon$ at an object $D\in Ob(\mathcal{D})$ is defined to be $\varepsilon_D=\Psi^{-1}_{ G(D),D}(Id_{ G(D)}).$
Substituting $F$ into the second variable of $\Psi_{-,-}$, we get
$$\Psi_{-,F(-)}:Mor_{\mathcal{D}}(F(-),F(-))\rightarrow Mor_{\mathcal{C}}(-,G\circ F(-)),$$ then the value of $\eta$ at an object $C\in Ob(\mathcal{C})$ is defined to be $\eta_C=\Psi_{C,F(C)}(Id_{F(C)}).$ The naturality of $\varepsilon,\eta$ comes from the naturality of $\Psi$, and the condition that $\Psi$ to be isomorphism implies the following commutative diagrams of natural transformations
$$
\begin{matrix}
\xymatrix{F\circ I_{\mathcal{C}}\ar[r]^{F(\eta)}\ar[d]_{\sim}&\ar[d]^{\varepsilon_{F}}F\circ G\circ F\\F\ar[r]^{\sim}& I_{\mathcal{D}}\circ F}
&&
\xymatrix{ I_{\mathcal{C}}\circ G\ar[r]^{\eta_{ G}}\ar[d]_{\sim}&\ar[d]^{ G(\varepsilon)} G\circ F\circ G\\ G\ar[r]^{\sim}& G\circ I_{\mathcal{D}}}
\end{matrix}
$$

We give another equivalent definition of an adjunction.
\begin{defn}
For two categories $\mathcal{C}$ and $\mathcal{D}$, an adjunction from $\mathcal{C}$ to $\mathcal{D}$ is a quadruple $(F, G,\varepsilon,\eta):\mathcal{C}\rightharpoonup\mathcal{D}$, where $F, G$ are functors
$$F:\mathcal{C}\rightleftharpoons\mathcal{D}: G,$$ and $\varepsilon,\eta$ are natural transformations (called co-unit and unit repectively)
$$\begin{matrix}\varepsilon:F\circ G\rightarrow  I_{\mathcal{D}},&&\eta: I_{\mathcal{C}}\rightarrow G\circ F\end{matrix}$$ satisfying $\varepsilon_{F}\circ F(\eta)=Id_{F}$ and $ G(\varepsilon)\circ\eta_{ G}=Id_{ G},$ where $Id_{F}:F\rightarrow F$ and $Id_{G}:G\rightarrow G$ are the identity natural transformation of $F$ and $G$ respectively.
\end{defn}

\begin{defn}
A monad $(T,\mu,\eta)$ in a category $\mathcal{C}$ consists of a functor $T:\mathcal{C}\rightarrow \mathcal{C}$ and two natural transformations $\mu:T\circ T\rightarrow T$ and $\eta: I_{\mathcal{C}}\rightarrow T,$ which make the following diagrams commute
$$
\begin{matrix}
\xymatrix{\ar[d]^{\mu_{T}}T\circ T\circ T\ar[r]^{T(\mu)}&T\circ T\ar[d]_{\mu}\\\ar[r]^{\mu}T\circ T&T}
&
\xymatrix{ I_{\mathcal{C}}\circ T\ar[r]^{\eta_{T}}&\ar[d]^{\mu}T\circ T&\ar[l]_{T(\eta)}T\circ  I_{\mathcal{C}}\\&\ar@{=}[ur]\ar@{=}[ul]T&}
\end{matrix}
$$
where $I_{\mathcal{C}}:\mathcal{C}\rightarrow \mathcal{C}$ is the identity functor of $\mathcal{C}.$
\end{defn}

\begin{defn}
A comonad $(D,\Delta,\varepsilon)$ in a category $\mathcal{D}$ consists of a functor $D:\mathcal{D}\rightarrow \mathcal{D}$ and two natural transformations $\Delta:D\rightarrow D\circ D$ and $\varepsilon: D\rightarrow I_{\mathcal{D}},$ which make the following diagrams commute
$$
\begin{matrix}
\xymatrix{\ar[d]^{\Delta}D\ar[r]^{\Delta}&D\circ D\ar[d]_{\Delta_D}\\\ar[r]^{D(\Delta)}D\circ D&D\circ D\circ D}
&
\xymatrix{&\ar[d]^{\Delta}D\ar@{=}[dl]\ar@{=}[dr]&\\D&\ar[l]^{D(\Delta)}D\circ D\ar[r]_{\Delta_D}&D}
\end{matrix}
$$
where $I_{\mathcal{D}}:\mathcal{D}\rightarrow \mathcal{D}$ is the identity functor of $\mathcal{D}.$
\end{defn}

Every adjunction $(F, G,\varepsilon,\eta):\mathcal{C}\rightharpoonup\mathcal{D}$ gives rises to a monad $(T,\mu,\eta)$ in $\mathcal{C}$ and a comonad $(D,\Delta,\varepsilon)$ in $\mathcal{D}$. Specifically, the monad is defined as
$$
\left\{
\begin{array}{rcl}
T= G\circ F&:&\mathcal{C}\rightarrow\mathcal{C}\\
\mu= G(\varepsilon_{F})&:& G\circ F\circ G\circ F\rightarrow G\circ F\\
\eta&:& I_{\mathcal{C}}\rightarrow G\circ F
\end{array}
\right.
$$
and the comonad is defined as
$$
\left\{
\begin{array}{rcl}
D=F\circ G&:&\mathcal{D}\rightarrow\mathcal{D}\\
\Delta=F(\eta_{ G})&:&F\circ G\rightarrow F\circ G\circ F\circ G\\
\varepsilon&:&F\circ G\rightarrow  I_{\mathcal{D}}
\end{array}
\right.
$$

The associativity of $\mu$ becomes the commutativity of the diagram $$\xymatrix{ G\circ F\circ G\circ F\circ G\circ  F\ar[rr]^{G\circ F\circ G(\varepsilon_{F})}\ar[d]_{ G(\varepsilon_{F\circ G\circ F})}
&& G\circ F\circ G\circ F\ar[d]^{ G(\varepsilon_{F})}
\\\ar[rr]^{ G(\varepsilon_{F})} G\circ F\circ G\circ F&& G\circ F}$$
which can be deduced from the commutativity of the diagram
$$\xymatrix{F\circ G\circ F\circ G\ar[rr]^{F\circ G(\varepsilon)}\ar[d]_{\varepsilon_{F\circ G}}
&&F\circ G\ar[d]^{\varepsilon}
\\\ar[rr]^{\varepsilon}F\circ G&& I_{\mathcal{D}}}$$
The commutativity of above diagram is just the consequence of naturality of $\varepsilon$ and functoriality of $F\circ G$.
The left and right unitary law of $\eta$  can be deduced from the equations  $\varepsilon_{F}\circ F(\eta)=Id_{F}$ and $ G(\varepsilon)\circ\eta_{ G}=Id_{ G}$ respectively. That is, we have the following commutative diagram:

$$\xymatrix{G\circ F\ar[r]^(0.4){\eta_{G\circ F}}\ar[dr]&\ar[d]^{G(\varepsilon_{F})}G\circ F\circ G\circ F&\ar[l]_(0.3){G\circ F(\eta)}\ar[dl]G\circ F\\&G\circ F&}$$

The co-associativity of $\Delta$ becomes the commutativity of the diagram $$\xymatrix{G\circ F \ar[rr]^{G(\eta_F)}\ar[d]_{ G(\eta_F)}
&& G\circ F\circ G\circ F\ar[d]^{ G\circ F\circ G(\varepsilon_{F})}
\\\ar[rr]^(0.4){ G(\varepsilon_{F\circ G\circ F})} G\circ F\circ G\circ F&& G\circ F\circ G\circ F\circ G\circ  F}$$
which can be deduced from the commutativity of the diagram
$$\xymatrix{ I_{\mathcal{C}} \ar[rr]^{\eta}\ar[d]_{\eta}
&&G\circ F\ar[d]^{G\circ F(\eta)}
\\\ar[rr]^{\eta_{G\circ F}}G\circ F&&G\circ F\circ G\circ F}$$
The commutativity of above diagram is just the consequence of naturality of $\eta$ and functoriality of $G\circ F$.
The left and right unitary law of $\eta$  can be deduced from the equations  $\varepsilon_{F}\circ F(\eta)=Id_{F}$ and $ G(\varepsilon)\circ\eta_{ G}=Id_{ G}$ respectively. That is, we have the following commutative diagram:

$$\xymatrix{&\ar[d]^{F(\eta_G)}F\circ G\ar@{=}[dl]\ar@{=}[dr]&\\F\circ G&\ar[l]^(0.6){F\circ G(\varepsilon)}F\circ G\circ F\circ G\ar[r]_(0.6){\varepsilon_{F\circ G}}&F\circ G}$$

\begin{defn}
An algebra over the monad $(T,\mu,\eta)$ is an object $A$ of $\mathcal{C}$ equipped with a morphism $\gamma:T(A)\rightarrow A$ such that
$$
\begin{matrix}
\xymatrix{T\circ T(A)\ar[r]^{T(\gamma)}\ar[d]_{\mu_A}&T(A)\ar[d]^{\gamma}\\T(A)\ar[r]^{\gamma}&A}
&
\xymatrix{ I_{\mathcal{C}}(A)\ar[r]^{\eta_A}\ar@{=}[dr]&T(A)\ar[d]^{\gamma}\\&A}
\end{matrix}
$$
\end{defn}

\begin{defn}
Let $(A,\gamma_A)$ and $(B,\gamma_B)$ be two algebras over $(T,\mu,\eta)$, a morphism of them is a morphism $f:A\rightarrow B$ in $\mathcal{C}$ such that
$$
\begin{matrix}
\xymatrix{T(A)\ar[r]^{\gamma_A}\ar[d]_{T(f)}&\ar[d]^{f}A\\T(B)\ar[r]^{\gamma_B}&B}
\end{matrix}
$$
\end{defn}

\begin{defn}
A coalgebra over the comonad $(D,\Delta,\varepsilon)$ is an object $C$ of $\mathcal{D}$ equipped with a morphism $\lambda:C\rightarrow D(C)$ such that
$$
\begin{matrix}
\xymatrix{C\ar[r]^{\lambda}\ar[d]^{\lambda}&D(C)\ar[d]^{\Delta_C}\\D(C)\ar[r]^{D(\lambda)}&D\circ D(C)}
&
\xymatrix{C\ar@{=}[dr]\ar[r]^{\lambda}&D(C)\ar[d]^{\varepsilon_D}\\&C}
\end{matrix}
$$
\end{defn}

\begin{defn}
Let $(C_1,\lambda_1)$ and $(C_2,\lambda_2)$ be two coalgebras over $(D,\Delta,\varepsilon)$, a morphism of them is a morphism $g:C_1\rightarrow C_2$ in $\mathcal{D}$ such that
$$
\begin{matrix}
\xymatrix{C_1\ar[d]^{g}\ar[r]^{\lambda_1}&\ar[d]^{D(g)}D(C_1)\\C_2\ar[r]^{\lambda_2}&D(C_2)}

\end{matrix}
$$
\end{defn}


\section{Acknowledgments}

Xuexing Lu want to thank Dr.Liang Kong for valuable advices in preparing for this paper  and  he also would like to thank Professor Ke Wu for encouragement. Yu Ye is supported by NSFC 11431010.





\end{document}